%% file: thesis.tex
\documentclass[10pt,twoside, openright]{report}

\title{Homeomorphisms, homotopy equivalences and chain complexes}
\author{Spiros Adams-Florou}
\date{\today}

\usepackage{graphicx,amssymb,amstext,amsmath,amsthm,cancel,mathrsfs,amsfonts,latexsym, url, palatino, multicol, calc, ifthen, setspace, amscd, enumerate, psfrag, placeins, flafter, hyperref} 
\usepackage[phd,fancyhdr]{edmaths}
\usepackage[all]{xy}

\usepackage{braket} 
\newcommand{\bd}{\mathrm{bd}}
\newcommand{\id}{\mathrm{id}}
\newcommand{\R}{\ensuremath{\mathbb{R}}}

\newcommand{\F}{\mathcal{F}}

\newcommand{\Z}{\mathbb{Z}}

\newcommand{\A}{\mathbb{A}}

\newcommand{\C}{\mathscr{C}}
\newcommand{\rr}{\mathcal{R}}

\newcommand{\ttt}{\mathcal{T}}
\newcommand{\CC}{\mathcal{C}}

\newcommand{\mesh}{\mathrm{mesh}}
\newcommand{\comesh}{\mathrm{comesh}}
\newcommand{\rad}{\mathrm{rad}}
\newcommand{\diam}{\mathrm{diam}}
\newcommand{\ep}{\epsilon}
\newcommand{\rdots}{\mathinner{%
  \mkern1mu\raise1pt\hbox{.}%
  \mkern2mu\raise4pt\hbox{.}%
  \mkern2mu\raise7pt\vbox{\kern7pt\hbox{.}}\mkern1mu}}

\newcommand{\pr}{\mathrm{pr}}

\newcommand{\leqslantup}{\rotatebox[origin=c]{90}{$\leqslant$}}

\newcommand{\brc}[2]{\left\{ \begin{array}{c} #1 \\ #2 \end{array}\right. }
\newcommand{\lrangle}[1]{\langle #1 \rangle }
\newcommand{\brcc}[4]{\left\{ \begin{array}{cc} #1 & #2 \\ #3 & #4 \end{array}\right. }
\newcommand{\brccc}[4]{\left\{\hspace{-3mm} \begin{array}{cc} #1 & #2 \\ #3 & #4 \end{array}\right. }
\newcommand{\matrixcc}[4]{\left( \begin{array}{cc} #1 & #2 \\ #3 & #4 \end{array}\right) }
\newcommand{\gam}[2]{ \Gamma_{#1}(\mathring{#2}) }

\newcommand\FigDiff[1]{Figure~\ref*{#1} on page~\pageref*{#1}}
\newcommand\FigSame[1]{Figure~\ref*{#1}}
\newcommand\Figref[1]{\ifthenelse{\value{page}=\pageref{#1}}
{\FigSame{#1}}{\FigDiff{#1}}}

\theoremstyle{plain}
\newtheorem{introthm}{Theorem}
\newtheorem{thm}{Theorem}[chapter]
\newtheorem{lem}[thm]{Lemma}
\newtheorem{prop}[thm]{Proposition}
\newtheorem{cor}[thm]{Corollary}
\newtheorem{conj}[thm]{Conjecture}
\newtheorem{q}[thm]{Question}
\newtheorem{qq}{Question}

\newtheorem{deflem}[thm]{Definition/Lemma}

\theoremstyle{definition}
\newtheorem{defn}[thm]{Definition}
\newtheorem{ex}[thm]{Example}
\newtheorem{rmk}[thm]{Remark}
\newtheorem{notn}[thm]{Notation}

\begin{document}
%
\maketitle
\newpage



\pagestyle{fancy}
\fancyhead{} 
\renewcommand{\headheight}{28pt}
\cfoot{}
\fancyfoot[LE,RO]{\thepage}

\renewcommand{\chaptermark}[1]{\markboth{\chaptername \ \thechapter.\ #1}{}} 
\renewcommand{\sectionmark}[1]{\markright{\thesection.\ #1} {}}

\fancyhead[LE]{\small \slshape \leftmark}      
\fancyhead[RO]{\small \slshape \rightmark}     
\renewcommand{\headrulewidth}{0.3pt}    


\pagenumbering{roman}
\setcounter{page}{1}

\onehalfspacing

\addcontentsline{toc}{chapter}{Abstract}
\include{abstract}
\newpage
\addcontentsline{toc}{chapter}{Declaration}
\include{declaration}

\newpage
\addcontentsline{toc}{chapter}{Acknowledgements}
\include{acknowledgements}

\singlespacing

\newpage
\addcontentsline{toc}{chapter}{Contents}
\tableofcontents

\newpage
\addcontentsline{toc}{chapter}{List of symbols}
\include{listofsymbols}

\newpage
\addcontentsline{toc}{chapter}{Conventions}
\include{conventions}

\onehalfspacing
\pagenumbering{arabic}
\setcounter{page}{1}

\include{introduction}
\include{Thesis1}
\include{Thesis2}

\renewcommand{\chaptername}[1]{Appendix A. }
\renewcommand{\chaptermark}[1]{\markboth{\chaptername \ #1}{}} 

\include{appendices}

\small{
\singlespacing

\addcontentsline{toc}{chapter}{Bibliography}
\bibliographystyle{amsalpha}
\bibliography{spirosbib}{}

}
\end{document}

%% file: abstract.tex
\chapter*{Abstract}
\noindent 
This thesis concerns the relationship between bounded and controlled topology and in particular how these can be used to recognise which homotopy equivalences of reasonable topological spaces are homotopic to homeomorphisms.  

Let $f:X\to Y$ be a simplicial map of finite-dimensional locally finite simplicial complexes. Our first result is that $f$ has contractible point inverses if and only if it is an $\ep$-controlled homotopy equivalences for all $\ep>0$, if and only if $f\times\id:X\times\R \to Y\times\R$ is a homotopy equivalence bounded over the open cone $O(Y^+)$ of Pedersen and Weibel. The most difficult part, the passage from contractible point inverses to bounded over $O(Y^+)$ is proven using a new construction for a finite dimensional locally finite simplicial complex $X$, which we call the \textit{fundamental $\ep$-subdivision cellulation} $X_\ep^\prime$. 

This whole approach can be generalised to algebra using geometric categories. In the second part of the thesis we again work over a finite-dimensional locally finite simplicial complex $X$, and use the $X$-controlled categories $\A^*(X)$, $\A_*(X)$ of Ranicki and Weiss (1990) together with the bounded categories $\CC_M(\A)$ of Pedersen and Weibel (1989). Analogous to the barycentric subdivision of a simplicial complex, we define the algebraic barycentric subdivision of a chain complex over that simplicial complex. The main theorem of the thesis is then that a chain complex $C$ is chain contractible in $\brc{\A^*(X)}{\A_*(X)}$ if and only if $\textit{``}C\otimes\Z\textit{''}\in\brc{\A^*(X\times\R)}{\A_*(X\times\R)}$ is boundedly chain contractible when measured in $O(X^+)$ for a functor $\textit{``}-\otimes\Z\textit{''}$ defined appropriately using algebraic subdivision. In the process we prove a squeezing result: a chain complex with a sufficiently small chain contraction has arbitrarily small chain contractions.

The last part of the thesis draws some consequences for recognising homology manifolds in the homotopy types of Poincar\'{e} Duality spaces. Squeezing tells us that a $PL$ Poincar\'{e} duality space with sufficiently controlled Poincar\'{e} duality is necessarily a homology manifold and the main theorem tells us that a $PL$ Poincar\'{e} duality space $X$ is a homology manifold if and only if $X\times\R$ has bounded Poincar\'{e} duality when measured in the open cone $O(X^+)$.
\vspace{10mm}
\normalsize

%% file: declaration.tex
\chapter*{Declaration}

\normalsize
I do hereby declare that this thesis was composed by myself and that
the work described within is my own, except where explicitly stated otherwise.

\vspace{20mm}
\hfill {\it Spiros Adams-Florou}

\hfill March 2012

%% file: acknowledgements.tex
\chapter*{Acknowledgements}
\noindent
\normalsize
First and foremost, I would like to thank my supervisor Andrew Ranicki. I am extremely grateful for all the help and advice he has given me over the last three and a half years, in particular for suggesting my research project and giving me the intuition necessary to solve it. Andrew is a very attentive supervisor; he is always available if you have a question and always seems to know the answer or at least where to find it!

I would like to thank both Wolfgang L\"{u}ck and Erik Pedersen for kindly arranging for me to visit the Universities of M\"{u}nster and Copenhagen during the autumns of 2008 and 2010 respectively. In both universities I was made to feel very welcome and enjoyed a stimulating research environment. I'm especially thankful to Erik for the many fruitful discussions I had with him during my stay.

I would also like to take this opportunity to thank Des Sheiham's parents Ivan and Teresa for their generous contributions towards my travel expenses for my trips to M\"{u}nster and Copenhagen. This was greatly appreciated, particularly as I found it difficult to secure funding for my trip to Copenhagen. 

I would like to show my gratitude to my friends and fellow graduate students, above all Mark Powell, Paul Reynolds, Patrick Orson and Julia Collins, from whom I have learned a lot during our many discussions together. I am grateful to all my friends who proofread part of my thesis and particularly Mark for making some very useful suggestions. Many thanks go to my officemates Ciaran Meachan, Jesus Martinez-Garcia and briefly Alexandre Martin for making our office a productive yet fun place to work; I have enjoyed the many hours spent at our blackboard discussing ideas. 

Last, but not least, I would like to thank my parents and my brother for all their love and for always supporting me in everything I have done.

%% file: listofsymbols.tex
\chapter*{List of symbols}
\noindent
\normalsize
\begin{tabular}{cll}
Symbol & Meaning & Reference \\ 
\hline
\\
$*$ & A point & \\
$\sim$ & An equivalence relation or a homotopy & \\
$\simeq$ & Homotopy equivalence or chain equivalence & \\ 
$\cong$ & (PL) homeomorphism or chain isomorphism & \\
$\A$ & An additive category & \\
$\A^*(X)$, $\A_*(X)$ & The $X$-controlled categories of Ranicki and Weiss & \ref{geometriccategories} \\
$\A(X)$ & Shorthand for either $\A^*(X)$ or $\A_*(X)$ & \ref{eithercategory} \\
bd$(f)$ & The bound of a controlled or bounded map $f$ & \ref{definebd} \\
& The bound of a chain map or chain equivalence $f$ & \ref{bdf2} \\
$B_\ep(m)$ & The open ball of radius $\ep$ around the point $m$ & \ref{openball} \\
$\partial\overline{B_\ep(m)}$ & The sphere of radius $\ep$ around the point $m$ & \ref{openball} \\
$\C(f)$ & The algebraic mapping cone of the chain map $f$ & \ref{algmapcone} \\
$\CC_M(\A)$ & The bounded geometric category of Pedersen and Weibel & \ref{bddcat} \\
$ch(\A)$ & The category of finite chain complexes in $\A$ and chain maps & \ref{definech} \\
comesh$(X)$ & The comesh size of the simplicial complex $X$ & \ref{meshandco} \\
$C(\sigma)$ & The part of the chain complex $C$ over $\mathring{\sigma}$ & \ref{geometriccategories} \\
$\partial \sigma$ & The boundary of the closed simplex $\sigma$ & \ref{interior} \\
$\partial X$ & The boundary of $X$ & \\
$\Delta^n$ & The standard $n$-simplex in $\R^{n+1}$ & \ref{standardsimplex} \\
$\Delta^{lf}_*(X)$ & The locally finite simplicial chain complex of $X$ & \ref{locallyfinitechains} \\
$\Delta^{-*}(X)$ & The simplicial cochain complex of $X$ & \\
diam$(\sigma)$ & The diameter of the simplex $\sigma$ & \ref{diam} \\
$D(\sigma,X)$ & The closed dual cell of $\sigma$ in $X$ & \ref{dualcells} \\
$\mathring{D}(\sigma,X)$ & The open dual cell of $\sigma$ in $X$ & \ref{dualcells} \\
$d_X$ & The boundary map of the simplicial chain complex of $X$ & \\
$\delta^X$ & The boundary map of the simplicial cochain complex of $X$ & \\ 
$d_{\tau\sigma}$ & Shorthand for $(d_C)_{\tau\sigma}$ & \ref{partsofs} \\
$d_{\rho_0\ldots\rho_i}$ & Shorthand for $d_{\rho_0\rho_1}d_{\rho_1\rho_2}\ldots d_{\rho_{i-1}\rho_i}$ & \ref{partsofs} \\
$f_{\tau,\sigma}$ or $f_{\tau\sigma}$ & The component of the morphism $f$ from $\sigma$ to $\tau$ & \ref{geometriccategories} \\
$\F(R)$ & The category of finitely generated free $R$-modules & \\
$f_{[Y^\prime],[Y]}$ & The assembled morphism from $Y$ to $Y^\prime$ & \ref{assembleobjects} \\
$\mathbb{G}_X(\A)$ & The $X$-graded geometric category & \ref{geometriccategories} \\
$\Gamma$ & A PL map $X \times Sd\, X \to X_\ep^\prime$ or $X \times Sd\, X \to Sd\, X$ used to & \ref{constructgam} \\
& decompose its image into the homotopies $\Gamma_{\sigma_0,\ldots,\sigma_i}(\rho)$ & \\
$\Gamma_{\sigma_0,\ldots,\sigma_i}$ & $\Gamma$ restricted to $\sigma_i\times\widehat{\sigma_0\ldots\sigma_i}$ & \ref{constructgam} \\
$\Gamma_{\sigma_0,\ldots,\sigma_i}(\rho)$ & Higher homotopies decomposing either $X_\ep^\prime$ or $Sd\, X$ & \ref{Defn:fundamentalcellulation} \\
$\gam{\sigma_0,\ldots,\sigma_i}{\rho}$ & Higher homotopies decomposing either $X_\ep^\prime$ or $Sd\, X$ & \ref{gamp}, \ref{gamp2}\\
$\mathcal{I}$ & The inclusion functor $\A(X\times\R) \to \CC_{O(X^+)}(\A)$ for $X$ a finite & \ref{finiteinclusionfunctor} \\ 
& simplicial complex & \\
$I_\sigma$ & The union of all open simplices in $Sd\, X$ sent by $r$ to $\mathring{\sigma}$ & \ref{Isigma} \\ 
$I^\infty$ & The Hilbert cube & \\
$j_X$ & The coning map $X\times\R \to O(X^+)$ & \ref{coningmap} \\
\end{tabular}
\newpage
\noindent
\hspace{-3mm}\begin{tabular}{cll}
$K(\sigma)$ & The simplicial complex such that $f^{-1}(\mathring{\sigma}) \cong \mathring{\sigma} \times K(\sigma)$ & \ref{Ksigmacontractible} \\
mesh$(X)$ & The mesh size of the simplicial complex $X$ & \ref{meshandco} \\
$M[Y]$ & The assembly of the object $M$ over the collection of simplices $Y$ & \ref{assembleobjects} \\
$N_\ep(\sigma)$ & The union of all simplices in $X$ containing a point within $\ep$ of $\sigma$ & \ref{keylemma} \\
$O(M^+)$ & The open cone of a metric space $(M,d)$ & \ref{openconedefn} \\
$o(\sigma)$ & The incentre of the simplex $\sigma$ & \ref{diam} \\
$P$ & A chain homotopy $s \circ t \sim \id$ & \ref{chapfive} \\
$P^*$ & The dual chain homotopy on simplicial cochains & \\
$R$ & A ring with identity & \\
$\rr$ & The covariant assembly functor assembling subdivided simplices & \ref{covariantassembly} \\
$\rr_r$ & The covariant assembly functor assembling the $I_\sigma$ defined by $r$ & \ref{rrr} \\
$\rr_i$ & The covariant assembly functor obtained by viewing $t^{j-i}(X\times\R)$ & \ref{otherformofexptranseq} \\
& as a subdivision of $t^j(X\times\R)$ & \\
$r$ & A simplicial approximation to the identity or chain inverse to $s$ & \ref{chapfive} \\
$r_*$ & The map $Sd_r\, C \to C$ induced by $r$ & \ref{algsubdiv} \\
$(r_*)_{\tau,\widetilde{\sigma},n}$ & The component of $r_*$ from $Sd\, C_n(\widetilde{\sigma})$ to $C_n(\tau)$ for $C\in \A^*(X)$ & \ref{partsofr} \\
$r_\sigma$ & The restriction of $r$ to a map $\Delta_*(I_\sigma) \to \Delta_*(\mathring{\sigma})$ & \\ 
$r^*$ & The dual map on simplicial cochains & \\
rad$(\sigma)$ & The radius of the simplex $\sigma$ & \ref{diam} \\
$s$ & The subdivision chain equivalence on simplicial chains induced & \ref{chapfive} \\
& by the identity map & \\
$s_*$ & The map $C \to Sd\, C$ induced by $s$ & \ref{algsubdiv} \\ 
$(s_*)_{\gam{\rho_0,\ldots,\rho_i}{\rho},\sigma,n}$ & Component of $s_*$ from $C(\sigma)_n$ to $Sd\, C_n [\gam{\rho_0,\ldots,\rho_i}{\rho}]$ for $C\in \A^*(X)$ & \ref{partsofs} \\
$s_\sigma$ & The restriction of $s$ to a map $\Delta_*(\mathring{\sigma}) \to \Delta_*(I_\sigma)$ & \\
$s^*$ & The dual map on simplicial cochains & \\
$\Sigma C$ & The suspension of the chain complex $C$ & \\
$\widetilde{Sd}_r$ & The algebraic subdivision functor $\A(X) \to ch(\A(Sd\, X))$ & \ref{Deflem:SdtildeFunctor} \\
$Sd\, C$ & The algebraic subdivision of a chain complex $C$ & \ref{Deflem:SdFunctor} \\
$Sd_r\,$ & The algebraic subdivision obtained from the chain inverse $r$ & \ref{Deflem:SdFunctor} \\ 
$Sd\, \sigma$ & The barycentric subdivision of the simplex $\sigma$ & \ref{barycentricsubdivision} \\
$Sd^i\, \sigma$ & The $i^{th}$ iterated barycentric subdivision of the simplex $\sigma$ & \ref{barycentricsubdivision} \\
$Sd\, X$ & The barycentric subdivision of the simplicial complex $X$& \ref{barycentricsubdivisionX} \\
$Sd^i\, X$ & The $i^{th}$ iterated barycentric subdivision of $X$ & \ref{barycentricsubdivisionX} \\
$Sd_i$ & The algebraic subdivision functor from viewing $t^{j-i}(X\times\R)$ as a & \ref{exptranseq} \\
& subdivision of $t^j(X\times\R)$ & \\
$\mathring{\sigma}$ & The open simplex that is the interior of the closed simplex $\sigma$ & \ref{interior} \\
$\widehat{\sigma}$ & The barycentre of the simplex $\sigma$ & \ref{barycentre} \\
$[\sigma]$ & An orientation of the simplex $\sigma$ & \ref{simporient} \\
$[\sigma]\times[\tau]$ & The product orientation of two oriented simplices & \ref{prodorient} \\ 
Supp$(M)$ & The support of $M$ & \ref{definesupports} \\
$t^{-1}$ & Exponential translation & \ref{exptrans} \\
$t^i(X\times\R)$ & Bounded triangulations of $X\times\R$ & \ref{exptrans} \\
$T$ & A (spanning) forest & \ref{forest} \\
$T_v$ & A (spanning) tree on $r^{-1}(v)$ & \ref{forest} \\
$\ttt$ & The contravariant assembly functor assembling open dual cells & \ref{contravariantassembly} \\
$\mathbb{T}(\tau\times\sigma)$ & The set of consistently oriented triangulations of the cell $\tau\times\sigma$ & \ref{setT} \\
$v_0\ldots v_n$ & The simplex spanned by the vertices $\{v_0,\ldots, v_n\}$ & \ref{simplexasjoins} \\
$v_0\ldots\widehat{v_i}\ldots v_n$ & The face of $v_0\ldots v_n$ with the vertex $v_i$ omitted & \ref{face} \\
$[v_0,\ldots,v_n]$ & An orientation of the simplex $v_0\ldots v_n$ & \ref{simporient} \\
$V(\sigma)$ & The set of vertices in the simplex $\sigma$ & \ref{vertexset} \\
$[X]$ & The fundamental class of a Poincar\'{e} duality space $X$ & \ref{PDstuff} \\
$[X]_x$ & The image of $[X]$ under the map $\Delta_*(X) \to \Delta_*(X,X\backslash x)$ & \\ 
$||X||$ & The geometric realisation of the simplicial complex $X$ & \ref{geometricrealisation} \\
$|X_i|$ & The cardinality of the set $X_i$ & \ref{surjectivesimplicialmapssimplextosimplex} \\
$X_\ep^\prime$ & The fundamental $\ep$-subdivision cellulation of $X$ & \ref{Defn:fundamentalcellulation} \\
$``-\otimes\Z''$ & Algebraically crossing with $\Z\subset\R$ & \ref{algcrosswithR} \\
\end{tabular}

%% file: conventions.tex
\chapter*{Conventions} \label{conventions}
\begin{enumerate}
 \item In general $\rho$, $\tau$, $\sigma$ will be closed simplices in $X$ and $\widetilde{\rho}$, $\widetilde{\tau}$, $\widetilde{\sigma}$ will always be closed simplices in some subdivision of $X$. 
 \item We will favour talking about collections of open simplices as opposed to collections of closed simplices rel their boundary. In particular we choose to write $\Delta_*(\mathring{\sigma})$ in place of $\Delta_*(\sigma,\partial\sigma)$. This convention is taken to reflect the fact that a simplicial complex can be written as a disjoint union of all its open simplices.
 \item We use the following notation for a chain equivalence
\begin{displaymath}
\xymatrix{ (C, d_C, \Gamma_C) \ar@<0.5ex>[r]^{f} & (D,d_D, \Gamma_D) \ar@<0.5ex>[l]^{g}
} 
\end{displaymath}
to mean 
\begin{eqnarray*}
 d_C\Gamma_C + \Gamma_Cd_C &=& 1 - g\circ f \\
 d_D\Gamma_D + \Gamma_Dd_D &=& 1 - f\circ g. 
\end{eqnarray*}

\end{enumerate}

\cleardoublepage

%% file: introduction.tex
\chapter{Introduction}

Homotopy equivalences are not in general hereditary, meaning that they do not restrict to homotopy equivalences on preimages of open sets. For nice spaces, a map that is close to a homeomorphism is a hereditary homotopy equivalence, but what about the converse - when is a hereditary homotopy equivalence close to a homeomorphism? What local conditions when imposed on a map guarantee global consequences? In 1927 Vietoris proved the Vietoris mapping Theorem in \cite{Vietmapthm}, which can be stated as: for a surjective map $f:X\to Y$ between compact metric spaces, if $f^{-1}(y)$ is acyclic\footnote{Meaning that $\widetilde{H}_*(f^{-1}(y))=0$.} for all $y\in Y$, then $f$ induces isomorphisms on homology.\footnote{Originally this was for augmented Vietoris homology mod $2$, but was subsequently extended to more general coefficient rings by Begle in \cite{Begle1} and \cite{Begle2}.}

Certainly, a homeomorphism satisfies the hypothesis of this theorem, but to what extent does the converse hold? Under what conditions is a map homotopic to a homeomorphism or at least to a map with acyclic point inverses? Many people have studied surjective maps with point inverses that are well behaved in some sense, whether they be contractible, acyclic, cell-like etc. The idea is to weaken the condition of a map being a homeomorphism where all the point inverses are precisely points to the condition where they merely have the homotopy or homology of points.

A great success of this approach for manifolds was the \textit{CE approximation Theorem} of Siebenmann showing that the set of cell-like maps $M\to N$, for $M,N$ closed $n$-manifolds with $n\neq 4$, is precisely the closure of the set of homeomorphisms $M\to N$ in the space of maps $M\to N$.\footnote{Provided in dimension $3$ that $M$ contains no fake cubes.} 

The approach of controlled topology, developed by Chapman, Ferry and Quinn, is to have each space equipped with a control map to a metric space with which we are able to measure distances. Typical theorems involve a concept called squeezing, where one shows that if the input of some geometric obstruction is sufficiently small measured in the metric space, then it can be `squeezed' arbitrarily small. Chapman and Ferry improved upon the \textit{CE approximation Theorem} in dimensions $5$ and above with their \textit{$\alpha$-approximation Theorem}. Phrased in terms of a metric, it says that for any closed metric topological $n$-manifold $N$ with $n\geqslant 5$ and for all $\delta>0$, there exists an $\ep$ such that any map $f:M\to N$ with point inverses smaller than $\ep$ is homotopic to a homeomorphism through maps with point inverses smaller than $\delta$.

The approach of bounded topology is again to have a control map, but this time necessarily to a metric space $M$ without finite diameter. Rather than focus on how small the control is, bounded topology only requires that the control is finite. An advantage of this perspective is that, unlike controlled topology, it is functorial: the sum of two finite bounds is still finite whereas the sum of two values less than $\ep$ need not be less than $\ep$. 

Since the advent of controlled and bounded topology people have been studying the relationship between the two. Pedersen and Weibel introduced the open cone $O(X)\subset \R^{n+1}$ of $X$ a subset of $S^n$. This can be viewed as the union of all the rays in $\R^{n+1}$ out of the origin through points of $X$. More generally $O(M^+)$ can be defined for a metric space $M$, which can be viewed roughly as $M\times\R$ with a metric so that $M\times\{t\}$ is $t$ times as big as $M\times\{1\}$ for $t\geqslant 0$ and $0$ times as big for $t\leqslant 0$.\footnote{See Definition \ref{openconedefn} for a precise definition.} Pedersen and Weibel noted that bounded objects over the open cone behave like local objects over $M$. The advantage to working over the open cone is that lots of rather fiddly local conditions to check at all the points of a space are replaced by a single global condition. This is preferable, particularly when this global condition can be checked with algebra by verifying whether a single chain complex is contractible or not. 

Ferry and Pedersen studied bounded Poincar\'{e} duality spaces over $O(M^+)$ and stated in a footnote on pages $168-169$ of \cite{epsurgthy} that 
\begin{quotation}
``It is easy to see that if $Z$ is a Poincar\'{e} duality space with a map $Z \to K$ such that $Z$ has $\ep$-Poincar\'{e} duality for all $\ep>0$ when measured in $K$ (after subdivision), e.g. a homology manifold, then $Z\times\R$ is an $O(K_+)$-bounded Poincar\'{e} complex. The converse (while true) will not concern us here.'' 
\end{quotation}
This footnote is proved in this thesis for finite-dimensional locally finite simplicial complexes. 

We prove both topological and algebraic Vietoris-like theorems for these spaces together with their converses and as a corollary prove Ferry and Pedersen's conjecture for such spaces. In the process we prove a squeezing result for these spaces with consequences for Poincar\'{e} duality: if a Poincar\'{e} duality space has $\ep$-Poincar\'{e} duality for $\ep$ sufficiently small, then it has arbitrarily small Poincar\'{e} duality and is a homology manifold. Explicit values for $\ep$ are computed.

\vspace{5mm}
In this thesis we work exclusively with simplicial maps between finite-dimensional locally finite simplicial complexes. Such an $X$ naturally has a complete metric\footnote{Given by taking the path metric whose restriction to each $n$-simplex is the subspace metric from the standard embedding into $\R^{n+1}$.} so in order to study $X$ with controlled or bounded topology we need only take $\id:X\to X$ as our control map. With respect to this prescribed control map, $X$ automatically has both a tame and a bounded triangulation: $0< \comesh (X) < \mesh(X) < \infty$. Here $\mesh(X)$ denotes the bound of the largest simplex diameter (c.f. Definition \ref{diam}); having a finite mesh means having a bounded triangulation. Having $\comesh(X)$ non-zero is the dual notion to a bounded triangulation; rather than having each simplex contained inside a ball of uniformly bounded diameter we have each simplex (other than $0$-simplices) containing a ball of uniformly non-zero diameter. 

The first half of the thesis is concerned with proving the following topological Vietoris-like Theorem and its converse:
\begin{introthm}\label{abv}
If $f:X\to Y$ is a simplicial map between finite-dimensional locally finite simplicial complexes $X$,$Y$, then the following are equivalent:
\begin{enumerate}
 \item $f$ has contractible point inverses,
 \item $f$ is an $\ep$-controlled homotopy equivalence measured in $Y$ for all $\ep>0$,
 \item $f\times\id_\R: X\times\R \to Y\times\R$ is an $O(Y^+)$-bounded homotopy equivalence.
\end{enumerate}
\end{introthm}
\noindent Conditions $(1)$ and $(2)$ being equivalent is essentially well known for the case of finite simplicial complexes; see Proposition 2.18 of Jahren and Rognes' paper \cite{plmf} for example. The equivalence of conditions $(2)$ and $(3)$ is a new result, inspired by the work of Ferry and Pedersen.

The way we prove Theorem \ref{abv} is to first characterise surjective simplicial maps $f:X\to Y$ between finite-dimensional locally finite simplicial complexes. The preimage of each open simplex $\mathring{\sigma}\in Y$ is PL homeomorphic to the product $\mathring{\sigma}\times K(\sigma)$ for some finite-dimensional locally finite simplicial complex $K(\sigma)$. In fact the map $f$ is contractible if and only if $K(\sigma)$ is contractible for all $\sigma\in Y$. 

This characterisation means that we can always locally define a section of $f:X\to Y$ over each open simplex. Contractibility of $f$ allows us to patch these local sections together provided we stretch a neighbourhood of $\partial \sigma$ in $\mathring{\sigma}$ for all $\sigma$. We accomplish this with a new construction: the fundamental $\ep$-subdivision cellulation of $X$ which we denote by $X_\ep^\prime$. This cellulation is a subdivision of $X$ that is PL homotopic to the given triangulation on $X$ by a homotopy with tracks of length at most $\ep$. This cellulation is related to the homotopy between a simplicial complex $X$ and its barycentric subdivision $Sd\, X$ provided by cross sections of a prism triangulation of $X\times I$ from $X$ to $Sd\, X$. Patching using $X_\ep^\prime$ gives an $\ep$-controlled homotopy inverse to $f$, but since we have a continuous family of cellulations $X_\ep^\prime$ we get a continuous family of homotopy inverses, i.e.\ we get that $f\times\id_\R$ is an $O(Y^+)$-bounded homotopy equivalence. 

The other implications are comparatively straightforward, especially given the aforementioned characterisation of surjective simplicial maps. An $O(Y^+)$-bounded homotopy equivalence $f\times\id_\R: X\times\R \to Y\times\R$ with bound $B$ can be `sliced' at height $t$ to get an $\ep$-controlled homotopy equivalence $X\to Y$ with control approximately $B/t$. Then, if $f$ is an $\ep$-controlled homotopy equivalence for all $\ep$, we first show that $f$ must be surjective. Then we can use our characterisation to say $f^{-1}(\mathring{\sigma})\cong \mathring{\sigma}\times K(\sigma)$ for all $\sigma\in Y$. Let $g_\ep$ be an $\ep$-controlled homotopy inverse for $\ep$ small enough, then $f\circ g_\ep \simeq \id_X$ provides contractions $K(\sigma)\simeq 0$ for all $\sigma$, proving that $f$ is contractible.

\vspace{5mm}
In the second half of this thesis we develop algebra to prove an algebraic analogue of Theorem \ref{abv}, with an application to Poincar\'{e} duality in mind. Let $\A$ be an additive category and let $ch(\A)$ denote the category of finite chain complexes of objects and morphisms in $\A$ together with chain maps. The simplicial structure of the simplicial complex is very important and needs to be reflected in the algebra. For this we use Quinn's notion of geometric modules, or more generally geometric categories. 

Given a metric space $(X,d)$ and an additive category $\A$, a \textit{geometric category over $X$} has objects that are collections $\{M(x)\,|\,x\in X \}$ of objects in $\A$ indexed by points of $X$, written as a direct sum \[M = \sum_{x\in X} {M(x)}.\] A morphism \[f = \{f_{y,x}\}: L = \sum_{x\in X}L(x) \to M = \sum_{y\in X}M(y) \] is a collection $\{f_{y,x}:L(x) \to M(y)\,|\,x,y\in X\}$ of morphisms in $\A$ where in order to be able to compose morphisms we stipulate that 
$\{y\in X\,|\, f_{y,x}\neq 0\}$ is finite for all $x\in X$.

We consider several different types of geometric categories for a simplicial complex $X$:
\begin{enumerate}[(i)]
 \item The $X$-graded category $\mathbb{G}_X(\A)$, whose objects are graded by the (barycentres of) simplices $X$ and morphisms are not restricted (except to guarantee we can compose). We think of morphisms as matrices.
 \item The categories $\brc{\A^*(X)}{\A_*(X)}$ of Ranicki and Weiss \cite{ranickiweiss}, whose objects are also graded by the (barycentres of) simplices of $X$. Components of morphisms can only be non-zero from $\sigma$ to $\tau$ if $\brc{\tau\leqslant\sigma}{\sigma\leqslant\tau}$ where $\sigma\leqslant\tau$ means that $\sigma$ is a subsimplex of $\tau$. We think of morphisms as triangular matrices.
 \item The bounded category $\CC_{O(X^+)}(\A)$ of Pedersen and Weibel, whose objects are graded by points in $O(X^+)$ in a locally finite way. For each morphism $f$ there is a $k(f)$, called the \textit{bound of $f$}, such that components of the morphism must be zero between points further than $k(f)$ apart. We think of morphisms as band matrices.
\end{enumerate}

These are fairly natural categories to consider since the simplicial chain and cochain complexes, $\Delta^{lf}_*(X)$ and $\Delta^{-*}(X)$, are naturally chain complexes in $\A^*(X)$ and $\A_*(X)$ respectively, and for a Poincar\'{e} duality simplicial complex $X$ the Poicar\'{e} duality chain equivalence is a chain map in $\mathbb{G}_X(\A)$. 

Let $\A(X)$ denote $\A^*(X)$ or $\A_*(X)$. The algebraic analogue of Theorem \ref{abv} is the following algebraic Vietoris-like Theorem and its converse:
\begin{introthm}\label{above}
If $X$ is a finite-dimensional locally finite simplicial complex and $C$ is a chain complex in $\A(X)$, then the following are equivalent:
\begin{enumerate}
 \item $C(\sigma)\simeq 0 \in \A$ for all $\sigma \in X$, i.e.\ $C$ is locally contractible over each simplex in $X$, 
 \item $C\simeq 0 \in \A(X)$, i.e.\ $C$ is globally contractible over $X$,
 \item $\textit{``}C\otimes\Z\textit{''}\simeq 0\in\mathbb{G}_{X\times\R}(\A)$ with finite bound measured in $O(X^+)$.
\end{enumerate}
\end{introthm}

For $X$ a Poincar\'{e} duality space, we can apply this to the algebraic mapping cone of a Poincar\'{e} duality chain equivalence to get the following consequence for Poincar\'{e} duality:

\begin{introthm}\label{cor}
If $X$ is a finite-dimensional locally finite simplicial complex, then the following are equivalent:
\begin{enumerate}
 \item $X$ has $\ep$-controlled Poincar\'{e} duality for all $\ep>0$ measured in $X$. 
 \item $X\times\R$ has bounded Poincar\'{e} duality measured in $O(X^+)$. 
\end{enumerate}
\end{introthm}
In particular, condition $(1)$ is equivalent to $X$ being a homology manifold (\cite{RanSingularities}), so this gives us a new way to detect homology manifolds.

\vspace{5mm}
The way we prove Theorem \ref{above} is as follows. A simplicial chain complex $X$ can be barycentrically subdivided. Denote this subdivision by $Sd\, X$. This procedure induces a subdivision chain equivalence on simplicial chains $s_*:\Delta^{lf}_*(X) \to \Delta^{lf}_*(Sd\, X)$. We can view $\Delta^{lf}_*(Sd\, X)$ as a chain complex in $\A^*(X)$ by \textit{assembling over the subdivided simplices}, that is we associate $\Delta_*(Sd\, \mathring{\sigma})$ to $\sigma$ for all $\sigma\in X$. In fact $s_*$ turns out to also be a chain equivalence in $\A^*(X)$. The key point is that with simplicial complexes we can subdivide and reassemble them and the effect this has on the simplicial chain complex is to give you a chain complex chain equivalent to the one you started with. 

Modelled on the effect that barycentric subdivision has on the simplicial chain and cochain complexes, it is possible to define an \textit{algebraic subdivision functor} \[Sd: ch(\A(X)) \to ch(\A(Sd\,X))\] such that for an appropriately defined covariant assembly functor \[\rr:ch(\A(Sd\,X)) \to ch(\A(X)),\] $\rr Sd\, C \simeq C \in \A(X)$ for all chain complexes $C$. Locally barycentric subdivision replaces an open simplex $\mathring{\sigma}$ with its subdivision $Sd\, \mathring{\sigma}$. Algebraic subdivision mimics this replacement by thinking of the part of $C\in\brc{\A^*(X)}{\A_*(X)}$ over $\mathring{\sigma}$, namely $C(\sigma)$, as $C(\sigma)\otimes_\Z\Z$ where $\Z$ is viewed as $\brc{\Sigma^{-|\sigma|}\Delta_*(\mathring{\sigma})}{\Sigma^{|\sigma|}\Delta^{-*}(\mathring{\sigma})}$. Here $\Sigma$ denotes suspension of the chain complex. This is replaced with $\brc{\Sigma^{-|\sigma|}\Delta_*(Sd\, \mathring{\sigma})}{\Sigma^{|\sigma|}\Delta^{-*}(Sd\, \mathring{\sigma})}$ in $Sd\, C$. Though this is not quite accurate the idea is correct. See chapter \ref{chapeight} for the precise details.

It is also possible to assemble contravariantly by assembling over dual cells yielding a functor \[\brc{\rr:ch(\A^*(Sd\,X)) \to ch(\A_*(X))}{\rr:ch(\A_*(Sd\,X)) \to ch(\A^*(X)).}\] Subdividing and then assembling contravariantly allows us to pass between $\A^*(X)$ and $\A_*(X)$.

The categories $\A(X)$ capture algebraically the notion of $\ep$-control for all $\ep>0$ for the following reason. If $X$  has mesh$(X)<\infty$ measured in $X$ then a chain complex $C\in \A(X)$ has bound at most $\mathrm{mesh}(X)$ as non-zero components of morphisms cannot go further than from a simplex to its boundary or vice versa. If $X$ is finite-dimensional, then the bound of $Sd^i\, C \in \A(Sd^i\, X)$ is at most \[\left(\dfrac{\mathrm{dim}(X)}{\mathrm{dim}(X)+1}\right)^i\mathrm{mesh}(X)\] which tends to $0$ as $i\to \infty$. So, given a $C\simeq 0\in \A(X)$ we can subdivide it to get a representative with bound as small as we like.

We already saw that the Poincar\'{e} Duality chain equivalence is naturally a chain map in $\mathbb{G}_X(\A)$ from a chain complex in $\A_*(X)$ to a chain complex in $\A^*(X)$. Subdividing the simplicial cochain complex of $X$ and assembling contravariantly we may think of it as a chain complex in $\A^*(X)$.
Thus it makes sense to ask when the Poincar\'{e} duality chain equivalence is in fact a chain equivalence in $\A^*(X)$. If we can show this then we can subdivide it to get $\ep$-controlled Poincar\'{e} duality for all $\ep>0$, thus making $X$ necessarily a homology manifold. 

A version of squeezing holds for these categories: 
\begin{introthm}[Squeezing Theorem]
Let $X$ be a finite-dimensional locally finite simplicial complex. There exists an $\ep=\ep(X)>0$ and an integer $i=i(X)$ such that if there exists a chain equivalence $Sd^i\,C \to Sd^i\, D$ in $\mathbb{G}_{Sd^i\, X}(\A)$ with control $<\ep$ for $C,D\in \A(X)$, then there exists a chain equivalence
\begin{displaymath}
 \xymatrix{ f: C \ar[r]^-{\sim} & D
}
\end{displaymath}
in $\A(X)$ without subdividing.
\end{introthm}
This answers our question about Poincar\'{e} Duality: 
\begin{introthm}[Poincar\'{e} Duality Squeezing]
Let $X$ be a finite-dimensional locally finite simplicial complex. There exists an $\ep=\ep(X)>0$ and an integer $i=i(X)$ such that if $Sd^i X$ has $\ep$-controlled Poincar\'{e} duality then, subdividing as necessary, $X$ has $\ep$-controlled Poincar\'{e} duality for all $\ep>0$ and hence is a homology manifold.
\end{introthm}

Again the open cone can be used to characterise when a chain complex is chain contractible in $\A(X)$. Algebraic subdivision is used to define a functor \[\textit{``}-\otimes\Z\textit{''}:ch(\A(X)) \to ch(\A(X\times\R)).\] This is done by giving $X\times\R$ a bounded triangulation over $O(X^+)$ which is $Sd^j\, X$ on $X\times\{v_j\}$ for points $\{v_j\}_{j\geqslant 1}$ chosen so that \[|v_j-v_{j+1}| = \frac{\mathrm{dim}(X)+1}{\mathrm{dim}(X)+2}|v_{j+1}-v_{j+2}|\] and which is $X$ on $X\times\{v_j\}$ for points $\{v_j\}_{j\leqslant 0}$ chosen so that $|v_j-v_{j+1}|=1$. The functor $\textit{``}-\otimes\Z\textit{''}$ then sends a chain complex $C\in ch(\A(X))$ to a chain complex which is $Sd^j\, C$ on $X\times\{v_j\}$ for $j\geqslant 1$ and $C$ for $j\leqslant 0$. If $C\simeq 0$ in $\A(X)$ then $Sd^j\, C \simeq 0$ in $\A(Sd^j\, X)$ for all $j\geqslant 0$, which shows that $\textit{``}C\otimes\Z\textit{''}\simeq 0$ in $\A(X\times\R)$. For $X$ a finite simplicial complex a stronger condition than the converse is also true.

To prove this we exploit algebraic consequences of the fact that $X\times\R$ is a product. We define a PL automorphism $t^{-1}: X\times\R \to X\times\R$ called \textit{exponential translation} which is the PL map defined by sending $v_j\mapsto v_{j-1}$. This induces a map on geometric categories over $X\times\R$ sending a chain equivalence with finite bound $B$ to a chain equivalence with bound $\dfrac{\mathrm{dim}(X)+1}{\mathrm{dim}(X)+2}B$. Iterates of this give us chain equivalences with bound as small as we like. 

We say that a chain complex is \textit{exponential translation equivalent} if exponential translates are chain equivalent to each other (when assembled so as to be in the same category). By the way it was defined $\textit{``}C\otimes\Z\textit{''}$ is exponential translation equivalent, so for a chain contraction $\textit{``}C\otimes\Z\textit{''}\simeq 0\in\mathbb{G}_{X\times\R}(\A)$ with finite bound measured in $O(X^+)$ we can take a representative with contraction with bound as small as we like. Combining this with the Squeezing Theorem allows us to obtain a chain contraction over a slice $X\times\{t\}$ for $t$ large enough. The fact that the metric increases in the open cone as you go towards $\{+\infty\}$ in $\R$ and the fact that the chain contraction had finite bound over $O(X^+)$ to begin with, means that the chain contraction on the slices $X\times \{t\}$ have control proportional to $1/t$ measured in $X$. We know what a slice of $\textit{``}C\otimes\Z\textit{''}$ looks like, because it is $Sd^j\, C$ on $X\times\{v_j\}$ for $j\geqslant 0$. So we have chain contractions of $Sd^j\, C$ for large $j$ with control as small as we like. Assembling such a chain contraction gives a chain contraction $C\simeq 0 \in \A(X)$.

Both the simplicial chain and cochain complexes of $X\times\R$ are exponential translation equivalent, so we can apply Theorem \ref{above} to Poincar\'{e} Duality. This gives Theorem \ref{cor} which is a simplicial complex version of the unproven footnote of \cite{epsurgthy}. 

The reason the approach above works so well is that $f\times\id$ is already a product, so by translating it in the $\R$ direction we do not change anything. A natural continuation to the work presented here is to tackle splitting problems where this sliding approach will not work, in which case we expect there to be a $K$-theoretic obstruction over each simplex. 

\vspace{5mm}
This thesis is split into two main parts: a topological half consisting of Chapters \ref{chaptwo} to \ref{chapsix} and an algebraic half consisting of Chapters \ref{chapseven} to \ref{chapten}.

Chapter \ref{chaptwo} surveys many results from the literature concerned with when a map is close to a homeomorphism. Chapter \ref{chapthree} introduces controlled topology and bounded topology as well as the open cone construction and its use to relate the two. Chapter \ref{chapfour} recaps some definitions related to simplicial complexes and explains the natural path space metric on a locally finite simplicial complex. Orientations of simplices are discussed and the fundamental $\ep$-subdivision cellulation is defined. Chapter \ref{chapfive} is a study of the subdivision chain equivalences induced on simplicial chains and cochains by barycentric subdivision. It is explained how chain inverses and chain homotopies can be carefully chosen to have the properties required to prove the squeezing results of the thesis. In Chapter \ref{chapsix} the topological Vietoris-like Theorem is proven directly with the use of the fundamental $\ep$-subdivision cellulation construction. There is also a slight digression on a notion of `triangular' homotopy equivalence of simplicial complexes inspired by the categories $\A(X)$ used later. 

In Chapter \ref{chapseven} we introduce the geometric categories $\mathbb{G}_X(\A)$ and $\A(X)$, discuss assembly and prove that a chain complex in $\A(X)$ is locally contractible if and only if it is globally contractible. Algebraic subdivision is defined in Chapter \ref{chapeight} where plenty of examples are also given. In Chapter \ref{chapnine} we boundedly triangulate $X\times\R$ and define the functor $\textit{``}-\otimes\Z\textit{''}$. Finally in Chapter \ref{chapten} we prove the algebraic squeezing result and use it together with exponential translation equivalence to prove the algebraic Vietoris-like Theorem. We then apply this theorem to Poincar\'{e} duality and homology manifolds. Appendix \ref{appendixa} contains some information about locally finite homology.

%% file: Thesis1.tex
\chapter{When is a map close to a homeomorphism?}\label{chaptwo}
In this chapter we survey several results from the literature concerned with the following questions:
\begin{qq}\label{q1}
When is a map $f:M^n\to N^n$ of $n$-dimensional topological manifolds a limit of homeomorphisms?
\end{qq}
\begin{qq}\label{q2}
When is a map $f:M^n\to N^n$ of $n$-dimensional topological manifolds homotopic to a homeomorphism (via a small homotopy)?
\end{qq}
\begin{qq}\label{q3}
 What can be said in the non-manifold case?
\end{qq}
This chapter will set the scene for the work in the first part of this thesis where we consider similar questions for simplicial maps of finite-dimensional locally finite simplicial complexes.

\section{Cell-like mappings as limits of homeomorphisms}
Since the $60$s a very fruitful approach to determine when a map $f$ is close to a homeomorphism has been to impose local conditions on $f$ that are sufficiently strong to deduce global consequences. The idea is that the point inverses of a homeomorphism are all precisely points. Thus, if a map is close to a homeomorphism in some appropriate sense, then the point inverses should also be close to points, again in some suitable sense. 

The first such notion of a set being close to a point that we shall consider is Morton Brown's concept of a cellular subset of a manifold:

\begin{defn}[\cite{BrownSchoenflies}]\label{cellular}
A subset $X\subset M^n$ of an $n$-dimensional topological manifold $M$ is called \textit{cellular} if there exist topological $n$-cells $Q_1, Q_2,\ldots \subset M$ with $Q_{i+1}\subset \mathring{Q}_i$ for all $i$ and \[ \bigcap_{i=1}^{\infty}{Q_i} = X.\]
A map $f:M^n \to X$ from an $n$-dimensional manifold is called $\textit{cellular}$ if for all $x\in X$, the point inverses $f^{-1}(x)$ are cellular subsets of $M$. 
\qed\end{defn}

\begin{ex}
A tree embedded in $\R^2$ is a cellular set.
\begin{figure}[ht]
\begin{center}
{
\includegraphics[width=5cm]{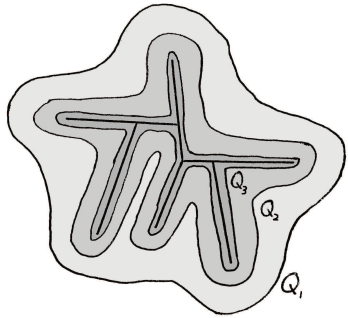}
}
\caption{A tree is cellular in $\R^2$.}
\label{Fig:tree}
\end{center}
\end{figure}
\qed\end{ex}

A cellular subset of a manifold is close to a point in the following sense:
\begin{thm}[\cite{BrownSchoenflies}, \cite{BrownMonotone}]
$S^n\backslash C \cong \R^n$ for a compact subset $C\subset S^n$ if and only if $C$ is a cellular subset of $S^n$.
\end{thm}
The proof of this theorem is also reviewed in \cite{Ed78} where the condition $S^n\backslash C \cong \R^n$ is referred to by Edwards as $C$ being \textit{point-like}. Brown remarks in \cite{BrownSchoenflies} that if $X\subset M^n$ is a compact cellular subset, then $M/X$ is a manifold and the projection map $\pi:M \to M/X$, which is clearly a cellular map, is a limit of homeomorphisms. The idea of the proof is to send the concentric $n$-cells $Q_i$ around $X$ to a local base of $n$-cells around the point $\pi(X)$. 

Brown noted that for Euclidean space $\R^n$, $\R^n/X$ is a manifold if and only if $X$ is cellular in $\R^n$, but this is not true for general manifolds as pointed out in for example \cite{lach77} where it is observed that $S^2 = S^1\times S^1/S^1\vee S^1$ but $S^1\vee S^1$ is not cellular in $S^1\times S^1$.

Conversely, Finney showed in \cite{finney} that if $f:M\to N$ is a limit of homeomorphisms then it is necessarily a cellular map. This leads one to conjecture that cellular maps are precisely limits of homeomorphisms, but it turns out that cellularity is not quite the right condition; we can be more general.

The problem is that whether the image of an embedding $\phi:X\to M^n$ is cellular in $M$ or not depends on the embedding $\phi$ rather than being an intrinsic property of the space $X$. Any finite-dimensional cellular subset of a manifold, except a point, can be non-cellularly embedded in Euclidean space of greater than twice its dimension. Classic examples include the Artin-Fox wild arc in $\R^3$ (see Figure 6 of \cite{ArtinFox48}), the Antoine-Alexander horned ball in $\R^3$ (see Figure 3 of \cite{Ed78}), and polyhedral copies of the dunce hat in $\R^4$ (\cite{duncehat}). Blankinship also showed in \cite{blankinship} that an arc may be non-cellularly embedding in $\R^n$ for $n\geqslant 3$. This motivated Lacher to consider \textit{embeddability as a cellular subset of some manifold} rather than cellularity as this is an intrinsic property of the space:

\begin{defn}[\cite{LachCellANR}] 
A space $X$ is \textit{cell-like} if there exists an embedding $\phi$ of $X$ into some manifold $M$ such that $\phi(X)$ is cellular in $M$. A mapping $f:X\to Y$ is \textit{cell-like} if $f^{-1}(y)$ is a cell-like space for each $y\in Y$.
\qed\end{defn}

In studying the Hauptvermutung, Sullivan studied the following type of homotopy equivalence: 
\begin{defn}\label{heredproperhtpyequiv}
A proper mapping $f:X\to Y$ is a \textit{hereditary proper homotopy equivalence} if for all open sets $U\subset Y$, $f|_{f^{-1}(U)}:f^{-1}(U)\to U$ is a proper homotopy equivalence.
\qed\end{defn}

\begin{defn}
An \textit{absolute neighbourhood retract (ANR)} is a locally compact separable metric space which can be embedded as a closed subset of $I^\infty\times[0,\infty)$ in such a manner that some neighbourhood $U$ of the image retracts to the image.

A finite-dimensional ANR is called a \textit{Euclidean neighbourhood retract (ENR)} as it can be embedded as a closed subset of Euclidean space such that a neighbourhood of the image retracts to the image.
\qed\end{defn}

Cell-like mappings were studied at the time due to their equivalence with hereditary proper homotopy equivalences as maps of ENRs:
\begin{thm}[Homotopy characterisation of cell-like maps, \cite{LachCellANR} Thm 2]\label{thm:lachcellanr2}
Let $f:X\to Y$ be a proper onto mapping of ENRs. Then the following are equivalent:
\begin{enumerate}[(a)]
 \item $f$ is cell-like,
 \item for each contractible open subset $U\subset Y$, $f^{-1}(U)$ is contractible,
 \item $f$ is a hereditary proper homotopy equivalence. 
\end{enumerate}
\end{thm}

Notice here that properness is essential: a cell-like mapping of ANRs\footnote{Or in fact a contractible mapping of ENRs.} need not in general be a homotopy equivalence, let alone a hereditary proper one. As noted, for example in \cite{SiebCE}, the numeral $6$ is the image of an injective, continuous map of the real line which is a cell-like map, neither proper nor a homotopy equivalence. A proper cell-like map is called a \textit{CE} map.

Of further interest is the fact that a hereditary proper homotopy equivalence of closed $n\geqslant 5$ dimensional $PL$ manifolds is often homotopic to a $PL$ homeomorphism:
\begin{thm}[Sullivan]
Let $f:M^n\to N^n$ be a hereditary proper homotopy equivalence of closed $n\geqslant 5$ dimensional $PL$ manifolds. Then $f$ is homotopic to a $PL$ homeomorphism $:M\to N$ provided $\pi_1(M) = H^3(M;\Z_2) = 0$. 
\end{thm}
\begin{proof}
See \cite{hauptvermutungbook}. 
\end{proof}

Returning to the conjecture that cellular maps are precisely limits of homeomorphisms, Siebenmann proved

\begin{thm}[Approximation Theorem A, \cite{SiebCE}]\label{lotsofpeople}
Let $f:M\to N$ be a CE map of metric topological $m$-manifolds, $m\neq 4$, with or without boundary, such that $f|\partial M$ gives a CE map $\partial M \to \partial N$. Let $\ep: M \to (0,\infty)$ be continuous. Suppose that at least one of the following three conditions holds.

\begin{enumerate}[(a)]
 \item $m\neq 3,4,5$
 \item $m=5$ and $f:\partial M \to \partial N$ is a homeomorphism
 \item $m=3$, and for each $y \in N$, $f^{-1}(y)$ has an open neighbourhood that is prime for connected sum.\footnote{If $X$ is a $3$-manifold possibly with boundary, which is not homeomorphic to an (interior) connected sum $Y\sharp Z$ where $Z$ is a closed $3$-manifold and neither $Y$ nor $Z$ is a $3$-sphere, then we agree to call $X$ prime (for connected sum).} (call such an $f$ prime).
\end{enumerate}
Then there exists a homeomorphism $g:M \to N$ such that $d(f(x),g(x))<\ep(x)$ for all $x\in M$.
\end{thm}

\noindent Lacher rephrases this in \cite{lach77} in a way that more closely resembles the conjecture: 
\begin{quotation}
 \noindent ``The set of cell-like maps $M\to N$ (where $M$ and $N$ are closed $n$-manifolds, $n\neq 4$) is precisely the closure of the set of homeomorphisms $M\to N$ in the space of maps $M\to N$. (The case $n=3$ requires also that $M$ contains no fake cubes and was done earlier by S. Armentrout.)''
\end{quotation}
See \cite{trout} for the case $n=3$ and \cite{hilbcube} for the case $n=\infty$. In particular we obtain that such a cell-like map is homotopic to a homeomorphism. Siebenmann states that his interest in this theorem is the following conjecture (also made by Kirby):

\begin{conj}[\cite{SiebCE}]
If $M$ is a closed metric topological manifold and $\delta>0$ is prescribed, one can find $\ep=\ep(M,\delta)>0$ so that the following holds: Given any map $f:M\to N$ onto a closed manifold such that, for all $y$ in $N$, $f^{-1}(y)$ has diameter $<\ep$, there exists a homotopy of $f$ to a homeomorphism, through maps with point preimages of diameter $<\delta$.
\end{conj}

Chapman and Ferry proved their \textit{$\alpha$-approximation Theorem} which generalises Theorem \ref{lotsofpeople} and confirms Kirby and Siebenmann's conjecture for $n\geqslant 5$. First they define

\begin{defn}[\cite{chapfer}]
 Let $\alpha$ be an open cover of $N$, then a proper map $f:M\to N$ is said to be an \textit{$\alpha$-equivalence} if there exists a $g:N \to M$ and homotopies $\theta_t:f \circ g\sim \id_N$, $\phi_t:g\circ f\sim \id_M$, such that 
\begin{enumerate}[(1)]
 \item $\forall m\in M$, $\exists U\in \alpha$ containing $\{f\phi_t(m)\,|\,0\leqslant t\leqslant 1 \}$,
 \item $\forall n\in N$, $\exists U\in \alpha$ containing $\{\theta_t(n)\,|\,0\leqslant t\leqslant 1 \}$.
\end{enumerate}
\qed\end{defn}
\noindent They then prove
\begin{thm}[$\alpha$-approximation Theorem, \cite{chapfer}]
Let $N^n$ be a topological $n$-manifold, $n\geqslant 5$. For every open cover $\alpha$ of $N$ there is an open cover $\beta$ of $N$ such that any $\beta$-equivalence $f:M^n\to N^n$ which is already a homeomorphism from $\partial M \to \partial N$ is $\alpha$-close to a homeomorphism $h:M\to N$ (i.e.\, for each $m\in M$, there is a $U\in \alpha$ containing $f(m)$ and $h(m)$).
\end{thm}
This generalises Theorem \ref{lotsofpeople} in the case $n\geqslant 5$ since it follows from \cite{celllikemappings1} that any $CE$ map $f:M\to N$ is an $\alpha$-equivalence, for any $\alpha$.

\section{The non-manifold case} 
As in the previous section we consider maps whose point inverses are close to being points. In general being homotopic to a homeomorphism is a bit too much to hope for unless we impose much stronger local conditions on our map. This is illustrated by the following example:
\begin{ex}\label{Ex:T}
Let $f:T\to [0,1]$ be the vertical projection of the capital letter $T$ onto its horizontal bar. Clearly $T$ and $[0,1]$ are not PL homeomorphic, yet $f$ is a proper, cell-like, hereditary proper homotopy equivalence, as well as an $\alpha$-equivalence for all open covers $\alpha$ of $[0,1]$ and a simple homotopy equivalence.
\qed\end{ex}
Such a map is a limit of PL homeomorphisms in a different sense, i.e.\ we do not take the closure of the set of homeomorphisms $X\to Y$ in the space of maps $X\to Y$, but rather in something larger. 
\begin{ex}
Consider again the example of the letter $T$. Let $k:T\times I \to T$ be a linear deformation retract of $T$ onto the horizontal bar. Letting $k_t=k(-,t)$, this is a homotopy from $k_0 = \id_T$ to $k_1=f$ through PL homeomorphisms $\{k_t\}_{t>0}$ onto their image. In this sense $f= \lim_{t\to 0} k_t$ is a limit of PL homeomorphisms.

Another point of view is that $f$ is the limit of the identity PL homeomorphism from $T$ to $T_t$ where $T_t$ is $T$ with a metric that is the usual metric on the horizontal bar but where the vertical bar is given $t$ times its usual length.
\qed\end{ex}

\begin{ex}\label{motivatingexamples}
Arguing similarly, a collapse\footnote{In the sense of simple homotopy theory.} $f:K\to L$ is the limit $\lim_{t\to 0}k_t$ of PL homeomorphisms $k_t$, where $k:K\times I \to K$ is the deformation retract of $K$ onto $L$ given by the collapse. 
\qed\end{ex}
Now suppose we have a cell-like map $f$ of PL non-manifolds. Bearing in mind the previous examples, the strongest conclusion we might hope for is that $f$ is a simple homotopy equivalence. We do in fact see this in the results of Cohen and Chapman below.
\begin{defn}
We say that a map $f:X\to Y$ is \textit{contractible} if it has contractible point inverses.
\qed\end{defn}
\begin{rmk}\label{rmkrmk}
From the definitions it is relatively easy to see that for a compact ANR $X$, $X$ is cell-like if and only if $X$ is contractible. For a proof, see for example \cite{Ed78} page 113. 
\qed\end{rmk}
\begin{cor}\label{Cor:CEisCont}
For a proper map $f$ of ANRs, $f$ is CE if and only if it is contractible.  
\end{cor}

\begin{thm}[\cite{cohen}]\label{cohensresult}
A contractible PL map of finite polyhedra is a simple homotopy equivalence. 
\end{thm}
Using Corollary \ref{Cor:CEisCont} we could replace the word contractible with CE in this theorem. As stated earlier, Lacher followed Cohen's result in \cite{LachCellANR} proving Theorem \ref{thm:lachcellanr2}. In Theorem 1.1 of \cite{SiebCE}, Siebenmann observes that by Lacher's work in \cite{LachCellANR} one can deduce
\begin{thm}[\cite{SiebCE}, Theorem 1.1]
 Let $f:X \to Y$ be a map of ENRs. If $f$ is CE, then $f$ is a proper homotopy equivalence. Conversely, if a proper map $f$ is a homotopy equivalence over small neighbourhoods of each point of $Y$\footnote{This means an $\alpha$-equivalence for sufficiently small cover $\alpha$.}, then $f$ is CE.
\end{thm}
After this, without the assumption of properness, Chapman showed
\begin{thm}[\cite{hilbcube}]
Let $f:X\to Y$ be a cell-like mapping of compact ANRs, then $f$ is a simple homotopy equivalence.
\end{thm}
In the case of finite simplicial complexes all the above results are nicely summarised in Proposition 2.1.8 of \cite{plmf} in which the authors refer to a contractible map as a \textit{simple map}. In this thesis we stick to the terminology contractible.

In \cite{prassidis}, Prassidis points out 
\begin{rmk}\label{rmkonctrl}
A CE map $f:Y \to X$ of locally compact ANRs is a controlled homotopy equivalence with control in $X$, i.e.\ an $\alpha$-equivalence for all $\ep$-nets $\alpha$ of $X$. 
\qed\end{rmk}

The work of Cohen and Akin on transverse cellular maps is also worth mentioning as it provides an answer to Question \ref{q2} for surjective maps of non-manifold simplicial complexes:

\begin{defn}[\cite{Akin72}]
A surjective simplicial map $f:K\to L$ is called \textit{transverse cellular} if for all simplices $\sigma \in L$, $D(\sigma, f):=f^{-1}(D(\sigma,L))$ is homeomorphic rel $\partial D(\sigma, f)$ to a cone on $\partial D(\sigma,f):=f^{-1}(\partial D(\sigma,L))$.\footnote{See chapter \ref{chapfour} for the definition of dual cells $D(\sigma,X)$.} 
\qed\end{defn}

\begin{prop}[\cite{Akin72}]
 If $f:K \to L$ is transverse cellular, then $f$ is homotopic to a PL homeomorphism $K\to L$ through maps taking $D(\sigma,f)$ to $D(\sigma, L)$. 
\end{prop}
The condition of transverse cellularity is much stronger than the conditions we will consider in this thesis, as demonstrated by its consequences.

For arbitrary simplicial complexes the notion of cell-like and that of contractible do not in general coincide. 
\begin{ex}\label{earliercomments}
The real line $\R$ is a contractible space that is not cell-like as it is not compact. Conversely, for a general simplicial complex the wedge of two cones on cantor sets can be cell-like but not contractible, see for example page 113 of \cite{Ed78}. However, examples like the latter are excluded if we are dealing with locally finite simplicial complexes. In fact a locally finite cell-like simplicial complex is necessarily contractible - this follows from Remark \ref{rmkrmk} noting that a cell-like simplicial complex is a compact ANR. 
\qed\end{ex}
We choose to work with contractible simplicial maps; using simplicial maps in some sense remedies the difference between contractible and cell-like as it prevents pathologies like the numeral $6$ map from occurring - we cannot have limit points like that of the numeral six without violating local finiteness. More precisely, a cell-like map between locally finite simplicial complexes must necessarily be a proper map by Example \ref{earliercomments}.

The results of the first half of this thesis can be thought of as extending Cohen's result, Theorem \ref{cohensresult}, for PL maps of finite-dimensional compact polyhedra to PL maps of finite-dimensional locally compact polyhedra, and answering Question \ref{q1} in the sense of Example \ref{motivatingexamples} for surjective simplicial maps of spaces of this class.

\chapter{Controlled and bounded topology}\label{chapthree}
The idea of using estimates in geometric topology dates back to Connell and Hollingsworth's paper \cite{ConnHoll} in which the authors introduce estimates in order to compute algebraic obstructions.

After this, controlled topology was developed by Chapman (\cite{Chapcontrol}), Ferry (\cite{Ferryep}) and Quinn (\cite{Quinn1}, \cite{Quinn2}). The general idea of controlled topology is to put an estimate on a geometric obstruction (usually by use of a metric). It is often possible to prove that if the size of the estimate is sufficiently small then the obstruction must vanish. This approach has proved a successful tool for the topological classification of topological manifolds.

A downside to controlled topology is that it is not functorial; the composition of two maps with control less than $\ep$ is a map with control less than $2\ep$. This motivates Pedersen's development in \cite{PedKminusi} and \cite{PedInvariants} of bounded topology, where the emphasis is no longer on how small an estimate is but rather just that it is finite. Functoriality is then satisfied by the fact that the sum of two finite bounds remains finite. Typically controlled topology is concerned with small control measured in a compact space, whereas bounded topology is concerned with finite bounds measured in a non-compact space.

In order to actually be able to make estimates, it is desirable that the space we are working with comes equipped with a metric. This is not strictly necessary: if our space $X$ has at least a map $p:X\to M$ to a metric space $M$, called a \textit{control map}, then we can measure distances in $M$. We can't just use any map however; a constant map would not be able to distinguish points in $X$ so we would be unable to determine anything about $X$ using it. 

\section{Relating controlled and bounded topology}
First we look at controlled topology in more detail. Let $(M,d)$ be a metric space. 
\begin{defn}
Let $p:X\to M$, $q:Y\to M$ be control maps. We say a map $f:(X,p) \to (Y,q)$ is \textit{$\ep$-controlled} if $f$ commutes with the control maps $p$ and $q$ up to a discrepancy of $\ep$, i.e.\ for all $x\in X$, $d(p(x), qf(x))<\ep$. We call $\ep$ the \textit{control} of the map $f$.

A map $f:(X,p) \to (Y,q)$ is called a \textit{controlled map}, if it is $\ep$-controlled for all $\ep$, i.e.\ if $p=qf$. 
\qed\end{defn}

\begin{defn}
We say that a controlled map $f:(X,p) \to (Y,q)$ is an $\ep$-controlled homotopy equivalence, if there exists an $\ep$-controlled homotopy inverse $g$ and $\ep$-controlled homotopies $h_1:g\circ f \sim \id_X$ and $h_2:f\circ g\sim \id_Y$. We call $\ep$ the \textit{control} of the homotopy equivalence $f$. 

We say that a controlled map is a \textit{controlled homotopy equivalence} if it is an $\ep$-controlled homotopy equivalence for all $\ep>0$.
\qed\end{defn}

\begin{rmk}
Note that the condition of being an $\ep$-controlled homotopy equivalence means that the maps $f$ and $g$ do not move points more than a distance $\ep$ \textbf{when measured in $M$} and that the homotopy tracks are no longer than $\ep$ again \textbf{when measured in $M$}. It is perfectly possible that the homotopy tracks are large in $X$ or $Y$ and only become small after mapping to $M$.
\qed\end{rmk}

We consider control maps $(Y,q)$ and $(Y,q^\prime)$ to be controlled equivalent if for all maps $f:X\to Y$, $f:(X,p) \to (Y,q)$ is controlled if and only if $f:(X,p) \to (Y,q^\prime)$ is controlled. This holds if and only if $q=q^\prime$.

Now consider instead bounded topology. Suppose in addition that our metric space $(M,d)$ does not have a finite diameter.

\begin{defn}\label{definebd}
Let $p:X\to M$, $q:Y\to M$ be control maps. A map $f:(X,p) \to (Y,q)$ is called \textit{bounded} if there exists a $B<\infty$ such that for all $x\in X$, $d(p(x), qf(x))<B$. We call the least such $B$ with this property the \textit{bound} of the map $f$ and denote it by bd$(f)$.
\qed\end{defn}

\begin{defn}\label{definebd2}
We say that a bounded map $f:(X,p) \to (Y,q)$ is a $Y$-bounded homotopy equivalence, if there exists a bounded homotopy inverse $g$ together with bounded homotopies $h_1:g\circ f \sim \id_X$ and $h_2:f\circ g\sim \id_Y$. We call $B=\max\{\mathrm{bd}(f),\mathrm{bd}(g), \mathrm{bd}(h_1), \mathrm{bd}(h_2)\}$ the \textit{bound} of the homotopy equivalence $f$. 
\qed\end{defn}

We consider control maps $(Y,q)$ and $(Y,q^\prime)$ to be boundedly equivalent if for all maps $f:X\to Y$, $f:(X,p) \to (Y,q)$ is bounded if and only if $f:(X,p) \to (Y,q^\prime)$ is bounded. This holds if and only if there exists a $k<\infty$ such that for all $x\in X$, $d(q(x),q^\prime(x))<k$.

Morally, equivalent control maps capture precisely the same information about $X$. Bounded topology is coarse in the sense that local topology is invisible to the control map; only the global behaviour out towards infinity is noticed. Controlled topology detects things with non-zero size, but is blind to the arbitrarily small, for example a missing point in $X$ cannot be detected. This is all of course provided the control map is a good one. Let the diameter of a subset of $X$ be the largest distance between any two points of that set measured in $M$. Some of the properties we might want the control map to satisfy are:

\begin{defn}
\begin{enumerate}[(i)]
 \item A control map $p:X\to M$ is called \textit{eventually continuous} if there exists a $k$ and an open covering $\{U_\alpha\}$ of $X$ such that for all $\alpha$, the diameter of $p(U_\alpha)$ is less than $k$.
 \item A \textit{bounded simplicial complex} over $M$ is a pair $(X,p)$ of a simplicial complex $X$ and an eventually continuous control map $p:X\to M$ such that there exists a $k$ with diam$(p(\sigma))<k$ for all simplices $\sigma\in X$.
\end{enumerate}
\qed\end{defn}
\begin{defn}
 Let $(X,p)$ be a bounded simplicial complex over $M$. 
\begin{enumerate}[(i)]
 \item $(X,p)$ is \textit{boundedly $(-1)$-connected} if there exists a $k>0$ such that for all points $m\in M$ there is an $x\in X$ with $d(p(x),m)<k$. In other words ``$p$ is surjective up to a finite bound''.
 \item $(X,p)$ is \textit{boundedly $0$-connected} if for all $d>0$ there exists a $k=k(d)$ such that for $x,y\in X$, if $d(p(x),p(y))<d$ then there is a path from $x$ to $y$ in $X$ with diameter less than $k(d)$.
\end{enumerate}
\qed\end{defn}
We could also ask that the fundamental group be tame and bounded in the sense that non-trivial loops have representatives with non-zero finite diameter when measured in $M$. For more details about all these conditions and their consequences, see for example \cite{epsurgthy}. 

In this thesis we consider only finite-dimensional locally finite simplicial complexes. In chapter \ref{chapfour} we show that such a space $X$ comes equipped with a natural complete metric so we can just take our control map to be the identity. In addition we will see that with respect to this metric, $X$ has a tame and bounded triangulation so all the above conditions will be satisfied for the identity control map (on each path-connected component at least). Thus, the identity map is as good a control map as we could hope for. 

When considering a map $f: (X,p) \to (Y,q)$ of such spaces we will usually measure distances in the target, i.e. with $M=Y$, $q=\id_Y$ and $p=f$.

\section{The open cone}
The open cone was first considered by Pedersen and Weibel in \cite{kthyhom} where it was defined for subsets of $S^n$. This definition was extended to more general spaces by Anderson and Munkholm in \cite{AndMunk}. We make the following definition:
\begin{defn}\label{openconedefn}
Let $(M,d)$ be a complete metric space. The \textit{open cone} $O(M^+)$ is defined to be the identification space $M\times\R/\sim$, where $(m,t)\sim(m^\prime,t)$ for all $m,m^\prime\in M$ if $t\leqslant 0$. Letting $d_X$ be the metric on $X$ we define a metric $d_{O(M^+)}$ on $O(M^+)$ by setting 
\begin{eqnarray*}
 d_{O(M^+)}((m,t),(m^\prime,t)) &=& \brcc{td_X(m,m^\prime),}{t\geqslant 0,}{0,}{t\leqslant 0,} \\
 d_{O(M^+)}((m,t),(m,s)) &=& |t-s|.
\end{eqnarray*}
 Then define $d_{O(M^+)}((m,t),(m^\prime,s))$ to be the infimum over all paths from $(m,t)$ to $(m^\prime,s)$, which are piecewise geodesics in either $M\times\{r\}$ or $\{n\}\times\R$, of the length of the path. I.e.
\[d_{O(M^+)}((m,t),(m^\prime,s)) = \max\{\min\{t,s\},0\}d_X(m,m^\prime) + |t-s|.\]
This metric has been carefully chosen so as to make $M\times\{t\}$, $t$ times as big as $M\times\{1\}$ for all $t\geqslant 0$ and $0$ times as big for all $t\leqslant 0$.
\qed\end{defn}
This is precisely the metric used by Anderson and Munkholm in \cite{AndMunk} and also by Siebenmann and Sullivan in \cite{SiebSull}, but there is a notable distinction: we do not require necessarily that our metric space $(M,d)$ have a finite bound.

\begin{ex}
For $M$ a proper subset of $S^n$ with the subspace metric, the open cone $O(M^+)$ can be thought of as all the points in the lines out from the origin in $\R^{n+1}$ through points in $M^+:= M \cup \{pt\}$. This is not the same as the metric given by Definition \ref{openconedefn} but is equivalent. 
\qed\end{ex}

\begin{defn}\label{coningmap}
There is a natural map $j_X: X\times\R \to O(X^+)=X\times\R/\sim$ given by the quotient map 
\begin{eqnarray*}
X\times\R &\to& X\times\R/\sim \\ 
(x,t) &\mapsto& [(x,t)].
\end{eqnarray*}
We call this the \textit{coning map}.
\qed\end{defn}

In \cite{epsurgthy} Ferry and Pedersen suggest a relation between controlled topology on a space $X$ and bounded topology on the open cone over that space $O(X^+)$ when they write in a footnote:
\begin{quotation}
``It is easy to see that if $Z$ is a Poincar\'{e} duality space with a map $Z \to K$ such that $Z$ has $\ep$-Poincar\'{e} duality for all $\ep>0$ when measured in $K$ (after subdivision), e.g. a homology manifold, then $Z\times\R$ is an $O(K_+)$-bounded Poincar\'{e} complex. The converse (while true) will not concern us here.'' 
\end{quotation}
The proposed relation is the following: 
\begin{conj}\label{arbcontisbddcontoncone}
$f:(X,qf) \to (Y,q)$ is an $M$-bounded homotopy equivalence with bound $\ep$ for all $\ep>0$ if and only if $f\times \id_\R: (X\times\R, j_M(qf\times\id_\R)) \to (Y\times \R, j_M(q\times\id_\R))$ is an $O(M^+)$-bounded homotopy equivalence.
\end{conj}

In this thesis we work with finite-dimensional locally finite simplicial complexes and prove Conjecture \ref{arbcontisbddcontoncone} in this setting, where for the map $f$ we measure in the target space with the identity control map and for $f\times\id_R$ we measure in the open cone over $Y$ with control map the coning map.  

In the next chapter we explain how all locally finite simplicial complexes are naturally (complete) metric spaces, thus allowing us to use the identity map $\id:X\to X$ as our control map.

\chapter{Preliminaries for simplicial complexes}\label{chapfour}
In this chapter we recall for the benefit of the reader some basic facts about simplicial complexes and make some new definitions which we shall require in this thesis.
\section{Simplicial complexes, subdivision and dual cells}
\begin{defn}\label{standardsimplex}
The \textit{standard $n$-simplex} $\Delta^n$ in $\R^{n+1}$ is the convex hull of the standard basis vectors: \[\Delta^n:= \{(t_1,\ldots,t_{n+1})\in \R^{n+1}\,|\, t_i\geqslant0, \sum_{i=1}^{n+1}t_i = 1 \}. \]
\qed\end{defn}
Recall we build an abstract simplicial complex as an identification space of unions of standard simplices with some subsimplices identified. 
\begin{defn}
We say that a simplicial complex $X$ is \textit{locally finite} if each vertex $v\in X$ is contained in only finitely many simplices.
\qed\end{defn}
\begin{defn}
Let $\sigma$ be an $n$-simplex with vertices labelled $\{v_0, \ldots, v_n\}$ and let $\tau$ be an $m$-simplex with vertices $\{w_0,\ldots,w_m\}$. Suppose $\sigma$ and $\tau$ are embedded in $\R^k$, then we say that $\sigma$ and $\tau$ are \textit{joinable} if the vertices $\{v_0,\ldots, v_n,w_0,\ldots, w_m\}$ are all linearly independent. The \textit{join} of $\sigma$ and $\tau$, written $\sigma\tau$, is then the $(n+m+1)$-simplex spanned by this vertex set, i.e.\ the convex hull of this vertex set. 
\qed\end{defn}
\begin{defn}\label{simplexasjoins}
Using the notation of joins, $v_0\ldots v_n$ denotes the $n$-simplex spanned by $\{v_0,\ldots, v_n\}$: \[ v_0\ldots v_n := \{\sum_{i=0}^{n}{t_iv_i}\,|\, \sum_{i=0}^{n}t_i = 1 \}. \]
\qed\end{defn}
\begin{defn}\label{vertexset}
We denote the \textit{vertex set} of the simplex $\sigma=v_0\ldots v_n$ by $V(\sigma)$.
\qed\end{defn}
\begin{defn}\label{barycentre}
We refer to the coordinates $(t_0,\ldots,t_n)$ as \textit{barycentric coordinates} and to the point $\widehat{\sigma}:=(1/(n+1),\ldots,1/(n+1))$ as the \textit{barycentre} of $\sigma=v_0\ldots v_n$. 
\qed\end{defn}
For rigour we should point out that each $n$-simplex of $X$ is thought of as standardly embedded in $\R^{n+1}$ so the barycentric coordinates make sense in $\R^{n+1}$.
\begin{defn}\label{face}
 If we delete one of the $n+1$ vertices of an $n$-simplex $\sigma=v_0\ldots v_n$, then the remaining $n$ vertices span an $(n-1)$-simplex $v_0\ldots\widehat{v_i}\ldots v_n$, called a \textit{face} of $\sigma$. Here $\widehat{v_i}$ is used to indicate that the vertex $v_i$ is omitted. This is not to be confused with the notation for barycentre, this should always be clear from context.
\qed\end{defn}
\begin{defn}\label{interior}
The union of all faces of $\sigma$ is the \textit{boundary} of $\sigma$, $\partial \sigma$. The open simplex $\mathring{\sigma} = \sigma\backslash \partial \sigma$ is the \textit{interior} of $\sigma$.  For an $n$-simplex $\sigma$, $|\sigma|:=n$ will denote the \textit{dimension} of $\sigma$.
\qed\end{defn}

\begin{defn}\label{geometricrealisation}
Forgetting the simplicial structure on $X$ and just considering it as a topological (or metric) space, $||X||$ will denote the \textit{geometric realisation} of $X$.
\qed\end{defn}

\begin{lem}\label{standardmetric}
Any locally finite simplicial complex $X$ can be given a complete metric whose restriction to any $n$-simplex $\sigma$ is just the subspace metric from the standard embedding of $\sigma$ in $\R^{n+1}$. We call this the \textit{standard metric on $X$}.
\end{lem}
\begin{proof}
The idea is that each simplex has the standard subspace metric and we extend this to the path metric on the whole of $X$: any $x,y\in X$ in the same path-connected component of $X$ can be joined by a path $\gamma$. By compactness of $[0,1]$, this path can be split up into the concatenation of a finite number of paths $\gamma = \gamma_i\circ\ldots \circ\gamma_1$ with each $\gamma_j$ contained entirely in a single simplex $\sigma_j$. Using the standard metric on $\sigma_j$ we get a length for $\gamma_j$, summing all these lengths gives a length for $\gamma$. We then define the distance between $x$ and $y$ to be the infimum over all paths of the path length. If $x$ and $y$ are in different path-connected components then the distance between them is defined to be infinite. See $\S 4$ of \cite{bartelssqueezing} or Definition $3.1$ of \cite{HR95} for more details.
\end{proof}

\begin{defn}\label{openball}
Let $(M,d)$ be a metric space. Define the open ball of radius $\ep$ around the point $m$ to be \[B_\ep(m) = \{ m^\prime\in M \,|\, d(m,m^\prime)<\ep\}.\] Define the sphere of radius $\ep$ around the point $m$ to be \[\partial\overline{B_\ep(v)}= \{ m^\prime\in M \,|\, d(m,m^\prime)=\ep\}.\]
\qed\end{defn}

\begin{defn}\label{diam}
Let $p:X\to (M,d)$ be a continuous control map. For all $\sigma\in X$, define the \textit{diameter} of $\sigma$ to be \[\mathrm{diam}(\sigma):= \sup_{x,y\in\sigma}{d(p(x),p(y))}.\]

\noindent Define the \textit{radius} of $\sigma$ by \[\mathrm{rad}(\sigma) = \sup_{x\in\sigma}\inf_{y\in\partial\sigma}{d(p(x),p(y))}.\] The radius of $\sigma$ can be thought of as the radius of the largest open ball in $M$ whose intersection with $p(\sigma)$ will fit entirely inside $p(\sigma)$, i.e.\ does not intersect $p(\partial \sigma)$. 

Let $o(\sigma)$ denote the \textit{incentre} of $\sigma$: the centre of the largest inscribed ball. 
\qed\end{defn}

\begin{ex}
\Figref{Fig:new4} illustrates the diameter, radius and incentre of a $2$-simplex with control map an embedding into Euclidean space.

\begin{figure}[ht]
\begin{center}
{
\psfrag{o}{$o(\sigma)$}
\psfrag{rad}{rad$(\sigma)$}
\psfrag{diam}{diam$(\sigma)$}
\includegraphics[width=7cm]{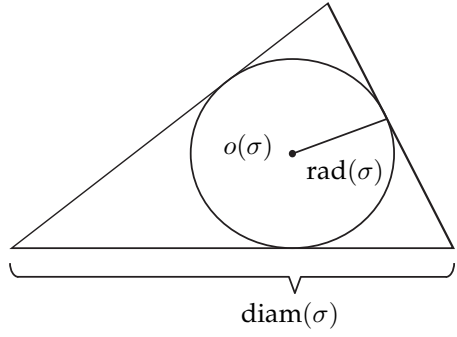}
}
\caption{Diameter, radius and incentre.}
\label{Fig:new4}
\end{center}
\end{figure}
\qed\end{ex}

\begin{rmk}
Note that if $\tau<\sigma$, then rad$(\sigma)\leqslant \mathrm{rad}(\tau)$.
\qed\end{rmk}

\begin{defn}\label{meshandco}
Let \[\mesh(X):= \sup_{\sigma\in X}\{\mathrm{diam}(\sigma)\}.\] If mesh$(X)<\infty$, then we say that $X$ has a \textit{bounded triangulation}. 

Let \[\comesh(X):= \inf_{\sigma\in X, |\sigma|\neq 0}\{\mathrm{rad}(\sigma)\}.\] If comesh$(X)>0$, then we say that $X$ has a \textit{tame triangulation}.
\qed\end{defn}

\begin{lem}
With the subspace metric from the standard embedding $\Delta^n\subset \R^{n+1}$ and control map $1:\Delta^n\to \Delta^n$,  
\begin{eqnarray*}
 \diam (\Delta^n) &=& \sqrt{2} \\
 \rad (\Delta^n) &=& \frac{1}{\sqrt{n(n+1)}}.
\end{eqnarray*}
\end{lem}
\begin{proof}
Recall that with the standard embedding \[\Delta^n:= \{(t_1,\ldots,t_{n+1})\in \R^{n+1}\,|\, t_i\geqslant0, \sum_{i=1}^{n+1}t_i = 1 \}. \] By the convexity of $\Delta^n$, the points in $\Delta^n$ which are furthest apart are the vertices and these all have separation $\sqrt{2}$, whence $\diam (\Delta^n) = \sqrt{2}$.

$\rad(\Delta^n)$ is the radius of the largest $n$-sphere that fits inside $\Delta^n$. This sphere necessarily intersects each face of $\Delta^n$ once, and this intersection must be the barycentre of that face by symmetry. Again by symmetry considerations the centre of the sphere must be the barycentre of $\Delta^n$, so its radius is the separation of the barycentre of $\Delta^n$ and the barycentre of any face, namely
\begin{eqnarray*}
 \rad(\Delta^n) &=& \bigg{\lvert} \left(\frac{1}{n+1},\ldots,\frac{1}{n+1}\right) - \left( 0, \frac{1}{n},\ldots, \frac{1}{n}\right) \bigg{\lvert} \\
 &=& \frac{1}{\sqrt{n(n+1)}}.
\end{eqnarray*}
\end{proof}

\begin{cor}
Let $X$ be an $n$-dimensional locally finite simplicial complex. Then with respect to the standard metric on $X$
\begin{eqnarray*}
 \mesh (X) &=& \sqrt{2}, \\
 \comesh (X) &=& \frac{1}{\sqrt{n(n+1)}},
\end{eqnarray*}
so using the control map $\id: X\to X$, $X$ has a tame and bounded triangulation. 
\end{cor}

\begin{defn}\label{barycentricsubdivision}
Let $\sigma = v_0\ldots v_n$ be an $n$-simplex embedded in Euclidean space. We define the \textit{barycentric subdivision of $\sigma$}, $Sd\, \sigma$, by
\[Sd\, \sigma := \bigcup_{\sigma_0<\ldots<\sigma_k\leqslant \sigma}{\widehat{\sigma_0}\ldots\widehat{\sigma_k}}.\] We denote the $i^{th}$ iterated barycentric subdivision of $\sigma$ by $Sd^i\,\sigma$.   
\qed\end{defn}

\begin{lem}
Let $\sigma$ be an $n$-simplex embedded in Euclidean space. Then with control map the embedding all simplices $\tau \in Sd\, \sigma$ satisfy \[\diam(\tau) < \frac{n}{n+1}\diam(\sigma).\]
\end{lem}
\begin{proof}
See page $120$ of \cite{hatcher}.
\end{proof}

\begin{cor}
Let $X$ be an $n$-dimensional simplicial complex, $n<\infty$, embedded in Euclidean space. Then with control map the embedding \[\mesh(Sd\, X) < \frac{n}{n+1}\mesh(X).\]
\end{cor}

\begin{rmk}
There may be a similar result putting a lower bound on the comesh of a subdivided simplex, but we shall not need it in this thesis.
\qed\end{rmk}

\begin{defn}\label{barycentricsubdivisionX}
We extend the definition of barycentric subdivision from embedded simplices to abstract locally finite simplicial complexes by subdividing each standardly embedded simplex and identifying common faces as before. We write this as \[Sd\, X:=\bigcup_{\sigma_0<\ldots<\sigma_k\in X}{\widehat{\sigma_0}\ldots\widehat{\sigma_k}},\]again denoting the $i^{th}$ iterated barycentric subdivision of $X$ by $Sd^i\, X$.
\qed\end{defn}

\begin{cor}
Let $X$ be a locally finite $n$-dimensional simplicial complex with $n<\infty$. Then measuring in $X$ with the standard metric \[\mesh(Sd\, X) < \frac{n}{n+1}\mesh(X).\]
\end{cor}

\begin{rmk}
We had two choices with how to measure subdivisions:
\begin{enumerate}[(i)]
 \item Like in \cite{bartelssqueezing}, when subdividing a locally finite complex $X$ we could have used the standard metric on $Sd\, X$ given by Lemma \ref{standardmetric}. The effect this would have had is to increase the diameter of an $n$-simplex when subdivided by a factor of at least $\dfrac{n+1}{n}$, but to leave the mesh and comesh unchanged. 
 \item Instead, what we have done is to choose to \textbf{measure all subdivisions of $X$ in the original unsubdivided $X$ via the identity map on geometric realisations}. With this choice $\diam(Sd^i\,\sigma) = \diam(\sigma)$ for all $\sigma\in X$ and for finite-dimensional $X$ 
\begin{equation}\label{subdivscales}
\mesh(Sd\, X) < \frac{\mathrm{dim}(X)}{\mathrm{dim}(X)+1}\mesh(X). 
\end{equation}
\end{enumerate}
\qed\end{rmk}

We reiterate as it is extremely important:
\begin{rmk}
If our control map for $X$ is $\id:X\to (X,d)$, where $d$ is the standard metric on $X$, then the control map for the barycentric subdivision is \[ \id:Sd\, X \to (X,d).\] 
\qed\end{rmk}

\begin{defn}\label{dualcells}
For any simplex $\sigma\in X$ we define the \textit{closed dual cell} , $D(\sigma, X)$, by \[D(\sigma,X):= \{ \widehat{\sigma_0}\ldots\widehat{\sigma_k}\in Sd\, X\,|\,\sigma\leqslant\sigma_0<\ldots<\sigma_k\in X\}\] with boundary \[\partial D(\sigma,X):= \{ \widehat{\sigma_0}\ldots\widehat{\sigma_k}\in Sd\, X\,|\,\sigma<\sigma_0<\ldots<\sigma_k\in X\}.\] We will call the interior of the closed dual cell the \textit{open dual cell} and denote it by
\[\mathring{D}(\sigma,X):= \{ \widehat{\sigma_0}\ldots\widehat{\sigma_k}\in Sd\, X\,|\,\sigma=\sigma_0<\ldots<\sigma_k\in X\}.\]  Note that for open dual cells $\mathring{D}(\sigma,X)= D(\sigma,X) - \partial D(\sigma,X)$.

\qed\end{defn}
\begin{ex}
\Figref{Fig:new0} shows a $2$-simplex $\sigma$ and examples of an open and a closed dual cell in its barycentric subdivision.
\begin{figure}[ht]
\begin{center}
{
\psfrag{0}[c][c]{$\rho_0$}
\psfrag{1}{$\rho_1$}
\psfrag{2}{$\rho_2$}
\psfrag{01}{$\tau_0$}
\psfrag{02}[bl][bl]{$\tau_1$}
\psfrag{12}{$\tau_2$}
\psfrag{s}{$\sigma$}
\psfrag{D}[br][br]{$D(\tau_0,\sigma)$}
\psfrag{oD}[bl][bl]{$\mathring{D}(\rho_2,\sigma)$}
\includegraphics[width=9cm]{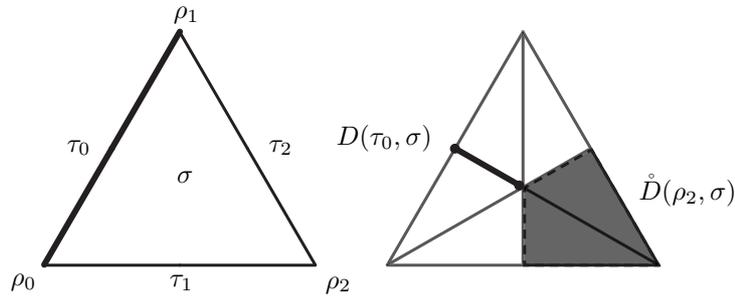}
}
\caption{Dual Cells}
\label{Fig:new0}
\end{center}
\end{figure}
We shall use this labelling of subsimplices of $\sigma$ in all $2$-simplex examples in this thesis.
\qed\end{ex}

\newpage \begin{rmk}
Any point $x\in X$ belongs to both a unique open simplex $\mathring{\sigma} \in X$ and a unique open dual cell $\mathring{D}(\tau,X)\subset Sd\, X.$ 
We will often think of $X$ as decomposed either into open simplices or open dual cells.\qed\end{rmk}

\section{Orientations and cellulations}
\begin{defn}\label{simporient}
Let $\sigma=v_0\ldots v_n$ be an $n$-simplex. An \textit{orientation} of a simplex is an equivalence class of choices of orderings $[v_{i_0},\ldots,v_{i_n}]$ of its vertices, where $[v_{i_0},\ldots,v_{i_n}]=[v_{j_0},\ldots,v_{j_n}]$ if and only if $(v_{i_0}\ldots v_{i_n})=\rho(v_{j_0}\ldots v_{j_n})$ for an even permutation $\rho \in S_n$.
\qed\end{defn}

\begin{defn}
Let a product of simplices $\sigma\times\tau$ be called a \textit{cell}. Another cell $\sigma^\prime\times\tau^\prime$ such that $\sigma^\prime<\sigma$ and $\tau^\prime<\tau$ is called a \textit{subcell}. A subcell is called a \textit{face} if it is codimension $1$.
\qed\end{defn}

\begin{defn}
A \textit{cellulation} of a simplicial complex $X$ is a decomposition of $X$ as a union of cells, such that the intersection of any two cells is a common subcell. 
\qed\end{defn}

\begin{ex}
The product of a $2$-simplex and a $1$-simplex is a cellulation, but if we subdivide it into a $3$-simplex and a square-based pyramid intersecting in a $2$-simplex this is not a cellulation. This is illustrated in \Figref{Fig:pyramid}.
\begin{figure}[ht]
\begin{center}
{
\includegraphics[width=4.5cm]{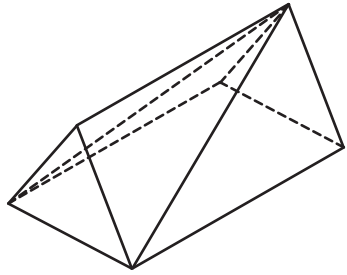}
}
\caption{A subdivision that is not a cellulation.}
\label{Fig:pyramid}
\end{center}
\end{figure}
\qed\end{ex}

\begin{defn}\label{setT}
Let $\tau\times\sigma$ be a cell triangulated with a choice of orientation for each $(|\tau|+|\sigma|)$-simplex. We say this triangulation is \textit{consistently oriented} if with respect to these orientations the simplicial boundary map $d_{|\tau|+|\sigma|}$ is only non-zero to simplices in the boundary $\partial(\tau\times\sigma)$. 

Let $\mathbb{T}(\tau\times\sigma)$ denote the set of consistently oriented triangulations of the cell $\tau\times\sigma$. We will use the letter $u$ to refer to elements of $\mathbb{T}(\tau\times\sigma)$ and $-u$ will be the same triangulation as $u$ but with the opposite orientation. 

Let $\tau^\prime\times \sigma^\prime$ be a face of $\tau\times\sigma$. For any $u\in\mathbb{T}(\tau\times\sigma)$, define $u|_{\tau^\prime\times \sigma^\prime} \in \mathbb{T}(\tau^\prime\times \sigma^\prime)$ to be image of $u$ under the simplicial boundary map $d_{|\sigma|+|\tau|}$ restricted to the face $\tau^\prime\times \sigma^\prime$. 
\qed\end{defn}

In order to define an orientation of a cell we first define an equivalence relation on the set of consistently orientated triangulations of the cell, then quotienting out by this equivalence we define an orientation to be an equivalence class of consistently orientated triangulations.  

Suppose inductively that for all cells $\tau^\prime\times \sigma^\prime$ of dimension less that dim$(\tau\times\sigma)$ we have an equivalence relation $\sim$ on $\mathbb{T}(\tau^\prime\times \sigma^\prime)$ such that $\mathbb{T}(\tau^\prime\times \sigma^\prime)/\sim$ has just two elements. 

Define an equivalence relation $\sim$ on $\mathbb{T}(\tau\times\sigma)$ by saying $u\sim u^\prime$ if \[[u|_{\tau^\prime\times \sigma^\prime}] = [u^\prime|_{\tau^\prime\times \sigma^\prime}] \in \mathbb{T}(\tau^\prime\times \sigma^\prime)/\sim\]for a face $\tau^\prime\times \sigma^\prime\subset \tau\times \sigma$. We see that this equivalence is well-defined and independent of the face used to define it by noting that all faces can be consistently oriented to form a cycle and that the image under the simplicial boundary map of a consistent orientation of $\tau\times\sigma$ is a consistent orientation of the boundary. 

Since $\mathbb{T}(\tau^\prime\times \sigma^\prime)/\sim$ has two elements this means that there are at most two equivalence classes in $\mathbb{T}(\tau\times\sigma)/\sim$, but since $[u|]=-[-u|]$ we get precisely two classes in $\mathbb{T}(\tau\times\sigma)/\sim$. 

Note that for simplices, $\mathbb{T}(\sigma)/\sim$ is just the set of orientations of $\sigma$ in the conventional sense. Thus the base case in our induction is the $0$-simplex with its usual orientation in $\{\pm 1\}$.

\begin{defn}
Define an \textit{orientation} of a cell $\tau\times\sigma$ to be an equivalence class of consistently oriented triangulations of $\tau\times\sigma$ in $\mathbb{T}(\tau\times\sigma)/\sim$.
\qed\end{defn}

\begin{defn}\label{prodorient}
Let $\sigma = v_0\ldots v_n$ and $\tau = w_o\ldots w_m$ be $n$- and $m$-simplices respectively. Let $\sigma$ have orientation $[\sigma]=[v_0,\ldots,v_n]$ and let $\tau$ have orientation $[\tau]=[w_0,\ldots, w_m]$. Define the \textit{product orientation} $[\sigma]\times[\tau]$ of the cell $v_0\ldots v_n\times w_0\ldots w_m$ to be the orientation determined by giving the simplex $(v_0\times w_0)\ldots(v_n\times w_0)(v_n\times w_1)\ldots(v_n\times w_m)$ in the subdivided product the orientation \[ [v_0\times w_0,\ldots,v_n\times w_0,v_n\times w_1,\ldots,v_n\times w_m] \] and orienting all other $(n+m)$-simplices consistently.
\qed\end{defn}

\begin{lem}\label{Lem:superderivation}
The simplicial boundary map behaves as a super-derivation with respect to the product orientation: \[d_{\Delta_*(\sigma\times\tau)}([\sigma]\times[\tau]) = d_{\Delta_*(\sigma)}[\sigma] \times [\tau] + (-1)^{|\sigma|}([\sigma]\times d_{\Delta_*(\tau)}[\tau]). \]
\end{lem}

\begin{proof}
The orientations $d_{\Delta_*(\sigma)}[\sigma] \times [\tau]$ and $(-1)^{|\sigma|}([\sigma]\times d_{\Delta_*(\tau)}[\tau])$ are determined by the orientation of a single $(n+m-1)$-simplex in each of $\partial\sigma\times\tau$ and $\sigma\times\partial\tau$. The product orientation $[\sigma]\times[\tau]$ is determined by the orientation \[ [v_0\times w_0,\ldots,v_n\times w_0,v_n\times w_1,\ldots,v_n\times w_m].\]Applying the simplicial boundary map to this we compute that 
\begin{eqnarray*}
&& d_{\Delta_*(\sigma\times\tau)}[v_0\times w_0,\ldots,v_n\times w_0,v_n\times w_1,\ldots,v_n\times w_m]\\
&=& [d_{\Delta_*(\sigma)}[\sigma]\times[w_0],v_n\times w_1,\ldots,v_n\times w_m] \\ 
&& + (-1)^n[v_0\times w_0,\ldots,v_{n-1}\times w_0, [v_n]\times d_{\Delta_*(\tau)}[\tau]] \\
&& - (-1)^n[v_0\times w_0,\ldots,v_{n-1}\times w_0, v_n\times w_1,\ldots, v_n\times w_m].
\end{eqnarray*}
where the condensed notation \[[d_{\Delta_*(\sigma)}[\sigma]\times[w_0],v_n\times w_1,\ldots,v_n\times w_m]\] means \[\sum_{i=0}^{|\sigma|}(-1)^i[v_0\times w_0,\ldots,\widehat{v_i\times w_0},\ldots,v_n\times w_1,\ldots,v_n\times w_m]\]and similarly for the second term.

The restriction of the right hand side to $\partial\sigma\times\tau$ is $[d_{\Delta_*(\sigma)}[\sigma]\times[w_0],v_n\times w_1,\ldots,v_n\times w_m]$ which is precisely the orientation determined by $d_{\Delta_*(\sigma)}[\sigma]\times[\tau]$. The restriction to $\sigma\times \partial\tau$ is $(-1)^n[v_0\times w_0,\ldots,v_{n-1}\times w_0, [v_n]\times d_{\Delta_*(\tau)}[\tau]]$ which is the orientation determined by $(-1)^n([\sigma]\times d_{\Delta_*(\tau)}[\tau])$ as required.
\end{proof}

\begin{ex}
Let $\sigma=v_0v_1$ with orientation $[v_0,v_1]$ and let $\tau=w_0w_1$ with orientation $[w_0,w_1]$, then the product orientation $[\sigma]\times[\tau]$ is $\sigma\times\tau$ decomposed as \[(v_0\times w_0)(v_1\times w_0)(v_1\times w_1)\cup (v_0\times w_0)(v_0\times w_1)(v_1\times w_1)\] with orientations \[ [v_0\times w_0,v_1\times w_0,v_1\times w_1] - [v_0\times w_0,v_0\times w_1,v_1\times w_1].\] The product the other way, $[\tau]\times [\sigma]$, is $\sigma\times \tau$ decomposed the same way but with the opposite orientation, namely \[ [v_0\times w_0,v_0\times w_1,v_1\times w_1] - [v_0\times w_0,v_1\times w_0,v_1\times w_1].\]
\vspace{-7mm}\begin{figure}[ht]
\begin{center}
{
\psfrag{sxt}[][]{$[\sigma]\times[\tau]$}
\psfrag{txs}[][]{$[\tau]\times[\sigma]$}
\psfrag{s}[r][r]{$\sigma$}
\psfrag{t}[b][b]{$\tau$}
\includegraphics[width=7cm]{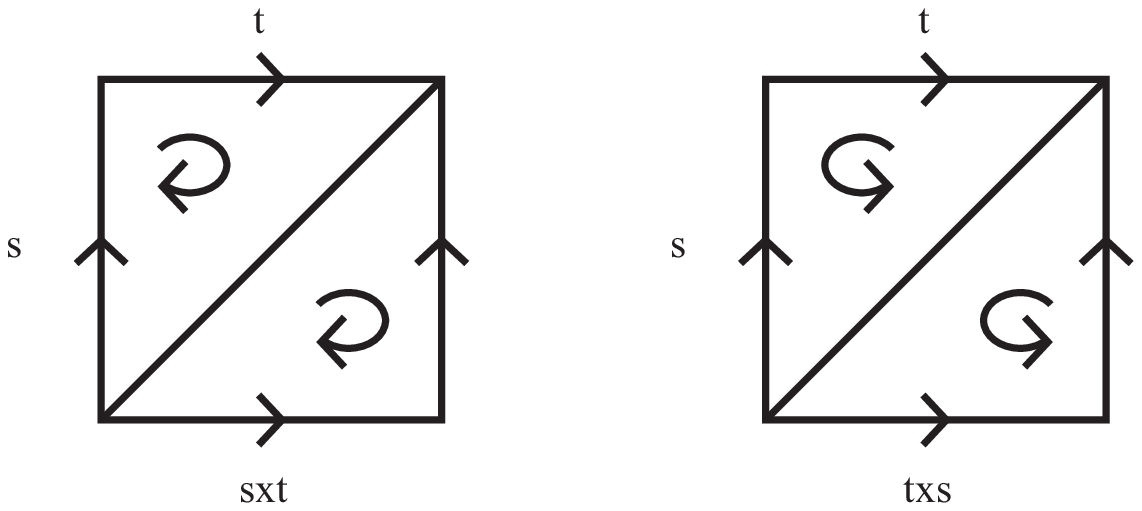}
}
\caption{Product orientations.}
\label{Fig:productorients}
\end{center}
\end{figure}
\qed\end{ex}

\section{The fundamental $\ep$-subdivision cellulation}
Given a simplicial complex $X$ and its barycentric subdivision $Sd\, X$, we can triangulate the prism $||X||\times [0,1]$ so that $||X||\times \{0\}$ is $X$ and $||X||\times\{1\}$ is $Sd\, X$ (see \cite{hatcher} page 112). The slices in between, $||X||\times\{t\}$ for $t\in (0,1)$, form a continuous family of cellulations of $||X||$ from $X$ to $Sd\, X$. Clearly there is straight line homotopy between $||X||\times\{t\}$ and $||X||\times\{0\}=X$ through these cellulations. In this section we mimic this family of cellulations but take care to control the bound of the straight line homotopies to $X$. The cellulation that is a distance $\ep$ from $X$ will be called the fundamental $\ep$-subdivision cellulation of $X$. 

\begin{defn}\label{Defn:fundamentalcellulation}
For $\ep<\comesh(X)$ define the \textit{fundamental $\ep$-subdivision cellulation of $X$}, denoted $X_\ep^\prime$, to be the cellulation with
\begin{enumerate}[(i)]
 \item a vertex $\Gamma_v(v)$ at each vertex in $X$ and for all $\tau\geqslant v$, a vertex $\Gamma_\tau(v)$ at the intersection of the $1$-simplex $\widehat{v}\widehat{\tau}$ and the sphere of radius $\ep$ centred at $v$: \[\Gamma_\tau(v):= \partial\overline{B_\ep(v)} \cap \widehat{v}\widehat{\tau}. \]
 \item a $k$-simplex $\Gamma_\sigma(\tau)$ spanned by $\{\Gamma_\sigma(v_{j_0}),\ldots,\Gamma_\sigma(v_{j_k})\}$ for all subsimplices $\tau = v_{j_0}\ldots v_{j_k}\leqslant \sigma$ and for all $\sigma$.
 \item a product of simplices $\Gamma_{\sigma_0,\ldots \sigma_i}(\tau) \cong \tau\times \Delta^i$ defined iteratively with boundary \[\bigcup_{j=0}^i{\Gamma_{\sigma_0,\ldots,\widehat{\sigma_j},\ldots, \sigma_i}(\tau)} \cup \bigcup_{\rho<\tau}{\Gamma_{\sigma_0,\ldots, \sigma_i}(\rho).} \] 
\end{enumerate}
\qed\end{defn}

\begin{ex}\label{starstar}
Let $X$ be the $2$-simplex $\sigma$ labelled as in \Figref{Fig:new0} then the fundamental $\ep$-subdivision cellulation of $X$ is as in \Figref{Fig:fundamentalcellulation}. Each $\Gamma_{\sigma_0,\ldots,\sigma_i}(\tau)$ is the closed cell pointed to by the arrow.
\begin{figure}[ht]
\begin{center}
{
\psfrag{ep}{$\ep$}
\psfrag{t1st1}{$\Gamma_{\tau_0,\sigma}(\tau_0)$}
\psfrag{t1t1}{$\Gamma_{\tau_0}(\tau_0)$}
\psfrag{st1}{$\Gamma_{\sigma}(\tau_0)$}
\psfrag{r2t2sr2}{$\Gamma_{\rho_1,\tau_1,\sigma}(\rho_1)$}
\psfrag{r2sr2}{$\Gamma_{\rho_1,\sigma}(\rho_1)$}
\psfrag{t2sr2}{$\Gamma_{\tau_1,\sigma}(\rho_1)$}
\psfrag{r2t2r2}{$\Gamma_{\rho_1,\tau_1}(\rho_1)$}
\psfrag{r3sr3}[tr][tr]{$\Gamma_{\rho_2,\sigma}(\rho_2)$}
\psfrag{sr3}{$\Gamma_{\sigma}(\rho_2)$}
\psfrag{r3r3}{$\Gamma_{\rho_2}(\rho_2)$}
\psfrag{s}[][]{$\Gamma_{\sigma}(\sigma)$}
\includegraphics[width=11cm]{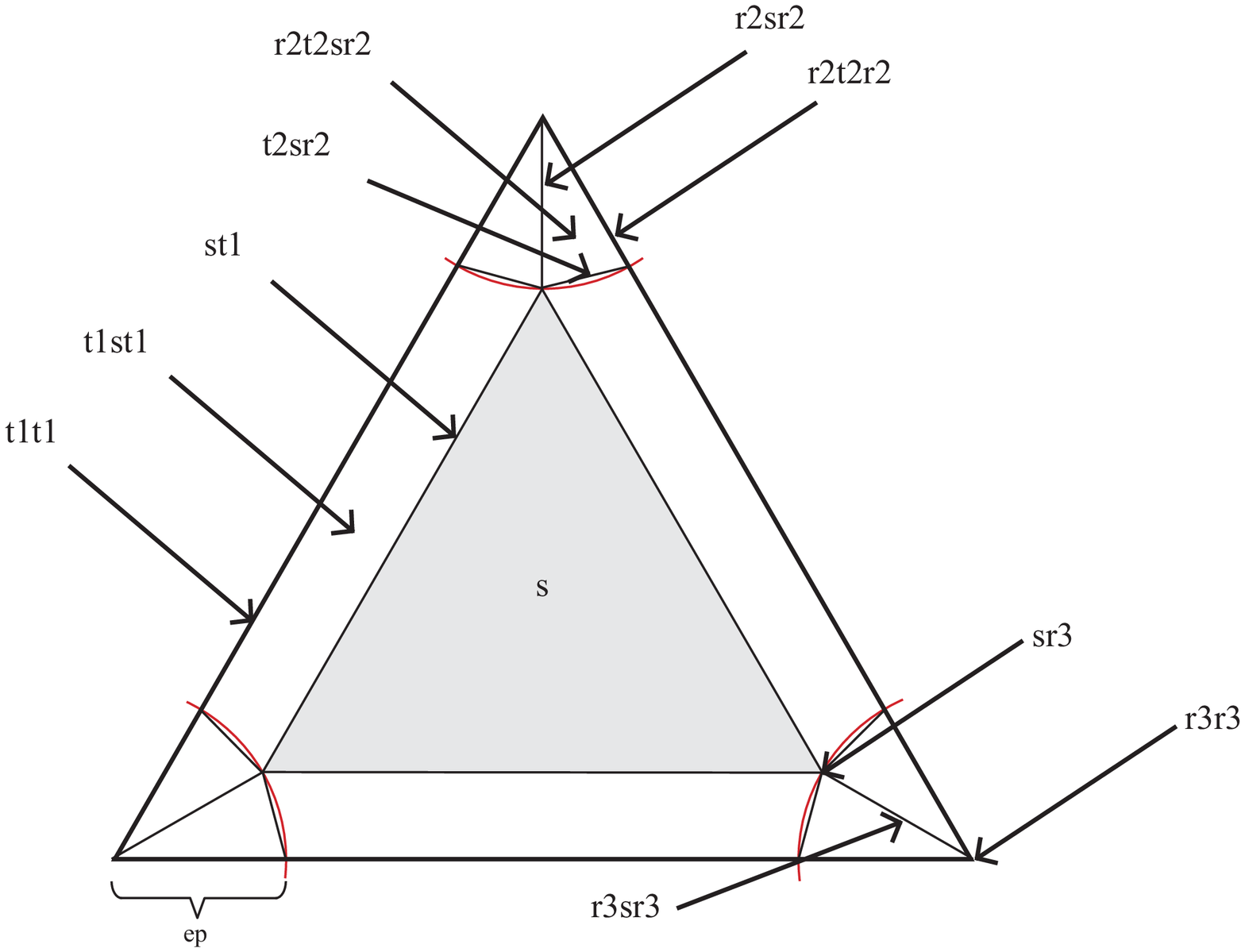}
}
\caption{The cellulation $X_\ep^\prime$ for a $2$-simplex.}
\label{Fig:fundamentalcellulation}
\end{center}
\end{figure}
\qed\end{ex}

\begin{rmk}
We refer to the $\Gamma_{\sigma_0,\ldots \sigma_i}(\tau)$ as \textit{higher homotopies} for the following reasons:
\begin{itemize}
 \item $\Gamma_{\sigma_0}(\tau)$ is to be thought of as the image of a PL $0$-homotopy, i.e. the image of a map $\Gamma_{\sigma_0}: \tau\times\Delta^0 \to X_\ep^\prime$.
 \item $\Gamma_{\sigma_0,\sigma_1}(\tau)$ is to be thought of as the image of a PL $1$-homotopy between the PL $0$-homotopies $\Gamma_{\sigma_0}$ and $\Gamma_{\sigma_1}$, i.e. the image of a map $\Gamma_{\sigma_0,\sigma_1}: \tau\times\Delta^1 \to X_\ep^\prime$.
 \item $\Gamma_{\sigma_0, \sigma_1, \sigma_2}(\tau)$ is to be thought of as the image of a PL $2$-homotopy between the PL $1$-homotopies $\Gamma_{\sigma_0,\sigma_1}\circ \Gamma_{\sigma_1,\sigma_2}$ and $\Gamma_{\sigma_0,\sigma_2}$, i.e. the image of a map $\Gamma_{\sigma_0, \sigma_1, \sigma_2}: \tau\times\Delta^2 \to X_\ep^\prime$. 
 \item similarly for higher homotopies.
\end{itemize}
\qed\end{rmk}

\begin{rmk}\label{epsubdivinprod}
Let $\tau=v_0\ldots v_n$ with barycentric coordinates $(s_0,\ldots, s_n)$. Each $\Gamma_{\sigma_0,\ldots, \sigma_m}(\tau)$ is linearly isomorphic to $\tau\times \Delta^m$ via the isomorphism 
\begin{eqnarray*}
\tau\times \Delta^m &\to& \Gamma_{\sigma_0,\ldots, \sigma_m}(\tau) \\
(s_0,\ldots, s_n,t_0,\ldots,t_m) &\mapsto& \sum_{i=0}^n\sum_{j=0}^m s_it_j\Gamma_{\sigma_i}(v_j). 
\end{eqnarray*}
Sometimes it will be convenient to think of this $\Delta^m$ as the simplex $\widehat{\sigma_0}\ldots\widehat{\sigma_m}$ so that $X_\ep^\prime$ will be embedded in $X\times Sd\, X$ as \[X_\ep^\prime = \bigcup_{\sigma_0<\ldots< \sigma_m \subset X}\sigma_0 \times \widehat{\sigma_0}\ldots\widehat{\sigma_m}.\] We will also think of the homotopies $\Gamma_{\sigma_0,\ldots, \sigma_m}(\tau)$ as images of a PL isomorphism
\begin{eqnarray*}
 \Gamma: X_\ep^\prime \subset X\times Sd\, X &\to& X_\ep^\prime \subset X \\
\tau\times \widehat{\sigma_0}\ldots\widehat{\sigma_m} &\mapsto& \Gamma_{\sigma_0,\ldots, \sigma_m}(\tau). 
\end{eqnarray*}
\qed\end{rmk}
It will be convenient to introduce the following notation.
\begin{defn}\label{gamp}
Let $\gam{\sigma_0,\ldots, \sigma_m}{\tau}$ denote the image under $\Gamma$ of $\mathring{\tau}\times \widehat{\sigma_0}\ldots\widehat{\sigma_m}$.
\qed\end{defn}

\begin{rmk}\label{crushinggammas}
For all $\ep< \comesh(X)$, the fundamental $\ep$-subdivision cellulation $X_\ep^\prime$ (viewed as a PL map as in Remark \ref{epsubdivinprod}) is homotopic to the triangulation of $X$ through fundamental $\delta$-subdivision cellulations, $0<\delta<\ep$.
\qed\end{rmk}

\begin{defn}\label{epsurgcellorient}
Given a choice of orientation $[\tau]$ for all the simplices $\tau\in X$, there is a canonical way to orient the simplices of $X_\ep^\prime$ considered as a subset of $X\times Sd\, X$. We give $\tau \times \widehat{\sigma_0}\ldots \widehat{\sigma_m}$ the product orientation \[ [\tau]\times [\widehat{\sigma_0},\ldots, \widehat{\sigma_m}].\] We can then use the isomorphism $\Gamma$ to orient $X_\ep^\prime\subset X$. We will call this the \textit{standard orientation of $X_\ep^\prime$.}
\qed\end{defn}

\begin{rmk}
Giving $X_\ep^\prime$ the standard orientation, Lemma \ref{Lem:superderivation} implies that
\begin{equation}\label{boundaryofgamma}
 \partial(\Gamma_{\sigma_0,\ldots, \sigma_m}(\tau)) = \bigcup_{j=0}^{m}(-1)^j\Gamma_{\sigma_0,\ldots, \widehat{\sigma_j}, \ldots, \sigma_m}(\tau) \cup \bigcup_{\rho<\tau}\Gamma_{\sigma_0,\ldots, \sigma_m}(\rho),
\end{equation}
 by which we mean that for $\widetilde{\tau}$ a face of $\widetilde{\sigma}\in\Gamma_{\sigma_0,\ldots, \sigma_m}(\tau)$ and \[\widetilde{\tau} \in \brc{\Gamma_{\sigma_0,\ldots, \widehat{\sigma_j}, \ldots, \sigma_m}(\tau)}{\Gamma_{\sigma_0,\ldots, \sigma_m}(\rho),}\]
the restriction of the simplicial boundary map is \[\brc{(-1)^j}{1}:\Delta_{|\widetilde{\sigma}|}(\widetilde{\sigma}) \to \Delta_{|\widetilde{\sigma}|-1}(\widetilde{\tau}).\] 
\qed\end{rmk}

\begin{ex}\label{Ex:SigmaOrients}
Continuing on from Example \ref{starstar} with $\sigma$ oriented as in \Figref{Fig:prethesis5},
\begin{figure}[ht]
\begin{center}
{
\psfrag{s}{$\sigma$}
\psfrag{p1}{$\rho_0$}
\psfrag{p2}[][]{$\rho_1$}
\psfrag{p3}{$\rho_2$}
\psfrag{t1}{$\tau_0$}
\psfrag{t2}{$\tau_1$}
\psfrag{t3}{$\tau_2$}
\psfrag{+}{$+$}
\psfrag{-}{$-$}
\includegraphics[width=5cm]{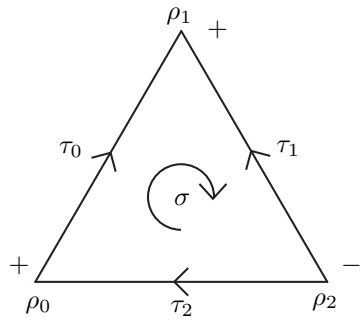}
}
\caption{A choice of orientations for all subsimplices of $\sigma$.}
\label{Fig:prethesis5}
\end{center}
\end{figure}
we see that the standard orientation of $X_\ep^\prime$ is as in \Figref{Fig:StandardOrient}.
\begin{figure}[ht]
\begin{center}
{
\psfrag{+}{$+$}
\psfrag{-}{$-$}
\includegraphics[width=9cm]{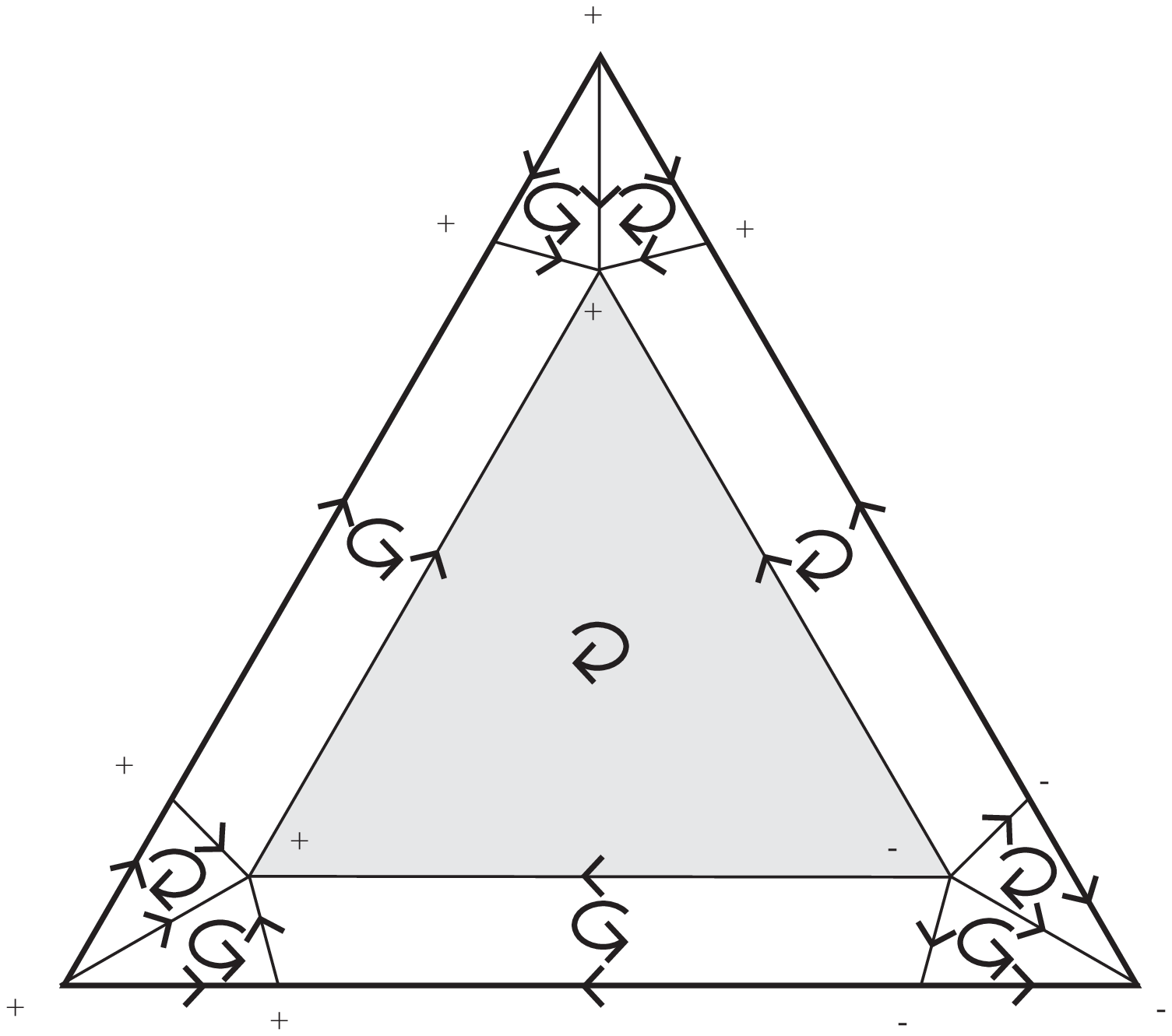}
}
\caption{Standard orientations of $X_\ep^\prime$.}
\label{Fig:StandardOrient}
\end{center}
\end{figure}
\qed\end{ex}

\newpage Before we are able to prove the main result of the topological half of the thesis, it will be useful to take a close look at the effect of barycentric subdivision on the simplicial chain complex of $X$.

\chapter{Subdivision and the simplicial chain complex}\label{chapfive} 
It is well known that the simplicial chain complex associated to any subdivision $X^\prime$ of a simplicial complex $X$ is chain equivalent\footnote{In fact this equivalence is a retraction.} to the simplicial chain complex of $X$ with the identity map $\id_X:X \to X^\prime$ inducing the chain equivalence \[s:= (\id_X)_*: \Delta^{lf}_*(X) \to \Delta^{lf}_*(X^\prime).\] 
Here we must use locally finite simplicial chains as our simplicial complex $X$ may not be finite. See Appendix \ref{appendixa} for a recap of locally finite homology. 

A choice of chain inverse $r:\Delta^{lf}_*(X^\prime) \to \Delta^{lf}_*(X)$ corresponds to a choice of simplicial approximation to the identity map $\id_X: X^\prime \to X$. We will use the same notation $r$ when thinking either of this map as a simplicial map or as the induced map on the level of chains. 

In this chapter we explicitly review the construction of the chain inverse $r$ and the chain homotopy $P:s\circ r \sim \id_{\Delta^{lf}_*(X^\prime)}$ for barycentric subdivision and iterated barycentric subdivisions. Importantly, it is shown that $r$ and $P$ can be chosen so as to retract a neighbourhood of each simplex $\sigma\in X$ onto that simplex. This fact will be used in subsequent chapters to prove squeezing results. 

\section{Barycentric subdivision}
In this section we deal with a single barycentric subdivision. We explain how to obtain chain inverses $r$ and for each $r$ a canonical chain homotopy $P$. Given a choice of $r$, $Sd\, X$ can be decomposed into a collection of homotopies that encode $r$ and the canonical $P$.

Consider the barycentric subdivision $Sd\, X$ and the map on simplicial chains induced by the identity $s: \Delta_*(\sigma)\to \Delta_*(Sd\, \sigma)$ for some $n$-simplex $\sigma = v_0\ldots v_n \in X$. Given a choice of orientation $[v_{i_0},\ldots, v_{i_n}]$ of $\sigma$, if we give all the $n$-simplices in $Sd\, \sigma$ the same orientation, i.e.\ the simplex $\widehat{v_{j_0}}\ldots\widehat{v_{j_0}\ldots v_{j_n}}$ is given the orientation $[\widehat{v_{j_0}\ldots v_{i_0}},\ldots, \widehat{v_{j_0}\ldots v_{i_n}}]$, then with respect to these orientations \[s= \left(\begin{array}{c}1 \\ \vdots \\ 1 \end{array}\right): \Delta_n(\sigma) \to \Delta_n(Sd\, \sigma) .\]

\begin{ex}
Suppose we have oriented the $2$-simplex $\sigma$ as in \Figref{Fig:prethesis5}, then the orientations of the subdivision that would make $s= \left(\begin{array}{c}1 \\ \vdots \\ 1 \end{array}\right)$ would be as in \Figref{Fig:prethesis5b}.
\begin{figure}[ht]
\begin{center}
{
\psfrag{+}[][]{$+$}
\psfrag{-}[][]{$-$}
\includegraphics[width=5cm]{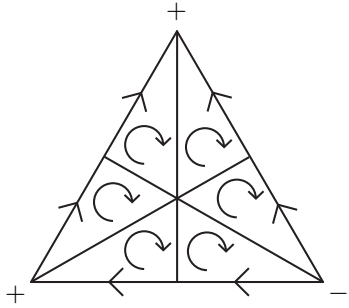}
}
\caption{Giving the same orientations to $Sd\, \sigma$.}
\label{Fig:prethesis5b}
\end{center}
\end{figure}
\qed\end{ex}

\begin{defn}\label{forest}
Let $Y$ be a closed subcomplex of $Sd\, X$. By a \textit{tree on $Y$} we mean a contractible graph whose vertices are vertices of $Y$ and whose edges are $1$-simplices in $Y$. A union of disjoint trees on $Y$ is called a \textit{forest on $Y$}. If a tree (resp. forest) on $Y$ contains every vertex in $Y$ we call it a \textit{spanning tree (resp. forest) on $Y$}.  
\qed\end{defn}

For the barycentric subdivision, a simplicial approximation to the identity $r:Sd\,X \to X$ corresponds to a choice for each simplex $v_0\ldots v_n$ of which vertex $v_j$ to map the barycentre $\widehat{v_0\ldots v_n}$ to. The following lemma shows that for a given choice of $r$, we can construct a canonical chain homotopy $P$ using a canonical spanning forest on $Sd\, X$. 

\begin{lem}\label{PFromTrees}
Let $s: \Delta^{lf}_*(X) \to \Delta^{lf}_*(Sd\, X)$ be the barycentric subdivision chain equivalence and let $r:\Delta^{lf}_*(Sd\, X) \to \Delta^{lf}_*(X)$ be a choice of chain inverse. Then 
\begin{enumerate}[(i)]
 \item we can define from $r$ a canonical spanning forest $T$ on $Sd\, X$,
 \item such a choice of spanning forest defines a canonical chain homotopy $P: s\circ r\sim \id_{\Delta^{lf}_*(Sd\, X)}$.
\end{enumerate}
\end{lem}

\begin{proof}
\begin{enumerate}[(i)]
 \item Let $r$ be given. For all vertices $v\in V(X)$ we construct a spanning tree $T_v$ on $r^{-1}(v)$ by connecting all simplices directly to $v$. The canonical spanning forest is then the union of all these disjoint trees \[T = \bigcup_{v\in V(X)}T_v.\] This is a spanning forest because each vertex of $Sd\, X$ is in $r^{-1}(v)$ for some $v$.
 \item We now define $P$ iteratively. Let any choices of orientations for simplices in $Sd\, X$ be given. For each $w\in V(Sd\, X)$ there exists a unique path in $T_{r(w)}$ from $w$ to $r(w)$.\footnote{Which for a canonical tree is the union of either zero or one $1$-simplices.} Set $P(w)$ to be the union of all the $1$-simplices in this path oriented consistently so that $dP = [w] -[r(w)]$. Continue to define $P$ iteratively by sending an $n$-simplex $w_0\ldots w_n$ to all the $(n+1)$-simplices bounded by \[w_0\ldots w_n \cup r(w_0)\ldots r(w_n) \cup \bigcup_{i=0}^{n}{P(w_0\ldots \widehat{w_i}\ldots w_n)}\] and oriented consistently so that $dP = [w_0\ldots w_n] - [r(w_0)\ldots r(w_n)] - Pd.$ We can orient correctly since we are effectively giving $P(w_0\ldots w_n)$ the product orientation $[\Delta^1 \times w_0\ldots w_n]$ which by Lemma \ref{Lem:superderivation} satisfies \[ d[\Delta^1 \times w_0\ldots w_n] = d[\Delta^1]\times [w_0\ldots w_n] - [\Delta^1]\times d[w_0\ldots w_n],\]i.e. \[ dP(w_0\ldots w_n) = [w_0\ldots w_n] - [r(w_0)\ldots r(w_n)] - Pd(w_0\ldots w_n).\]
\end{enumerate}
\end{proof}

\begin{rmk}
Note that not all possible chain homotopies $P$ can be obtained this way. Let \[T = \bigcup_{v\in V(X)}T_v\] be a spanning forest such that $T_v$ spans $r^{-1}(v)$ for all $v$, but where each $T_v$ need not be canonical. Part $(ii)$ of the proof above still defines a chain contraction $P$, but it is not a canonical one. 
\qed\end{rmk}

\begin{ex}
Let $X$ be the $2$-simplex $\rho_0\rho_1\rho_2$ labelled as in \Figref{Fig:prethesis5} and let us choose spanning forests $T$ and $T^\prime$ as in \Figref{Fig:NextGen6}, where the $T_{\rho_i}$ are all canonical and the $T^\prime_{\rho_i}$ are not.
\begin{figure}[ht]
\begin{center}
{
\psfrag{Tp1}{$T_{\rho_0}$}
\psfrag{Tp2}{$T_{\rho_1}$}
\psfrag{Tp3}{$T_{\rho_2}$}
\psfrag{Sp1}{$T^\prime_{\rho_0}$}
\psfrag{Sp2}{$T^\prime_{\rho_1}$}
\psfrag{Sp3}{$T^\prime_{\rho_2}$}
\includegraphics[width=9cm]{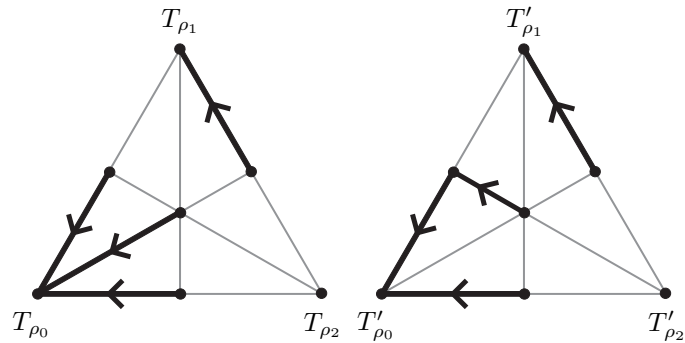}
}
\caption{Canonical spanning trees $T_{\rho_i}$ and non-canonical $T^\prime_{\rho_i}$.}
\label{Fig:NextGen6}
\end{center}
\end{figure}
Then the corresponding chain homotopies $P$ and $P^\prime$ are illustrated in \Figref{Fig:NextGen7}.
\begin{figure}[ht]
\begin{center}
{
\psfrag{t}{$\tau$}
\psfrag{stt}{$sr(\tau)$}
\psfrag{Pt}{$P(\tau)$}
\psfrag{St}{$P^\prime(\tau)$}
\includegraphics[width=9cm]{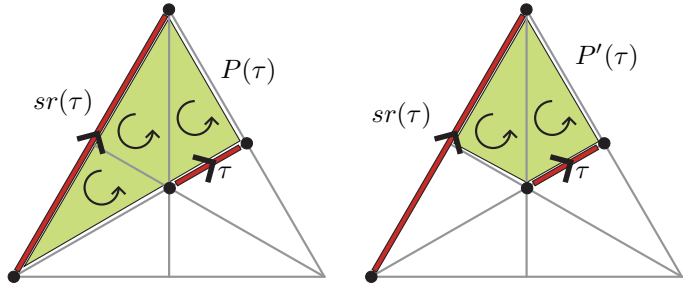}
}
\caption{$P(\tau)$ and $P^\prime(\tau)$ shaded in yellow for $\tau \in Sd\, \sigma$.}
\label{Fig:NextGen7}
\end{center}
\end{figure}
\qed\end{ex}

Simplicial approximations to the identity $r:Sd\, X \to X$ are characterised by which $|\sigma|$-simplex in $Sd\, \sigma$ is stretched by $r$ to fill $\sigma$ for each $\sigma\in X$.

\begin{lem}\label{tdeterminesGammas}
Chain inverses $r:\Delta^{lf}_*(Sd\, X) \to \Delta^{lf}_*(X)$ are in one-one correspondence with choices for all $\sigma\in X$ of a single $|\sigma|$-simplex $\Gamma_\sigma(\sigma)$ such that for $\tau\leqslant \sigma$, if $\Gamma_\sigma(\sigma) \cap \tau$ is a $|\tau|$-simplex then $\Gamma_\tau(\tau)= \Gamma_\sigma(\sigma)\cap \tau$.
\end{lem}

\begin{proof}
Given $r$ define $\Gamma_\sigma(\sigma)$ to be the unique simplex mapping onto $\sigma$ under $r$. If $\Gamma_\sigma(\sigma) \cap \tau$ is a $|\tau|$-simplex then it must necessarily be $\Gamma_\tau(\tau)$ since it is in $Sd\, \tau$ and maps onto $\tau$ under $r$.

Conversely, suppose we have consistently chosen $\Gamma_\sigma(\sigma)$ for all $\sigma\in X$. $\Gamma_\sigma(\sigma)$ intersects the subdivision of the $i$-skeleton of $\sigma$, $Sd\, (\sigma^{(i)})$, in precisely one $i$-dimensional face. Denote this face $\widehat{\sigma_0}\ldots\widehat{\sigma_i}$ where $|\sigma_j|=j$. Define $r(\widehat{\sigma_0}\ldots\widehat{\sigma_i})$ to be the unique vertex in $\sigma_i$ not also in $\sigma_{i-1}$. Doing this for all $\sigma \in X$ defines a simplicial approximation to the identity $r$ that maps $\Gamma_\sigma(\sigma)$ onto $\sigma$ for all $\sigma\in X$. 
\end{proof}

\begin{lem}\label{constructgam}
Given choices of orientations of all $\sigma\in X$, a choice of subdivision chain equivalence
\begin{displaymath}
 \xymatrix@1{ (\Delta^{lf}_*(X),d_{\Delta^{lf}_*(X)},0) \ar@<0.5ex>[rr]^-{s} && (\Delta^{lf}_*(Sd\,X),d_{\Delta^{lf}_*(Sd\,X)},P) \ar@<0.5ex>[ll]^-{r}
}
\end{displaymath}
decomposes $Sd\, X$ into the union of a collection of canonically oriented homotopies $\Gamma_{\sigma_0,\ldots,\sigma_i}(\sigma_0)$, one for each sequence of inclusions \[\sigma_0<\ldots <\sigma_i,\] analogous to the fundamental $\ep$-subdivision cellulation, and satisfying \[\partial\Gamma_{\sigma_0,\ldots, \sigma_i}(\sigma_0) = \bigcup_{j=0}^i(-1)^j\Gamma_{\sigma_0,\ldots,\widehat{\sigma}_j,\ldots, \sigma_i}(\sigma_0) \cup \Gamma_{\sigma_0,\ldots, \sigma_i}(\partial\sigma_0).\]
\end{lem}

\begin{proof}
Suppose we have specified orientations for all $\sigma\in X$ and have chosen a subdivision chain equivalence as in the statement of the lemma with $P$ the canonical chain homotopy. Let $\Gamma_\sigma(\sigma)$ be the unique $|\sigma|$-simplex mapping under $r$ to $\sigma$ given by Lemma \ref{tdeterminesGammas}. We think of each $\Gamma_\sigma$ as the map \begin{equation}\label{Gammasigmas} 
\Gamma_\sigma: \sigma \to \Gamma_\sigma(\sigma)
\end{equation} 
which is the linear isomorphism with inverse \[ r|: \Gamma_\sigma(\sigma)\to \sigma.\]

Now let $X_\ep^\prime$ be the fundamental $\ep$-subdivision cellulation for some $\ep$ thought of as embedded in $X\times Sd\, X$ as in Remark \ref{epsubdivinprod}. Define a PL map $\Gamma: X_\ep^\prime \subset X\times Sd\, X \to Sd\, X$ by mapping for all $\sigma_0<\ldots < \sigma_m \leqslant \tau$
\begin{eqnarray*}
 \tau\times \widehat{\sigma}_0\ldots\widehat{\sigma}_m &\to& Sd\, X, \\
 (x,t_0,\ldots,t_m) &\mapsto& \sum_{j=0}^{m}t_j\Gamma_{\sigma_j}(x)
\end{eqnarray*}
where the maps $\Gamma_{\sigma_j}$ are those from equation \ref{Gammasigmas}. This map is not to be confused with the one used to define the higher homotopies of $X_\ep^\prime$; we intentionally use the same notation as these decompositions are completely analogous.

The map $\Gamma:X_\ep^\prime \to Sd\,X$ surjects onto $Sd\, X$. Like with the labelling of $X_\ep^\prime$, let $\Gamma_{\sigma_0,\ldots,\sigma_m}(\tau)$ denote the image of $\tau\times \widehat{\sigma}_0\ldots\widehat{\sigma}_m$ under the map $\Gamma$. The cells of $X_\ep^\prime$ are oriented canonically with the product orientation as in Definition \ref{epsurgcellorient}. We push forward these orientations to give consistent orientations for all simplices in $Sd\, X$. Since the orientations are push forwards their orientations satisfy \[\partial\Gamma_{\sigma_0,\ldots, \sigma_i}(\sigma_0) = \bigcup_{j=0}^i(-1)^j\Gamma_{\sigma_0,\ldots,\widehat{\sigma}_j,\ldots, \sigma_i}(\sigma_0) \cup \Gamma_{\sigma_0,\ldots, \sigma_i}(\partial\sigma_0)\] since the orientations of $X_\ep^\prime$ do. 
\end{proof}

\begin{defn}\label{gamp2}
As in Definition \ref{gamp}, let $\gam{\sigma_0,\ldots, \sigma_i}{\tau}$ denote \[\Gamma_{\sigma_0,\ldots, \sigma_i}(\tau) \backslash \Gamma_{\sigma_0,\ldots, \sigma_i}(\partial \tau).\]
\qed\end{defn}

\begin{ex}
Let $\sigma$ be the $2$-simplex $\rho_0\rho_1\rho_2$ again labelled as in \Figref{Fig:prethesis5} and let $r$ be the following subdivision chain inverse
\begin{eqnarray*}
 f: Sd\, \sigma &\to& \sigma \\
 \{\widehat{\rho_0}, \widehat{\rho_0\rho_2}, \widehat{\rho_0\rho_1}, \widehat{\rho_0\rho_1\rho_2}\}  &\mapsto& \rho_0 \\
 \{\widehat{\rho_1}, \widehat{\rho_1\rho_2} \} &\mapsto& \rho_1 \\
 \{\widehat{\rho_2}\} &\mapsto& \rho_2,
\end{eqnarray*}
then following the construction in Lemma \ref{constructgam} $Sd\, \sigma$ is decomposed with $0$-homotopies as shown in \Figref{Fig:constructgam} and higher homotopies as shown in \Figref{Fig:constructgamhigher}. 
\begin{figure}[ht]
\begin{center}
{
\psfrag{p1p1}{$\gam{\rho_0}{\rho_0}$}
\psfrag{t1p1}{$\gam{\tau_0}{\rho_0}$}
\psfrag{t1p2}{$\gam{\tau_0}{\rho_1}=\gam{\rho_1}{\rho_1}$}
\psfrag{t1t1}{$\gam{\tau_0}{\tau_0}$}
\psfrag{sp1}{$\gam{\sigma}{\rho_0}$}
\psfrag{st1}{$\gam{\sigma}{\tau_0}$}
\psfrag{sp2}{$\gam{\sigma}{\rho_1}=\gam{\tau_1}{\rho_1}$}
\psfrag{st3}{$\gam{\sigma}{\tau_2}$}
\psfrag{ss}{$\gam{\sigma}{\sigma}$}
\psfrag{st2}{$\gam{\sigma}{\tau_1}=\gam{\tau_1}{\tau_1}$}
\psfrag{t3p1}{$\gam{\tau_2}{\rho_0}$}
\psfrag{t3t3}{$\gam{\tau_2}{\tau_2}$}
\psfrag{sp3}{$\gam{\sigma}{\rho_2}=\gam{\tau_1}{\rho_2}=\gam{\tau_2}{\rho_2}=\gam{\rho_2}{\rho_2}$}
\includegraphics[width=9cm]{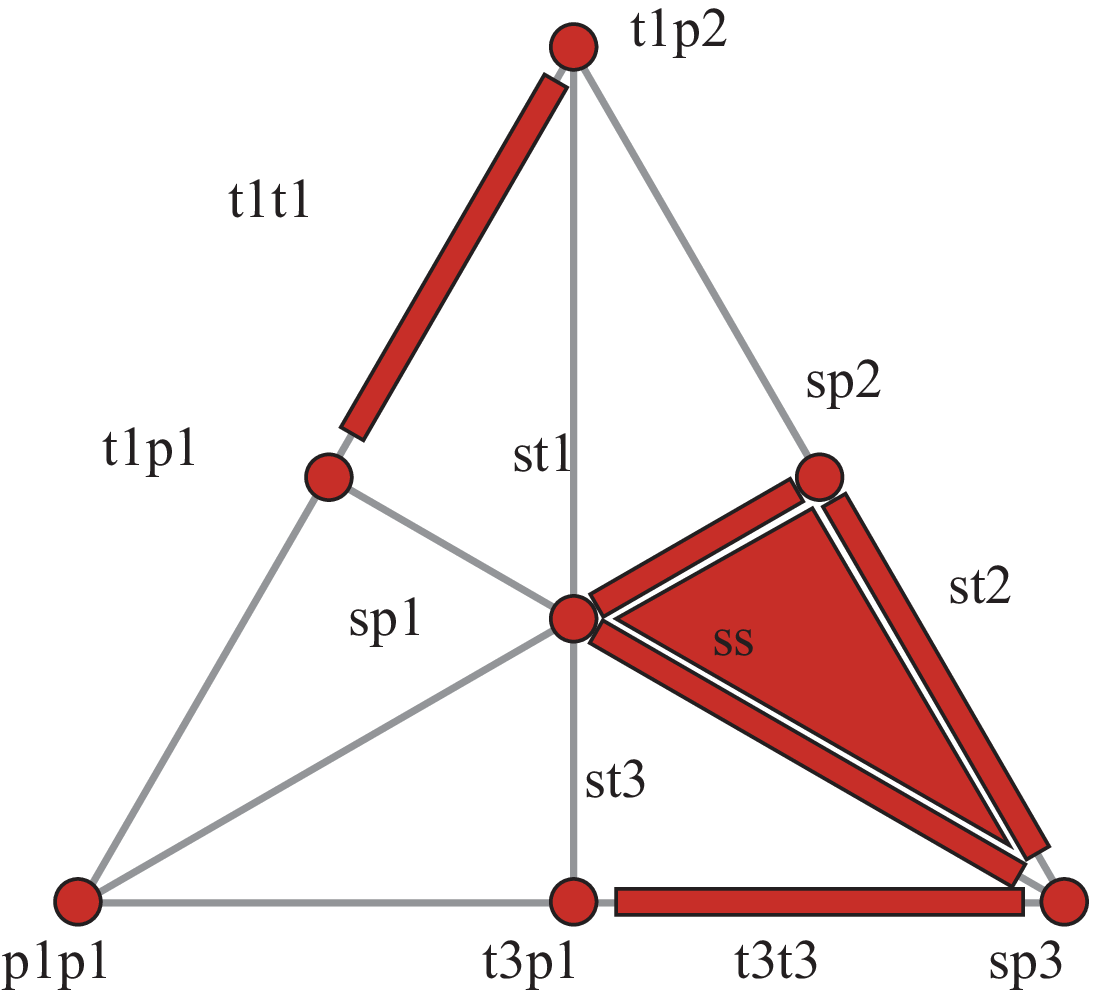}
}
\caption{Decomposing $Sd\, \sigma$: The $0$-homotopies $\gam{\tau}{\rho}$ for all $\rho\leqslant\tau\leqslant\sigma$.}
\label{Fig:constructgam}
\end{center}
\end{figure}
\begin{figure}[ht]
\begin{center}
{
\psfrag{t1sp2}{\shortstack[]{$\gam{\tau_0,\sigma}{\rho_1}=\gam{\rho_1,\tau_1}{\rho_1}$ \\$=\gam{\rho_1,\sigma}{\rho_1}$}}
\psfrag{t1st1}{$\gam{\tau_0,\sigma}{\tau_0}$}
\psfrag{t1sp1}{$\gam{\tau_0,\sigma}{\rho_0}$}
\psfrag{p1t1p1}{$\gam{\rho_0,\tau_0}{\rho_0}$}
\psfrag{p1t1sp1}{$\gam{\rho_0,\tau_0,\sigma}{\rho_0}$}
\psfrag{p1t3sp1}{$\gam{\rho_0,\tau_2,\sigma}{\rho_0}$}
\psfrag{p1sp1}{$\gam{\rho_0,\sigma}{\rho_0}$}
\psfrag{p1t3p1}{$\gam{\rho_0,\tau_2}{\rho_0}$}
\psfrag{t3sp1}{$\gam{\tau_2,\sigma}{\rho_0}$}
\psfrag{t3st3}{$\gam{\tau_2,\sigma}{\tau_2}$}
\includegraphics[width=13cm]{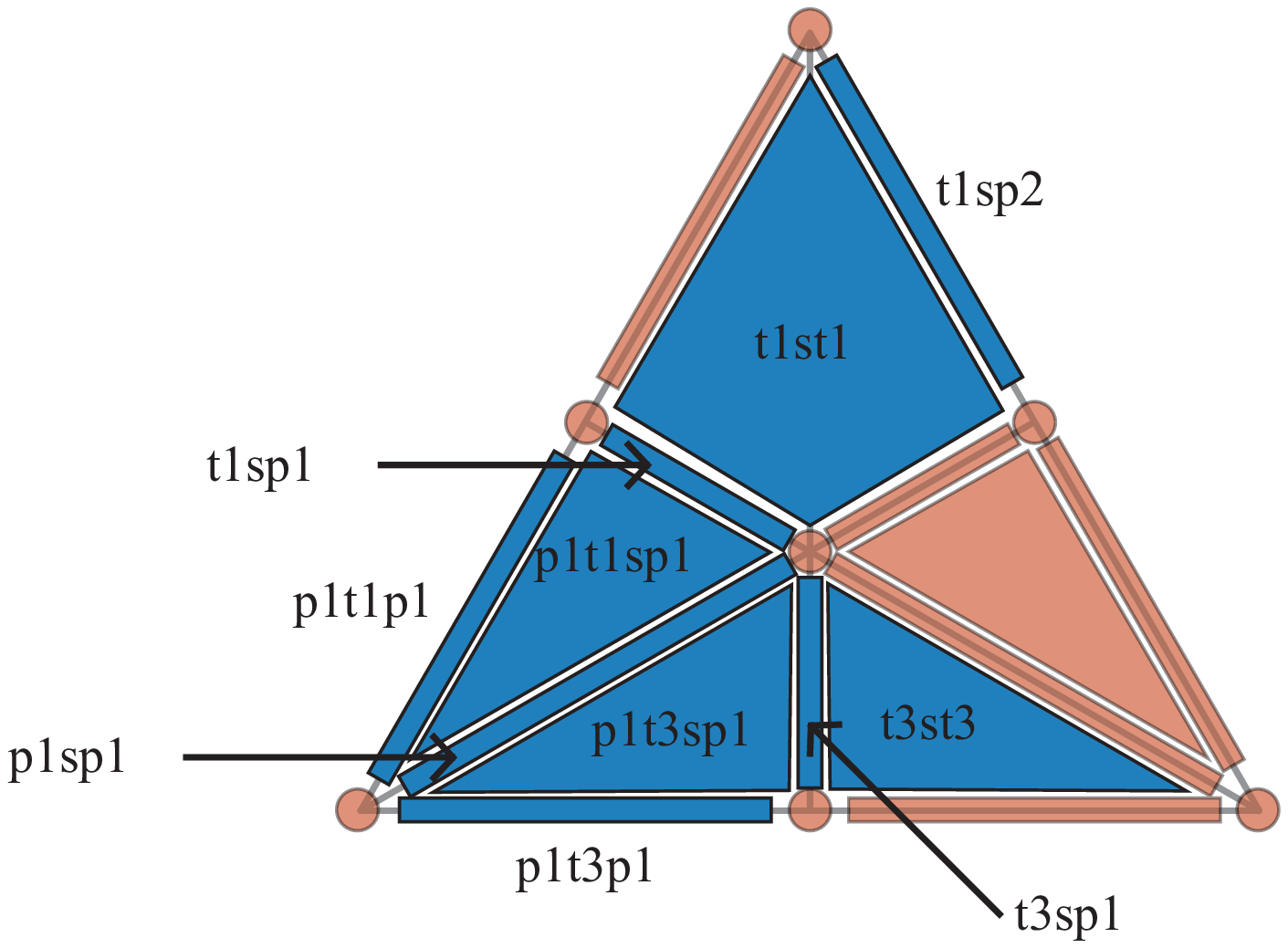}
}
\caption{Decomposing $Sd\, \sigma$: The higher homotopies $\gam{\sigma_1,\ldots,\sigma_i}{\tau}$.}
\label{Fig:constructgamhigher} 
\end{center}
\end{figure}

Further if we orient the subsimplices of $\sigma$ as in \Figref{Fig:prethesis5}, then the orientations attributed to the homotopies are as depicted in \Figref{Fig:constructgamorient}.
\begin{figure}[ht]
\begin{center}
{
\psfrag{+}{$+$}
\psfrag{-}{$-$}
\includegraphics[width=7cm]{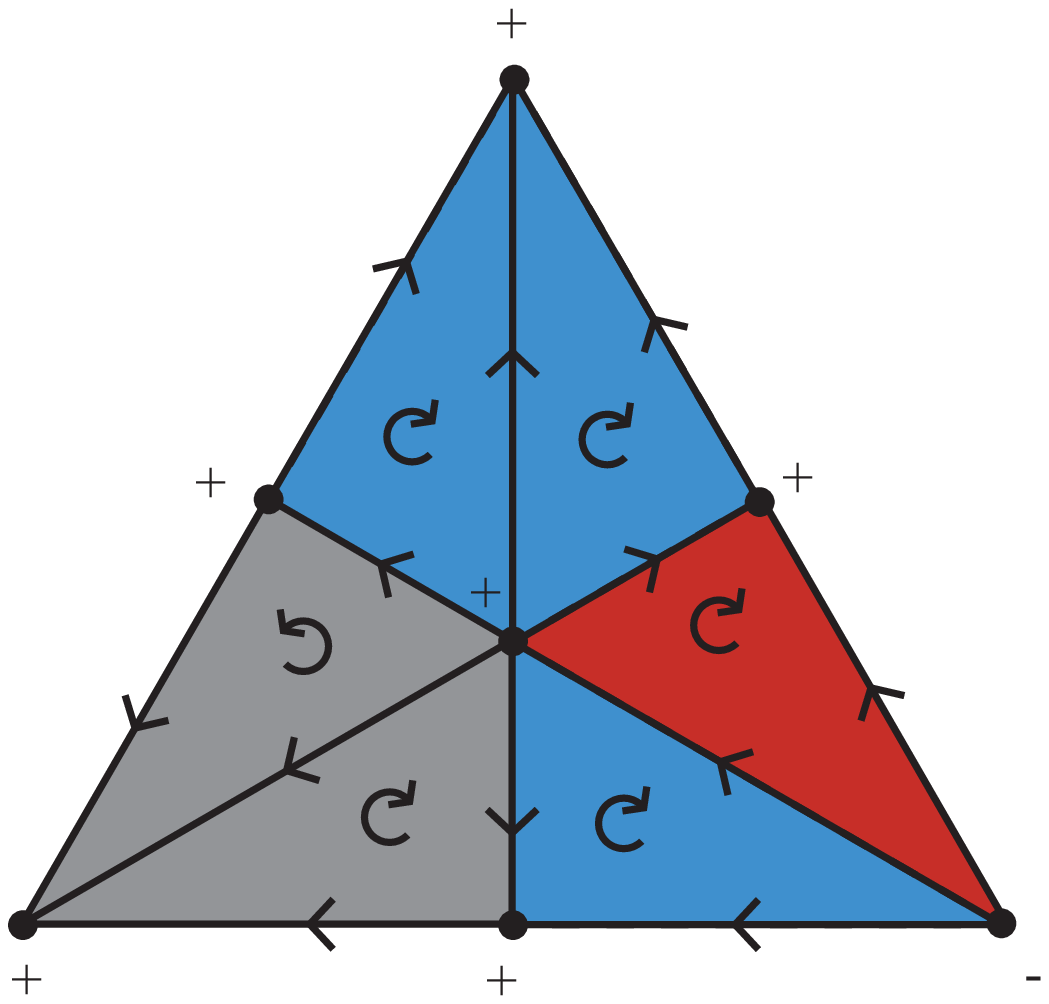}
}
\caption{Correct orientations for all $\gam{\rho_0,\ldots,\rho_i}{\rho}$.}
\label{Fig:constructgamorient}
\end{center}
\end{figure}
\qed\end{ex}

\begin{defn}\label{Isigma}
Thinking of $r:Sd\, X \to X$ as a simplicial map, for any simplex $\sigma\in X$, define $I_\sigma\subset Sd\, X$ to be the set of open simplices mapping onto $\mathring{\sigma}$ under $r$, or equivalently $I_\sigma$ is the union over all sequences of inclusions $\sigma\leqslant \sigma_0<\ldots<\sigma_i$ of $\Gamma_{\sigma_0,\ldots,\sigma_i}(\mathring{\sigma})$: 

\begin{equation}\label{Isigmadefn}
I_\sigma:= \bigcup_{\tau \in r^{-1}(\sigma)} {\mathring{\tau}} = \bigcup_{\sigma_0<\ldots<\sigma_i\leqslant\sigma}\Gamma_{\sigma_0,\ldots,\sigma_i}(\mathring{\sigma}).
\end{equation}
\qed\end{defn}

\begin{ex}
\Figref{Fig:prethesis4} shows what the various $I_\tau$, $\tau\leqslant\sigma$ are for our example of the $2$-simplex $\sigma$.
\begin{figure}[ht]
\begin{center}
{
\psfrag{s}[l][r]{$I_{\sigma}$}
\psfrag{p1}[l][r]{$I_{\rho_1}$}
\psfrag{p2}{$I_{\rho_2}$}
\psfrag{p3}[t][]{$I_{\rho_3}$}
\psfrag{t1}{$I_{\tau_0}$}
\psfrag{t2}{$I_{\tau_1}$}
\psfrag{t3}[tl][tr]{$I_{\tau_2}$}
\includegraphics[width=7cm]{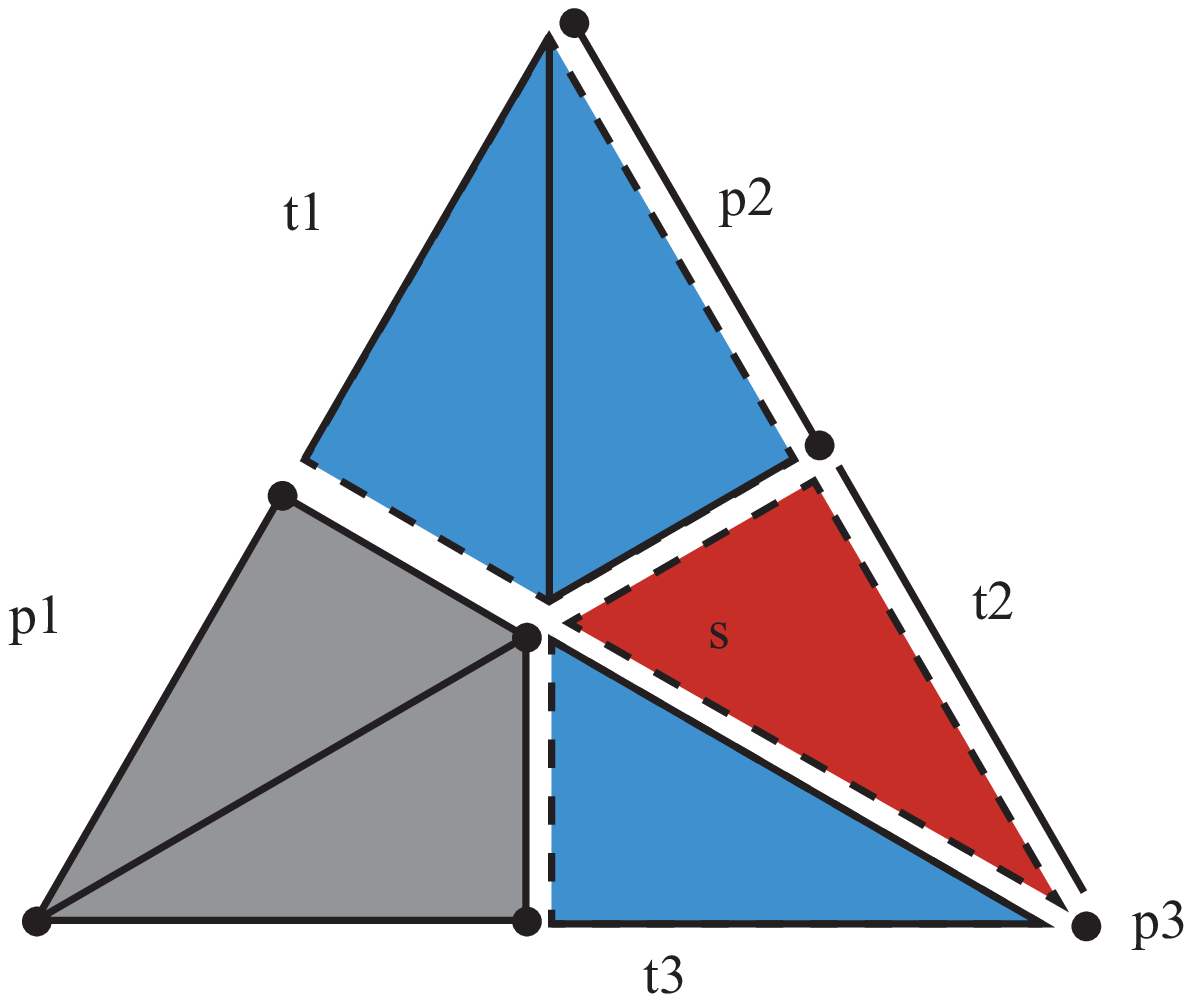}
}
\caption{The $I_\tau$'s for $\sigma$.}
\label{Fig:prethesis4}
\end{center}
\end{figure}
\qed\end{ex}

\begin{rmk}\label{Isigmadecomp}
Note that, just like $X$ can be decomposed into the disjoint union of all the open simplices, we can also decompose $Sd\, X$ into the disjoint union of all the $I_\sigma$, $\sigma\in X$.
\qed\end{rmk}

\section{Composing multiple subdivisions}
We would like to understand more general subdivisions than just a single barycentric subdivision. All subdivisions $X^\prime$ of a simplicial complex $X$ have $Sd^i\, X$ as a subdivision for sufficiently large $i$. So, a first step towards understanding more general subdivisions is to consider iterated barycentric subdivisions. For this we consider the effects of composing barycentric subdivision chain equivalences and how to obtain a chain inverse $r$ and canonical chain homotopy $P$ for the composition. 

\begin{lem}\label{circcheq}
Given chain equivalences
\begin{displaymath}
\xymatrix@1{ (C, d_C, \Gamma_C) \ar@<0.5ex>[rr]^-{f} && (D,d_D, \Gamma_D), \ar@<0.5ex>[ll]^-{g} & (D, d_D, \Gamma^\prime_D) \ar@<0.5ex>[rr]^-{h} && (E,d_E, \Gamma^\prime_E) \ar@<0.5ex>[ll]^-{i}
} 
\end{displaymath}
the composition
\begin{displaymath}
\xymatrix@1{ (C, d_C, \Gamma_C + g\circ\Gamma^\prime_D\circ f ) \ar@<0.5ex>[rr]^-{h\circ f} && (E,d_E, \Gamma^\prime_E + h\circ \Gamma_D\circ i) \ar@<0.5ex>[ll]^-{g\circ i} } 
\end{displaymath}
is a chain equivalence.
\end{lem}
\begin{proof}
Straightforward verification.
\end{proof}

\begin{deflem}\label{CanUseSingleTreeLem}
 Consider the composition 
\begin{equation} \label{eqi}
\xymatrix{ (\Delta^{lf}_*(X), d_{\Delta^{lf}_*(X)},0) \ar@<0.5ex>[rr]^-{s_2\circ s_1} && (\Delta^{lf}_*(Sd^2\,X), d_{\Delta^{lf}_*(Sd^2\,X)}, P_2 + s_2\circ P_1\circ r_2) \ar@<0.5ex>[ll]^-{r_1\circ r_2} }
\end{equation}
of two canonical subdivision chain equivalences
\begin{equation} \label{eqii}
\xymatrix{ (\Delta^{lf}_*(X), d_{\Delta^{lf}_*(X)},0) \ar@<0.5ex>[rr]^-{s_1} && (\Delta^{lf}_*(Sd\,X), d_{\Delta^{lf}_*(Sd\,X)}, P_1) \ar@<0.5ex>[ll]^-{r_1} }
\end{equation}\vspace{-3mm}
\begin{equation} \label{eqiii}
\xymatrix{ (\Delta^{lf}_*(Sd\,X), d_{\Delta^{lf}_*(Sd\,X)},0) \ar@<0.5ex>[rr]^-{s_2} && (\Delta^{lf}_*(Sd^2\,X), d_{\Delta^{lf}_*(Sd^2\,X)}, P_2) \ar@<0.5ex>[ll]^-{r_2} }
\end{equation}
\end{deflem}
\noindent as in Lemma \ref{circcheq}. Let \[T_1 = \bigcup_{v\in V(X)}{T_{1,v}}\]be the canonical trees used to define the canonical $P_1$ as in Lemma \ref{PFromTrees}. Similarly let \[T_2 = \bigcup_{w\in V(Sd\,X)}{T_{2,w}}\]be the canonical trees used to define the canonical $P_2$. Then $P_2 + s_2\circ P_1\circ r_2$ is defined by the tree \[T_3 = T_2 \cup s_2(T_1).\] 
We will call $P_2 + s_2\circ P_1\circ t_2$ the \textit{canonical chain homotopy} for the choice of chain inverse $r_2r_1$, and the tree $T_3$ used to define it will also be called \textit{canonical}.

\begin{proof}
$P_2$ is given by $T_2$: for all $v\in V(Sd^2\,X)$, $P_2(v)$ is the path along the edge $vr_2(v)$ in $T_2$ from $v$ to $r_2(v)\in V(Sd\, X)$. Similarly, $P_1r_2(v)$ is the path along the edge $r_2(v)r_1r_2(v)$ in $Sd\,X$ from $r_2(v)$ to $r_1(r_2(v))$. Then $s_2P_1r_2(v)$ is the same path considered in $Sd^2\,X$. 

Thus composing paths, $(P_2 + s_2P_1r_2)(v)$ is a path $v \to r_2(v) \to r_1r_2(v)$ in $T_2\cup s_2(T_1)$. This path cannot be a loop as it is the union of two straight paths, one in $T_2$ and the other in $s_2(T_1)$. Thus $T_2\cup s_2(T_1)$ is a collection of spanning trees for the $(r_1r_2)^{-1}(v)$ as required.
\end{proof} 
So we see that taking the union of the canonical trees provides us with our chain inverse for the composition.
\begin{rmk}
It should be noted that not all chain inverses $r$ to iterated barycentric subdivisions $s:\Delta^{lf}_*(X)\to \Delta^{lf}_*(Sd^i\,X)$ can be decomposed into compositions $r=r_1\circ\ldots\circ r_i$ of chain inverses $r_j: \Delta^{lf}_*(Sd^j\, X) \to \Delta^{lf}_*(Sd^{j-1}\,X)$.
\qed\end{rmk}
Generalising this to the composition of more than two barycentric subdivision chain equivalences is not that much more work.
\begin{cor}\label{CanUseSingleTreeCor}
For any iterated barycentric subdivision $s:\Delta^{lf}_*(X)\to \Delta^{lf}_*(Sd^i\,X)$ and choice of inverse $r$ of the form $r_1\circ\ldots\circ r_i$ there is a canonical chain homotopy 
\begin{displaymath}
\xymatrix{ (\Delta^{lf}_*(X), d_{\Delta^{lf}_*(X)},0) \ar@<0.5ex>[rr]^-{s} && (\Delta^{lf}_*(Sd^i\,X), d_{\Delta^{lf}_*(Sd^i\,X)}, P) \ar@<0.5ex>[ll]^-{r} }
\end{displaymath}
where \[P= P_i + s_iP_{i-1}r_i + \ldots + (s_i\circ\ldots\circ s_2)P_1(r_2\circ\ldots\circ r_i)\]
and is defined by the tree \[T = T_i \cup s_i(T_{i-1}) \cup \ldots \cup s_i\circ\ldots\circ s_2(T_1),\] where each $T_i$ is the canonical tree used to construct the chain homotopy $P_i$.
\end{cor}

\begin{proof}
First note that for any iterated barycentric subdivision $s=s_i\circ\ldots\circ s_1$. All we need to check is that for all $v\in Sd^i\, X$ the path \[(P_i + s_iP_{i-1}t_i + \ldots + (s_i\circ\ldots s_2)P_1(r_2\circ\ldots\circ r_i))(v) \]is not a loop in $Sd^i\,X$. This path is the union of edges \[ v \mapsto r_i(v) \mapsto \ldots \mapsto (r_1\circ\ldots\circ r_i)(v),\] where $r_j\circ\ldots\circ r_i(v) \mapsto r_{j-1}\circ\ldots\circ r_i(v)$ is an edge in $Sd^j\, X$. Suppose this is a loop, then we must have that $r_j\circ\ldots\circ r_i(v) = r_k\circ\ldots\circ r_i(v)$ for some $j,k$. Thinking about the path in reverse, since \[\sum_{l=1}^{n}(1/2)^l < 1\] for all $n<\infty$ the path cannot get back to a vertex in $Sd^j\, X$ once it has deviated from it as each subsequent path travels half as far as the previous. This forces all the $r_l\circ\ldots\circ r_i(v)$ for $l$ between $j$ and $k$ to be the same point, so the loop is just stationary at a single point. Hence the path is not a loop as claimed.
\end{proof}

\begin{ex}\label{ExUnionTrees}
Let $T_1$ and $T_2$ be canonical trees as given in \Figref{Fig:insert1}.
\begin{figure}[ht]
\begin{center}
{
\psfrag{T1}{$T_1$}
\psfrag{T2}{$T_2$}
\includegraphics[width=11cm]{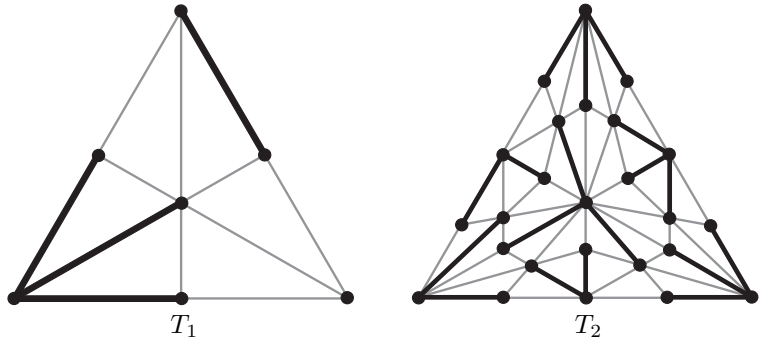}
}
\caption{Constructing trees $T_1$ and $T_2$.}
\label{Fig:insert1}
\end{center}
\end{figure}
Then their union $T = T_2 \cup s_2(T_1)$ as in \Figref{Fig:insert2} is canonical for $r_2r_1$.
\begin{figure}[ht]
\begin{center}
{
\psfrag{T3}[t][t]{$T$}
\includegraphics[width=5.5cm]{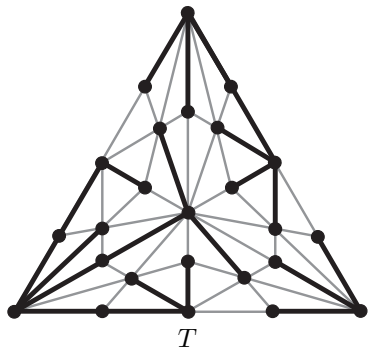}
}
\caption{The tree $T = T_2 \cup s_2(T_1)$.}
\label{Fig:insert2}
\end{center}
\end{figure}
\qed\end{ex}
Before, in the case of a single barycentric subdivision, we saw in Lemma \ref{constructgam} that a chain homotopy $P$ decomposed $Sd\, X$ into higher homotopies $\Gamma_{\sigma_0,\ldots, \sigma_i}(\sigma)$. The same is true for the $P$ obtained from iterated barycentric subdivisions:
\begin{cor}\label{BigTreeUnion}
Using $T = T_n \cup s_n(T_{n-1}) \cup \ldots \cup s_i\circ\ldots\circ s_2(T_1)$ we can decompose $Sd^n\,X$ into homotopies proceeding as in Lemma \ref{constructgam}.
\end{cor}

\begin{ex}
Using the union of trees $T = T_2 \cup s_2(T_1)$ from \Figref{Fig:insert2} we obtain homotopies as labelled in \Figref{Fig:insert3}.
\begin{figure}[ht]
\begin{center}
{
\psfrag{ss}{$\gam{\sigma}{\sigma}$}
\psfrag{p1t1sp1}{$\gam{\rho_1,\tau_0,\sigma}{\rho_1}$}
\psfrag{p0t1sp0}[br][br]{$\gam{\rho_0,\tau_0,\sigma}{\rho_0}$}
\psfrag{p0t3sp0}{$\gam{\rho_0,\tau_2,\sigma}{\rho_0}$}
\psfrag{t2st2}[b][b]{$\gam{\tau_1,\sigma}{\tau_1}$}
\psfrag{t3st3}[][]{$\gam{\tau_2,\sigma}{\tau_2}$}
\psfrag{t1st1}[b][b]{$\gam{\tau_0,\sigma}{\tau_0}$}
\includegraphics[width=10cm]{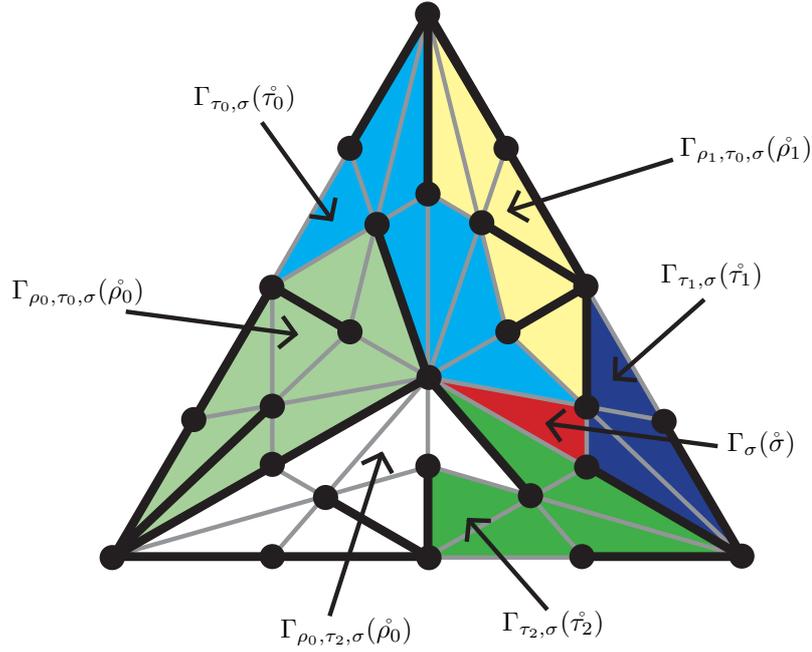}
}
\caption{Homotopies constructed from $T_2 \cup s_2(T_1)$.}
\label{Fig:insert3}
\end{center}
\end{figure}
\qed\end{ex}
Now, with the subdivision chain equivalences better understood, we can show how to carefully choose a chain inverse $r$ so that, for $P$ the canonical chain homotopy obtained from $r$, for all $\sigma\in X$, $P$ maps a neighbourhood in $\sigma$ of $\partial \sigma$ \textit{towards} $\partial\sigma$, and $r$ crushes this neighbourhood all the way to the boundary $\partial\sigma$. Having such an $r$ and $P$ will allow us later to prove squeezing results. The rough idea is that if a map $f$ maps a simplex to a small neighbourhood of that simplex, then $r\circ f$ will stay inside the simplex.

\begin{lem}\label{keylemma}
Let $X$ be a finite-dimensional, locally finite simplicial complex, so that in particular $0<\comesh(X)< \mesh(X) <\infty$, and let $s: \Delta^{lf}_*(X) \to \Delta^{lf}_*(Sd^i\,X)$ be some iterated barycentric subdivision. Then we may choose a chain inverse $r$ and a chain homotopy $P$ such that for all $\sigma\in X$ and for all $0<\ep < \rad(\sigma) - 2 \mesh(Sd^i\,\sigma)$
\begin{eqnarray*}
 P(\Delta_*(N_\ep(Sd^i(\partial\sigma))\cap \sigma )) &\subset& \Delta_{*+1}(N_\ep(Sd^i\, (\partial\sigma))\cap \sigma),\\
 r(\Delta_*(N_\ep(Sd^i(\partial\sigma))\cap \sigma )) &\subset& \Delta_{*}(Sd^i\, (\partial\sigma))
\end{eqnarray*}
where $N_\ep(Sd^i\,(\partial\sigma))$ is the union of all simplices in $Sd^i\, \sigma$ containing a point within $\ep$ of any point in $\partial\sigma$. 
\end{lem}

\begin{proof}
Since $\ep< \rad(\sigma) - 2 \mesh(Sd^i\,\sigma)$ a ball of radius $2\mesh(Sd^i\,\sigma)$ around the incentre of $\sigma$ is contained in the complement of $N_\ep(\partial\sigma)$. Therefore we can find a $|\sigma|$-simplex in $Sd^i\,X$ contained entirely in $Sd^i\,\sigma \backslash \overline{N_\ep(Sd^i\,(\partial\sigma))}$. Let $\Gamma^i_{\sigma}(\sigma)$ be a choice of such a $|\sigma|$-simplex. We construct $r$ in such a way as to stretch $\Gamma^i_{\sigma}(\sigma)$ to fill the whole of $\sigma$ whilst crushing everything else towards the boundary.

Each $|\sigma|$-simplex in $Sd^i\,\sigma$ is contained in exactly one simplex in $Sd^j\, \sigma$ for $j=0,\ldots,i$. Let $\Gamma^j_{\sigma}(\sigma)$ denote the unique $|\sigma|$-simplex in $Sd^j\, \sigma$ containing $\Gamma^i_{\sigma}(\sigma)$. By Lemma \ref{tdeterminesGammas}, each $r_i$ corresponds to consistent choices of top dimensional simplices in $Sd\, \tau$ to stretch to fill each simplex $\tau \in Sd^i\, X$. Thus the $\Gamma^j_{\sigma}(\sigma)$ determine the map $r_j$ on their vertices. For all vertices which are not vertices of $\Gamma^j_{\sigma}(\sigma)$ for some $j$ we proceed as follows.

We may extend our definition of $r_1$ from $\Gamma^1_{\sigma}(\sigma)$ to the whole of $Sd\,\sigma$ in any way we like. Then proceed iteratively defining each successive $r_j$ to map those barycentres $\widehat{\widetilde{\tau}}$ of a simplex $\widetilde{\tau}\in Sd^{j-1}\, \sigma$ which are not vertices of $\Gamma^j_{\sigma}(\sigma)$ to any vertex $v$ of $\widetilde{\tau}$ that is closer to the boundary $\partial\sigma$ than $\widetilde{\tau}$ is. We can always find a vertex closer otherwise this would contradict convexity of $\widetilde{\tau}$. Then taking $r=r_i\circ\ldots\circ r_1$ and the corresponding canonical $P$ satisfies the properties of the lemma by construction. The only simplex not crushed onto the boundary is $\Gamma^i_{\sigma}(\sigma)$ and this was chosen to be in $Sd^i\,\sigma \backslash \overline{N_\ep(Sd^i\,(\partial\sigma))}$.
\end{proof}

\begin{ex}
Let $X$ be the $2$-simplex $\sigma$ as before. \Figref{Fig:insert4} demonstrates choices of $T_1$ and $T_2$ constructed with the proof of Lemma \ref{keylemma}. 
\begin{figure}[ht]
\begin{center}
{
\psfrag{T1}[t][t]{$T_1$}
\psfrag{T2}[t][t]{$T_2$}
\includegraphics[width=11cm]{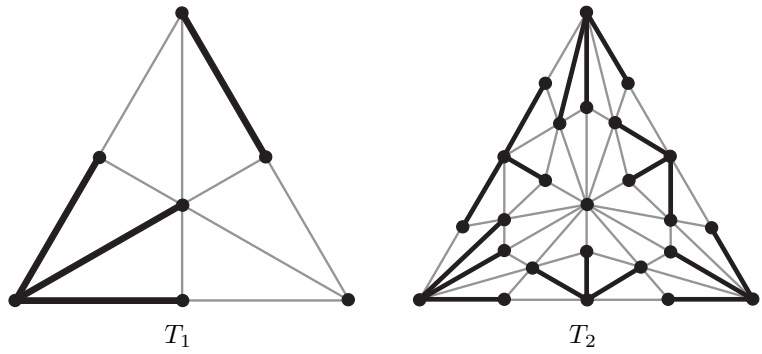}
}
\caption{Trees constructed to retract to the boundary.}
\label{Fig:insert4}
\end{center}
\end{figure}
In \Figref{Fig:insert5} the region shaded in red, $Sd^2\,\sigma \backslash \mathring{D}(\widetilde{\sigma},Sd\,\sigma)$, is crushed to the boundary by $r$ and contracted towards the boundary by the $P$ constructed with $T=T_2\cup s_2(T_1)$.
\begin{figure}[ht]
\begin{center}
{
\psfrag{T3=T2us2(T1)}[t][t]{$T=T_2\cup s_2(T_1)$}
\includegraphics[width=5.5cm]{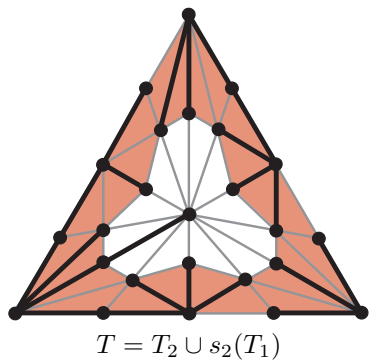}
}
\caption{The neighbourhood of $\partial\sigma$ that is retracted by $P$.}
\label{Fig:insert5}
\end{center}
\end{figure}
\qed\end{ex}
In fact we can choose $\ep$ uniformly across all simplices by virtue of the fact that $X$ is locally finite and finite-dimensional.
\begin{prop}\label{setupthesqueeze}
Let $X$ be a locally finite, $n$-dimensional simplicial complex, then there exists an $\ep=\ep(X)>0$, an integer $i=i(X)$ and a chain equivalence
\begin{displaymath}
\xymatrix{ (\Delta^{lf}_*(X), d_{\Delta^{lf}_*(X)},0) \ar@<0.5ex>[rr]^-{s} && (\Delta^{lf}_*(Sd^i\,X), d_{\Delta^{lf}_*(Sd^i\,X)}, P) \ar@<0.5ex>[ll]^-{r} }
\end{displaymath}
such that for all $\sigma\in X$, 
\begin{eqnarray*}
 P(\Delta_*(N_{\ep^\prime}(Sd^i \sigma))) &\subset& \Delta_{*+1}(N_{\ep^\prime}(Sd^i \sigma)),\\
 r(\Delta_*(N_{\ep^\prime}(Sd^i \partial \sigma))) &\subset& \Delta_{*}(\partial\sigma) \\
 s(\Delta_*(\mathring{\sigma})) &\subset& \Delta_*(Sd^i\, \mathring\sigma)
\end{eqnarray*}
for all $0\leqslant \ep^\prime \leqslant \ep$.
\end{prop}

\begin{proof}
Since $X$ is locally finite and finite-dimensional, $0<\comesh(X)<\mesh(X)<\infty$, so we can choose 
\begin{eqnarray*}
\ep(X)&:=& \alpha\comesh(X),\quad \mathrm{for}\; \mathrm{any} \; \alpha\in(0,1), \\
i(X)&:=& \bigg{\lceil}\ln\left(\frac{(1-\alpha)\comesh(X)}{2\mesh(X)}\right)\bigg{/}\ln\left(\frac{n}{n+1}\right)\bigg{\rceil} +1.
\end{eqnarray*}
We see that
\begin{eqnarray*}
\ep = \alpha\comesh(X) &<& \comesh(X)-2\left(\frac{n}{n+1}\right)^i\mesh(X) \\
    &\leqslant& \comesh(X) - 2\mesh(Sd^i\,X) \\
    &\leqslant& \rad(\sigma) - 2\mesh(Sd^i\,\sigma),\; \forall \sigma\in X,
\end{eqnarray*}
so Lemma \ref{keylemma} holds for every simplex. 
\end{proof}

\begin{cor}\label{setupmapsqueeze}
Let $f:X\to Y$ be a simplicial map between finite-dimensional, locally finite simplicial complexes, then there exists an $\ep(X,Y)$ and an $i(X,Y)$ and chain equivalences
\begin{displaymath}
\xymatrix{ (\Delta^{lf}_*(X), d_{\Delta^{lf}_*(X)},0) \ar@<0.5ex>[rr]^-{s_X} && (\Delta^{lf}_*(Sd^i\,X), d_{\Delta^{lf}_*(Sd^i\,X)}, P_X) \ar@<0.5ex>[ll]^-{r_X} \\
(\Delta^{lf}_*(Y), d_{\Delta^{lf}_*(Y)},0) \ar@<0.5ex>[rr]^-{s_Y} && (\Delta^{lf}_*(Sd^i\,Y), d_{\Delta^{lf}_*(Sd^i\,Y)}, P_Y) \ar@<0.5ex>[ll]^-{r_Y}}
\end{displaymath}
such that for all $\sigma\in Y$, 
\begin{eqnarray*}
 P_Y(\Delta_*(N_\ep(Sd^i \sigma))) &\subset& \Delta_{*+1}(N_\ep(Sd^i \sigma)),\\
 r_Y(\Delta_*(N_\ep(Sd^i \sigma))) &\subset& \Delta_{*}(Sd^i\, \sigma), \\
 P_X(\Delta_*(f^{-1}(N_\ep(Sd^i \sigma)))) &\subset& \Delta_{*+1}(f^{-1}(N_\ep(Sd^i \sigma))),\\
 r_X(\Delta_*(f^{-1}(N_\ep(Sd^i \sigma)))) &\subset& \Delta_{*}(f^{-1}(Sd^i\, \sigma)).
\end{eqnarray*}
\end{cor}

\begin{proof}
Note that for any $\ep>0$ and for all $\sigma\in Y$, 
\begin{equation}\label{transferep}
f^{-1}(N_{\ep}(Sd^i \sigma))=N_{\ep}(f^{-1}(Sd^i \sigma)) 
\end{equation}
where the former has $\ep$ measured in $Y$ and the latter in $X$ using the standard metrics on $X$ and $Y$. Apply Proposition \ref{setupthesqueeze} separately to each of $X$ and $Y$ and choose 
\begin{eqnarray*}
\ep(X,Y)&=& \mathrm{min}(\ep(X),\ep(Y)), \\
i(X,Y)&=&\mathrm{max}(i(X),i(Y)).
\end{eqnarray*}
Then the corollary follows from combining Proposition \ref{setupthesqueeze} and equation $(\ref{transferep})$.
\end{proof}
With the same choice of $r$ and $P$ we can consider the dual chain equivalence for cochains. A similar result to Proposition \ref{setupthesqueeze} can be proven which is also useful later for squeezing.
\begin{prop}\label{setupthedualsqueeze}
Choosing $\ep(X)$ and $i(X)$ as in Proposition \ref{setupthesqueeze}, the dual chain equivalence 
\begin{displaymath}
\xymatrix{ (\Delta^{-*}(X), \delta^{\Delta^{-*}(X)},0) \ar@<0.5ex>[rr]^-{r^*} && (\Delta^{-*}(Sd^i\,X), \delta^{\Delta^{-*}(Sd^i\,X)}, P^*) \ar@<0.5ex>[ll]^-{s^*} }
\end{displaymath}
to the chain equivalence provided by Proposition \ref{setupthesqueeze} satisfies 
\begin{enumerate}[(i)]
 \item $r^*(\Delta^{-*}(\mathring{\sigma})) \subset \Delta^{-*}(\bigcup_{\tau\geqslant\sigma}(Sd^i\,\tau \backslash N_\ep(\partial\tau\backslash \sigma)))$,
 \item $s^*(\Delta^{-*}(Sd^i\, \mathring{\sigma})) \subset \Delta^{-*}(\bigcup_{\tau\geqslant \sigma}\mathring{\tau})$, 
 \item $P^*(\Delta^{-*}(Sd^i\,\tau\backslash N_{\ep^\prime}(\partial\tau\backslash\sigma))) \subset \Delta^{-*+1}(Sd^i\,\tau\backslash N_{\ep^\prime/2}(\partial\tau\backslash\sigma))$, for all $\tau\geqslant\sigma$, $0\leqslant \ep^\prime\leqslant \ep$.  
\end{enumerate}
\end{prop}

\begin{proof}
\begin{enumerate}[(i)]
 \item The support of $r^*(\mathring{\sigma})$ is all the simplices in $Sd^i\, X$ that map under $r$ to $\mathring{\sigma}$. By Proposition \ref{setupthesqueeze}, for all $\sigma\in X$, an $\ep$-neighbourhood of $\partial\sigma$ is mapped to $\partial \sigma$ by $r$, thus the preimage of $\mathring{\sigma}$ under $r$ is contained in \[\bigcup_{\tau\geqslant\sigma}(Sd^i\,\tau \backslash N_\ep(\partial\tau\backslash \sigma)),\] from which the result follows.
 \item The support of $s^*(Sd^i\, \mathring{\sigma})$ is all the simplices in $X$ that map under $s$ to $Sd^i\, \mathring{\sigma}$. By Proposition \ref{setupthesqueeze}, $s$ only maps from a simplex interior to its closure which immediately implies the result.
 \item The support of $P^*(\mathring{\widetilde{\sigma}})$ is all the simplices in $Sd^i\,X$ that map under $P$ to $\mathring{\widetilde{\sigma}}$. Recall that we constructed $P$ in Lemma \ref{keylemma} so that it maps simplices \textit{towards the boundary}. This means that the image in $Sd^i\tau$ under $P^*$ of $Sd^i\tau \backslash N_{\ep^\prime}(\partial \tau)$ is contained in $Sd^i\tau \backslash N_{\ep^\prime}(\partial \tau)$ for all $0\leqslant\ep^\prime\leqslant\ep$. This is not quite the claim. Consider now $P^*$ applied to $Y=Sd^i\,\tau\backslash N_{\ep^\prime}(\partial\tau\backslash\sigma)$. Everything outside $Y$ is closer to the boundary than $Y$, except possibly near $\sigma$. We just need to determine how far out of $Y$ it is possible for $P^*$ to map, but $P^*$ can only map to points closer to $Y$ than $\partial\tau\backslash \sigma$. Thus the image is supported on \[Sd^i\,\tau\backslash N_{\ep^\prime/2}(\partial\tau\backslash\sigma)\] as claimed. \Figref{Fig:better1} illustrates this argument; on the left the dotted lines indicate $\ep$-neighbourhoods of the boundary and on the right they indicate $\ep/2$-neighbourhoods of the boundary. The lower face is $\sigma$ and $\tau$ is the whole simplex.

\begin{figure}[ht]
\begin{center}
{
\psfrag{s}[][]{$\sigma$}
\psfrag{blah}[][]{$Y$}
\includegraphics[width = 5.5cm]{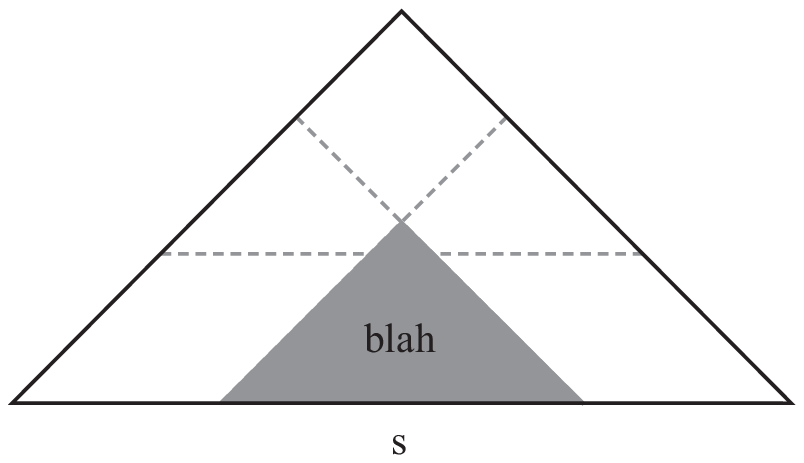}
\psfrag{blah}[][]{$P^*(Y)$}
\includegraphics[width = 5.5cm]{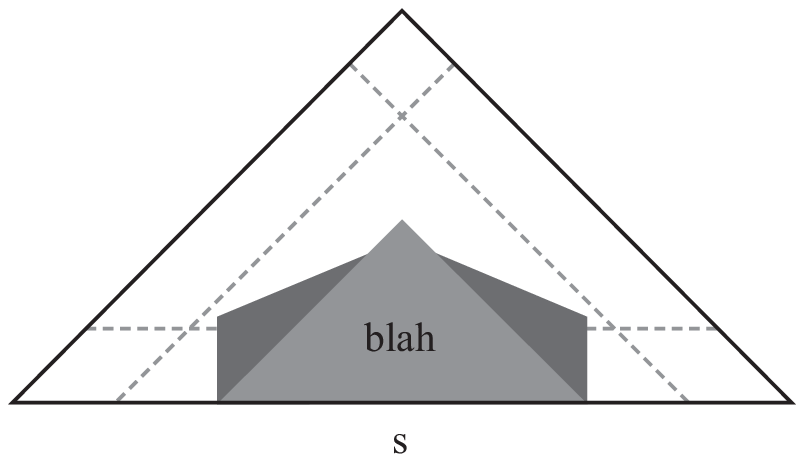}
}
\caption{$P^*$ applied to $Y=Sd^i\,\tau\backslash N_{\ep^\prime}(\partial\tau\backslash\sigma)$.}
\label{Fig:better1}
\end{center}
\end{figure}
\end{enumerate}
\end{proof}

\chapter{Controlled topological Vietoris-like theorem for simplicial complexes}\label{chapsix}
In this chapter we will prove the main topological result of this thesis:

\begin{thm}\label{maintopthm}
Let $f:X\to Y$ be a simplicial map between finite-dimensional locally finite simplicial complexes, then the following are equivalent:
\begin{enumerate}[(1)]
 \item $f$ has contractible point inverses,
 \item $f$ is an $\ep$-controlled homotopy equivalence measured in $Y$ for all $\ep>0$,
 \item $f\times\id_\R: X\times\R \to Y\times\R$ is an $O(Y^+)$-bounded homotopy equivalence.
\end{enumerate}
\end{thm}

Note in the statement above that $f$ is not required to be a proper map, so following Example \ref{earliercomments} these conditions do not imply that $f$ is cell-like (but are implied by $f$ being cell-like).

A contractible map is necessarily surjective as the empty set is not contractible. We will prove that conditions $(2)$ and $(3)$ also force $f$ to be surjective, so a good place to start is by discussing surjective simplicial maps and in particular contractible ones.

\section{A study of contractible simplicial maps}
By examining surjective simplicial maps we prove $(1)\Rightarrow(3)$ of Theorem \ref{maintopthm} directly:

\begin{prop}\label{constructhtpyinv}
Let $f:X\to Y$ be a contractible simplicial map of finite-dimensional locally finite simplicial complexes, then $f\times\id:(X\times\R,j_Y(f\times\id)) \to (Y\times\R,j_Y)$ is an $O(Y^+)$-bounded homotopy equivalence.
\end{prop}
In order to prove this proposition, we will require a few lemmas:
\begin{lem}\label{surjectivesimplicialmapssimplextosimplex}
Let $f$ be a surjective simplicial map from an $m$-simplex $\sigma=v_0\ldots v_m$ to an $n$-simplex $\tau=w_0\ldots w_n$. Let $X_i = \{v_{i,1},\ldots, v_{i,{r_i}} \}$ denote the set of vertices of $\sigma$ mapped by $f$ to $w_i$ and let $r_i = |X_i|$. Then there is a $PL$ homeomorphism
\[f^{-1}(\mathring{\tau}) \cong \mathring{\tau} \times \prod_{i=0}^n \Delta^{r_i-1}.\]
\end{lem}
\begin{proof}
Let $(t_0,\ldots, t_n)$ be barycentric coordinates for $\tau$ with respect to the vertices $w_0,\ldots, w_n$. Each set $X_i$ spans an $(r_i-1)$-simplex 
\[\{\sum_{j=1}^{r_i}s_{i,j}v_{i,j}\,|\, s_{i,j}\geqslant 0, \sum_{j}s_{i,j} = 1 \} \]
in the preimage over any point $(t_0,\ldots, t_n)$ in $\tau$ for which $t_i\neq 0$. Over the interior of $\tau$ we have that $t_i\neq 0$ for all $i$, so each point inverse is $PL$ homeomorphic to the product 
\[\prod_{i=0}^n \Delta^{r_i-1}\] 
and these fit together by linearity to give \[f^{-1}(\mathring{\tau}) \cong \mathring{\tau} \times \prod_{i=0}^n \Delta^{r_i-1}.\]
Explicitly a $PL$ homeomorphism is given by 
\begin{displaymath}
 \xymatrix{ f^{-1}(\mathring{\tau}) \ar@<0.5ex>[rr]^-{\phi} && \mathring{\tau} \times \Delta^{r_0-1}\times \ldots \Delta^{r_n-1} \ar@<0.5ex>[ll]^-{\psi}
}
\end{displaymath}
with
\begin{eqnarray*}
 \phi(\sum_{i=0}^n\sum_{j=1}^{r_i}{\lambda_{i,j}v_{i,j}}) &=&  (\sum_{i=0}^n{\Lambda_i w_i}, \sum_{j=1}^{r_0}\Lambda_0^{-1}\lambda_{0,j}v_{0,j}, \ldots, \sum_{j=1}^{r_n}\Lambda_n^{-1}\lambda_{n,j}v_{n,j}), \\
\psi(\sum_{i=0}^n{t_iw_i}, \sum_{j=1}^{r_0}s_{0,j}v_{0,j}, \ldots, \sum_{j=1}^{r_n}s_{n,j}v_{n,j}) &=& \sum_{i=0}^n\sum_{j=1}^{r_i}{t_is_{i,j}v_{i,j}} 
\end{eqnarray*}
where $\Lambda_i := \sum_{j=1}^{r_i}{\lambda_{i,j}}$.
\end{proof}

\begin{ex}
Let $X=v_0v_1v_2v_3$ be a $3$-simplex and $Y=w_0w_1$ a $1$-simplex. Let $f:X\to Y$ be the simplicial map defined by sending $v_0,v_1$ to $w_0$ and $v_2,v_3$ to $w_1$. Then
\begin{eqnarray*}
 f^{-1}(\mathring{w}_0) = v_0v_1 &\cong& \mathring{w}_0 \times \Delta^1, \\
 f^{-1}(\mathring{w}_1) = v_2v_3 &\cong& \mathring{w}_1 \times \Delta^1, \\
 f^{-1}(\mathring{Y}) = X \backslash (v_0v_1\cup v_2v_3) &\cong& \mathring{Y} \times \Delta^2.
\end{eqnarray*}
This is illustrated in \Figref{Fig:Last}.
\begin{figure}[ht]
\begin{center}
{
\psfrag{1}[r][r]{$v_0$}
\psfrag{2}[r][r]{$v_1$}
\psfrag{3}[r][r]{$w_0$}
\psfrag{4}[b][b]{$\Delta^1\times\Delta^1$}
\psfrag{5}[l][l]{$v_2$}
\psfrag{6}[l][l]{$v_3$}
\psfrag{7}[l][l]{$w_1$}
\psfrag{8}[l][l]{$f$}
\includegraphics[width=5cm]{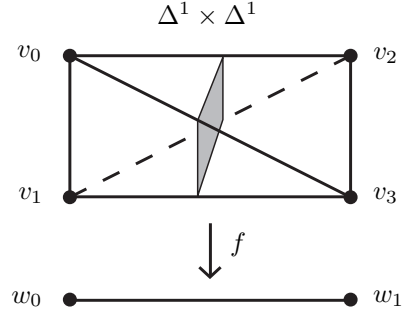}
}
\caption{Seeing that inverse images of open simplices are products.}
\label{Fig:Last}
\end{center}
\end{figure}
\qed\end{ex}

For a simplicial surjection onto a single simplex $\sigma$, we can piece together the PL homeomorphisms given by Lemma \ref{surjectivesimplicialmapssimplextosimplex}.

\begin{lem}\label{Ksigmacontractible}
Let $f:X \to \tau$ be a surjective simplicial map from a finite-dimensional locally finite simplicial complex $X$ to an $n$-simplex $\tau$. Then 
\begin{enumerate}[(i)]
 \item for all $\rho\leqslant \tau$, $f^{-1}(\mathring{\rho}) \cong \mathring{\rho} \times K(\rho)$ for some finite-dimensional locally finite simplicial complex $K(\rho)$,
 \item if $f$ is a contractible map, then $K(\rho)$ is contractible, for all $\rho\leqslant \tau$, 
 \item if $K(\rho)$ is contractible for all $\rho\leqslant \tau$, then $f$ is a contractible map.
\end{enumerate}
\end{lem}

\begin{proof}\label{generalsurjectivesimplicialmapstoasimplex}
\begin{enumerate}[(i)]
 \item  For each $\sigma\in X$ that surjects onto $\tau$, let $f_\sigma = f|: \sigma \to \tau$. We know by Lemma \ref{surjectivesimplicialmapssimplextosimplex} that \[f_\sigma^{-1}(\mathring{\tau})\cong \mathring{\tau}\times K(\tau,\sigma)\] where $K(\tau,\sigma)$ is some product of closed simplices. Let $\rho<\sigma$ also surject onto $\tau$, then $K(\tau,\rho)\subset K(\tau,\sigma)$ so we can build up the preimage as
\[ K(\tau):=  \coprod_{\mathring{\sigma}\in f^{-1}(\mathring{\tau})}{K(\tau,\sigma)}\scalebox{1.6}{\raisebox{-1mm}{$\diagup$}}\hspace{-1mm}\raisebox{-3mm}{$\sim$} \]where we identify $K(\tau,\rho)\subset K(\tau,\sigma)$ with $K(\tau,\rho)\subset K(\tau,\sigma^\prime)$ for all $\rho \subset \sigma\cap \sigma^\prime$. $K$ is by construction a cellulation, but by subdividing we can triangulate $K$ to make it a simplicial complex as required. 
\item If $f$ is contractible then $f^{-1}(x)\simeq *$ for all $x\in \tau$. For any $\rho\leqslant\tau$, pick $x\in\mathring{\rho}$. Then $f^{-1}(x)=\{x\}\times K(\rho)$, so $K(\rho)$ is contractible.
\item If $K(\rho)\simeq *$ for all $\rho\leqslant\tau$, then for all $x\in\tau$ there exists a unique $\rho\leqslant\tau$ such that $x\in\mathring{\rho}$. This means that $f^{-1}(x) = \{x\}\times K(\rho)\simeq *$ for all $x$. Thus $f$ is a contractible map.
\end{enumerate}
\end{proof}

With this lemma we now have enough to prove Proposition \ref{constructhtpyinv}. Given a surjective simplicial map $f:X\to Y$ of finite-dimensional locally finite simplicial complexes, for all $\sigma\in Y$ we know that $f^{-1}(\mathring{\sigma})\cong \mathring{\sigma}\times K(\sigma)$ for some contractible $K(\sigma)$. This means $f$ is a trivial fibration over each $\mathring{\sigma}$ so we can define a local section. Roughly speaking, the contractibility of $K(\sigma)$ for all $\sigma$ allows us to piece together the local sections to get a global homotopy inverse $g_\ep$, for all $\ep>0$, that is an \textit{approximate section} in the sense that $f\circ g_\ep \simeq \id_Y$ via homotopy tracks of diameter $<\ep$. This is precisely the notion of an approximate fibration as defined by Coram and Duvall in \cite{CoramDuvall}:

\begin{defn}
Let $p:X\to B$ be a map to a metric space $B$ and $\ep>0$. A map $q:E\to X$ is called a $\ep$-fibration if the lifting problem
\begin{displaymath}
 \xymatrix{
Z\times\{0\} \ar[d]_-{i} \ar[r]^-{f} & E \ar[d]^{q}\\
Z\times [0,1] \ar@{-->}[ur]^-{\widetilde{F}} \ar[r]^-{F} & X
}
\end{displaymath}
has a solution $\widetilde{F}$ such that $\widetilde{F}\circ i = f$ and $d(pq\widetilde{F}(z,t),pF(z,t))<\ep$ for all $(z,t)\in Z\times [0,1]$. The map $q$ is called an \textit{approximate fibration} if it is an $\ep$-fibration for all $\ep$.
\qed\end{defn}

\begin{proof}[Proof of Proposition \ref{constructhtpyinv}]
Let $f:X\to Y$ be a contractible map between finite-dimensional locally finite simplicial complexes. We explicitly construct a one parameter family $g_\ep$ of homotopy inverses with homotopy tracks of diameter at most $\ep$ when measured in the target space $Y$. For all $0<\ep<\comesh(Y)$ we can give $Y$ the fundamental $\ep$-subdivision cellulation as defined in Definition \ref{Defn:fundamentalcellulation}. For all $\sigma\in Y$, by Lemma \ref{surjectivesimplicialmapssimplextosimplex} there exist $PL$ isomorphisms
\begin{displaymath}
 \xymatrix{ \phi_\sigma: \mathring{\sigma}\times K(\sigma) \ar[r]^-{\simeq}  & f^{-1}(\mathring{\sigma}). 
}
\end{displaymath}
For each $\sigma\in Y$, choose a point $\gamma_{\sigma}(1) \in K(\sigma)$. We think of $\gamma_\sigma(1)$ as the image of a map $\gamma_\sigma: \Delta^0 \to K(\sigma)$. Define $g$ on $\Gamma_\sigma(\sigma)$ as the closure of the map 
\begin{displaymath}
 \xymatrix{ \Gamma_\sigma(\mathring{\sigma}) \ar[r]^-{\Gamma_\sigma^{-1}} & \mathring{\sigma}\times\Delta^0 \ar[rrr]^-{\left(\begin{array}{cc} \id_{\mathring{\sigma}} & 0 \\ 0 & \gamma_{\sigma}\end{array} \right)} &&& \mathring{\sigma}\times K(\sigma) \ar[r]^-{\phi_\sigma} & f^{-1}(\mathring{\sigma}). 
}
\end{displaymath}
Suppose now we have defined $g$ continuously on $\Gamma_{\sigma_0,\ldots,\sigma_i}(\mathring{\sigma}_0)$ for all sequences of inclusions $\sigma_0<\ldots<\sigma_i$ for $i\leqslant n$. Let $\sigma_0<\ldots<\sigma_{n+1}$ be a sequence of inclusions, then there is a continuous map defined by
\begin{eqnarray*}
\partial \Delta^{n+1} = S^n &\to& K(\sigma_0) \\
(t_0,\ldots,t_{n+1}) &\mapsto& \gamma_{\sigma_0,\ldots,\widehat{\sigma}_j,\ldots,\sigma_{n+1}}(t_0,\ldots,\widehat{t}_j,\ldots,t_{n+1}),\;\mathrm{for}\; t_j=0
\end{eqnarray*}
noting that for all $\sigma_j\geqslant \sigma_0$, $K(\sigma_j)\subset K(\sigma_0)$. Contractibility of $K(\sigma_0)$ means we may extend this map continuously to the whole of $\Delta^{n+1}$ to obtain a map \[\gamma_{\sigma_0,\ldots, \sigma_{n+1}}:\Delta^{n+1}\to K(\sigma_0).\] We then define $g$ on $\Gamma_{\sigma_0,\ldots,\sigma_{n+1}}(\mathring{\sigma}_0)$ as the closure of the map
\begin{equation}\label{defnofg}
 \xymatrix@R=3mm{ \Gamma_{\sigma_0,\ldots,\sigma_{n+1}}(\mathring{\sigma}_0) \ar[rr]^-{\Gamma_{\sigma_0,\ldots,\sigma_{n+1}}^{-1}} && \mathring{\sigma}_0\times \Delta^{n+1} \ar[rrr]^-{\left(\begin{array}{cc} \id_{\mathring{\sigma}_0} & 0 \\ 0 & \gamma_{\sigma_0,\ldots, \sigma_{n+1}}\end{array} \right)}  &&& \mathring{\sigma}_0\times K(\sigma_0) \ar[r]^-{\phi_{\sigma_0}} & f^{-1}(\mathring{\sigma}_0).
}
\end{equation}
This is continuous and agrees with the definition of $g$ so far on $\partial \Gamma_{\sigma_0,\ldots,\sigma_{n+1}}(\mathring{\sigma}_0)$ by construction. Thus proceeding by induction we get a well defined map $g:Y\to X$ which we claim is an $\ep$-controlled homotopy inverse to $f$. By equation $(\ref{defnofg})$ and the observation that $f\circ\phi_{\sigma_0}$ is projection on the first factor $\pr_1: \mathring{\sigma}_0\times K(\sigma_0) \to \mathring{\sigma}_0$, we see that \[f\circ g| = \pr_1: \mathring{\sigma}_0\times \Delta^{n+1} \to \mathring{\sigma}_0.\] There are straight line homotopies \[h_{\sigma_0,\ldots,\sigma_i}:\id_{\Gamma_{\sigma_0,\ldots,\sigma_i}(\mathring{\sigma}_0)}\simeq \pr_1 : \Gamma_{\sigma_0,\ldots,\sigma_i}(\mathring{\sigma}_0)\times I \to \Gamma_{\sigma_0,\ldots,\sigma_i}(\mathring{\sigma}_0),\] as constructed in Remark \ref{crushinggammas}, which fit together to give a global homotopy $h_1:f \circ g\simeq \id_Y$ with homotopy tracks over a distance of at most $\ep$ measured in $Y$.

Consider $g\circ f$. This sends $f^{-1}(\Gamma_{\sigma_0,\ldots,\sigma_i}(\mathring{\sigma}_0))$ to $g(\Gamma_{\sigma_0,\ldots,\sigma_i}(\mathring{\sigma}_0))= \sigma_0\times \gamma_{\sigma_0,\ldots,\sigma_i}(\Delta^i)$. We define the homotopy $h_2:g\circ f \simeq \id_X$ inductively.

For all $\sigma\in X$ there is a homotopy $k_\sigma:K(\sigma)\simeq\{\gamma_\sigma(1)\}$, so $h_\sigma\times k_\sigma$ is a homotopy \[f^{-1}(\Gamma_{\sigma_0}(\mathring{\sigma}_0)) \cong \Gamma_{\sigma_0}(\mathring{\sigma}_0)\times K(\sigma) \simeq \mathring{\sigma}_0 \times \{\gamma_{\sigma_0}(1)\} \cong g(\Gamma_{\sigma_0}(\mathring{\sigma}_0)). \]
Suppose now that we have defined a homotopy $h_2:g\circ f \simeq \id_X$ for all $\Gamma_{\sigma_0,\ldots,\sigma_i}(\mathring{\sigma}_0)$ for $i\leqslant n$. We extend this to $i=n+1$. Consider \[f^{-1}(\Gamma_{\sigma_0,\ldots,\sigma_{n+1}}(\mathring{\sigma}_0))\cong \Gamma_{\sigma_0,\ldots,\sigma_{n+1}}(\mathring{\sigma}_0)\times K(\sigma_{n+1}).\] We seek a homotopy to $g(\Gamma_{\sigma_0,\ldots,\sigma_{n+1}}(\mathring{\sigma}_0)) \cong \mathring{\sigma}_0\times \gamma_{\sigma_0,\ldots,\sigma_{n+1}}(\Delta^{n+1})$ compatible with the homotopy we already have defined on $\partial \Gamma_{\sigma_0,\ldots,\sigma_{n+1}}(\mathring{\sigma}_0)$. $h_2$ is already defined on \[\partial \overline{f^{-1}(\Gamma_{\sigma_0,\ldots,\sigma_{n+1}}(\mathring{\sigma}_0))}\cong \partial\overline{\Gamma_{\sigma_0,\ldots,\sigma_{n+1}}(\mathring{\sigma}_0)\times K(\sigma_{n+1})} = \partial(\Gamma_{\sigma_0,\ldots,\sigma_{n+1}}(\sigma_0)\times K(\sigma_{n+1})),\]i.e.\ we have a map 
\begin{equation}\label{extendhomotopy}
\partial(\sigma_0\times \Delta^{n+1})\times K(\sigma_{n+1})\times I = S^{n+|\sigma_0|} \times K(\sigma_{n+1})\times I \to f^{-1}(\Gamma_{\sigma_0,\ldots,\sigma_{n+1}}(\sigma_0)).
\end{equation}
The preimage $f^{-1}(\Gamma_{\sigma_0,\ldots,\sigma_{n+1}}(\sigma_0))$ is contractible because \[f^{-1}(\Gamma_{\sigma_0,\ldots,\sigma_{n+1}}(\mathring{\sigma}_0))= \bigcup_{j=0}^{n+1}f^{-1}(\Gamma_{\sigma_0,\ldots,\sigma_j}(\mathring{\sigma}_0))= \bigcup_{j=0}^{n+1}{\Gamma_{\sigma_0,\ldots,\sigma_j}(\mathring{\sigma}_0) \times K(\sigma_j)}.\]
We already saw that $K(\sigma_j)\subset K(\sigma_i)$ for $i\leqslant j$. All the $K(\sigma_j)$ are contractible so we can find deformation retracts \[K(\sigma_0)\to K(\sigma_1)\to \ldots \to K(\sigma_{n+1})\] which shows that \[\bigcup_{j=0}^{n+1}{\Gamma_{\sigma_0,\ldots,\sigma_j}(\mathring{\sigma}_0) \times K(\sigma_j)} \simeq \bigcup_{j=0}^{n+1}{\Gamma_{\sigma_0,\ldots,\sigma_j}(\mathring{\sigma}_0) \times K(\sigma_{n+1})} = \Gamma_{\sigma_0,\ldots,\sigma_j}(\sigma_0)\times K(\sigma_{n+1})\simeq *.\]
Now since $f^{-1}(\Gamma_{\sigma_0,\ldots,\sigma_{n+1}}(\sigma_0))$ is contractible we can extend $(\ref{extendhomotopy})$ from $S^{n+|\sigma_0|}$ to $D^{n+|\sigma_0|+1}$, i.e.\ to $f^{-1}(\Gamma_{\sigma_0,\ldots,\sigma_{n+1}}(\sigma_0)).$ Thus by construction we obtain $h_2:g\circ f \simeq \id_X$. We may construct $h_2$ in such a way that it projects to $h_1$ under $f$, therefore the homotopy tracks have the same length bounded by $\ep$.

For any two $\ep,\ep^\prime < \comesh(Y)$, $g_\ep\simeq g_{\ep^\prime}$ via
\begin{eqnarray*}
 G_{\ep, \ep^\prime}: Y\times I &\to& X \\
 (y,t) &\mapsto& g_{t\ep +(1-t)\ep^\prime}(y),
\end{eqnarray*}
and similarly for the homotopies $h_1$ and $h_2$: 
\begin{eqnarray*}
 (H_1)_{\ep, \ep^\prime}: Y\times I\times I &\to& Y \\
 (y,t,s) &\mapsto& (h_1)_{t\ep +(1-t)\ep^\prime}(y,s),
\end{eqnarray*}
\begin{eqnarray*}
 (H_2)_{\ep, \ep^\prime}: X\times I\times I &\to& X \\
 (x,t,s) &\mapsto& (h_2)_{t\ep +(1-t)\ep^\prime}(x,s),
\end{eqnarray*}
so $f\times\id_\R: X\times \R \to Y\times \R$ is an $O(Y^+)$-bounded homotopy equivalence with inverse 
\begin{eqnarray*}
 g:Y\times\R &\to& X\times\R \\
(y,t) &\mapsto& \brcc{g_1(y),}{t\leqslant 1,}{g_{1/t}(y),}{t>1} 
\end{eqnarray*}
and homotopies 
\begin{eqnarray*}
 h_1:(f\times\id_\R)\circ g \simeq \id_{Y\times\R}: Y\times\R\times I &\to& Y\times\R \\
(y,t,s) &\mapsto& \brcc{(h_2)_1(y,s),}{t\leqslant 1,}{(h_2)_{1/t}(y,s),}{t>1} 
\end{eqnarray*}
\begin{eqnarray*}
 h_2:g\circ (f\times\id_\R) \simeq \id_{X\times\R}: X\times\R\times I &\to& X\times\R \\
(x,t,s) &\mapsto& \brcc{(h_1)_1(x,s),}{t\leqslant 1,}{(h_1)_{1/t}(x,s),}{t>1.} 
\end{eqnarray*}

Note also that $f|:f^{-1}(\tau)\to \tau$ is a homotopy equivalence for all $\tau\in Y$ by restricting $g, h_1$ and $h_2$. See section \ref{Ytriang} for a brief discussion of such homotopy equivalences which we call \textit{$Y$-triangular homotopy equivalences}.
\end{proof}

\begin{ex}
Let $Y=\tau_1\cup\tau_2\cup\tau_3\subset \R^2$ for
\begin{eqnarray*}
 \sigma_1 &:=& \lrangle{(0,0), (2,0), (1,1)},\\
 \sigma_2 &:=& \lrangle{(2,0), (1,1), (2,2)},\\
 \sigma_3 &:=& \lrangle{(1,1), (2,2), (0,2)},
\end{eqnarray*}
where we also label the intersections of these simplices as
\begin{eqnarray*}
 \tau_1 &:=& \lrangle{(2,0), (1,1)},\\
 \tau_2 &:=& \lrangle{(1,1), (2,2)},\\
 \rho &:=& \{(1,1)\},
\end{eqnarray*}
as pictured in \Figref{Fig:workedex1}.
\begin{figure}[ht]
\begin{center}
{
\psfrag{x}{$x$}
\psfrag{y}{$y$}
\psfrag{s1}{$\sigma_1$}
\psfrag{s2}{$\sigma_2$}
\psfrag{s3}{$\sigma_3$}
\psfrag{t1}{$\tau_1$}
\psfrag{t2}{$\tau_2$}
\psfrag{r}{$\rho$}
\psfrag{2}{$2$}
\includegraphics[width=6cm]{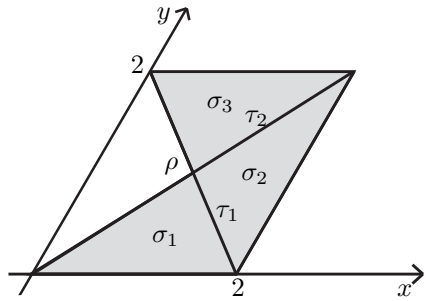}
}
\caption{The space $Y$}
\label{Fig:workedex1}
\end{center}
\end{figure}

Let \[X:= \bigcup_{i=0}^{2}{\sigma_{i+1}\times\{i\}} \cup \bigcup_{i=0}^{1}{\tau_{i+1}\times [i,i+1]}\] where we triangulate the squares as necessary. The space $X$ is illustrated in \Figref{Fig:workedex3}. Define $f:X\to Y$ by vertical projection $(x,y,z)\mapsto (x,y)$. 
\begin{figure}[ht]
\begin{center}
{
\psfrag{x}{$x$}
\psfrag{y}{$y$}
\psfrag{z}{$z$}
\includegraphics[width=6cm]{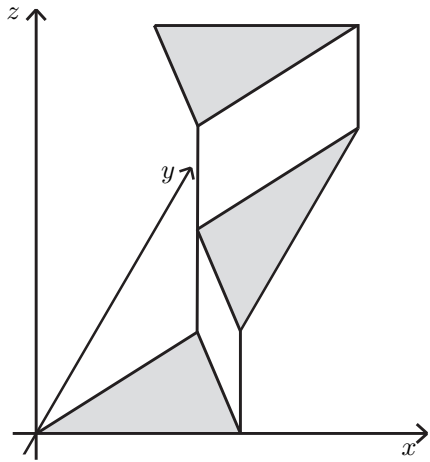}
}
\caption{The space $X$}
\label{Fig:workedex3}
\end{center}
\end{figure}
Following the proof of Proposition \ref{constructhtpyinv}, first we decompose each simplex in $Y$ into the homotopies $\Gamma_{\sigma_1,\ldots,\sigma_i}$ as illustrated in \Figref{Fig:workedex2}.
\begin{figure}[ht]
\begin{center}
{
\psfrag{x}{$x$}
\psfrag{y}{$y$}
\includegraphics[width=6cm]{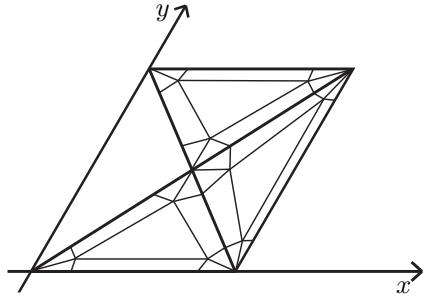}
}
\caption{The space $Y$ decomposed}
\label{Fig:workedex2}
\end{center}
\end{figure}
Then we choose local sections $\gamma_\tau$ for each open simplex $\mathring{\tau}\in Y$. These choices are shown in the table below.
\begin{displaymath}
\begin{array}{|c|c|c|}
\hline
\mathring{\tau} \subset & \mathrm{possible}\;\gamma_\tau & \mathrm{choice} \\
\hline
\hline
\sigma_1\backslash \tau_1 & \{0\} & \{0\}\\
\sigma_2\backslash (\tau_1\cup\tau_2) & \{1\} & \{1\}\\
\sigma_3\backslash \tau_3 & \{2\} & \{2\}\\
\hline
\tau_1\backslash\rho & [0,1] & \{0.5\}\\
\tau_2\backslash\rho & [1,2] & \{1.5\}\\
\hline
\rho & [0,2] & \{1\}\\
\hline
\end{array}
\end{displaymath}
With these choices, the global homotopy inverse we obtain is as illustrated in \Figref{Fig:workedex4}. Note that the homotopy tracks are tall in the space $X$ but after projecting to $X$ they are small.
\begin{figure}[ht]
\begin{center}
{
\psfrag{x}{$x$}
\psfrag{y}{$y$}
\psfrag{z}{$z$}
\psfrag{X}[r][r]{$X$}
\includegraphics[width=7cm]{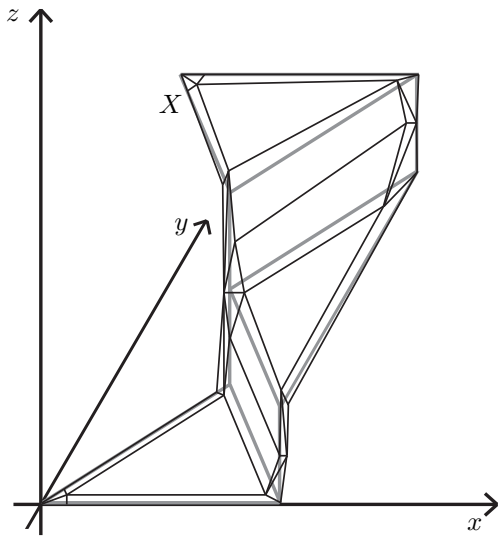}
}
\caption{The $\ep$-bounded homotopy inverse $g_\ep$.}
\label{Fig:workedex4}
\end{center}
\end{figure}
\qed\end{ex}

\section{Finishing the proof}
Given an $O(Y^+)$-bounded homotopy equivalence $f\times\id$ we can simply ``slice it'' further and further up in the $\R$ direction to get homotopy inverse $g_\ep:Y\to X$ to $f$ with control smaller and smaller. This proves $(3)\Rightarrow(2)$ of Theorem \ref{maintopthm}: 

\begin{prop}
Let $f\times\id: X\times\R \to Y\times\R$ be an $O(Y^+)$-bounded homotopy equivalence, then $f:X\to Y$ is an $\ep$-controlled homotopy equivalence, for all $\ep>0$.
\end{prop}

\begin{proof}
Let $f\times\id$ have inverse $g$ and homotopies $h_1:g\circ (f\times\id) \sim \id_{X\times\R}$ and $h_2:(f\times\id)\circ g \sim \id_{Y\times\R}$ all with bound $B<\infty$. Consider the slices at height $t$ and the following diagram:
\begin{displaymath}
 \xymatrix{
X\times\{t\} \ar[r]^-{f\times\id} \ar@<0.5ex>[d]^-{1} & Y\times\{t\} \ar[dl]_-{g|} \ar@<-0.5ex>[d]_-{1} \\
X\times\R \ar@<0.5ex>[u]^-{\id\times p_t} \ar@<0.5ex>[r]^-{f\times\id} & Y\times\R \ar@<-0.5ex>[u]_-{\id\times p_t} \ar@<0.5ex>[l]^-{g}
}
\end{displaymath}
where $p_t: \R \to \{t\}$ is projection onto $t\in\R$. 

Define $g_t:= (\id_X\times p_t)\circ g|: Y\times\{t\}\to X\times\{t\}.$ It can be checked that this is a homotopy inverse to $f\times\id: X\times\{t\} \to Y\times\{t\}$ with homotopies
\begin{eqnarray*}
 h_1^\prime &=& (\id_X\times p_t) \circ h_1|_{X\times\{t\}} : g_t\circ (f\times \id)| \sim \id_{X\times\{t\}} \\
 h_2^\prime &=& (\id_Y\times p_t) \circ h_2|_{Y\times\{t\}} : (f\times \id)\circ g_t \sim \id_{Y\times\{t\}}
\end{eqnarray*}
where we have used the commutativity of 
\begin{displaymath}
 \xymatrix{
X\times\{t\} \ar[r]^-{f\times\id} & Y\times\{t\} \\
X\times\R \ar[u]^-{\id_X\times p_t} \ar[r]^-{f\times\id} & Y\times\R \ar[u]_{\id_Y\times p_t}
}
\end{displaymath}
to note that 
\begin{eqnarray*}
 (f\times \id)\circ g_t &=& (f\times \id)\circ (\id_X\times p_t)\circ g| \\
&=& (\id_Y\times p_t)\circ (f\times \id)\circ g| \\
&\sim& (\id_Y\times p_t) \circ \id_{Y\times\{t\}} = \id_{Y\times\{t\}}.
\end{eqnarray*}

The bound of this homotopy equivalence is approximately $B$ when we measure it on $Y\times\{t\} \subset O(Y^+)$. The slice  $Y\times\{t\}$ has a metric $t$ times bigger than $Y=Y\times\{1\}$, so measuring this in $Y$ gives a homotopy equivalence $f:X\to Y$ with control proportional to $\dfrac{B}{t}$ as required.
\end{proof}

To prove the final implication of Theorem \ref{maintopthm} we first need to show that if $f$ is an $\ep$-controlled simplicial homotopy equivalence for all $\ep>0$ then $f$ must necessarily surject. Intuitively this is because every simplex in $Y$ has a non-zero radius so if that simplex is not surjected onto then $f$ cannot be $\ep$-controlled for $\ep$ less than its radius.

\begin{lem}\label{fsurjects}
Let $X$, $Y$ be simplicial complexes with $\comesh(Y)>0$. If $f:X\to Y$ is an $\ep$-controlled simplicial homotopy equivalence for all $\ep>0$ then $f$ must necessarily be surjective.
\end{lem}
\begin{proof}
Suppose $f$ is not surjective, let $\sigma$ be a simplex with $f^{-1}(\mathring{\sigma})=\emptyset$. We can choose such a $\sigma$ with $|\sigma|\geqslant 1$ and hence non-zero radius. This is because if a vertex $v$ is not in the image of $f$, then all simplices containing $v$ have interior not in the image of $f$. If it is not contained in any other simplices it forms a connected component of $Y$ that is not mapped to by the homotopy equivalence, which is absurd. 

So suppose we have chosen such a $\sigma$. Let $x$ denote the incentre of $\sigma$. Choose $\ep$ less than $\comesh(Y)$ and consider the homotopy track of $f\circ g(x)\sim x$. The diameter of the homotopy track must be greater than the radius of $\sigma$ and hence greater than $\ep$, since $f(g(x))\notin \mathring{\sigma}$. This is a contradiction, therefore $f$ must surject.
\end{proof}
Note the same approach also proves that an $O(Y^+)$-bounded homotopy equivalence $f\times\id$ must surject. 
\begin{prop}\label{threeimpliesone}
Let $f:X\to Y$ be a simplicial map between finite-dimensional locally finite simplicial complexes. If $f$ is an $\ep$-controlled homotopy equivalence for all $\ep>0$, then $f$ is a contractible map.
\end{prop}

\begin{proof}
By Lemma \ref{fsurjects} $f$ must necessarily be a surjective map. By Lemma \ref{surjectivesimplicialmapssimplextosimplex}, for all $\sigma\in Y$ \[f^{-1}(\mathring{\sigma})\cong \mathring{\sigma}\times K(\sigma)\] for some simplicial complex $K(\sigma)$. Pick $x\in \mathring{\sigma}$, then there exists an $\ep$ such that $x\in B_{\ep}(x)\subset \mathring{\sigma}$, whence $f^{-1}(B_\ep(x))\cong B_\ep(x)\times K(\sigma)$. Since $f$ is an $\ep$-controlled homotopy equivalence for all $\ep$ we can find an $\ep/2$-controlled homotopy inverse $g$ together with homotopies such that $f^{-1}(x) \simeq g(x)=*$ in $f^{-1}(B_\ep(x))$. Projecting this homotopy to $K(\sigma)$ gives a contraction of $K(\sigma)$ to $\pr_2(g(x))=*$. Thus $K(\sigma)\simeq *$ for all $\sigma\in Y$. Applying Lemma \ref{Ksigmacontractible} $(iii)$ we deduce that $f$ is a contractible map.
\end{proof}
This completes the proof of Theorem \ref{maintopthm}.

\section{Topological squeezing}\label{Ytriang}
There is an alternate approach one can take in order to try to prove Theorem \ref{maintopthm}. This brief section introduces a type of homotopy equivalence that aims to mirror the approach we take algebraically in the second half of the thesis;\footnote{It is meant to mirror the notion of a chain complex $C$ being chain contractible in $\A(X)$.} this is included purely to be suggestive.

\begin{defn}
Let $(X,p)$ and $(Y,q)$ be topological spaces with control maps to a finite-dimensional locally finite simplicial complex $Z$ equipped with its standard metric. We say that $f:X\to Y$ is a $Z$-triangular map if for all $\sigma\in Z$, \[q(f(p^{-1}(\mathring{\sigma}))) \subset \sigma.\]
\qed\end{defn}
\begin{rmk}
Note that the composition of $Z$-triangular maps is again $Z$-triangular. 
\qed\end{rmk}
\begin{defn}
A map $f: X \to Y$ is a \textit{$Z$-triangular homotopy equivalence} if there exists a homotopy inverse $g:Y \to X$ and homotopies $h_1:f\circ g\simeq \id_Y$, $h_2:g\circ f\simeq \id_X$ such that all of $f$, $g$, $h_1$ and $h_2$ are $Z$-triangular.
\qed\end{defn}
In this section we are only concerned with maps of finite-dimensional locally finite simplicial complexes measured in the target space. 
\begin{rmk}\label{rmkYtriang}
Let $f:(X,f)\to (Y,\id_Y)$ be a surjective map of finite-dimensional locally finite simplicial complexes. Then
\begin{enumerate}[(i)]
 \item $f$ is automatically $Y$-triangular, since for all $\sigma\in Y$, $id_Y(f(f^{-1}(\mathring{\sigma}))) = \mathring{\sigma} \subset \sigma.$
 \item $f$ is a $Y$-triangular homotopy equivalence if there exists a homotopy inverse $g$ and homotopies $h_1:f\circ g\simeq \id_Y$, $h_2:g\circ f\simeq \id_X$ such that for all $\sigma\in Y$, 
\begin{eqnarray}
 f(g(\mathring{\sigma})) &\subset& \sigma, \label{geqn} \\
 h_1(\mathring{\sigma}) &\subset& \sigma, \label{h1eqn} \\
 f(h_2(f^{-1}(\mathring{\sigma}))) &\subset& \sigma. \label{h2eqn}
\end{eqnarray}
\item $f$ is a $Y$-triangular homotopy equivalence if and only if there exists a homotopy inverse $g:Y\to X$ and homotopies $h_1:f\circ g\simeq \id_Y$, $h_2:g\circ f\simeq \id_X$ such that for all closed simplices $\sigma\in Y$ 
\begin{displaymath}
 \xymatrix{ f^{-1}(\sigma) \ar@<0.5ex>[r]^-{f|} & \sigma \ar@<0.5ex>[l]^-{g|}
}
\end{displaymath}
is a homotopy equivalence with homotopies $h_1|$ and $h_2|$. 
\end{enumerate}
\qed\end{rmk}

One might hope that it is possible to subdivide a $Y$-triangular homotopy equivalence so as to obtain an $Sd\, Y$-triangular homotopy equivalence. This would show that a $Y$-triangular homotopy equivalence is an $\ep$-controlled homotopy equivalence for all $\ep>0$. We can show the homotopy converse:
\begin{thm}[Squeezing]\label{topsqueezing}
 Let $X,Y$ be finite-dimensional locally finite simplicial complexes, then there exists an $\ep=\ep(X,Y)$ such that for any simplicial map $f:X\to Y$, if $f$ is an $\ep$-controlled homotopy equivalence, then $f$ is homotopic to a $Y$-triangular homotopy equivalence.
\end{thm}

\begin{proof}
Let $\ep^\prime(X)$ and $i(X)$ be chosen as in Corollary \ref{setupmapsqueeze} so that there exist subdivision chain equivalences 
\begin{displaymath}
\xymatrix{ (\Delta^{lf}_*(X), d_{\Delta^{lf}_*(X)},0) \ar@<0.5ex>[rr]^-{s_X} && (\Delta^{lf}_*(Sd^i\,X), d_{\Delta^{lf}_*(Sd^i\,X)}, P_X) \ar@<0.5ex>[ll]^-{r_X} \\
(\Delta^{lf}_*(Y), d_{\Delta^{lf}_*(Y)},0) \ar@<0.5ex>[rr]^-{s_Y} && (\Delta^{lf}_*(Sd^i\,Y), d_{\Delta^{lf}_*(Sd^i\,Y)}, P_Y). \ar@<0.5ex>[ll]^-{r_Y}}
\end{displaymath}
We can think of these chain equivalences on the level of spaces where $s_X=\id_X$, $s_Y=\id_Y$, $r_X$ and $r_Y$ are PL maps, $P_X: r_X\sim \id_X$ and $P_Y:r_Y\sim\id_Y$. Thus we have the following subdivision homotopy equivalences:
\begin{displaymath}
\xymatrix{ (X,P_X) \ar@<0.5ex>[rr]^-{s_X=\id_X} && (X, P_X) \ar@<0.5ex>[ll]^-{r_X} \\
(Y,P_Y) \ar@<0.5ex>[rr]^-{s_Y=\id_Y} && (Y, P_Y). \ar@<0.5ex>[ll]^-{r_Y}}
\end{displaymath}
Setting $\ep = \frac{1}{3}\ep^\prime$, by Corollary \ref{setupmapsqueeze} we have that for all $\sigma\in Y$, 
\begin{eqnarray}
 r_X(N_{k\ep}(f^{-1}(\sigma))) &\subset& f^{-1}(\sigma), \label{eqnrX} \\
 r_Y(N_{k\ep}(\sigma)) &\subset& \sigma, \label{eqnrY} \\
 P_X(N_{k\ep}(f^{-1}(\sigma)),I) &\subset& N_{k\ep}(f^{-1}(\sigma)), \label{eqnPX} \\
 P_Y(N_{k\ep}(\sigma),I) &\subset& N_{k\ep}(\sigma), \label{eqnPY}
\end{eqnarray}
for all $0\leqslant k \leqslant 3$.

Let $g_\ep$ be an $\ep$-controlled homotopy inverse to $f$ together with homotopies $h_1:f\circ g_\ep\simeq\id_Y$ and $h_2:g_\ep\circ f\simeq \id_X$. Let $f^\prime = r_Y\circ f$ and $g^\prime=r_X\circ g_\ep$. Since $r_X$ and $r_Y$ are both simplicial approximations to the identity we have that $r_X\sim \id_X$ and $r_Y\sim \id_Y$. Whence $g^\prime$ is a homotopy inverse to $f^\prime$.

By remark \ref{rmkYtriang} $(i)$, $f$ is automatically a $Y$-triangular map. Also, $r_Y$ retracts a $5\ep$-neighbourhood of each simplex $\sigma\in Y$ to that simplex which means it is also a $Y$-triangular map. Consequently $f^\prime$ is $Y$-triangular as it is the composition of $Y$-triangular maps. We see that $g^\prime$ is $Y$-triangular because
\begin{eqnarray*}
 f(g^\prime(\id_Y^{-1}(\mathring{\sigma}))) &=& f(r_X(g_\ep(\mathring{\sigma}))) \\
 &\subset& f(r_X(N_\ep(f^{-1}(\sigma)))) \\
 &\subset& f(f^{-1}(\sigma)) = \sigma, 
\end{eqnarray*}
where first we used the fact that $g_\ep$ has control $\ep$ and then we used equation $(\ref{eqnrX})$.

Now we check that the new homotopies are also $Y$-triangular:
\begin{eqnarray*}
f^\prime\circ g^\prime &=& r_Y\circ f\circ s_X\circ r_X \circ g_\ep\circ s_Y \\
&\simeq& r_Y\circ f \circ g_\ep\circ s_Y \\
&\simeq& r_Y\circ s_Y \\
&\simeq& \id_Y.
\end{eqnarray*}
Since the composition of $Y$-triangular homotopies is again $Y$-triangular it suffices to check each of the above homotopies is $Y$-triangular. 

The first is $r_Y\circ f\circ P_X \circ g_\ep\circ s_Y$. We verify that it is $Y$-triangular using an arbitrary $\mathring{\sigma}\in Y$:
\begin{eqnarray*}
 f(r_Y\circ f \circ P_X(g_\ep \circ s_Y(\id_Y^{-1}(\mathring{\sigma})),I)) &\subset& f(r_Y(f(P_X(N_\ep(f^{-1}(\sigma)),I)))) \\
 &\subset& f(r_Y(f(N_\ep(f^{-1}(\sigma))))) \\
 &\subset& f(r_Y(N_{2\ep}(f^{-1}(\sigma)))) \\
 &\subset& f(f^{-1}(\sigma)) = \sigma.
\end{eqnarray*}
Here in the same order we have used the fact that $g_\ep$ has control $\ep$, equation $(\ref{eqnPX})$, the fact $f$ has control $\ep$ and equation $(\ref{eqnrY})$. 

The second homotopy, $r_Y\circ h_1 \circ s_Y$, is seen to be $Y$-triangular by the following calculation:
\begin{eqnarray*}
 \id_Y(r_Y\circ h_1 (s_Y(\id_Y^{-1}(\mathring{\sigma})),I)) &=& r_Y\circ h_1 (\mathring{\sigma},I) \\
 &\subset& r_Y(N_\ep(\sigma)) \\
 &\subset& \sigma.
\end{eqnarray*}

The final homotopy $P_Y$ is $Y$-triangular since by equation $(\ref{eqnPY})$, \[\id_Y(P_Y(\id_Y^{-1}(\mathring{\sigma}),I)) \subset \sigma.\]
A similar analysis shows that the composition $g^\prime\circ f^\prime$ is $Y$-triangular homotopic to $\id_X$ which completes the proof.
\end{proof}

One might hope that using this squeezing theorem it is possible to prove a splitting theorem in a similar manner to Chapter \ref{chapten}.

%% file: Thesis2.tex
\chapter{Chain complexes over $X$}\label{chapseven}
We now pass from topology to algebra and consider an algebraic analogue of Theorem \ref{maintopthm}. To do this effectively we need to keep track of the simplicial structure with the algebra. Geometric categories are perfectly suited for this task.

The idea to consider algebraic objects parametrised by a topological space was introduced by Coram and Duvall in \cite{ConnHoll} in which they define geometric groups. Quinn extended this concept with his definition of geometric modules in \cite{QuinnGeomAlg}. We work with a further generalisation - geometric categories, whose objects are formal direct sums of objects in an additive category $\A$ parametrised by points in a space $X$, and whose morphisms are collections of morphisms of $\A$ between the summands of the objects. 

\section{Definitions and examples}\label{subsection5pt1}
Let $\A$ be an additive category.
\begin{defn}
A chain complex $C$ of objects and morphisms in $\A$ is called \textit{finite} if $C_i=0$ for all but finitely many $i\in \Z$.
\qed\end{defn}
\begin{defn}\label{definech}
We denote by $ch(\A)$ the additive category of finite chain complexes in $\A$ and chain maps.
\qed\end{defn}
All the chain complexes we consider in this thesis will be finite.

\begin{notn}
Let $C$ be a chain complex in $ch(\A)$. In what follows we will use the following slight abuse of notation: $C\otimes\Z$. This does not mean that $C_n$ is a group for each $n$ and that we are tensoring over its $\Z$ action. It is a notational convenience, and when we replace $\Z$ with a finite free $\Z$-module chain complex that is chain equivalent to $\Z$, we are actually appealing to the property of additive categories that finite direct sums exist (see \cite{maclane} page 250):

For all pairs of objects $A_1,A_2\in \A$ there exists an object $B$ and four morphisms forming a diagram
\begin{displaymath}
 \xymatrix@1{A_1 \ar@<0.5ex>[r]^-{i_1}  & B \ar@<0.5ex>[l]^-{\pi_1} \ar@<-0.5ex>[r]_-{\pi_2} & A_2 \ar@<-0.5ex>[l]_-{i_2} 
}
\end{displaymath}
 with $\pi_1i_1=\id_{A_1}$, $\pi_2i_2=\id_{A_2}$, $i_1\pi_1+i_2\pi_2=\id_{B}$.
\end{notn}

\begin{defn}\label{geometriccategories}
\begin{enumerate}[(i)]
 \item Given a simplicial complex $X$ and an additive category $\A$ define the \textit{$X$-graded} category $\mathbb{G}_X(\A)$ to be the additive category whose objects are collections of objects of $\A$, $\{M(\sigma)\,|\, \sigma \in X \}$, indexed by the simplices of $X$, written as a direct sum \[\sum_{\sigma\in X} M(\sigma) \] and whose morphisms \[f = \{f_{\tau,\sigma}\}: L = \sum_{\sigma\in X}L(\sigma) \to M = \sum_{\tau\in X}M(\tau) \]are collections $\{f_{\tau,\sigma}:L(\sigma) \to M(\tau)\,|\,\sigma,\tau \in X \}$ of morphisms in $\A$ such that for each $\sigma\in X$, the set $\{\tau\in X\,|\, f_{\tau,\sigma}\neq 0 \}$ is finite. It is convenient to regard $f$ as a matrix with one column $\{f_{\tau,\sigma}\,|\,\tau\in X \}$ for each $\sigma\in X$ (containing only finitely many non-zero entries) and one row $\{f_{\tau,\sigma}\,|\,\sigma\in X \}$ for each $\tau\in X$.

The composition of morphisms $f:L\to M$, $g:M\to N$ in $\mathbb{G}_X(\A)$ is the morphism $g\circ f:L\to N$ defined by \[(g\circ f)_{\rho, \sigma} = \sum_{\tau\in X} g_{\rho,\tau}f_{\tau,\sigma}: L(\sigma) \to N(\rho) \]where the sum is actually finite.
 \item Let $\left\{\begin{array}{c} \A^*(X) \\ \A_*(X) \end{array}\right.$ be the additive category with objects $M$ in $\mathbb{G}_X(\A)$ and with morphisms $f:M\to N$ such that $f_{\tau,\sigma}:M(\sigma)\to N(\tau)$ is $0$ unless $\left\{ \begin{array}{c} \tau\leqslant \sigma \\ \tau\geqslant \sigma \end{array}\right.$, i.e.\ satisfying \[\left\{ \begin{array}{c} f(M(\sigma))\subseteq \sum_{\tau\leqslant \sigma}{N(\tau)} \\ f(M(\sigma))\subseteq \sum_{\tau\geqslant \sigma}{N(\tau)}. \end{array}\right. \]
\end{enumerate}
\qed\end{defn}

The categories $\A^*(X)$ and $\A_*(X)$ are originally due to Ranicki and Weiss \cite{ranickiweiss}.
\begin{rmk}\label{lochfyne}
For a finite simplicial complex we can order our simplices by dimension from smallest to largest to observe that morphisms in $\left\{ \begin{array}{c} \A^*(X) \\ \A_*(X) \end{array}\right.$ are $\left\{ \begin{array}{c} \mathrm{upper} \\ \mathrm{lower} \end{array}\right.$ triangular matrices. Not all triangular matrices correspond to valid $\left\{ \begin{array}{c} \A^*(X) \\ \A_*(X) \end{array}\right.$ morphisms, since a component of a morphism cannot be non-zero between two simplices if neither is a face of the other.

Locally finite simplicial complexes are countable, so again we can think of the morphisms as triangular with respect to some choice of ordering of the simplices. We do this carefully by first listing the $0$-simplices arbitrarily and then proceeding inductively. Suppose we have an ordering of all $i$-simplices, for $i<n$, such that all faces of any given simplex appear earlier in the list. We then add each of the $(n+1)$-simplices $\sigma$ to the list at the place immediately after the last simplex in the list contained in $\sigma$. Each simplex is thus put in a finite position in the list.  
\qed\end{rmk}

\begin{notn}\label{eithercategory}
Due to the similar nature of the categories $\A^*(X)$ and $\A_*(X)$ we will often want to consider both categories simultaneously. We therefore introduce the notation $\A(X)$ to denote \textit{either $\A^*(X)$ or $\A_*(X)$}. 
\end{notn}

\begin{defn}\label{definesupports}
Let $M$ be an object of $\A(X)$. We define the \textit{support of $M$}, denoted by Supp$(M)$, to be the union of all open simplices $\mathring{\sigma}$ such that $M(\sigma)\neq 0$. Similarly for a chain complex $C\in ch(\A(X))$, \[\mathrm{Supp}(C) := \{\mathring{\sigma}\,|\,\sigma\in X, C(\sigma)\neq 0\}.\] 
\qed\end{defn}

\begin{defn}
Often it will be convenient to refer to chain complexes or chain equivalences as \textit{$X$-graded} or \textit{``over $X$''}. This will mean that they are chain complexes or chain equivalences in $\mathbb{G}_X(\A)$. 
\qed\end{defn}

Many examples and applications will involve the simplicial chain and cochain complexes of a simplicial complex, for which our category $\A$ will be the category $\F(R)$ of finitely generated free $R$-modules for some ring $R$. In keeping with \cite{bluebk} we use the following notation:

\begin{defn}
 Let $R$ be a ring. We denote by $\brc{A(R)^*(X)}{A(R)_*(X)}$ the category $\brc{\A^*(X)}{\A_*(X)}$ for $\A=\F(R)$.
\qed\end{defn}

\begin{ex}\label{sup}
The categories $\A^*(X)$ and $\A_*(X)$ are the correct categories to be considering because for $\A = \F(\Z)$ the simplicial $\left\{ \begin{array}{c} \mathrm{chain} \\ \mathrm{cochain} \end{array}\right.$ complex $\left\{ \begin{array}{c} \Delta^{lf}_*(X) \\ \Delta^{-*}(X) \end{array}\right.$ of $X$ is naturally a chain complex in $\left\{ \begin{array}{c} \A^*(X) \\ \A_*(X) \end{array}\right.$ with $\left\{ \begin{array}{cc} \Delta^{lf}_i(X)(\sigma) = \Z\langle \mathring{\sigma} \rangle, & i=|\sigma| \\ \Delta^{-i}(X)(\sigma) = \Z\langle \mathring{\sigma} \rangle, & -i=|\sigma| \end{array}\right.$ and $0$ otherwise.
\qed\end{ex}

\begin{ex}\label{easyex}
Let $C$ be the simplicial chain complex of the unit interval $[0,1]$. Then $C$ is the chain complex 
\begin{displaymath}
 \xymatrix{
0 \ar[r] & \Z \ar[r]^-{\binom{1}{-1}} & \Z^2 \ar[r] & 0.
}
\end{displaymath}
We consider this as a chain complex in $\A^*([0,1])$ as in Example \ref{sup} as follows:
\begin{displaymath}
 \xymatrix@R=3mm{
C(0): & 0 \ar[r] & 0 \ar[r] & \Z \ar[r] & 0 \\
C(01): & 0 \ar[r] & \Z \ar[ru]^-{-1} \ar[r] \ar[dr]_-{1} & 0 \ar[r] & 0 \\
C(1): & 0 \ar[r] & 0 \ar[r] & \Z \ar[r] & 0.
}
\end{displaymath}
\qed\end{ex}

In order to obtain finer and finer control we will need to ``subdivide'' a chain complex in $ch(\A(X))$ yielding one in $ch(\A(Sd\, X))$. First we look at the effect of barycentric subdivision $Sd\, X$ on the simplicial chain and cochain complexes:

\begin{ex}\label{geomcx}
Like in Example \ref{sup} above, $\left\{ \begin{array}{c} \Delta^{lf}_*(Sd\, X) \\ \Delta^{-*}(Sd\, X) \end{array}\right.$ is naturally a chain complex in $\left\{ \begin{array}{c} A(\Z)^*(Sd\, X) \\ A(\Z)_*(Sd\, X) \end{array}\right.$ with $\left\{ \begin{array}{cc} \Delta^{lf}_i(Sd\, X)(\tau) = \Z\langle \tau \rangle, & i=|\tau| \\ \Delta^{-i}(Sd\, X)(\tau) = \Z\langle \tau \rangle, & -i=|\tau| \end{array}\right.$ and $0$ otherwise.
\qed\end{ex}

We may assemble objects, morphisms and hence chain complexes over $Sd\, X$ either covariantly or contravariantly to obtain ones over $X$. These procedures give rise to covariant and contravariant functors. Assembly takes a collection of points on which parts of $C$ are supported and groups these parts of $C$ together by direct summing them and supporting them at a single point. As we shall see in a moment, we have to be slightly careful how we do this in order to produce something in $\A^*(X)$ or $\A_*(X)$. 

\begin{defn}\label{restrictioncomplex}
Let $Y$ be a collection of open simplices in $X$. We define restrictions to $Y$ in the obvious way:
\begin{enumerate}[(i)]
 \item Let $M$ be an object in $\mathbb{G}_{X}(\A)$. We define the \textit{restriction of $M$ to $Y$} to be the object $M|_Y$ in $\mathbb{G}_{X}(\A)$ given by \[(M|_Y)(\sigma):= \brcc{M(\sigma),}{\mathring{\sigma}\in Y}{0,}{\mathrm{o/w}.}\]
 \item Let $f:M\to N$ be a morphism in $\mathbb{G}_{X}(\A)$. We define the \textit{restriction of $f$ to $Y$} to be the morphism $f|_Y: M|_Y \to N|_Y$ in $\mathbb{G}_{X}(\A)$ given by \[ (f|_Y)_{\tau,\sigma} := \brcc{f_{\tau,\sigma},}{\mathring{\sigma},\mathring{\tau}\in Y}{0,}{\mathrm{o/w}.}\]
 \item Let $C$ be a chain complex in $\mathbb{G}_{X}(\A)$. We define the \textit{restriction of $C$ to $Y$}, denoted by $C_Y$, by 
 \begin{eqnarray*}
  (C_Y)_n &:=& C_n|_Y \\
  (d_{C_Y})_n &:=& (d_C)_n|_Y.
 \end{eqnarray*}
In general this restriction is not a chain complex.
\end{enumerate}
\qed\end{defn}

\begin{lem}
Let $C$ be a chain complex in $\A(X)$ and let $Y$ be a collection of open simplices in $X$, then a sufficient condition for $C|_Y$ being a chain complex in $\A(X)$ is that  
\begin{equation}\label{five}
\forall \mathring{\rho},\mathring{\sigma} \in Y \; \mathrm{with} \;\, \rho\leqslant \sigma, \;\, \mathring{\tau}\in Y \;  \forall \rho\leqslant \tau\leqslant \sigma.
\end{equation} 
\end{lem}

\begin{proof}
For $\A^*(X)$, since $C$ is a chain complex we know that $d_C^2=0$. For $\rho\leqslant \sigma$, \[(d_C^2)_{\rho,\sigma} = \sum_{\rho\leqslant \tau\leqslant \sigma}{(d_C)_{\rho\tau}(d_C)_{\tau\sigma}} \] 
so if any $\mathring{\tau}$ are missing from $Y$, $(d_C)|_Y^2$ may not be zero and hence $C$ may not be a chain complex. It certainly is, however, if all such $\mathring{\tau}$ are present. A similar argument also holds for $\A_*(X)$.
\end{proof}

\begin{defn}\label{assembleobjects}
Let $Y$, $Y^\prime$ be finite collections of open simplices in $X$. 
\begin{enumerate}[(i)]
 \item Let $M$ be an object in $\mathbb{G}_{X}(\A)$. We define the \textit{assembly of $M$ over $Y$} to be the object $M[Y]$ in $\A$ given by \[M[Y] := \sum_{\mathring{\sigma}\in Y}{M(\sigma)}.\]
 \item Let $f:M\to N$ be a morphism in $\mathbb{G}_{X}(\A)$. We define the \textit{assembly of $f$ from $Y$ to $Y^\prime$} to be the morphism $f_{[Y^\prime],[Y]}: M[Y] \to N[Y^\prime]$ in $\A$ given by the matrix \[ f_{[Y^\prime],[Y]} := \{f_{\mathring{\sigma},\mathring{\tau}}\}_{\mathring{\sigma}\in Y^\prime,\mathring{\tau}\in Y}\] with respect to the direct sum decompositions of $M[Y]$ and $N[Y^\prime]$ as defined above.
 \item Let $C$ be a chain complex in $\A(X)$ and let the collection of open simplices $Y$ additionally satisfy condition $(\ref{five})$, then the \textit{assembly of $C$ over $Y$}, $C[Y]$, given by 
\begin{eqnarray*}
 C[Y]_n &:=& C_n[Y] \\
 (d_{C[Y]})_n &:=& ((d_C)_n)_{[Y],[Y]}
\end{eqnarray*}
is a chain complex in $\A$. 
\end{enumerate}
\qed\end{defn}

For the barycentric subdivision $Sd\, X$ of a locally finite simplicial complex $X$, both of the following collections of open simplices satisfy condition $(\ref{five})$, so for chain complexes in $\A(X)$ we can assemble over these collections to obtain chain complexes in $\A$:
\begin{enumerate}
 \item The collection of all open simplices in the subdivision of an open simplex $\mathring{\sigma}\in X$, 
 \item The collection of all open simplices in the open dual cell $\mathring{D}(\sigma,X)$ of any simplex $\sigma\in X$.
\end{enumerate}
Grouping these collections for all simplices in $X$ gives rise to covariant and contravariant assembly functors respectively.

\begin{defn}\label{covariantassembly}
We define the \textit{covariant assembly} functor $\rr: \mathbb{G}_{Sd\,X}(\A) \to \mathbb{G}_{X}(\A)$ as follows:
\begin{itemize}
 \item An object $M$ in $\mathbb{G}_{Sd\,X}(\A)$ is sent to the object $\rr(M)$ in $\mathbb{G}_{X}(\A)$ defined by \[\rr(M)(\sigma):= M[\mathring{\sigma}].\]
 \item A morphism $f:M \to N$ in $\mathbb{G}_{Sd\,X}(\A)$ is sent to the morphism $\rr(f):\rr(M)\to \rr(N)$ in $\mathbb{G}_{X}(\A)$ defined by \[ \rr(f)_{\tau,\sigma,n} := (f_n)_{[\mathring{\sigma}],[\mathring{\tau}]}.\] 
\end{itemize}
Since $\rr$ does not change the direction of morphisms it restricts to give functors
\[ \rr: \brc{\A^*(Sd\, X) \to \A^*(X)}{\A_*(Sd\, X) \to \A_*(X).}\]
\qed\end{defn}

\begin{ex}\label{covregrp}(Covariant assembly) 
Let $C$ be a chain complex in $\A^*(Sd\, \sigma)$ for the $2$-simplex $\sigma$ labelled as usual, i.e.\ $d_C$ only has non-zero components from simplices of $Sd\, \sigma$ to their boundary. Then $\rr(C)$ is assembled as depicted in \Figref{Fig:new22}, where the whole block in the middle is $\rr(C)(\sigma)$. Note in particular that $d_{\rr(C)}$ again only has non-zero components from simplices of $\sigma$ to its boundary, so in particular $\rr(C)$ is a chain complex in $\A^*(\sigma)$ as claimed.
\begin{figure}[ht]
\begin{center}
{
\psfrag{0}[tr][]{$\rr(C)(\rho_0)$}
\psfrag{1}[b][b]{$\rr(C)(\rho_1)$}
\psfrag{2}[tl][]{$\rr(C)(\rho_2)$}
\psfrag{01}[r][r]{$\rr(C)(\tau_0)$}
\psfrag{02}[t][]{$\rr(C)(\tau_2)$}
\psfrag{12}[l][l]{$\rr(C)(\tau_1)$} 
\includegraphics[width=7cm]{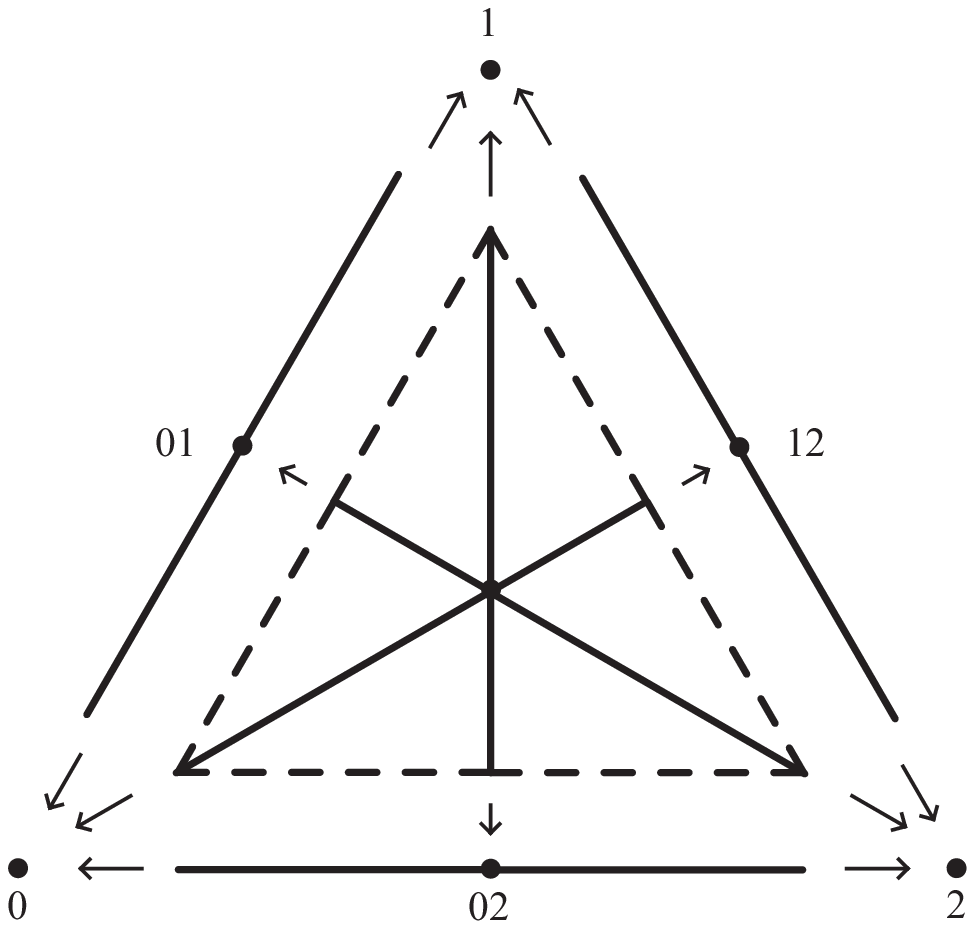}
}
\caption{$\rr(C)$ as a chain complex in $\A^*(\sigma)$.}
\label{Fig:new22}
\end{center}
\end{figure}

Suppose instead that $C$ is in $\A_*(Sd\, \sigma)$. Then assembling gives exactly the same diagram but with all the arrows reversed, and hence $\rr(C)$ would be a chain complex in $\A_*(\sigma)$.
\qed\end{ex}

\begin{defn}\label{contravariantassembly}
We define the \textit{contravariant assembly} functor $\ttt: \mathbb{G}_{Sd\,X}(\A) \to \mathbb{G}_{X}(\A)$ as follows:
\begin{itemize}
 \item An object $M$ in $\mathbb{G}_{Sd\,X}(\A)$ is sent to the object $\ttt(M)$ in $\mathbb{G}_{X}(\A)$ defined by \[\ttt(M)(\sigma):= M[\mathring{D}(\sigma,X)].\]
 \item A morphism $f:M \to N$ in $\mathbb{G}_{Sd\,X}(\A)$ is sent to the morphism $\ttt(f):\ttt(M)\to \ttt(N)$ defined by \[ \ttt(f)_{\tau,\sigma,n} := (f_n)_{[\mathring{D}(\tau,X)],[\mathring{D}(\sigma,X)]}.\]
 \item Since $\ttt$ reverses the direction of morphisms it restricts to give functors
\[ \ttt: \brc{\A^*(Sd\, X) \to \A_*(X)}{\A_*(Sd\, X) \to \A^*(X).}\]
\end{itemize}
\qed\end{defn}

\begin{ex}\label{contraregrp}(Contravariant assembly)
Let $C$ be a chain complex in $\A^*(Sd\, \sigma)$ for the $2$-simplex $\sigma$, i.e.\ $d_C$ only has non-zero components from simplices of $Sd\, \sigma$ to their boundary. Then $\ttt(C)$ is assembled as depicted in \Figref{Fig:new1}. Note in particular that $d_{\ttt(C)}$ now only has non-zero components from boundaries of simplices in $\sigma$ to their interiors, so in particular $\ttt(C)$ is a chain complex in $\A_*(\sigma)$ as claimed.
\begin{figure}[ht]
\begin{center}
{
\psfrag{0}[r][r]{$\ttt(C)(\rho_0)$}
\psfrag{1}[b][b]{$\ttt(C)(\rho_1)$}
\psfrag{2}[l][l]{$\ttt(C)(\rho_2)$}
\psfrag{01}[][r]{$\ttt(C)(\tau_0)$}
\psfrag{02}[t][]{$\ttt(C)(\tau_2)$}
\psfrag{12}[l][l]{$\ttt(C)(\tau_1)$}
\includegraphics[width=7cm]{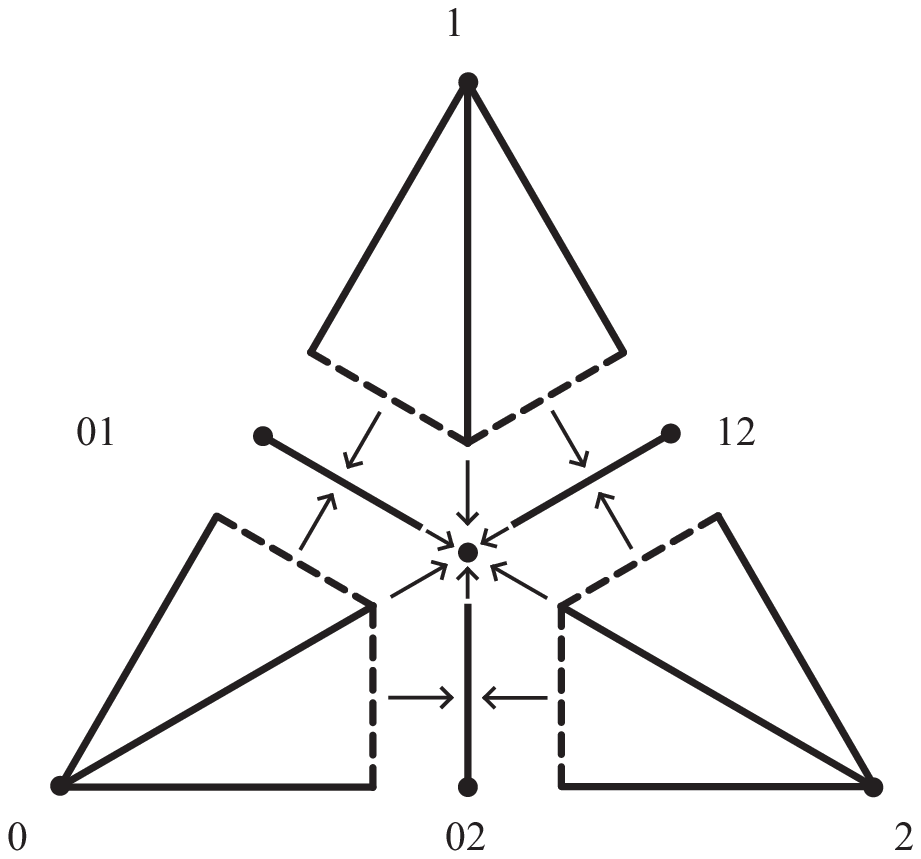}
}
\caption{$\ttt(C)$ as a chain complex in $\A_*(\sigma)$.}
\label{Fig:new1}
\end{center}
\end{figure}

Suppose instead that $C$ is in $\A_*(Sd\, \sigma)$. Then assembling gives exactly the same diagram but with all the arrows reversed, and hence $\ttt(C)$ would be a chain complex in $\A^*(\sigma)$.
\qed\end{ex}
We now give a simple explicit example.
\begin{ex}
Let $X$ be the $1$-simplex $v_0v_1$ so that $Sd\, X = \widehat{v_0}\widehat{v_0v_1} \cup \widehat{v_0v_1}\widehat{v_1}$. Let $C$ be any chain complex in $\A^*(Sd\, X)$, then explicitly $C$ has the form
\begin{displaymath}
 \xymatrix@R=3mm{
C(\widehat{v_0}): & \ldots \ar[r] & C_n(\widehat{v_0}) \ar[r]^-{\alpha} & C_{n-1}(\widehat{v_0}) \ar[r] & \ldots \\
C(\widehat{v_0}\widehat{v_0v_1}): & \ldots \ar[r] \ar[ru] \ar[dr] & C_n(\widehat{v_0}\widehat{v_0v_1}) \ar[r]^-{\gamma} \ar[ru]^-{\beta} \ar[dr]^-{\delta}  & C_{n-1}(\widehat{v_0}\widehat{v_0v_1}) \ar[r] \ar[ru] \ar[dr] & \ldots \\
C(\widehat{v_0v_1}): & \ldots \ar[r] & C_n(\widehat{v_0v_1}) \ar[r]^-{\epsilon} & C_{n-1}(\widehat{v_0v_1}) \ar[r] & \ldots \\
C(\widehat{v_0v_1}\widehat{v_1}): & \ldots \ar[r] \ar[ru] \ar[dr]  & C_n(\widehat{v_0v_1}\widehat{v_1}) \ar[r]^-{\eta} \ar[ru]^-{\zeta} \ar[dr]^-{\theta} & C_{n-1}(\widehat{v_0v_1}\widehat{v_1}) \ar[r] \ar[ru] \ar[dr] & \ldots \\
C(\widehat{v_1}): & \ldots \ar[r] & C_n(\widehat{v_1}) \ar[r]^-{\iota} & C_{n-1}(\widehat{v_1}) \ar[r] & \ldots
}
\end{displaymath}
Assembling covariantly gives $\rr(C)$ as below:
\begin{displaymath}
 \xymatrix{
\rr{C}(v_0) = C(\widehat{v_0}): & \ldots \ar[r] & C_n(\widehat{v_0}) \ar[rr]^-{\alpha} && C_{n-1}(\widehat{v_0}) \ar[r] & \ldots \\
\rr{C}(v_0v_1) = C[(v_0,v_1)]: & \ldots \ar[r] \ar[ur] \ar[dr] & C[(v_0,v_1)]_n \ar[rr]^{\mu} \ar[urr]^-{(\beta,0,0)} \ar[drr]^-{(0,0,\theta)} &&  C[(v_0,v_1)]_{n-1} \ar[r] \ar[ru] \ar[dr] & \ldots \\
\rr{C}(v_1) = C(\widehat{v_1}): & \ldots \ar[r] & C_n(\widehat{v_1}) \ar[rr]^-{\iota} && C_{n-1}(\widehat{v_1}) \ar[r] & \ldots
}
\end{displaymath}
where $C[(v_0,v_1)]_n = C_n(\widehat{v_0}\widehat{v_0v_1})\oplus C_n(\widehat{v_0v_1}) \oplus C_n(\widehat{v_0v_1}\widehat{v_1})$ and \[\mu = \left(\begin{array}{ccc} \gamma & 0 & 0 \\ \delta & \ep & \zeta \\ 0 & 0 & \eta \end{array} \right).\]
Assembling contravariantly gives $\ttt(C)$ as below:

\begin{displaymath}
 \xymatrix@C=0.8cm{
\ttt(C)(v_0) = C[\mathring{D}(v_0,X)]: & \ldots \ar[r] \ar[dr] & C_n[\mathring{D}(v_0,X)] \ar[rr]^-{{\left(\begin{array}{cc} \alpha & \beta \\ 0 & \gamma \end{array}\right)}} \ar[drr]^-{(0,\delta)} && C_{n-1}[\mathring{D}(v_0,X)] \ar[r] \ar[dr] & \ldots \\
\ttt(C)(v_0v_1) = C(\widehat{v_0v_1}): & \ldots \ar[r] & C_n(\widehat{v_0v_1}) \ar[rr]^-{\epsilon} && C_{n-1}(\widehat{v_0v_1}) \ar[r] & \ldots \\
\ttt(C)(v_1) = C[\mathring{D}(v_1,X)]: & \ldots \ar[r] \ar[ur] & C_n[\mathring{D}(v_1,X)] \ar[rr]_-{{\left(\begin{array}{cc} \iota & \theta \\ 0 & \eta \end{array}\right)}} \ar[urr]^-{(0,\zeta)} && C_{n-1}[\mathring{D}(v_1,X)] \ar[r] \ar[ur] & \ldots
}
\end{displaymath}
where $C_n[\mathring{D}(v_0,X)] = C_n(\widehat{v_0})\oplus C_n(\widehat{v_0}\widehat{v_0v_1})$ and $C_n[\mathring{D}(v_1,X)] = C_n(\widehat{v_0v_1}\widehat{v_1})\oplus C_n(\widehat{v_1})$.
\qed\end{ex}

\begin{defn}\label{algmapcone}
Let $f:C\to D$ be a chain map. We define the \textit{algebraic mapping cone of $f$}, denoted by $\C(f:C\to D)$, to be the chain complex with 
\begin{equation}\label{conenchain}
\C(f:C\to D)_n:= C_n\oplus D_{n+1}
\end{equation} 
and boundary maps 
\begin{equation}\label{coneboundarymap}
(d_{\C(f)})_n = \left(\begin{array}{cc} (d_C)_n & 0 \\ f_n & -(d_D)_{n+1} \end{array}\right).
\end{equation} 
\qed\end{defn}

The types of chain maps we will mostly be concerned with are chain equivalences. A chain map $f$ is a chain equivalence if and only if the algebraic mapping cone on $f$ is chain contractible. We will favour working with chain complexes and considering whether or not they are contractible and this will also deal with whether chain maps are chain equivalences or not. Of course it is possible that the chain map is induced by an actual map on the spaces.

\begin{lem}\label{conestillincat}
The algebraic mapping cone of an $\brc{\A^*(X)}{\A_*(X)}$ chain map $f:C\to D$ is a chain complex in $\brc{\A^*(X)}{\A_*(X).}$
\end{lem}
\begin{proof}
$\C(f)(\sigma) = C_n(\sigma)\oplus D_{n+1}(\sigma)$ and the boundary map 
\[(d_{\C(f)})_{\tau,\sigma,n} = \left(\begin{array}{cc} (d_C)_{\tau,\sigma,n} & 0 \\ f_{\tau,\sigma,n} & -(d_D)_{\tau,\sigma,n+1} \end{array}\right)\] is only non-zero for $\brc{\tau \leqslant \sigma}{\tau\geqslant\sigma.}$
\end{proof}

\begin{lem}\label{carvesimpmap}
Let $f:Y \to X$ be a proper simplicial map between locally finite simplicial complexes. Then
\begin{enumerate}[(i)]
 \item $f$ induces an $A(\Z)^*(X)$ chain map \[f_* : \Delta^{lf}_*(Y) \to \Delta^{lf}_*(X) \] by setting for all $\sigma\in X$ \[ \Delta^{lf}_*(Y)(\sigma):= \Delta^{lf}_*(f^{-1}(\mathring{\sigma})),\] and viewing $\Delta^{lf}_*(X)$ as a chain complex in $A(\Z)^*(X)$ in the standard way as explained in Example \ref{sup}.
 \item $f$ induces an $A(\Z)_*(X)$ chain map \[(Sd\,f)_*: \Delta^{lf}_*(Sd\,Y) \to \Delta^{lf}_*(Sd\,X) \] by setting for all $\sigma\in X$ \[ \Delta^{lf}_*(Sd\,Y)(\sigma):= \Delta^{lf}_*(f^{-1}(\mathring{D}(\sigma,X))),\] and viewing $\Delta^{lf}_*(Sd\, X)$ as a chain complex in $A(\Z)_*(X)$ in the standard way as explained in Example \ref{contraregrp}.
\end{enumerate}
\end{lem}

\begin{cor}
Taking the algebraic mapping cone of $\brc{f_*}{(Sd\, f)_*}$ we can view $\brc{\C(f_*)}{\C((Sd\,f)_*)}$ as a chain complex in $\brc{A(\Z)^*(X)}{A(\Z)_*(X).}$ 
\end{cor}

\begin{ex}\label{chainmapoverX}
Consider the simplicial map 
\begin{eqnarray*}
 f: \sigma=v_0v_1v_2 &\to& \tau=w_0w_1 \\
\{v_0, v_1\} &\mapsto& \{w_0\} \\
\{v_2 \} &\mapsto& \{w_1\}.
\end{eqnarray*}
Then with respect to the decomposition $[0,1] = \{0\}\cup (0,1)\cup \{1\}$ the map $\alpha=(Sd\, f)_*$ is a diagonal matrix \[\left(\begin{array}{ccc}a&0&0 \\ 0&b&0 \\ 0&0&c \end{array}\right) \] for the maps $a,b,c$ as pictured in \Figref{Fig:Sdf}: 
\begin{figure}[ht]
\begin{center}
{
\psfrag{f0}[r][r]{$a$}
\psfrag{f1}[r][r]{$b$}
\psfrag{f2}[l][l]{$c$}
\includegraphics[width=7cm]{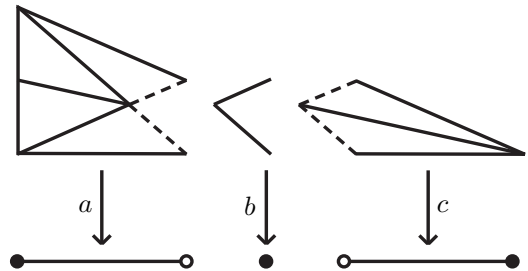}
}
\caption{How to think of $(Sd\,f)_*$ as a morphism in $A(\Z)_*(X)$.}
\label{Fig:Sdf}
\end{center}
\end{figure}
\qed\end{ex}
For the rest of the thesis we will consider chain maps on the level of chains irrespective of whether they are induced by actual maps on spaces or not.

\section{Contractibility in categories over simplicial complexes}
In this section we demonstrate the significance of the categories $\A(X)$, namely that a chain complex in $\A(X)$ is globally chain contractible if and only if it is locally chain contractible in $\A$ over each simplex. This result is instrumental in defining subdivision functors as it allows us to replace each piece of a chain complex in $\A(X)$ with something chain equivalent in $\A$ that is supported more finely over the barycentric subdivision. Morally this result is true because morphisms in $\A(X)$ are triangular matrices and the result is analogous to the statement that a triangular matrix is invertible if and only if all its diagonal entries are. 

\begin{prop}\label{chcont}
Let $X$ be a locally finite simplicial complex, and let $C\in\A(X)$. 
Then, 
\begin{enumerate}[(i)]
 \item $C$ is chain contractible in $\A(X)$ if and only if $C(\sigma)$ is chain contractible in $\A$ for all $\sigma\in X$.
 \item A chain map $f:C\to D$ of chain complexes in $\A(X)$ is a chain equivalence if and only if $f_{\sigma,\sigma}:C(\sigma)\to D(\sigma)$ is a chain equivalence in $\A$ for all $\sigma\in X$.
\end{enumerate}
\end{prop}
This is a well known result (see Prop. 4.7. \cite{bluebk}) for which we present a new proof.
\begin{proof}
We prove $(i)$ for the category $\A^*(X)$. Exactly the same proof works for $\A_*(X)$ replacing lower triangular matrices with upper. Part $(ii)$ follows by applying part $(i)$ to the algebraic mapping cone of the chain equivalence $f:C\to D$ which is a chain complex in $\A_*(X)$ by Lemma \ref{conestillincat}.

($\Rightarrow$)
Suppose we have a chain contraction $P_C:C\simeq 0\in\A^*(X)$. Then for all $\tau,\sigma\in X$, \[(d_CP_C + P_Cd_C)_{\tau\sigma} = \brcc{1_{C(\sigma)},}{\tau=\sigma}{0,}{\mathrm{otherwise}} \]so for the case $\tau=\sigma$ we observe that 
\[d_{\sigma\sigma}P_{\sigma\sigma} + P_{\sigma\sigma}d_{\sigma\sigma} = 1_{C(\sigma)}, \]
thus $P_{\sigma}:= (P_C)_{\sigma\sigma}$ is a chain contraction in $\A$ for $C(\sigma)$.

($\Leftarrow$)
\noindent 
Since $C(\sigma)$ is chain contractible for all $\sigma\in X$ there exist $P_{\sigma}$ such that
\begin{equation}
d_{\sigma\sigma}P_{\sigma} + P_{\sigma}d_{\sigma\sigma} = 1_{C(\sigma)}.
\end{equation}
We construct an $\A^*(X)$ chain contraction $P_C$ by setting
\[ P_{\tau\sigma}:= \sum_{i=0}^{|\sigma|-|\tau|}\sum_{\tau=\sigma_0<\ldots< \sigma_i=\sigma}{(-1)^iP_{\sigma_0}d_{\sigma_0\sigma_1}P_{\sigma_1}\ldots d_{\sigma_{i-1}\sigma_i}P_{\sigma_i}}.\]
First we check that $dP+Pd$ gives the identity down the diagonal. For $\tau=\sigma$, $i=0$ and \[ (d_CP_C+P_Cd_C)_{\tau\sigma} = d_{\sigma\sigma}P_{\sigma} + P_{\sigma}d_{\sigma\sigma} = 1_{C(\sigma)}.\]
Next we show that away from the diagonal $dP=Pd=0$ and so in particular $dP+Pd=0$. Suppose $\rho<\sigma$, then
\begin{eqnarray*}
(d_CP_C)_{\rho\sigma} &=& \sum_{\rho\leqslant \tau\leqslant \sigma}{d_{\rho\tau}P_{\tau\sigma}}\\
&=& d_{\rho\rho}P_{\rho\sigma} + \sum_{\rho< \tau\leqslant \sigma}{d_{\rho\tau}P_{\tau\sigma}}\\
&=& d_{\rho\rho}\sum_{i=0}^{|\sigma|-|\rho|}\sum_{\rho=\sigma_0<\ldots<\sigma_i=\sigma}{(-1)^iP_{\sigma_0}d_{\sigma_0\sigma_1}\ldots d_{\sigma_{i-1}\sigma_i}P_{\sigma_i}} + \sum_{\rho< \tau\leqslant \sigma}{d_{\rho\tau}P_{\tau\sigma}} \\
&=& d_{\rho\rho}P_{\rho}\sum_{\rho<\sigma_1\leqslant\sigma}{d_{\rho\sigma_1}\sum_{i=0}^{|\sigma|-|\sigma_i|}\sum_{\sigma_1<\ldots<\sigma_i=\sigma}{(-1)^iP_{\sigma_1}d_{\sigma_1\sigma_2}\ldots d_{\sigma_{i-1}\sigma_i}P_{\sigma_i}}}  \\
&& + \sum_{\rho< \tau\leqslant \sigma}{d_{\rho\tau}P_{\tau\sigma}} \\
&=& (1-P_{\rho\rho}d_{\rho\rho})(\sum_{\rho<\sigma_1\leqslant\sigma}{d_{\rho\sigma_1}(-1)P_{\sigma_1\sigma}}) + \sum_{\rho< \tau\leqslant \sigma}{d_{\rho\tau}P_{\tau\sigma}} \\
&=& -\sum_{\rho<\sigma_1\leqslant \sigma}{d_{\rho\sigma_1}P_{\sigma_1\sigma}} + \sum_{\rho< \tau\leqslant \sigma}{d_{\rho\tau}P_{\tau\sigma}} \\
&=& 0
\end{eqnarray*}

The proof that $(P_Cd_C)_{\rho\sigma}=0$ is essentially the same. Note that this formula for $P_{\tau\sigma}$ can only be non-zero if $\tau\leqslant\sigma$ because otherwise there are no inclusions $\tau=\sigma_0<\ldots< \sigma_i=\sigma$ to sum over. Thus $P$ as constructed is a chain contraction in $\A^*(X)$ as required.
\end{proof}

%

\chapter{Algebraic subdivision}\label{chapeight}
Simplicial complexes satisfy the following basic properties:

\begin{enumerate}[(i)]
 \item Any simplicial complex $X$ can be decomposed into a disjoint union of contractible spaces, namely its open simplices.
 \item Any bounded open simplex can be decomposed into the disjoint union of a collection of contractible spaces (the open simplices in its barycentric subdivision) which are strictly smaller\footnote{Smaller in the sense that the diameters of all the simplices in the subdivision are smaller than the diameter of the open simplex, and importantly they are smaller by a fixed scale factor dependent on the dimension of $X$.} than the original open simplex.
 \item Repeating this as many times as we like we can decompose a finite-dimensional simplicial complex, which initially had uniformly bounded simplex diameters, into one with arbitrarily small simplex diameters. 
\end{enumerate}
We would like chain complexes in $ch(\A(X))$ to have analogous properties. A chain complex in $ch(\A(X))$ already satisfies $(i)$ in that it has a piece over each simplex. In this chapter we strive to obtain analogues of properties $(ii)$ and $(iii)$ by defining an algebraic subdivision functor 
\[Sd: \brc{ch(\A^*(X)) \to ch(\A^*(Sd\, X))}{ch(\A_*(X)) \to ch(\A_*(Sd\, X))}\]which generalises the algebraic effect that barycentric subdivision has on the simplicial $\brc{\mathrm{chain}}{\mathrm{cochain}}$ complex.

In addition, if we seek to extend this work to a more general class of nice spaces, the spaces should satisfy analogues of the properties $(i)-(iii)$. In particular the same approach should work for finite-dimensional locally finite CW complexes.

\section{Algebraic subdivision functors}
The goal of this section is to define algebraic subdivision:
\begin{thm}
Given a barycentric subdivision chain equivalence $s:\Delta_*(X) \to \Delta_*(Sd\, X)$ together with a choice of chain inverse $r$ and (canonical) chain homotopy $P$, there are defined algebraic subdivision functors \[Sd_r: \brc{ch(\A^*(X)) \to ch(\A^*(Sd\, X))}{ch(\A_*(X)) \to ch(\A_*(Sd\, X))}\]depending on $r$.
\end{thm}
More is true in fact; in Section \ref{svptthr} it is proven that if $\rr_r$ is the covariant assembly functor that assembles all the $I_\sigma$, then for all chain complexes in $ch(\A(X))$ we have \[\rr_rSd_r\, C \simeq C \in \A(X).\]

Before diving into a rather technical proof we first attempt to convince the reader why algebraic subdivision is possible by explaining the approach taken. The strategy is to mimic the effect that barycentric subdivision has on the simplicial $\brc{\mathrm{chain}}{\mathrm{cochain}}$ complexes, where 
\begin{eqnarray*}
 \Delta_*(\mathring{\sigma}) &\simeq& \Delta_*(I_\sigma), \\
 \Delta^{-*}(\mathring{\sigma}) &\simeq& \Delta^{-*}(I_\sigma),
\end{eqnarray*}
since $\mathring{\sigma}$ is homotopy equivalent to $I_\sigma$ for all $\sigma\in X$. Now, 
\[\Z = \brc{\Sigma^{-|\sigma|}\Delta_*(\mathring{\sigma})}{\Sigma^{-|\sigma|}\Delta^{-*}(\mathring{\sigma})}\]
so locally we think of $C(\sigma)$ as 
\begin{eqnarray*}
C(\sigma)\otimes_\Z\Z &=& \brc{C(\sigma)\otimes_\Z\Sigma^{-|\sigma|}\Delta_*(\mathring{\sigma})}{C(\sigma)\otimes_\Z\Sigma^{-|\sigma|}\Delta^{-*}(\mathring{\sigma})} \\
 &\simeq& \brc{C(\sigma)\otimes_\Z\Sigma^{-|\sigma|}\Delta_*(I_\sigma) }{C(\sigma)\otimes_\Z\Sigma^{|\sigma|}\Delta^{-*}(I_\sigma) } 
\end{eqnarray*}
thought of as spread over $I_\sigma$. We then piece together these local replacements carefully.

First we define a functor $\widetilde{Sd}_r: \A(X) \to ch(\A(Sd\,X))$. This functor will map morphisms of $\A(X)$ to chain maps, therefore it will send a chain complex $C\in ch(\A(X))$ to a bicomplex $(\widetilde{Sd}_r\, (C_*))_{i,j} = \widetilde{Sd}_r\, (C_i)_j$ with differentials 
\begin{eqnarray*}
 (d_{\widetilde{Sd}_r\, (C_i)})_j: \widetilde{Sd}_r\, (C_*)_{i,j} &\to& \widetilde{Sd}_r\, (C_*)_{i,j-1} \\
 \widetilde{Sd}_r\,((d_C)_i)_j: \widetilde{Sd}_r\, (C_*)_{i,j} &\to& \widetilde{Sd}_r\, (C_*)_{i-1,j}.
\end{eqnarray*}
We then define a functor $Sd_r: ch(\A(X)) \to ch(\A(Sd\,X))$ by sending $(C,d_C)\in ch(\A(X))$ to the chain complex obtained from the bicomplex $(\widetilde{Sd}_r\, (C_*))_{i,j}$ in the standard way, i.e.\
\begin{eqnarray*}
 (Sd_r\, C)_n &:=& \sum_{i+j=n}\widetilde{Sd}_r\, (C_*)_{i,j}, \\
 (d_{Sd_r\, C})_n &:=& \sum_{i+j=n}\{\widetilde{Sd}_r\,((d_C)_i)_j + (-1)^n(d_{\widetilde{Sd}_r\, (C_i)})_j\},
\end{eqnarray*}
where all the direct sums will be finite.

\begin{deflem}\label{Deflem:SdtildeFunctor}
\begin{enumerate}[(i)]
 \item Define $\widetilde{Sd}_r: \A^*(X) \to ch(\A^*(Sd\, X))$ by $M \mapsto \widetilde{Sd}_r\, M$ where
\begin{eqnarray*}
 \widetilde{Sd}_r\,M(\widetilde{\sigma}) &:=& M(\sigma)\otimes \Sigma^{-|\sigma|}\Delta_*(\mathring{\widetilde{\sigma}}) = M(\sigma)_{*+|\sigma|-|\widetilde{\sigma}|}, \quad \mathring{\widetilde{\sigma}}\in I_\sigma \\
(d_{\widetilde{Sd}_r\, M})_{\widetilde{\tau},\widetilde{\sigma}} &:=& \brcc{\id_{M(\sigma)}\otimes(d_{\Sigma^{-|\sigma|}\Delta_*(I_\sigma)})_{\widetilde{\tau},\widetilde{\sigma}},}{\mathring{\widetilde{\tau}},\mathring{\widetilde{\sigma}}\in I_\sigma}{0,}{\mathrm{otherwise}} 
\end{eqnarray*}
and $\{ f: M \to N\} \mapsto \{\widetilde{Sd}_r\, f: \widetilde{Sd}_r\, M \to \widetilde{Sd}_r\, N \}$ where
\[(\widetilde{Sd}_r\, f)_{\widetilde{\tau},\widetilde{\sigma},n} := \brcc{f_{\tau,\sigma,n+|\sigma|-|\widetilde{\sigma}|,}}{I_\tau\ni \widetilde{\tau} \leqslant \widetilde{\sigma} \in I_\sigma, \quad |\widetilde{\tau}|-|\tau| = |\widetilde{\sigma}|-|\sigma|   }{0,}{\mathrm{otherwise}.} \]
\item Define $\widetilde{Sd}: \A_*(X) \to ch(\A_*(Sd\, X))$ by $M \mapsto \widetilde{Sd}_r\, M$ where
\begin{eqnarray*}
 \widetilde{Sd}_r\,M(\widetilde{\sigma}) &:=& M(\sigma)\otimes \Sigma^{|\sigma|}\Delta^{-*}(\mathring{\widetilde{\sigma}}) = M(\sigma)_{*+|\widetilde{\sigma}|-|\sigma|}, \quad \mathring{\widetilde{\sigma}}\in I_\sigma \\
(d_{\widetilde{Sd}_r\, M})_{\widetilde{\tau},\widetilde{\sigma}} &:=& \brcc{\id_{M(\sigma)}\otimes(\delta^{\Sigma^{|\sigma|}\Delta^{-*}(I_\sigma)})_{\widetilde{\tau},\widetilde{\sigma}},}{\mathring{\widetilde{\tau}},\mathring{\widetilde{\sigma}}\in I_\sigma}{0,}{\mathrm{otherwise}} 
\end{eqnarray*}
and $\{ f: M \to N\} \mapsto \{\widetilde{Sd}_r\, f: \widetilde{Sd}_r\, M \to \widetilde{Sd}_r\, N \}$ where
\[(\widetilde{Sd}_r\, f)_{\widetilde{\tau},\widetilde{\sigma},n} := \brcc{f_{\tau,\sigma,n+|\widetilde{\sigma}|-|\sigma|,}}{I_\tau\ni \widetilde{\tau} \geqslant \widetilde{\sigma} \in I_\sigma, \quad |\widetilde{\tau}|-|\tau| = |\widetilde{\sigma}|-|\sigma|}{0,}{\mathrm{otherwise}.} \]
\end{enumerate}
\noindent Then in each case $\widetilde{Sd}_r$ is a functor.
\end{deflem}

\begin{rmk}\label{Rmk:ExamineMaps}
For $\widetilde{Sd}_r: \A^*(X) \to ch(\A^*(Sd\, X))$ we have 
 \begin{enumerate}[(i)]
  \item $d_{\widetilde{Sd}_r\, N}$ is only non-zero from $\widetilde{\sigma}$ to $\widetilde{\tau}$ if they are both in the same $I_\sigma$ \textit{and} $|\widetilde{\tau}|=|\widetilde{\sigma}|-1$, i.e.\ $d_{\widetilde{Sd}_r\, N}$ reduces the quantity $|\widetilde{\sigma}|-|\sigma|$ by $1$.
  \item If $d_{\widetilde{Sd}_r\, N}$ is non-zero from $\widetilde{\sigma}$ to $\widetilde{\tau}$ then $\widetilde{\sigma}\in\Gamma_{\sigma_0, \ldots, \sigma_m}(\mathring{\sigma})$ for some sequence of inclusions $\sigma\leqslant\sigma_0<\ldots< \sigma_m$ and $\widetilde{\tau}$ is necessarily in $\Gamma_{\sigma_0, \ldots,\widehat{\sigma_j},\ldots, \sigma_m}(\mathring{\sigma})$ for some $j$.
  \item $\widetilde{Sd}_r\, f$ is only non-zero from $\widetilde{\sigma}$ to $\widetilde{\tau}$ if $I_\tau\ni \widetilde{\tau} \leqslant \widetilde{\sigma} \in I_\sigma$ and $|\widetilde{\sigma}|-|\sigma|=|\widetilde{\tau}|-|\tau|$, i.e.\ $\widetilde{Sd}_r\, f$ does not change the quantity $|\widetilde{\sigma}|-|\sigma|$.
  \item If $\widetilde{Sd}_r\, f$ is non-zero from $\widetilde{\sigma}$ to $\widetilde{\tau}$ then $\widetilde{\sigma}\in\Gamma_{\sigma_0, \ldots, \sigma_m}(\mathring{\sigma})$ for some sequence of inclusions $\sigma\leqslant\sigma_0<\ldots< \sigma_m$ and $\widetilde{\tau}$ is necessarily in $\Gamma_{\sigma_0, \ldots, \sigma_m}(\mathring{\tau})$ for the same sequence of inclusions.
 \end{enumerate}
For $\widetilde{Sd}_r: \A_*(X) \to ch(\A_*(Sd\, X))$, conditions $(i)-(iv)$ above hold switching the direction of the non-zero component to be from $\widetilde{\tau}$ to $\widetilde{\sigma}$.
\qed\end{rmk}

\begin{proof}[Proof of \ref{Deflem:SdtildeFunctor}]
\begin{enumerate}[(i)]
 \item Note that $(\widetilde{Sd}_r\, f)_{\widetilde{\tau},\widetilde{\sigma}, n}$ is only non-zero if $\widetilde{\tau}\leqslant \widetilde{\sigma}$, so $(\widetilde{Sd}_r\, f)_n$ is a morphism in $\A^*(Sd\, X)$.

By Remark \ref{Rmk:ExamineMaps} we know that $\widetilde{Sd}_r\, f\circ d_{\widetilde{Sd}_r\, M}$ and $d_{\widetilde{Sd}_r\, N} \circ \widetilde{Sd}_r\, f$ can only be non-zero between $\widetilde{\sigma}$ and $\widetilde{\tau}$ if $|\widetilde{\tau}|-|\tau|=|\widetilde{\sigma}|-|\sigma|-1$ and $I_\tau\ni \widetilde{\tau} \leqslant \widetilde{\sigma} \in I_\sigma$. This means that we must have $\widetilde{\sigma}\in\Gamma_{\sigma_0, \ldots, \sigma_m}(\mathring{\sigma})$ for some sequence of inclusions $\sigma\leqslant\sigma_0<\ldots< \sigma_m$ and $\widetilde{\tau}$ in $\Gamma_{\sigma_0, \ldots,\widehat{\sigma_j},\ldots, \sigma_m}(\mathring{\tau})$ for some $j$. Then $\widetilde{Sd}_r\, f$ being a chain map is given by the commutativity of the following diagram:
\begin{displaymath}
 \xymatrix@C=10.5mm{ M_{n+|\sigma|-|\widetilde{\sigma}|}[\Gamma_{\sigma_0,\ldots,\sigma_m}(\mathring{\sigma})] \ar[rr]^-{(d_{\widetilde{Sd}_r\, M})_{\widetilde{\tau},\widetilde{\sigma}}=(-1)^j} \ar[d]^-{(\widetilde{Sd}_r\, f)_{\widetilde{\tau}^\prime,\widetilde{\sigma}, n}=f_{\tau,\sigma,n+|\sigma|-|\widetilde{\sigma}|}} && M_{n+|\sigma|-|\widetilde{\sigma}|-1}[\Gamma_{\sigma_0,\ldots,\widehat{\sigma_j},\ldots,\sigma_m}(\mathring{\sigma})] \ar[d]^-{(\widetilde{Sd}_r\, f)_{\widetilde{\rho},\widetilde{\tau}, n-1}=f_{\tau,\sigma,n+|\sigma|-|\widetilde{\sigma}|}}  \\
 N_{n+|\sigma|-|\widetilde{\sigma}|}[\Gamma_{\sigma_0,\ldots,\sigma_m}(\mathring{\tau})] \ar[rr]_-{(d_{\widetilde{Sd}_r\, N})_{\widetilde{\rho},\widetilde{\tau}^\prime}=(-1)^j}&& N_{n+|\sigma|-|\widetilde{\sigma}|-1}[\Gamma_{\sigma_0,\ldots,\widehat{\sigma_j},\ldots,\sigma_m}(\mathring{\tau})]
}
\end{displaymath}
where $\widetilde{\sigma}$,$\widetilde{\tau}\in I_\sigma$, $\widetilde{\tau}^\prime,\widetilde{\rho}\in I_\tau$ and 
\begin{displaymath}
\xymatrix@R=5mm@C=5mm{
\widetilde{\sigma} \ar@{}[r]|-{\geqslant}  \ar@{}[d]|-{\leqslantup} & \widetilde{\tau}  \ar@{}[d]|-{\leqslantup} \\
\widetilde{\tau}^\prime \ar@{}[r]|-{\geqslant} & \widetilde{\rho}. \\
}
\end{displaymath}
\noindent Next we verify functoriality. Let $f:M\to N$ and $g:N\to P$ be morphisms in $\A^*(X)$, then
\begin{eqnarray*}
\widetilde{Sd}_r\, (g\circ f)_{\widetilde{\rho}, \widetilde{\sigma}, n} &:=& (g\circ f)_{\rho, \sigma,n+|\sigma|-|\widetilde{\sigma}|}: M_{n+|\sigma|-|\widetilde{\sigma}|}[\Gamma_{\sigma_0,\ldots,\sigma_m}(\mathring{\sigma})] \\
&& \hspace{1.8cm} \to P_{n+|\sigma|-|\widetilde{\sigma}|}[\Gamma_{\sigma_0,\ldots,\sigma_m}(\mathring{\sigma})]\\
&=& \sum_{\rho\leqslant\tau\leqslant\sigma}\left(g_{\rho, \tau, n+|\sigma|-|\widetilde{\sigma}|}\circ f_{\tau, \sigma, n+|\sigma|-|\widetilde{\sigma}|}:M_{n+|\sigma|-|\widetilde{\sigma}|}[\Gamma_{\sigma_0,\ldots,\sigma_m}(\mathring{\sigma})]\right. \\
&& \hspace{1.8cm} \to \left.N_{n+|\sigma|-|\widetilde{\sigma}|}[\Gamma_{\sigma_0,\ldots,\sigma_m}(\mathring{\tau})] \to P_{n+|\sigma|-|\widetilde{\sigma}|}[\Gamma_{\sigma_0,\ldots,\sigma_m}(\mathring{\rho})]\right) \\
&=& \sum_{\widetilde{\rho}\leqslant\widetilde{\tau}\leqslant\widetilde{\sigma}}((\widetilde{Sd}_r\, g)_{\widetilde{\rho}, \widetilde{\tau}, n}(\widetilde{Sd}_r\, f)_{\widetilde{\tau}, \widetilde{\sigma}, n}) \\
&=& (\widetilde{Sd}_r\, g \circ \widetilde{Sd}_r\, f)_{\widetilde{\rho}, \widetilde{\sigma}, n}.
\end{eqnarray*}
\item The same analysis holds.
\end{enumerate}
\end{proof}

\begin{deflem}\label{Deflem:SdFunctor}
Using the definition of $\widetilde{Sd}_r: \A^*(X) \to ch(\A^*(Sd\, X))$ we define a functor $Sd_r: ch(\A^*(X)) \to ch(\A^*(Sd\, X))$ by 
\begin{enumerate}[(i)]
 \item Send a chain complex $(C, d_C)\in ch(\A^*(X))$ to the chain complex $(Sd_r\, C, d_{Sd_r\, C})$ defined by:
\begin{equation}\label{SdCn}
 Sd_r\, C(\widetilde{\sigma})_n := (C(\sigma)\otimes \Sigma^{-|\sigma|}\Delta_*(\mathring{\widetilde{\sigma}}))_n = C(\sigma)_{n+|\sigma|-|\widetilde{\sigma}|}, \quad \widetilde{\sigma}\in I_\sigma,
 \end{equation}
\begin{equation}\label{dSdconcise}
 (d_{Sd_r\, C})_{\widetilde{\tau},\widetilde{\sigma},n} := \left\{\begin{array}{cc} 
(d_C)_{\tau,\sigma,n+|\sigma|-|\widetilde{\sigma}|}, & \mathrm{case}\, 1, \\ 
(-1)^n(\id_{C(\sigma)}\otimes d_{\Sigma^{-|\sigma|}\Delta_*(I_\sigma)})_{\widetilde{\tau},\widetilde{\sigma},n}, & \mathrm{case}\, 2, \\
0, & \mathrm{otherwise},
\end{array} \right.
\end{equation}
where case $1$ is $I_\tau\ni \widetilde{\tau} \leqslant \widetilde{\sigma} \in I_\sigma$ and $|\widetilde{\tau}|-|\tau|=|\widetilde{\sigma}|-|\sigma|$ so that \[(d_C)_{\tau,\sigma,n+|\sigma|-|\widetilde{\sigma}|}: C(\sigma)_{n+|\sigma|-|\widetilde{\sigma}|}\to C(\tau)_{n+|\sigma|-|\widetilde{\sigma}|-1}\] and case $2$ is $\widetilde{\sigma}\in I_\sigma = I_\tau \ni \widetilde{\tau}$, $\widetilde{\tau}\leqslant\widetilde{\sigma}$ and $|\widetilde{\tau}|-|\tau|=|\widetilde{\sigma}|-|\sigma|-1$ so that \[(-1)^n(\id_{C(\sigma)}\otimes d_{\Sigma^{-|\sigma|}\Delta_*(I_\sigma)})_{\widetilde{\tau},\widetilde{\sigma},n} = (-1)^n(\id_{C(\sigma)})_{n+|\sigma|-|\widetilde{\sigma}|}\otimes (d_{\Delta_*(Sd\, X)})_{\widetilde{\tau}, \widetilde{\sigma},|\widetilde{\tau}|}:\] 
\[Sd_r\, C(\widetilde{\sigma})_n =C(\sigma)_{n+|\sigma|-|\widetilde{\sigma}|}\otimes \Delta_*(\mathring{\widetilde{\sigma}})_{|\widetilde{\sigma}|}  \to Sd_r\, C(\widetilde{\tau})_{n-1}= C(\sigma)_{n+|\sigma|-|\widetilde{\sigma}|}\otimes \Delta_*(\mathring{\widetilde{\tau}})_{|\widetilde{\tau}|}\] which is just the map \[(-1)^n(d_{\Delta_*(Sd\, X)})_{\widetilde{\tau}, \widetilde{\sigma},|\widetilde{\tau}|}:C(\sigma)_{n+|\sigma|-|\widetilde{\sigma}|} \to C(\sigma)_{n+|\sigma|-|\widetilde{\sigma}|}\] where $(d_{\Delta_*(Sd\, X)})_{\widetilde{\tau}, \widetilde{\sigma},|\widetilde{\tau}|}$ is $(-1)^j$ for the $j$ such that $\widetilde{\sigma}\in \Gamma_{\sigma_0,\ldots, \sigma_m}(\mathring{\sigma})$ and $\widetilde{\tau}\in\Gamma_{\sigma_0,\ldots,\widehat{\sigma_j},\ldots, \sigma_m}(\mathring{\sigma})$.
\item Send a chain map $f:C\to D$ to the chain map $Sd_r\, f : Sd_r\, C \to Sd_r\, D$ defined by 
\begin{equation}\label{Sdfconcise}
(Sd_r\, f)_{\widetilde{\tau},\widetilde{\sigma},n} := \brcc{f_{\tau,\sigma,n+|\sigma|-|\widetilde{\sigma}|,}}{\mathrm{case}\, 1,}{0,}{\mathrm{otherwise}.}
\end{equation}
\end{enumerate}
Analogously we define $Sd_r: ch(\A_*(X)) \to ch(\A_*(Sd\, X))$ using the simplicial cochain complex in place of the simplicial chain complex.
\end{deflem}

\begin{proof}
We prove the case $Sd_r: ch(\A^*(X)) \to ch(\A^*(Sd\, X))$; the same treatment works for $Sd_r: ch(\A_*(X)) \to ch(\A_*(Sd\, X))$.

The introduction of the sign $(-1)^n$ in the definition of $d_{Sd_r\, C}$ makes $Sd_r\, C$ into a chain complex. Explicitly $Sd_r\, C$ is a chain complex by the anticommutativity of the following diagram:
\begin{displaymath}
 \xymatrix{ Sd_r\, C_n(\Gamma_{\sigma_0,\ldots, \sigma_m}(\mathring{\sigma})) \ar[rrr]^-{(d_C)_{\tau,\sigma,n+|\sigma|-|\widetilde{\sigma}|}} \ar[d]_-{(-1)^n(-1)^j} &&& Sd_r\, C_{n-1}(\Gamma_{\sigma_0,\ldots, \sigma_m}(\mathring{\tau})) \ar[d]^-{(-1)^{n-1}(-1)^j}  \\
Sd_r\, C_{n-1}(\Gamma_{\sigma_0,\ldots,\widehat{\sigma_j},\ldots, \sigma_m}(\mathring{\sigma})) \ar[rrr]^-{(d_C)_{\tau,\sigma,(n-1)+|\sigma|-(|\widetilde{\sigma}|-1)}} &&& Sd_r\, C_{n-2}(\Gamma_{\sigma_0,\ldots,\widehat{\sigma_j},\ldots, \sigma_m}(\mathring{\tau})). \\
}
\end{displaymath}

As with $\widetilde{Sd}_r\, f$, the map $(Sd_r\, f)_n$ is a morphism in $\A^*(Sd\, X)$ and is functorial. We are just left to check that $(Sd_r\, f): Sd_r\, C\to Sd_r\, D$ is a chain map, i.e.\ that \[(Sd_r\, f\circ d_{Sd_r\, C})_{\widetilde{\rho},\widetilde{\sigma}, n} = (d_{Sd_r\, D} \circ Sd_r\, f)_{\widetilde{\rho},\widetilde{\sigma}, n}, \; \forall \widetilde{\rho},\widetilde{\sigma}, n.\]

Suppose that $\widetilde{\sigma}\in \Gamma_{\sigma_0,\ldots, \sigma_m}(\mathring{\sigma})$ for some $\sigma\leqslant \sigma_0< \ldots< \sigma_m$ where $\widetilde{\sigma}\in I_\sigma$. Then $Sd_r\, f$ maps to $\Gamma_{\sigma_0,\ldots, \sigma_m}(\mathring{\tau})$ for all $\tau\leqslant\sigma$, and $d_{Sd_r\, C}$ maps to 
\begin{enumerate}[(i)]
 \item $\Gamma_{\sigma_0,\ldots,\widehat{\sigma_j},\ldots, \sigma_m}(\mathring{\sigma})$ for all $j=0,\ldots,m$, 
 \item $\Gamma_{\sigma_0,\ldots, \sigma_m}(\mathring{\tau}^\prime)$ for all $\tau^\prime\leqslant\sigma$, 
\end{enumerate}
so we have two cases to check. 
\begin{enumerate}[(i)]
 \item 
\begin{displaymath}
 \xymatrix@C=4.5cm{ Sd_r\, C_n(\widetilde{\sigma}) \ar[r]^-{(Sd_r\,f)_{\widetilde{\tau}, \widetilde{\sigma},n}= f_{\tau,\sigma,n+|\sigma|-|\widetilde{\sigma}|}} \ar[d]_-{(-1)^n(-1)^j} & Sd_r\, D_n(\widetilde{\tau}) \ar[d]^-{(-1)^n(-1)^j}  \\
Sd_r\, C_{n-1}(\widetilde{\tau}^\prime) \ar[r]^-{(Sd_r\,f)_{\widetilde{\rho}, \widetilde{\tau}^\prime,n-1}= f_{\tau,\sigma,n+|\sigma|-|\widetilde{\sigma}|}} & Sd_r\, D_{n-1}(\widetilde{\rho}) \\
}
\end{displaymath}
for $\widetilde{\tau} \in \Gamma_{\sigma_0,\ldots,\sigma_m}(\mathring{\tau})$, $\widetilde{\tau}^\prime \in \Gamma_{\sigma_0,\ldots,\widehat{\sigma_j},\ldots, \sigma_m}(\mathring{\sigma})$ and $\widetilde{\rho} \in \Gamma_{\sigma_0,\ldots,\widehat{\sigma_j},\ldots, \sigma_m}(\mathring{\tau})$. The above diagram commutes so case $(i)$ is verified.
\item 
\begin{displaymath}
 \xymatrix@C=4.5cm{ Sd_r\, C_n(\widetilde{\sigma}) \ar[r]^-{(Sd_r\,f)_{\widetilde{\tau}, \widetilde{\sigma},n}= f_{\tau,\sigma,n+|\sigma|-|\widetilde{\sigma}|}} \ar[d]_-{(d_C)_{\tau,\sigma,n+|\sigma|-|\widetilde{\sigma}|}} & Sd_r\, D_n(\widetilde{\tau}) \ar[d]^-{(d_D)_{\rho,\tau,n+|\sigma|-|\widetilde{\sigma}|}}  \\
Sd_r\, C_{n-1}(\widetilde{\tau}^\prime) \ar[r]^-{(Sd_r\,f)_{\widetilde{\rho}, \widetilde{\tau}^\prime,n-1}= f_{\tau,\sigma,n+|\sigma|-|\widetilde{\sigma}|}} & Sd_r\, D_{n-1}(\widetilde{\rho}) \\
}
\end{displaymath}
for $\widetilde{\tau} \in \Gamma_{\sigma_0,\ldots,\sigma_m}(\mathring{\tau})$, $\widetilde{\tau}^\prime \in \Gamma_{\sigma_0,\ldots, \sigma_m}(\mathring{\tau}^\prime)$ and $\widetilde{\rho} \in \Gamma_{\sigma_0,\ldots, \sigma_m}(\mathring{\rho})$. Restricting to case $(ii)$ contributions,
\begin{eqnarray*}
(d_{Sd_r\, D}\circ Sd_r\, f)_{\widetilde{\rho}, \widetilde{\sigma}, n} &=& \sum_{\rho\leqslant\tau\leqslant\sigma}((d_{Sd_r\, D})_{\widetilde{\rho}, \widetilde{\tau}, n}\circ Sd_r\, f_{\widetilde{\tau}, \widetilde{\sigma}, n}) \\ 
&=& \sum_{\rho\leqslant\tau\leqslant\sigma}((d_D)_{\rho, \tau, n+|\sigma|-|\widetilde{\sigma}|}\circ f_{\tau, \sigma, n+|\sigma|-|\widetilde{\sigma}|}) \\
&=& (d_D \circ f)_{\rho,\sigma,n+|\sigma|-|\widetilde{\sigma}|} \\
&=& (f \circ d_C)_{\rho,\sigma,n+|\sigma|-|\widetilde{\sigma}|} \\
&=& \sum_{\rho\leqslant\tau\leqslant\sigma}(f_{\rho, \tau, n+|\sigma|-|\widetilde{\sigma}|}\circ (d_C)_{\tau, \sigma, n+|\sigma|-|\widetilde{\sigma}|}) \\
&=& \sum_{\rho\leqslant\tau\leqslant\sigma}((Sd_r\, f)_{\widetilde{\rho}, \widetilde{\tau}, n}\circ (d_{Sd_r\, C})_{\widetilde{\tau}, \widetilde{\sigma}, n}) \\ 
&=& (Sd_r\, f\circ d_{Sd_r\, C})_{\widetilde{\rho}, \widetilde{\sigma}, n}. 
\end{eqnarray*}
\end{enumerate}
Thus $Sd_r\, f$ is a chain map and so $Sd_r: ch(\A^*(X)) \to ch(\A^*(Sd\, X))$ is a functor as required.
\end{proof}
For ease of notation we will sometimes suppress the dependence of $Sd_r$ on $r$ and just write $Sd$. Later we will see that for any choice of $r$ the functor $Sd_r$ generalises barycentric subdivision of the simplicial complex so this notational convenience will be justified.

Before showing that a subdivided chain complex may be reassembled to give one chain equivalent to the one we started with we take a quick break to present some examples of algebraic subdivision.

\section{Examples of algebraic subdivision}
So far, from the formulae defining the functor $Sd$, it is not so easy to get a feel for what the functor is doing; the aim of this section is to make this clear with some worked examples.
\begin{ex}\label{Ex:simpleexofsubdivonchains}
 Let $X = v_0v_1$ with orientations $[v_0,v_1]$, $[v_0]=[v_1]=+1$. Consider a general chain complex in $\A^*(X)$:
\begin{displaymath}
 \xymatrix@R=3mm{
C(v_0): & \ldots \ar[r] & C_n(v_0) \ar[r] & C_{n-1}(v_0) \ar[r] & \ldots \\
C(v_0v_1): & \ldots \ar[r] \ar[ur]^-{\alpha} \ar[dr]_-{\beta} & C_n(v_0v_1) \ar[r]  \ar[ur]^-{\alpha} \ar[dr]_-{\beta} & C_{n-1}(v_0v_1) \ar[r] \ar[ur]^-{\alpha} \ar[dr]_-{\beta} & \ldots \\
C(v_1): & \ldots \ar[r] & C_n(v_1) \ar[r] & C_{n-1}(v_1) \ar[r] & \ldots \\
}
\end{displaymath}
which we express diagrammatically as
\begin{displaymath}
 \xymatrix@R=3mm{
C(v_0)_* & C(v_0v_1)_* & C(v_1)_* \\
\stackrel{+}{\bullet} \ar@{-}[rr]|-{>} && \stackrel{+}{\bullet} \\
v_0 & v_0v_1 & v_1
}
\end{displaymath}
\noindent then using the chain inverse 
\begin{eqnarray*}
 r: Sd\, X &\to& X \\
\{ \widehat{v_0}, \widehat{v_0v_1}\} &\mapsto& \{v_0\} \\
\{ \widehat{v_1}\} &\mapsto& \{v_1\}
\end{eqnarray*}
we get 
\begin{eqnarray*}
 I_{\{v_0\}} &=& [\widehat{v_0}, \widehat{v_0v_1}]\\
 I_{(v_0,v_1)} &=& (\widehat{v_0v_1}, \widehat{v_1})\\
 I_{\{v_1\}} &=& [\widehat{v_1}].
\end{eqnarray*}
So we get $Sd\,C$ as 
\begin{eqnarray*}
 Sd\, C[I_{\{v_0\}}] &=& C(v_0) \otimes \Sigma^{0}\Delta_*(I_{\{v_0\}}) \\
 Sd\, C[I_{(v_0,v_1)}] &=& C(v_0v_1) \otimes \Sigma^{-1}\Delta_*(I_{(v_0,v_1)}) \\
 Sd\, C[I_{\{v_0\}}] &=& C(v_1) \otimes \Sigma^{0}\Delta_*(I_{\{v_1\}})  
\end{eqnarray*}
which we express diagrammatically as
\begin{displaymath}
 \xymatrix@R=3mm{
C(v_0)_* & C(v_0)_{*-1} & C(v_0)_* & C(v_0v_1)_* & C(v_1)_* \\
\stackrel{+}{\bullet} \ar@{-}[rr]|-{>} && \stackrel{+}{\bullet} \ar@{-}[rr]|-{>} && \stackrel{+}{\bullet} \\
\widehat{v_0} & (\widehat{v_0}, \widehat{v_0v_1}) & \widehat{v_0v_1} & (\widehat{v_0v_1},\widehat{v_1}) & \widehat{v_1}
}
\end{displaymath}
which can be seen explicitly as a chain complex in $\A^*(Sd\, X)$ as follows:
\begin{displaymath}
 \xymatrix@R=3mm{
Sd\,C(\widehat{v_0})=C(v_0)_*: & \ldots \ar[r] & C_n(v_0) \ar[r] & C_{n-1}(v_0) \ar[r] & \ldots \\
Sd\,C([\widehat{v_0},\widehat{v_0v_1}])= C(v_0)_{*-1}: & \ldots \ar[r] \ar[ur]^-{-1} \ar[dr]_-{1} & C_{n-1}(v_0) \ar[r]  \ar[ur]^-{-1} \ar[dr]_-{1} & C_{n-2}(v_0) \ar[r] \ar[ur]^-{-1} \ar[dr]_-{1} & \ldots \\
Sd\,C(\widehat{v_0v_1})= C(v_0)_*: & \ldots \ar[r] & C_n(v_1) \ar[r] & C_{n-1}(v_1) \ar[r] & \ldots \\
Sd\,C([\widehat{v_0v_1},\widehat{v_1}])= C(v_0v_1)_*: & \ldots \ar[r] \ar[ur]^-{\alpha} \ar[dr]_-{\beta} & C_n(v_0v_1) \ar[r]  \ar[ur]^-{\alpha} \ar[dr]_-{\beta} & C_{n-1}(v_0v_1) \ar[r] \ar[ur]^-{\alpha} \ar[dr]_-{\beta} & \ldots \\
Sd\,C(\widehat{v_1})= C(v_1)_*: & \ldots \ar[r] & C_n(v_1) \ar[r] & C_{n-1}(v_1) \ar[r] & \ldots \\
}
\end{displaymath}
\qed\end{ex}

\begin{ex}\label{Ex:lesssimpleexofsubdivonchains}
 Let $X$ be the $2$-simplex $\sigma$ with orientations as in Example \ref{Ex:SigmaOrients}. We express a general chain complex in $\A^*(\sigma)$ diagrammatically as
\begin{figure}[ht]
\begin{center}
{
\psfrag{p1}[tr][tr]{$C(\rho_0)_*$}
\psfrag{p2}[b][b]{$C(\rho_1)_*$}
\psfrag{p3}[tl][tl]{$C(\rho_2)_*$}
\psfrag{t1}[r][]{$C(\tau_0)_*$}
\psfrag{t2}[l][l]{$C(\tau_1)_*$}
\psfrag{t3}[t][t]{$C(\tau_2)_*$}
\psfrag{s}[l][r]{$C(\sigma)_*$}
\psfrag{+}[][]{$+$}
\includegraphics[width=7cm]{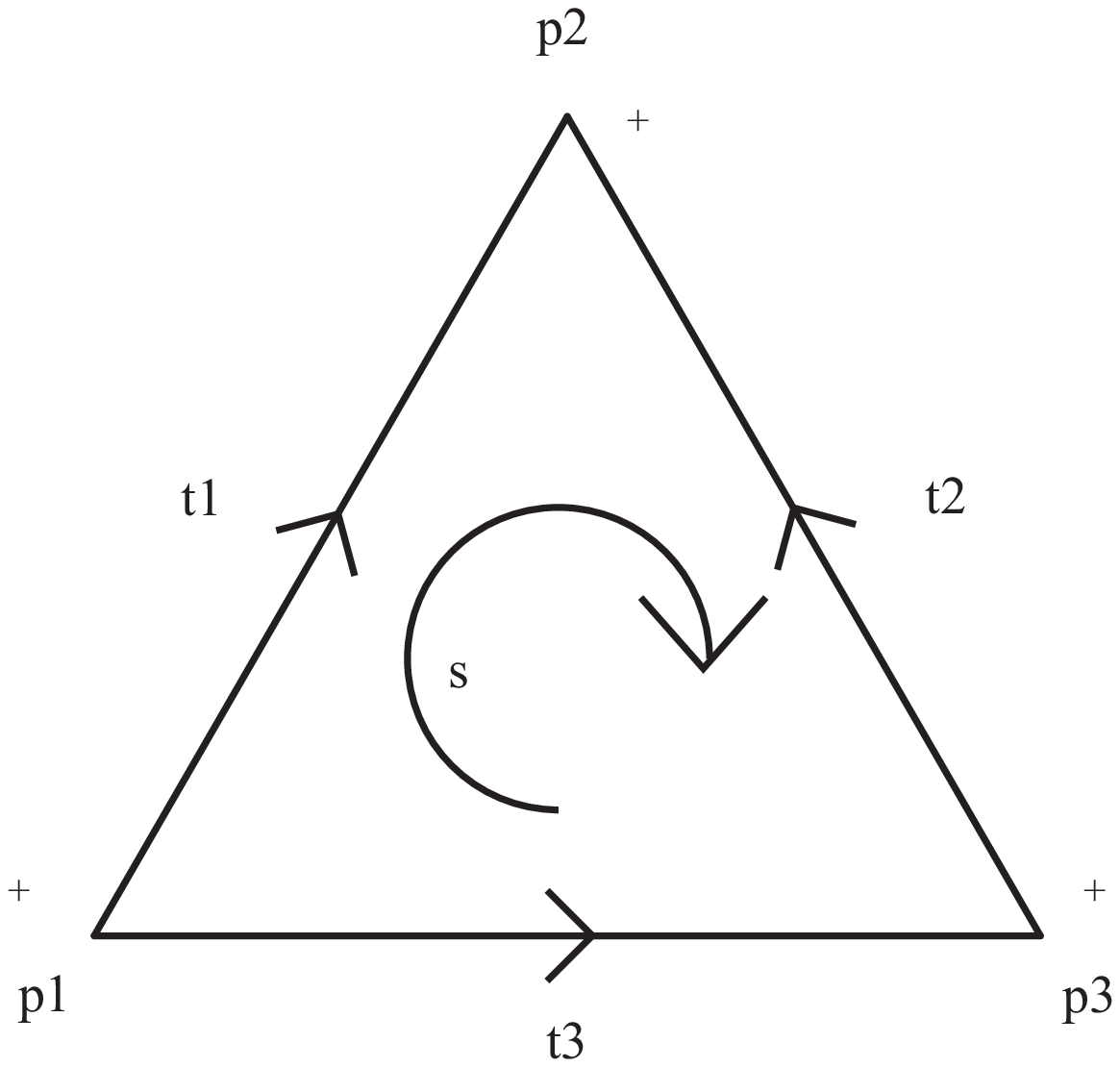}
}
\label{Fig:prethesis55}
\end{center}
\end{figure}

\noindent then using the chain inverse 
\begin{eqnarray*}
 r: Sd\, \sigma &\to& \sigma \\
\{ \widehat{\rho_0}, \widehat{\tau_0}, \widehat{\tau_2}, \widehat{\sigma}\} &\mapsto& \{\rho_0\} \\
\{ \widehat{\rho_1}, \widehat{\tau_1}\} &\mapsto& \{\rho_1\} \\
\{ \widehat{\rho_2} \} &\mapsto& \{\rho_2\}
\end{eqnarray*}

\noindent we get the following decomposition of $Sd\, \sigma$:
\begin{figure}[ht]
\begin{center}
{
\psfrag{s}[l][r]{$I_{\sigma}$}
\psfrag{p1}[l][r]{$I_{\rho_0}$}
\psfrag{p2}{$I_{\rho_1}$}
\psfrag{p3}[t][]{$I_{\rho_2}$}
\psfrag{t1}{$I_{\tau_0}$}
\psfrag{t2}{$I_{\tau_1}$}
\psfrag{t3}[tl][tr]{$I_{\tau_2}$}
\includegraphics[width=7cm]{images/prethesis4.eps}
}
\caption{The $I_\tau$'s for $\sigma$.}
\label{Fig:Sdsigmadecomposed}
\end{center}
\end{figure}

oriented as in \Figref{Fig:constructgamorient}. $Sd\, C$ is then the chain complex expressed diagrammatically in \Figref{Fig:Chaincxover2simplexafter}.
\begin{figure}[ht]
\begin{center}
{
\psfrag{s}{$C(\sigma)_n$}
\psfrag{p1}{$C(\rho_0)_n$}
\psfrag{p2}{$C(\rho_1)_n$}
\psfrag{p3}{$C(\rho_2)_n$}
\psfrag{t1}{$C(\tau_0)_n$}
\psfrag{t2}{$C(\tau_1)_n$}
\psfrag{t3}{$C(\tau_2)_n$}
\psfrag{p1+}{$C(\rho_0)_{n-1}$}
\psfrag{p1++}{$C(\rho_0)_{n-2}$}
\psfrag{p2+}{$C(\rho_1)_{n-1}$}
\psfrag{t1+}{$C(\tau_0)_{n-1}$}
\psfrag{t3+}{$C(\tau_2)_{n-1}$}
\includegraphics[width=13cm]{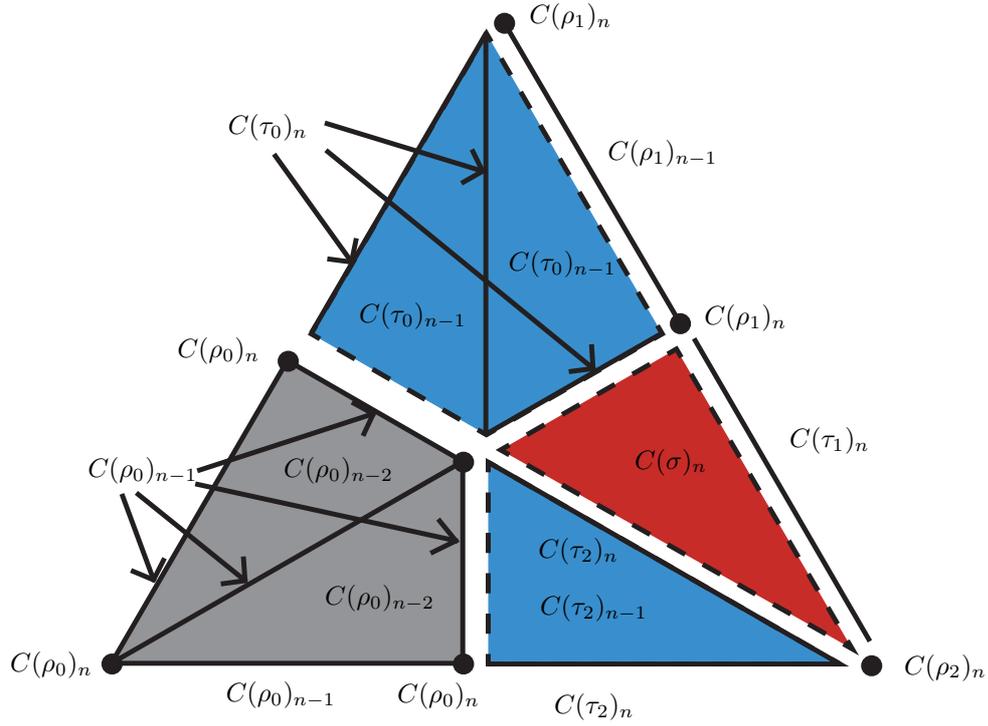}
}
\caption{The $n$-chains $(Sd\, C)_n$ for the $2$-simplex $\sigma$.}
\label{Fig:Chaincxover2simplexafter}
\end{center}
\end{figure}

\noindent Letting $d_{\tau^\prime,\tau}$ denote the component of the boundary map $(d_C)_{\tau^\prime,\tau, n}: C(\tau)_n \to C(\tau^\prime)_{n-1}$, the components of the boundary map $d_{Sd\, C}$ of the subdivision chain complex are as in \Figref{Fig:Chaincxover2simplexafterboundaries}. The components of $d_{Sd\, C}$ within each $I_\tau$ have been suppressed. We should also have $(-1)^n(d_{\Delta_*(Sd\, \sigma)})_{\widetilde{\rho},\widetilde{\tau}}$ going from any $\widetilde{\tau}$ to a codimension one $\widetilde{\rho}$ in the same $I_{\tau}$ and $(d_C)_{\sigma,\sigma}$ from each $\widetilde{\sigma}\in I_\sigma$ to itself.
\begin{figure}[ht]
\begin{center}
{
\psfrag{p1t3}{$d_{\rho_0\tau_2}$}
\psfrag{p1t1}{$d_{\rho_0\tau_0}$}
\psfrag{p2t1}{$d_{\rho_1\tau_0}$}
\psfrag{p2t2}{$d_{\rho_1\tau_1}$}
\psfrag{p3t2}{$d_{\rho_2\tau_1}$}
\psfrag{p3t3}{$d_{\rho_2\tau_2}$}
\psfrag{t1s}{$d_{\tau_0\sigma}$}
\psfrag{t3s}{$d_{\tau_2\sigma}$}
\psfrag{t2s}{$d_{\tau_1\sigma}$}
\psfrag{p1s}{$d_{\rho_0\sigma}$}
\psfrag{p3s}{$d_{\rho_2\sigma}$}
\psfrag{p2s}{$d_{\rho_1\sigma}$}
\includegraphics[width=11cm]{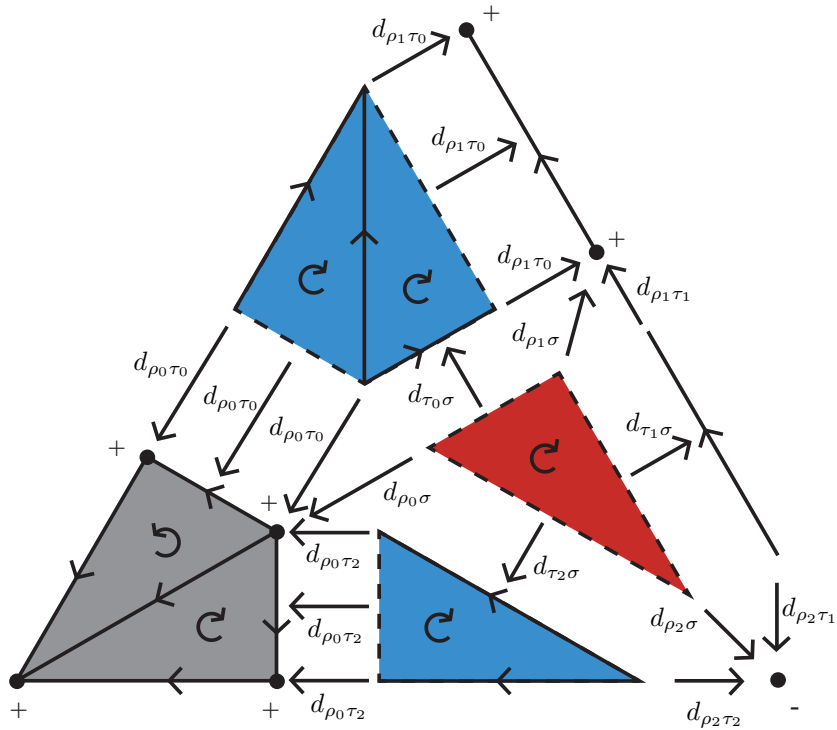}
}
\caption{The boundary map $d_{Sd\, C}$ for the $2$-simplex $\sigma$.}
\label{Fig:Chaincxover2simplexafterboundaries}
\end{center}
\end{figure}
\qed\end{ex}

\section{Algebraic subdivision chain equivalences}\label{svptthr}
Consider the simplicial $\brc{\mathrm{chain}}{\mathrm{cochain}}$ complex and the subdivision chain map induced by barycentric subdivision: \[s: \brc{\Delta^{lf}_*(X) \to \Delta^{lf}_*(Sd\, X)}{\Delta^{-*}(X) \to \Delta^{-*}(Sd\, X).}\]Assembling $\brc{\Delta^{lf}_*(Sd\, X)}{\Delta^{-*}(Sd\, X)}$ covariantly with $\rr$ (as in Definition \ref{covariantassembly}) we can view $s$ as a chain equivalence in $\brc{ch(\A^*(X))}{ch(\A_*(X)).}$ 

The same is true algebraically; for any choice of chain inverse $r$, the subdivision chain map $s$ induces a chain map $s_*:C \to Sd_r\, C$ of chain complexes in $\brc{ch(\A^*(X))}{ch(\A_*(X))}$ where we view $Sd_r\, C$ as a chain complex in $\brc{ch(\A^*(X))}{ch(\A_*(X))}$ by assembling covariantly over all the $I_\tau$'s with a functor which we shall call $\rr_r$ to show the dependence on $r$. 

\begin{thm}\label{algsubdiv}
Let 
\begin{displaymath}
\xymatrix@1{ (\Delta^{lf}_*(X), d_{\Delta^{lf}_*(X)},0) \ar@<0.5ex>[r]^-{s} & (\Delta^{lf}_*(Sd\, X), d_{\Delta^{lf}_*(Sd\, X)},P) \ar@<0.5ex>[l]^-{r}
}
\end{displaymath}
be the subdivision chain equivalence between the simplicial chain complex of $X$ and that of its barycentric subdivision, together with a choice of chain inverse $t$ and chain homotopy \[d_{\Delta^{lf}_*(Sd\, X)} P + P d_{\Delta^{lf}_*(Sd\, X)} = 1_{\Delta^{lf}_*(Sd\, X)}\] constructed as in chapter \ref{chapfive}. Then this induces a subdivision chain equivalence 
\begin{displaymath}
\xymatrix@1{ (C, d_C,0) \ar@<0.5ex>[r]^-{s_*} & (Sd_r\, C, d_{Sd_r\, C},P_*) \ar@<0.5ex>[l]^-{r_*}
}
\end{displaymath}
such that for the assembly functor $\rr_r$ that assembles the $I_\tau$, \[\rr_r Sd_r\, C \simeq C \in \A^*(X)\] for all $C\in \A^*(X)$. 
\end{thm}
First we deal with $C\in ch(\A^*(X))$:
\begin{rmk}
The global chain equivalence given by $s$, $r$ and $P$ restricts to a local chain equivalence for all $\sigma\in X$:
\begin{displaymath}
\xymatrix@1{ (\Delta_*(\mathring{\sigma}), 0,0) \ar@<0.5ex>[r]^-{s_\sigma:=s|} & (\Delta_*(I_\sigma), d_{\Delta_*(Sd\, X)}|_{I_\sigma},P_\sigma:=P|_{I_\sigma}) \ar@<0.5ex>[l]^-{r_\sigma:= r|}
}
\end{displaymath}
since $s(\mathring{\sigma}) = \gam{\sigma}{\sigma} \subset I_\sigma$, $t(I_\sigma) \subset \mathring{\sigma}$ by the definition of $I_\sigma$. This induces a local chain equivalence 
\begin{displaymath}
\xymatrix@C=3cm{ (C(\sigma), (d_C)_{\sigma,\sigma},0) \ar@<0.5ex>[r]^-{(s_\sigma)_*= \id_{C(\sigma)}\otimes \Sigma^{-|\sigma|}s_\sigma} & (Sd_r\, C|_{I_\sigma}, d_{Sd_r\, C}|_{I_\sigma},(P_\sigma)_*) \ar@<0.5ex>[l]^-{(r_\sigma)_*= \id_{C(\sigma)}\otimes \Sigma^{-|\sigma|}r_\sigma}
}
\end{displaymath}
viewing $C(\sigma)$ as $C(\sigma)\otimes \Z = C(\sigma)\otimes\Sigma^{-|\sigma|}\Delta_*(\mathring{\sigma})$ and noting that $Sd_r\, C|_{I_\sigma}=C(\sigma)\otimes \Sigma^{-|\sigma|}\Delta_*(I_\sigma)$.
\qed\end{rmk}

If the local chain maps $(s_\sigma)_*$ and $(r_\sigma)_*$ can be extended to give global chain maps $s_*$ and $r_*$ then we can apply Proposition \ref{chcont} to obtain chain equivalences in $\brc{\A^*(X)}{\A_*(X)}$. To extend the local $(s_\sigma)_*$ and $(r_\sigma)_*$ we need to introduce some notation:
\begin{notn}\label{partsofs}
\begin{enumerate}
 \item We will use \[(s_*)_{\gam{\rho_0,\ldots,\rho_i}{\rho},\sigma,n}: C(\sigma)_n \to Sd\, C_n[\gam{\rho_0,\ldots,\rho_i}{\rho}]\] to denote the component of $s_*$ from $C(\sigma)_n$ to $Sd\, C$ of \textbf{every $(|\rho|+i)$-dimensional simplex in $\gam{\rho_0,\ldots,\rho_i}{\rho}$} with respect to the given orientation of $\sigma$ and the standard orientations of simplices in $Sd\, X$. 

 This map will be zero in the case that there are no $(|\rho|+i)$-dimensional simplices in $\gam{\rho_0,\ldots,\rho_i}{\rho}$, otherwise it will be a column vector \[ (s_*)_{\gam{\rho_0,\ldots,\rho_i}{\rho},\sigma,n} \left(\hspace{-1mm} \begin{array}{c} 1\\ \vdots \\ 1 \end{array}\hspace{-1mm}\right) : C(\sigma)_n \to \hspace{-2mm}\sum_{ \begin{array}{c} \widetilde{\tau}\in \gam{\rho_0,\ldots,\rho_i}{\rho} \\ |\widetilde{\tau}|=|\rho|+i \end{array} }\hspace{-5mm} Sd\, C_n(\widetilde{\tau}) = \hspace{-2mm}\sum_{\begin{array}{c} \widetilde{\tau}\in \gam{\rho_0,\ldots,\rho_i}{\rho} \\ |\widetilde{\tau}|=|\rho|+i \end{array}} \hspace{-5mm} C(\rho)_{n+|\rho|-|\widetilde{\tau}|}.\]

 \item $d_{\rho\tau}$ will be shorthand for $(d_C)_{\rho,\tau}$. 
 \item $d_{\rho_0\rho_1\ldots\rho_i}$ will be shorthand for $d_{\rho_0\rho_1}d_{\rho_1\rho_2}\ldots d_{\rho_{i-1}\rho_i}$.
\end{enumerate}
\end{notn}

\begin{rmk}\label{Rmk:dsquarediszero}
In verifying our definition of $s_*$ works we will make \textbf{repeated use} of the fact that for all $\rho\leqslant\sigma$ \[0 = d^2_{\rho\sigma} = \sum_{\rho\leqslant\tau\leqslant\sigma}d_{\rho\tau}d_{\tau\sigma}\]
and so in particular that 
\begin{equation}\label{slidealong}
d_{\rho\rho}d_{\rho\sigma} =  - d_{\rho\sigma}d_{\sigma\sigma} + \sum_{\rho<\tau<\sigma}d_{\rho\tau}d_{\tau\sigma}.
\end{equation}
\qed\end{rmk}

\begin{prop}\label{Prop:sstardefined}
Set \[(s_*)_{\gam{\rho_0,\ldots,\rho_i}{\rho},\sigma,n}:= \left\{ \begin{array}{cc} \id_{C(\sigma)}, & \rho=\rho_0=\rho_i=\sigma, \\ (-1)^{(n+1)i}d_{\rho_0\rho_1\ldots\rho_i},& \rho=\rho_0\neq\rho_i=\sigma, \\ 0, & \mathrm{otherwise.} \end{array} \right. \] This defines a chain map $s_*: C\to Sd\, C$ which extends the local $\{s_\sigma\}_{\sigma\in X}$.
\end{prop}

\begin{proof}
First we verify that this definition extends the local chain maps $(s_\sigma)_*$, i.e.\ that \[(s_*)|= (s_\sigma)_*: C(\sigma) \to Sd\, C|_{I_\sigma}. \]
Recall that $(s_\sigma)_* = \id_{C(\sigma)}\otimes \Sigma^{-|\sigma|}s_\sigma$ where $s_\sigma$ maps $\sigma$ to $\gam{\sigma}{\sigma}$. Hence,
\[((s_\sigma)_*)_{\widetilde{\tau},\sigma} = \brcc{\id_{C(\sigma)}: C(\sigma) \to Sd\, C(\widetilde{\tau}),}{\widetilde{\tau} = \gam{\sigma}{\sigma},}{0:C(\sigma) \to Sd\, C(\widetilde{\tau}),}{\widetilde{\tau}\in I_\sigma\backslash \gam{\sigma}{\sigma}.} \]
All $\widetilde{\tau}\in I_\sigma$ are in $\gam{\rho_0,\ldots, \rho_i}{\sigma}$ for some sequence of inclusions $\sigma\leqslant \rho_0<\ldots< \rho_i$. For $(s_*)_{\widetilde{\tau},\sigma,n}$ to be non-zero we need $\rho_0=\sigma$ and $\rho_i=\sigma$, so $s_*$ is $0$ to all $\widetilde{\tau}\in I_\sigma$ except for $\gam{\sigma}{\sigma}$ for which the component is $\id_{C(\sigma)}$. Thus $s_*$ defined as in the proposition extends the local $(s_\sigma)_*$'s.

Next we seek to prove that $s_*$ defines a chain map, i.e.\ that \[s_*d_C - d_{Sd\, C}s_*=0.\] We prove this componentwise, showing that \[(s_*d_C - d_{Sd\, C}s_*)_{\gam{\rho_0,\ldots, \rho_i}{\rho},\sigma,n}=0: C(\sigma)_n \to Sd\, C_{n-1}[\gam{\rho_0,\ldots, \rho_i}{\rho}]\]for all $\rho\leqslant \rho_0<\ldots<\rho_i\leqslant \sigma$. Now we examine each term:

Consider $d_C$. This maps to all $\tau\leqslant \sigma$ and $(s_*)_{\gam{\rho_0,\ldots, \rho_i}{\rho},\tau,n-1}$ is only non-zero if $\rho_i\leqslant \tau$. Thus \[ 
(s_*d_C)_{\gam{\rho_0,\ldots, \rho_i}{\rho},\sigma,n} = \sum_{\rho_i\leqslant\tau\leqslant\sigma}{(s_*)_{\gam{\rho_0,\ldots,\rho_i}{\rho},\tau,n-1}(d_{\tau\sigma})_n.}
\]

Now consider $-d_{Sd\, C}s_*$. This can be non-zero in several different ways:
\begin{enumerate}[(i)]
 \item It can go via $\gam{\rho_0,\ldots, \rho_i}{\tau}$ for any $\rho\leqslant \tau\leqslant \rho_0$. This contributes \[- \sum_{\rho\leqslant\tau\leqslant\rho_0}(d_{Sd\, C})_{\gam{\rho_0,\ldots,\rho_i}{\rho},\gam{\rho_0,\ldots,\rho_i}{\tau},n}(s_*)_{\gam{\rho_0,\ldots,\rho_i}{\tau},\sigma,n }.\]
 \item It can go via $\gam{\rho^\prime_0,\ldots, \rho^\prime_{i+1}}{\rho}$ for $\{\rho^\prime_0,\ldots,\rho^\prime_{i+1}\}$ such that $\{\rho_0,\ldots,\rho_i\} = \{\rho_0^\prime,\ldots,\widehat{\rho}_j^\prime.\ldots,\rho_{i+1}^\prime\}$ for some $j$. Summing over all possible places the extra $\tau:=\rho^\prime_j$ can appear, this contributes 
\begin{eqnarray*}
&& - \sum_{\rho\leqslant\tau< \rho_0} (d_{Sd\, C})_{\gam{\rho_0,\ldots,\rho_i}{\rho},\gam{\tau,\rho_0,\ldots,\rho_i}{\rho},n}(s_*)_{\gam{\tau,\rho_0,\ldots,\rho_i}{\rho},\sigma,n} \\
&& - \sum_{j=1}^{i}\sum_{\rho_{j-1}<\tau< \rho_j} (d_{Sd\, C})_{\gam{\rho_0,\ldots,\rho_i}{\rho},\gam{\rho_0,\ldots,\rho_{j-1},\tau,\rho_j,\ldots,\rho_i}{\rho},n}(s_*)_{\gam{\rho_0,\ldots,\rho_{j-1},\tau,\rho_j,\ldots,\rho_i}{\rho},\sigma,n} \\
&& - \sum_{\rho_i< \tau \leqslant \sigma} (d_{Sd\, C})_{\gam{\rho_0,\ldots,\rho_i}{\rho},\gam{\rho_0,\ldots,\rho_i,\tau}{\rho},n}(s_*)_{\gam{\rho_0,\ldots,\rho_i,\tau}{\rho},\sigma,n}.
\end{eqnarray*}
\end{enumerate}
If $\rho<\rho_0$ and $\rho_i<\sigma$, then all the $(s_*)$ terms in $(s_*d_C - d_{Sd\, C}s_*)_{\gam{\rho_0,\ldots, \rho_i}{\rho},\sigma,n}$ are $0$. We check the other $3$ cases explicitly summing the terms in the order they were introduced. We will use $(s_*d_C - d_{Sd\, C}s_*)|$ as shorthand for $(s_*d_C - d_{Sd\, C}s_*)_{\gam{\rho_0,\ldots, \rho_i}{\rho},\sigma,n}$.

\noindent \underline{Case 1}: $\rho = \rho_0$ and $\rho_i=\sigma$.
\begin{eqnarray*}
 (s_*d_C - d_{Sd\, C}s_*)| &=& (-1)^{ni}d_{\rho_0\ldots\rho_{i-1}\rho_i^2} - (-1)^{(n+1)i}d_{\rho_0^2\rho_1\ldots\rho_i} \\
&& -0 - \sum_{j=1}^i\sum_{\rho_{j-1}<\tau<\rho_j}(-1)^j(-1)^n(-1)^{(n+1)(i+1)}d_{\rho_0\ldots\rho_{j-1}\tau\rho_j\ldots\rho_i} -0.
\end{eqnarray*}
This sum is zero because by applying equation $(\ref{slidealong})$ $i-1$ times we see that
\begin{eqnarray*}
 d_{\rho_0^2\rho_1\ldots\rho_i} &=& - \sum_{\rho_0<\tau<\rho_1}d_{\rho_0\tau\rho_1\ldots\rho_i} - d_{\rho_0\rho_1^2\rho_2\ldots\rho_i} \\
&=& - \sum_{\rho_0<\tau<\rho_1}d_{\rho_0\tau\rho_1\ldots\rho_i} + \sum_{\rho_1<\tau<\rho_2}d_{\rho_0\rho_1\tau\rho_2\ldots\rho_i} +    d_{\rho_0\rho_1\rho_2^2\rho_3\ldots\rho_i} \\
&=& \ldots \\
&=& \sum_{j=1}^i\sum_{\rho_{j-1}<\tau<\rho_j}(-1)^jd_{\rho_0\ldots\rho_{j-1}\tau\rho_j\ldots\rho_i} + (-1)^id_{\rho_0\ldots \rho_{i-1}\rho_i^2}.
\end{eqnarray*}
Substituting this in we obtain
\begin{eqnarray*}
 (s_*d_C - d_{Sd\, C}s_*)| &=& \{(-1)^{ni} - (-1)^{(n+1)i}(-1)^i\}d_{\rho_0\ldots\rho_i^2} \\
&& - \sum_{j=1}^i\sum_{\rho_{j-1}<\tau<\rho_j} \{(-1)^{j+n+(n+1)(i+1)} + (-1)^{(n+1)i+j}\}d_{\rho_0\ldots\rho_{j-1}\tau\rho_j\ldots\rho_i} \\ &=& 0. 
\end{eqnarray*}
\underline{Case 2}: $\rho = \rho_0$ and $\rho_i < \sigma$.
\begin{eqnarray*}
 (s_*d_C - d_{Sd\, C}s_*)| &=& 0 - (-1)^{(n+1)i}d_{\rho_0^2\rho_1\ldots\rho_i} \\
 && - (-1)^0(-1)^n(-1)^{(n+1)(i+1)}d_{\rho_0^2\rho_1\ldots\rho_i} \\
 && -0 -0 \\
 &=& 0.
\end{eqnarray*}
\underline{Case 3}: $\rho < \rho_0$ and $\rho_i=\sigma$.
\begin{eqnarray*}
 (s_*d_C - d_{Sd\, C}s_*)| &=& (-1)^{ni}d_{\rho_0\ldots\rho_i\sigma} - 0 \\
 && -0 -0 - (-1)^{i+1}(-1)^n(-1)^{(n+1)(i+1)}d_{\rho_0\ldots\rho_i\sigma} \\
 &=& 0.
\end{eqnarray*}
\end{proof}

\begin{rmk}
We need to be a little bit careful in the above proof in the case that $\gam{\rho_0,\ldots,\rho_i}{\rho}$ is not PL-homeomorphic to $\mathring{\rho}\times\Delta^i$, i.e.\ if $\gam{\rho_0,\ldots, \rho_i}{\rho}=\gam{\rho_0,\ldots, \rho_{i-1}}{\rho}$ for example. In this case the component of $(s_*)$ mapping to $\gam{\rho_0,\ldots, \rho_i}{\rho}$ is zero as is the component of $d_{Sd\, C}$ mapping to it. A case by case analysis verifies these pathological cases.
\qed\end{rmk}

We now extend the local chain maps $\{r_\sigma\}_{\sigma\in X}$ to a global chain map $r:Sd\, C \to C$. This is considerably easier than extending the $\{s_\sigma\}_{\sigma\in X}$ because defining $r_*$ with $r_\sigma$ down the diagonal is already a chain map without needing to introduce any off-diagonal terms like we had to with the $\{s_\sigma\}_{\sigma\in X}$.

\begin{prop}\label{partsofr}
Define $r_*:Sd\, C \to C$ by \[ (r_*)_{\tau,\widetilde{\sigma},n}:= \brcc{((r_\sigma)_*)_n: Sd\, C_n(\widetilde{\sigma}) \to C_n(\tau),}{\widetilde{\sigma}\subset I_\tau}{0,}{\mathrm{otherwise}}\] then $r_*: Sd\, C \to C$ is a chain map.
\end{prop}

\begin{proof}
We already know each $r_\sigma$ is a local chain map, so we just need to check that the components of $(r_*d_{Sd\, C} - d_C r_*)$ from $I_\sigma$ to $\tau \lneq \sigma$ are zero.

Suppose $\tau\lneq\sigma$ and $\widetilde{\sigma}\in I_\sigma$. Then $r_\sigma$ is only non-zero on the $|\sigma|$-simplices in $I_\sigma$ where it is the map $1$ with respect to the chosen orientations. Let $\widetilde{\sigma}$ be such a simplex, then \[ (r_*d_{Sd\, C} - d_C r_*)|_{\tau,\widetilde{\sigma}}= 1.d_{\tau\sigma} - d_{\tau\sigma}.1 =0.\]
Thus $r_*:Sd\, C\to C$ is a chain map. 
\end{proof}

The open subsets $I_\tau$ satisfy condition $(\ref{five})$ for all $\tau$, thus we are able to define a functor by assembling each $I_\tau$:
\begin{defn}\label{rrr}
Given a subdivision chain equivalence
\begin{displaymath}
\xymatrix@1{ (\Delta^{lf}_*(X), d_{\Delta^{lf}_*(X)},0) \ar@<0.5ex>[r]^-{s} & (\Delta^{lf}_*(Sd\, X), d_{\Delta^{lf}_*(Sd\, X)},P) \ar@<0.5ex>[l]^-{r}
}
\end{displaymath}
define an assembly functor \[\rr_r:\A^*(Sd\, X)\to \A^*(X)\] by \[\rr_r(C)(\sigma):= C[I_\sigma].\]Just like the assembly functors $\rr$ and $\ttt$ this is clearly functorial.
\qed\end{defn}
We have now done all the hard work and can easily prove Theorem \ref{algsubdiv}:
\begin{proof}[Proof of Theorem \ref{algsubdiv}]
The maps $s_*$ and $r_*$ defined as above are chain maps, such that for all $\sigma\in X$, \begin{displaymath}
\xymatrix@1{ (C(\sigma), (d_C)_{\sigma,\sigma},0) \ar@<0.5ex>[r]^-{s_*|=s_\sigma} & (\rr_r Sd_r\, C, d_{\rr_r Sd_r\, C}, (P_\sigma)_*) \ar@<0.5ex>[l]^-{r_*|=r_\sigma}
}
\end{displaymath}
is a chain equivalence in $\A$. Therefore by Proposition \ref{chcont}, \begin{displaymath}
\xymatrix@1{ (C, d_C,0) \ar@<0.5ex>[r]^-{s_*} & (\rr_r Sd_r\, C, d_{\rr_r Sd_r\, C},P_*) \ar@<0.5ex>[l]^-{r_*}
}
\end{displaymath}
is a chain equivalence in $\A^*(X)$ as required.
\end{proof}
Now we deal with the case $C\in \A_*(X)$:
\begin{thm}\label{Dualsubdivision}
The dual chain equivalence 
\begin{displaymath}
\xymatrix{ (\Delta^{-*}(X), \delta^{\Delta^{-*}(X)},0) \ar@<0.5ex>[rr]^-{r^*} && (\Delta^{-*}(Sd^i\,X), d^{\Delta^{-*}(Sd^i\,X)}, P^*) \ar@<0.5ex>[ll]^-{s^*} }
\end{displaymath}
induces a subdivision chain equivalence
\begin{displaymath}
\xymatrix@1{ (C, d_C,0) \ar@<0.5ex>[r]^-{(r^*)_*} & (Sd_r\, C, d_{Sd_r\, C},(P^*)_*) \ar@<0.5ex>[l]^-{(s^*)_*}
}
\end{displaymath}
such that
\[\rr_r Sd_r\, C \simeq C \in \A_*(X)\] for all $C\in \A_*(X)$.
\end{thm}

\begin{proof}
This proof is dual to the proof for $\A^*$. The global chain equivalence given by $s^*$, $r^*$ and $P^*$ restricts to a local chain equivalence for all $\sigma\in X$:
\begin{displaymath}
\xymatrix@1{ (\Delta^{-*}(\mathring{\sigma}), 0,0) \ar@<0.5ex>[rr]^-{r^*_\sigma:=r^*|} && (\Delta^{-*}(I_\sigma), \delta^{\Delta^{-*}(Sd\, X)}|_{I_\sigma},P^*_\sigma:=P^*|_{I_\sigma}). \ar@<0.5ex>[ll]^-{s^*_\sigma:= s^*|}
}
\end{displaymath}
This induces a local chain equivalence 
\begin{displaymath}
\xymatrix@C=3cm{ (C(\sigma), (d_C)_{\sigma,\sigma} ,0) \ar@<0.5ex>[r]^-{(r^*_\sigma)_*= \id_{C(\sigma)}\otimes \Sigma^{|\sigma|}r^*_\sigma} & (Sd_r\, C|_{I_\sigma}, d_{Sd_r\, C}|_{I_\sigma},(P^*_\sigma)_*) \ar@<0.5ex>[l]^-{(s^*_\sigma)_*= \id_{C(\sigma)}\otimes \Sigma^{|\sigma|}s^*_\sigma}
}
\end{displaymath}
viewing $C(\sigma)$ as $C(\sigma)\otimes \Z = C(\sigma)\otimes\Sigma^{|\sigma|}\Delta^{-*}(\mathring{\sigma})$ and noting that $Sd_r\, C|_{I_\sigma}=C(\sigma)\otimes \Sigma^{|\sigma|}\Delta^{-*}(I_\sigma)$.

Define $(r^*)_*:C \to Sd\, C$ by \[ ((r^*)_*)_{\widetilde{\sigma},\tau,n}:= \brcc{((r^*_\sigma)_*)_n: C_n(\tau) \to Sd\, C_n(\widetilde{\sigma}),}{\widetilde{\sigma}\subset I_\tau,|\widetilde{\sigma}|=|\tau|,}{0,}{\mathrm{otherwise}}\] then $(r^*)_*: C \to Sd\, C$ is a chain map.

Set \[((s^*)_*)_{\sigma,\gam{\rho_0,\ldots,\rho_i}{\rho},n}:= \left\{ \begin{array}{cc} \id_{C(\sigma)}, & \rho=\rho_0=\rho_i=\sigma, \\ (-1)^{(n+1)i}d_{\rho_i\ldots\rho_1\rho_0},& \rho=\rho_0\neq\rho_i=\sigma\ \\ 0, & \mathrm{otherwise.} \end{array} \right. \] This defines a chain map $(s^*)_*: Sd\, C\to C$ which extends the local $\{s^*_\sigma\}_{\sigma\in X}$'s. 

Therefore by Proposition \ref{chcont}, \begin{displaymath}
\xymatrix@1{ (C, d_C,0) \ar@<0.5ex>[r]^-{(r^*)_*} & (\rr_r Sd_r\, C, d_{\rr_r Sd_r\, C},(P^*)_*) \ar@<0.5ex>[l]^-{(s^*)_*}
}
\end{displaymath}
is a chain equivalence in $\A_*(X)$ as required.
\end{proof}

\section{Examples of chain equivalences}
In this section are a couple of low dimensional worked examples of how the global definition of $s_*$ is a chain equivalence.
\begin{ex}
Consider $X$ the closed $1$-simplex $\sigma=v_0v_1$ oriented as in Example \ref{Ex:simpleexofsubdivonchains} with chain inverse $r$ also given as in Example \ref{Ex:simpleexofsubdivonchains}. In the following diagrams the labelling of each simplex $\widetilde{\tau}$ will denote the map from $C(\sigma)_n$ to the chain complex $Sd\, C_n(\widetilde{\tau})$ or $Sd\, C_{n-1}(\widetilde{\tau})$. Following the definition of $s_*$ in Proposition \ref{Prop:sstardefined} the components of $s_*$ from $C(\sigma)$ to the various simplices in $Sd\, X$ are as in \Figref{Fig:sstaronesimplex}.
\begin{figure}[ht]
\begin{center}
{
\psfrag{p0p0}[][]{$\Gamma_{v_0}(v_0)$}
\psfrag{p0tp0}[][]{$\Gamma_{v_0,\sigma}(v_0)$}
\psfrag{tp0}[][]{$\Gamma_{\sigma}(v_0)$}
\psfrag{tt}[][]{$\Gamma_{\sigma}(\sigma)$}
\psfrag{tp1}{$\Gamma_{\sigma}(v_1)= \Gamma_{v_1}(v_1)$}
\psfrag{dp0t}[][]{$(-1)^{n+1}d_{v_0\sigma}$}
\psfrag{1}[][]{$1$}
\includegraphics[width=7cm]{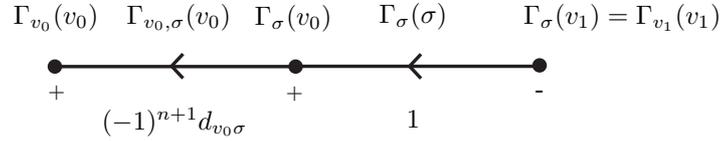}
}
\caption{$(s_*)_n$ for $\sigma=v_0v_1$.}
\label{Fig:sstaronesimplex}
\end{center}
\end{figure}

We verify explicitly that this $s_*$ is a chain map. Applying $d_{Sd\,C}$ to it, we get the map $(d_{Sd\,C}\circ s_*)_n|_{C(\sigma)}$ in Figure \ref{Fig:dSdCsstaronesimplex}.
\begin{figure}[ht]
\begin{center}
{
\psfrag{p0p0}[r][r]{$-(-1)^n(-1)^{n+1}d_{v_0\sigma}$}
\psfrag{p0tp0}{}
\psfrag{tp0}[][]{$d_{v_0\sigma}.1 + (-1)^n(-1)^{n+1}d_{v_0\sigma}$}
\psfrag{tt}{}
\psfrag{tp1}[l][l]{$d_{v_1\sigma}.1$}
\psfrag{dp0t}[][]{$(-1)^{n+1}d_{v_0v_0}d_{v_0\sigma}$}
\psfrag{1}[][]{$d_{\sigma\sigma}.1$}
\includegraphics[width=7cm]{images/prethesis9.eps}
}
\caption{$(d_{Sd\,C}s_*)_n$ for $\sigma=v_0v_1$.}
\label{Fig:dSdCsstaronesimplex}
\end{center}
\end{figure}
Then $(d_C)|_{C(\sigma)}$ is the map in Figure \ref{Fig:dConesimplex}, 
\begin{figure}[ht]
\begin{center}
{
\psfrag{t}[t][t]{$d_{\sigma\sigma}$}
\psfrag{p0}[t][t]{$d_{v_0\sigma}$}
\psfrag{p1}[t][t]{$d_{v_1\sigma}$}
\includegraphics[width=7cm]{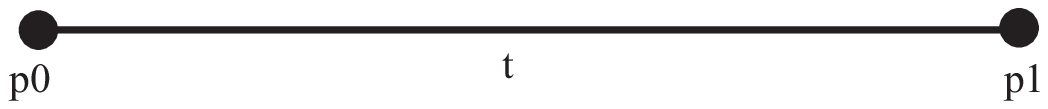}
}
\caption{$d_C$ for $\sigma=v_0v_1$.}
\label{Fig:dConesimplex}
\end{center}
\end{figure}
so that $(s_*d_C)|_{C(\sigma)}$ is the map in \Figref{Fig:sstardConesimplex},
\begin{figure}[ht]
\begin{center}
{
\psfrag{p0p0}[r][r]{$1.d_{v_0\sigma}$}
\psfrag{p0tp0}{}
\psfrag{tp0}[][]{$0$}
\psfrag{tt}{}
\psfrag{tp1}[l][l]{$1.d_{v_1\sigma}$}
\psfrag{dp0t}[][]{$(-1)^{n}d_{v_0\sigma}d_{\sigma\sigma}$}
\psfrag{1}[][]{$1.d_{\sigma\sigma}$}
\includegraphics[width=7cm]{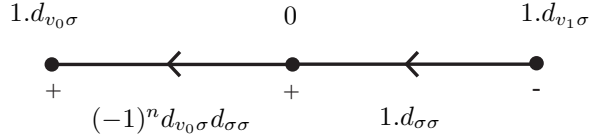}
}
\caption{$(s_*d_C)_n$ for $\sigma=v_0v_1$.}
\label{Fig:sstardConesimplex}
\end{center}
\end{figure}
which we see to be the same as $(d_{Sd\,C}s_*)_n$ using the fact that \[(-1)^{n+1}d_{{v_0}{v_0}}d_{{v_0}\sigma} = -(-1)^nd_{{v_0}\sigma}d_{\sigma\sigma}.\]
\qed\end{ex}

\begin{ex}\label{sstarchainmaptwod}
Let $X$ be the $2$-simplex $\sigma$ oriented as in Example \ref{Ex:SigmaOrients}, and let $r$ be the chain inverse as in Example \ref{Ex:lesssimpleexofsubdivonchains}, so that $Sd\, X$ decomposes into homotopies $\Gamma$ as in \Figref{Fig:constructgamhigher}, then $(s_*)_n|: C(\sigma)_n \to Sd\,C_n(Sd\, X)$ is as labelled in \Figref{Fig:sstartwosimplex}.
\begin{figure}[ht]
\begin{center}
{
\psfrag{p0p0}[tr][tr]{}
\psfrag{p0t1}[r][r]{}
\psfrag{p0t1s}[][]{$d_{\rho_0\tau_0\sigma}$}
\psfrag{p0s}[r][r]{$(-1)^{n+1}d_{\rho_0\sigma}$}
\psfrag{p0t3s}[][]{$d_{\rho_0\tau_2\sigma}$}
\psfrag{p0t3}[t][t]{}
\psfrag{0}[][]{}
\psfrag{t3t3s}[][]{$(-1)^{n+1}d_{\tau_2\sigma}$}
\psfrag{t3s}[t][t]{}
\psfrag{p2p2}[tl][tl]{}
\psfrag{t2s}[l][l]{}
\psfrag{ss}[][]{$\id_{C(\sigma)}$}
\psfrag{t1t1s}[r][r]{\small{$(-1)^{n+1}d_{\tau_0\sigma}$}}
\psfrag{t1t1s2}[l][l]{\small{$(-1)^{n+1}d_{\tau_0\sigma}$}}
\psfrag{p1s}[bl][bl]{$(-1)^{n+1}d_{\rho_1\sigma}$}
\psfrag{p1p1}[b][b]{}
\psfrag{t1s}[br][br]{}
\includegraphics[width=10cm]{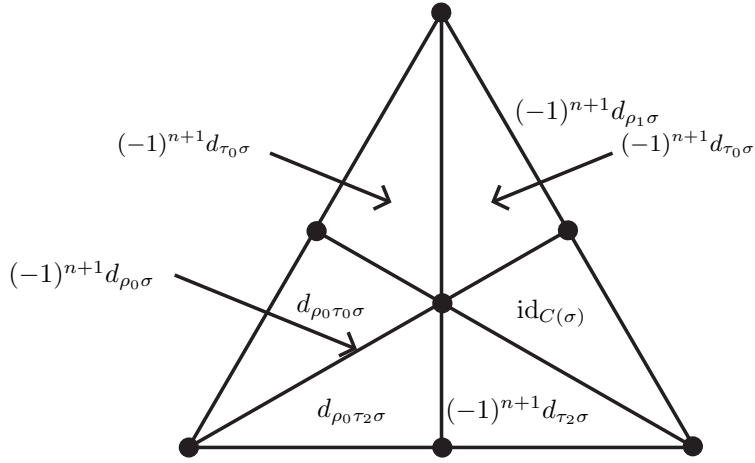}
}
\caption{$(s_*)_n$ for the $2$-simplex $\sigma$.}
\label{Fig:sstartwosimplex}
\end{center}
\end{figure}
Applying $d_{Sd\, C}$ we obtain $(d_{Sd\, C}s_*)_n$ as in \Figref{Fig:dSdCsstartwosimplex}.
\begin{figure}[ht]
\begin{center}
{
\psfrag{p0p0}[tr][tr]{$-(-1)^n(-1)^{n+1}d_{\rho_0\sigma}$}
\psfrag{p0t1}[r][r]{$(-1)^nd_{\rho_0\tau_0\sigma}$}
\psfrag{p0t1s}[][]{$d_{\rho_0^2\tau_0\sigma}$}
\psfrag{p0s}[r][r]{{\small{\shortstack{$(-1)^{n+1}d_{\rho_0^2\sigma} - (-1)^nd_{\rho_0\tau_0\sigma}$\\ $- (-1)^nd_{\rho_0\tau_2\sigma}$}}}}
\psfrag{p0t3s}[][]{$d_{\rho_0^2\tau_2\sigma}$}
\psfrag{p0t3}[t][t]{$(-1)^nd_{\rho_0\tau_2\sigma}$}
\psfrag{0}[][]{}
\psfrag{t3t3s}[][]{$(-1)^{n+1}d_{\tau_2^2\sigma}$}
\psfrag{t3s}[t][t]{$-(-1)^n(-1)^{n+1}d_{\tau_2\sigma}$}
\psfrag{p2p2}[tl][tl]{$d_{\rho_2\sigma}.1$}
\psfrag{t2s}[l][l]{$d_{\tau_1\sigma}.1$}
\psfrag{ss}[][]{$d_{\sigma\sigma}.1$}
\psfrag{t1t1s}[r][r]{$-(-1)^n(-1)^{n+1}d_{\tau_0\sigma}$}
\psfrag{t1t1s2}[l][l]{$-(-1)^n(-1)^{n+1}d_{\tau_0\sigma}$}
\psfrag{p1s}[bl][bl]{$(-1)^{n+1}d_{\rho_1^2\sigma}+(-1)^{n+1}d_{\rho_1\tau_0\sigma}$}
\psfrag{p1p1}[b][b]{$-(-1)^n(-1)^{n+1}d_{\rho_1\sigma}$}
\psfrag{t1s}[br][br]{$(-1)^{n+1}d_{\rho_1\tau_0\sigma}$}
\includegraphics[width=13cm]{images/NextGenn5.eps}
}
\caption{$(d_{Sd\, C}s_*)_n$ for the $2$-simplex $\sigma$.}
\label{Fig:dSdCsstartwosimplex}
\end{center}
\end{figure}
The map $d_C|:C(\sigma)_n \to C_{n-1}(X)$ is as in \Figref{Fig:dCtwosimplex},
\begin{figure}[ht]
\begin{center}
{
\psfrag{p0p0}[tr][tr]{$d_{\rho_0\sigma}$}
\psfrag{t1}[r][r]{$d_{\tau_0\sigma}$}
\psfrag{ss}[][]{$d_{\sigma\sigma}$}
\psfrag{t3}[t][t]{$d_{\tau_2\sigma}$}
\psfrag{p2p2}[tl][tl]{$d_{\rho_2\sigma}$}
\psfrag{t2}[l][l]{$d_{\tau_1\sigma}$}
\psfrag{p1p1}[b][b]{$d_{\rho_1\sigma}$}
\includegraphics[width=5cm]{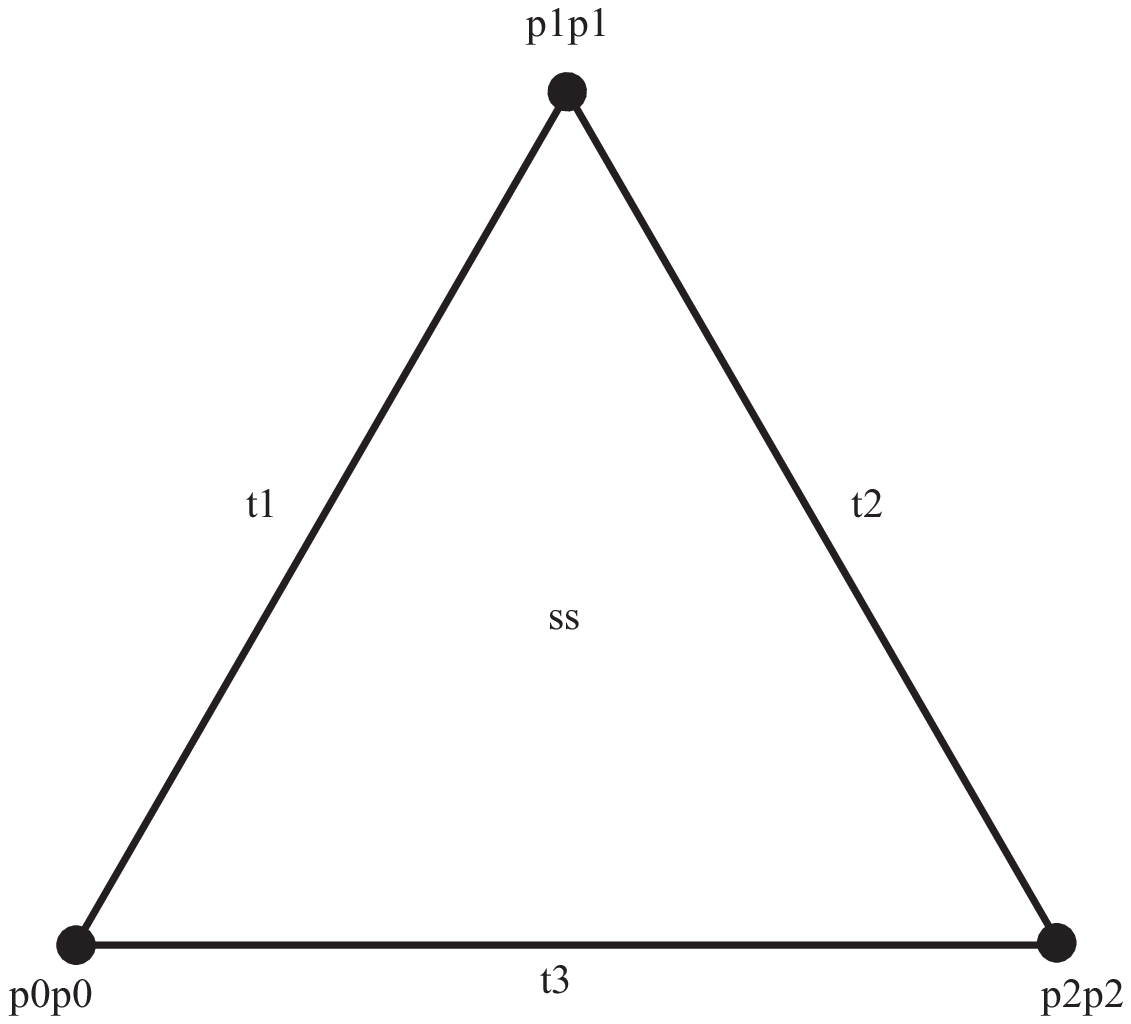}
}
\caption{$(d_C)_n$ for the $2$-simplex $\sigma$.}
\label{Fig:dCtwosimplex}
\end{center}
\end{figure}
thus $(s_*d_C)_n$ is as in \Figref{Fig:sstardCtwosimplex},
\begin{figure}[ht]
\begin{center}
{
\psfrag{p0p0}[tr][tr]{$d_{\rho_0\sigma}$}
\psfrag{p0t1}[r][r]{$(-1)^nd_{\rho_0\tau_0\sigma}$}
\psfrag{p0t1s}[][]{$d_{\rho_0\tau_0\sigma^2}$}
\psfrag{p0s}[r][r]{$(-1)^nd_{\rho_0\sigma^2}$}
\psfrag{p0t3s}[][]{$d_{\rho_0\tau_2\sigma^2}$}
\psfrag{p0t3}[t][t]{$(-1)^nd_{\rho_0\tau_2\sigma}$}
\psfrag{0}[][]{}
\psfrag{t3t3s}[][]{$(-1)^nd_{\tau_2\sigma^2}$}
\psfrag{t3s}[t][t]{$1.d_{\tau_2\sigma}$}
\psfrag{p2p2}[tl][tl]{$d_{\rho_2\sigma}$}
\psfrag{t2s}[l][l]{$1.d_{\tau_1\sigma}$}
\psfrag{ss}[][]{$1.d_{\sigma\sigma}$}
\psfrag{t1t1s}[][]{$(-1)^nd_{\tau_0\sigma^2}$}
\psfrag{t1t1s2}[][]{$(-1)^nd_{\tau_0\sigma^2}$}
\psfrag{p1s}[bl][bl]{$(-1)^n(d_{\rho_1\sigma^2}+d_{\rho_1\tau_1\sigma})$}
\psfrag{p1p1}[b][b]{$d_{\rho_1\sigma}$}
\psfrag{t1s}[br][br]{$1.d_{\tau_0\sigma}$}
\includegraphics[width=13cm]{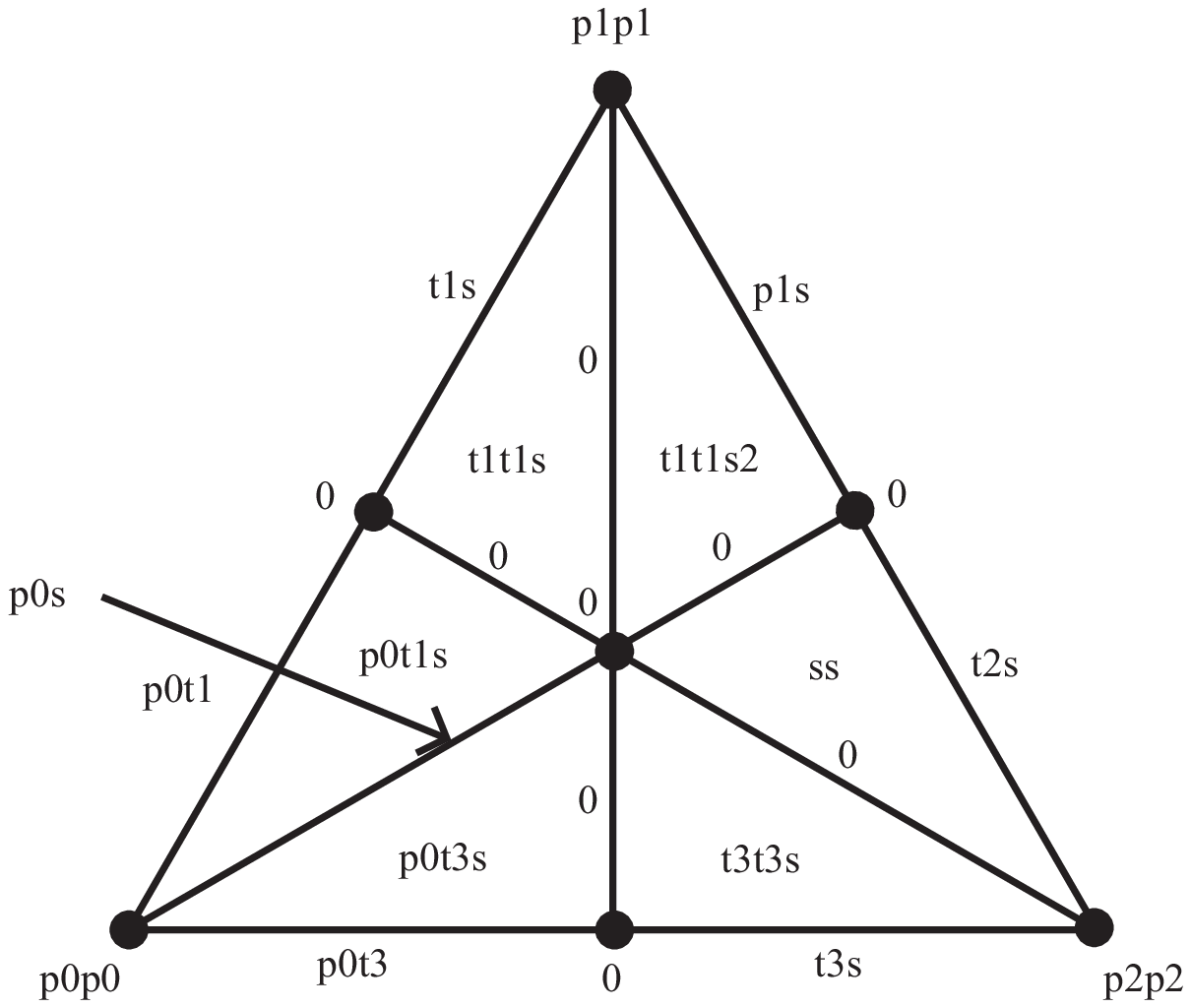}
}
\caption{$(s_*d_C)_n$ for the $2$-simplex $\sigma$.}
\label{Fig:sstardCtwosimplex}
\end{center}
\end{figure}
which is precisely the same as $(d_{Sd\, C}s_*)_n$ using Remark \ref{Rmk:dsquarediszero} repeatedly.
\qed\end{ex}

\section{Consequences of algebraic subdivision}
In this section we note a few properties of the subdivision functors and verify the claim that algebraic subdivision generalises the algebraic effect that barycentric subdivision has on the simplicial chain and cochain complexes.

\begin{lem}\label{Lem:SubcontIsCont}
$C\simeq 0 \in \A(X)$ if and only if $Sd\, C \simeq 0 \in \A(Sd\, X)$. 
\end{lem}
\begin{proof}
This is more or less immediate from Proposition \ref{chcont}. We prove the lemma for $\A^*(X)$. Exactly the same proof holds for $\A_*(X)$.

$(\Rightarrow)$: Since $C\simeq 0 \in \A^*(X)$, for all $\sigma\in X$, there is a local chain contraction $P_\sigma$ such that \[ (d_C)_\sigma (P_C)_\sigma + (P_C)_\sigma (d_C)_\sigma = 1_{C(\sigma)}.\] Now let $Sd\, C = Sd_r\, C$ for any choice of $r$. Consider $(d_{Sd\, C})_{\widetilde{\sigma}}:=(d_{Sd\, C})_{\widetilde{\sigma},\widetilde{\sigma}}$ for some $\widetilde{\sigma}\in I_\sigma$ for arbitrary $\sigma$. Define \[P_{\widetilde{\sigma}}: = (P_C)_{\sigma,{n+|\sigma|-|\widetilde{\sigma}|} } : Sd\, C_n(\widetilde{\sigma}) = C(\sigma)_{n+|\sigma|-|\widetilde{\sigma}|} \to Sd\, C_{n+1}(\widetilde{\sigma}) = C(\sigma)_{n+1+|\sigma|-|\widetilde{\sigma}|}.\] Then this is a local chain contraction for $Sd\, C$, i.e.\
\begin{eqnarray*}
 (d_{Sd\, C})_{\widetilde{\sigma}}P_{\widetilde{\sigma}} + P_{\widetilde{\sigma}}(d_{Sd\, C})_{\widetilde{\sigma}} &=& ((d_C)_{\sigma,n+1+|\sigma|-|\widetilde{\sigma}|}\otimes 1)((P_C)_{\sigma,n+1+|\sigma|-|\widetilde{\sigma}|}\otimes 1) \\ && + ((P_C)_{\sigma,n+|\sigma|-|\widetilde{\sigma}|}\otimes 1)((d_C)_{\sigma,n+|\sigma|-|\widetilde{\sigma}|}\otimes 1) \\ &=& (P_Cd_C + d_CP_C)_{\sigma,n+|\sigma|-|\widetilde{\sigma}|} \\ &=& 1.
\end{eqnarray*}
By Proposition \ref{chcont}, we can extend this to a global chain contraction in $\A^*(Sd\, X)$, and so $Sd\, C\simeq 0 \in \A^*(Sd\, X)$ as required. 

$(\Leftarrow)$: Converse to Proposition \ref{chcont}.
\end{proof}

\begin{lem}\label{Sdcommuteswithcones}
The subdivision functor $Sd\,$ commutes with taking algebraic mapping cones.
\end{lem}
\begin{proof}
As before it suffices to prove the statement for $Sd\,:\A^*(X) \to \A^*(Sd\, X)$. Let $Sd\, C$ denote $Sd_r\, C$ for any choice of $r$. We compute both of the compositions $\C(Sd\, f: Sd\, C \to Sd\, D)$ and $Sd\, \C(f:C\to D)$ using equations $(\ref{conenchain})$ and $(\ref{SdCn})$ and verify that they are exactly the same.
\begin{eqnarray*}
 \C(Sd\, f: Sd\, C \to Sd\, D)_n(\widetilde{\sigma}) &=& Sd\, C_n(\widetilde{\sigma}) \oplus Sd\, D_{n+1}(\widetilde{\sigma}) \\
 &=& C(\sigma)_{n+|\sigma|-|\widetilde{\sigma}|} \oplus D(\sigma)_{n+1+|\sigma|-|\widetilde{\sigma}|} \\
Sd\, \C(f:C\to D)(\widetilde{\sigma}) &=& \C(f:C\to D)(\sigma)_{n+|\sigma|-|\widetilde{\sigma}|}\\
&=& C(\sigma)_{n+|\sigma|-|\widetilde{\sigma}|} \oplus D(\sigma)_{n+|\sigma|-|\widetilde{\sigma}|+1} 
\end{eqnarray*}
so the $n$-chains are exactly the same. Next we compare the boundary maps using equations $(\ref{coneboundarymap})$, $(\ref{dSdconcise})$ and $(\ref{Sdfconcise})$. Suppose $\widetilde{\tau}\leqslant\widetilde{\sigma}$, then
\begin{eqnarray*}
 (d_{\C(Sd\, f: Sd\, C \to Sd\, D)})_{\widetilde{\tau}, \widetilde{\sigma}, n} &=& \matrixcc{(d_{Sd\, C})_{\widetilde{\tau}, \widetilde{\sigma}, n}}{0}{(Sd\, f)_{\widetilde{\tau}, \widetilde{\sigma}, n}}{(d_{Sd\, D})_{\widetilde{\tau}, \widetilde{\sigma}, n+1}} \\
\end{eqnarray*}
\[= \brccc{\matrixcc{(d_C)_{\tau,\sigma, n+|\sigma|-|\widetilde{\sigma}|}}{0}{(f)_{\tau,\sigma, n+|\sigma|-|\widetilde{\sigma}|} }{-(d_D)_{\tau,\sigma, n+1+|\sigma|-|\widetilde{\sigma}|}},}{|\tau|-|\widetilde{\tau}| = |\sigma|-|\widetilde{\sigma}|}{\left(\hspace{-2mm} \begin{array}{cc} (-1)^{n+|\sigma|-|\widetilde{\sigma}|}(d_{\Delta_*(Sd\, X)})_{\widetilde{\tau},\widetilde{\sigma}} & 0 \\ 0 & -(-1)^{n+1+|\sigma|-|\widetilde{\sigma}|}(d_{\Delta_*(Sd\, X)})_{\widetilde{\tau},\widetilde{\sigma}} \end{array}\hspace{-2mm}\right)\hspace{-1mm},}{|\tau|-|\widetilde{\tau}| = |\sigma|-|\widetilde{\sigma}|-1}
\]
\begin{eqnarray*}
 (d_{Sd\, \C(f: C \to D) })_{\widetilde{\tau}, \widetilde{\sigma}, n} &=& (d_{\C(f:C\to D)})_{\tau,\sigma, n+|\sigma|-|\widetilde{\sigma}|}
\end{eqnarray*}
\[= \brcc{\matrixcc{(d_C)_{\tau,\sigma, n+|\sigma|-|\widetilde{\sigma}|}}{0}{(f)_{\tau,\sigma, n+|\sigma|-|\widetilde{\sigma}|} }{-(d_D)_{\tau,\sigma, n+|\sigma|-|\widetilde{\sigma}|+1}},}{|\tau|-|\widetilde{\tau}| = |\sigma|-|\widetilde{\sigma}|}{(-1)^{n+|\sigma|-|\widetilde{\sigma}|}(d_{\Delta_*(Sd\, X)})_{\widetilde{\tau},\widetilde{\sigma}}\matrixcc{1}{0}{0}{1},}{|\tau|-|\widetilde{\tau}| = |\sigma|-|\widetilde{\sigma}|-1}
\]
which is exactly the same as $(d_{\C(Sd\, f: Sd\, C \to Sd\, D)})_{\widetilde{\tau}, \widetilde{\sigma}, n}$. Hence subdivision commutes with algebraic mapping cones.
\end{proof}

\begin{prop}\label{subdivgeneralised}
For any simplicial complex $X$ and any choice of subdivision chain equivalence 
\begin{displaymath}
\xymatrix@1{ (\Delta^{lf}_*(X), d_{\Delta^{lf}_*(X)},0) \ar@<0.5ex>[r]^-{s} & (\Delta^{lf}_*(Sd\, X), d_{\Delta^{lf}_*(Sd\, X)},P) \ar@<0.5ex>[l]^-{r}
}
\end{displaymath}
if we let $C = \Delta^{lf}_*(X)\in\A^*(X)$ be the simplicial chain complex of $X$, then $Sd_r\, C$ and $\Delta^{lf}_*(Sd\, X)$ are chain isomorphic in $\A^*(Sd\, X)$. Similarly for simplicial cochains. Thus for any $r$ our algebraic subdivision functor $Sd_r$ generalises the standard barycentric subdivision operation. 
\end{prop}

\begin{proof}
Recall that \[ C_i(\sigma) = (\Sigma^{|\sigma|}\Z)_i = \brcc{\Z,}{i=|\sigma|,}{0,}{\mathrm{otherwise.}}\] First we observe that for all $\widetilde{\sigma}\in Sd\, X$, $Sd_r\, C_i(\widetilde{\sigma}) \equiv \Delta_i(\mathring{\widetilde{\sigma}}).$ Let $\widetilde{\sigma}\in I_\sigma$, then
\begin{eqnarray*}
 Sd_r\, C_i(\widetilde{\sigma}) &=& (C_*(\sigma) \otimes_\Z \Sigma^{-|\sigma|}\Delta_*(\mathring{\widetilde{\sigma}}))_i \\
&=& ((\Sigma^{|\sigma|}\Z)_* \otimes_\Z (\Sigma^{-|\sigma|}\Delta_*(\mathring{\widetilde{\sigma}})))_i \\
&=& (\Sigma^{|\sigma|}\Sigma^{-|\sigma|}\Delta_*(\mathring{\widetilde{\sigma}}))_i \\
&=& \Delta_i(\mathring{\widetilde{\sigma}}).
\end{eqnarray*}
Having seen that the $i$-chains are identical we see how the boundary maps differ.
\begin{eqnarray*}
 ((d_{Sd_r\, C})|_{I_\sigma})_n  &=& (d_C \otimes d_ {\Sigma^{-|\sigma|}\Delta_*(I_\sigma)})_n \\
 &=& 1 \otimes (-1)^n( d_{\Sigma^{-|\sigma|}\Delta_*(I_\sigma)})_{n-|\sigma|} \\
 &=& (-1)^n( d_{\Delta_*(I_\sigma)})_n
\end{eqnarray*}
so within $I_\sigma$ the boundary maps of $Sd_r\, C$ and $\Delta^{lf}_*(Sd\, X)$ differ by a sign of $(-1)^n$. Between $I_\sigma$ and $I_\tau$ the sign of $d_{\tau\sigma}$ in the boundary map of $Sd_r\, C$ is always $+1$. Thus the boundary maps of the two chain complexes match between distinct $I_\sigma$, $I_\tau$. Bearing all this in mind \[ (-1)^{(n-|\sigma|) + \lfloor \frac{(n-|\sigma|)}{2}\rfloor} : Sd_r\, C_n(\widetilde{\sigma}) \to \Delta^{lf}_n(Sd\, X)(\widetilde{\sigma})\] is a chain isomorphism. The same holds for simplicial cochains. 
\end{proof}

\begin{rmk}
We have only defined a functor $Sd_r$ for \textbf{barycentric} subdivisions $Sd\, X$. By composition we also obtain functors for iterated barycentric subdivisions. Any subdivision $X^\prime$ of $X$ is a subdivision of some iterated barycentric subdivision\footnote{Since a simplex has a unique PL structure.} $Sd^j\, X$, therefore we can get a subdivision functor $\A(X) \to \A(X^\prime)$ by applying the barycentric subdivision functor $j$ times and then assembling covariantly over the subdivision from $X^\prime$ to $Sd^j\, X$. For this reason it is sufficient to consider only the barycentric subdivision.
\qed\end{rmk}
We end the chapter with a summary of the functors defined so far:
\begin{ex}\label{functorsummary}
We can summarise all the subdivision and assembly functors in one diagram:
\begin{displaymath}
\xymatrix@C=20mm{
\A_*(X) \ar@<0.5ex>[r]^-{Sd_{r_1}} & \A_*(Sd\, X) \ar@<0.5ex>[r]^-{Sd_{r_2}} \ar@<0.5ex>[l]^-{\rr_{r_1}} \ar[dl]_(0.7){\ttt} & \A_*(Sd^2\, X) \ar@<0.5ex>[r]^-{Sd_{r_3}} \ar@<0.5ex>[l]^-{\rr_{r_2}} \ar[dl]_(0.7){\ttt} & \ldots \ar@<0.5ex>[l]^{\rr_{r_3}} \ar[dl]_(0.7){\ttt} \\
\A^*(X) \ar@<0.5ex>[r]^-{Sd_{r_1}} & \A^*(Sd\, X) \ar@<0.5ex>[r]^-{Sd_{r_2}} \ar@<0.5ex>[l]^-{\rr_{r_1}} \ar[ul]_(0.3){\ttt} & \A^*(Sd^2\, X) \ar@<0.5ex>[r]^{Sd_{r_3}} \ar@<0.5ex>[l]^-{\rr_{r_2}} \ar[ul]_(0.3){\ttt} & \ldots \ar@<0.5ex>[l]^-{\rr_{r_3}} \ar[ul]_(0.3){\ttt} 
}
\end{displaymath}
In particular we can pass between $\A^*(X)$ and $\A_*(X)$ by applying $\ttt\circ Sd_r$.
\qed\end{ex}

\chapter{Controlled and bounded algebra}\label{chapnine}
Given an $\ep$-controlled map $f: (X,p) \to (Y, \id_Y)$ measured in $Y$, we saw at the end of Section \ref{subsection5pt1} that the induced map on simplicial chains can be thought of as a morphism in $\mathbb{G}_Y(\F(\Z))$. Since $f$ is $\ep$-controlled we have that $f_{\tau,\sigma}:\Delta_*(\sigma) \to \Delta_*(\tau)$ is zero if $d(\widehat{\tau},\widehat{\sigma})>\ep$. 

This motivates the generalisation of controlled topology to controlled algebra. The bounded categories $\CC_{Y}(\A)$ of Pedersen and Weibel are the natural generalisations and can be used to characterise when a map is a bounded homotopy equivalence.

In the case of a finite-dimensional locally finite simplicial complex $X$ it will turn out that the simplicial categories $\A(X)$ of Ranicki precisely capture algebraically the notion of $\ep$-control for all $\ep>0$. 

\begin{defn}\label{bdf2}
Let $(X,p)$ be a finite-dimensional locally finite simplicial complex with control map $p:X\to (M,d)$. Let $f:C\to D$ be a chain map of $X$-graded chain complexes, then we define the \textit{bound} of $f$ by \[\bd(f):= \sup_{f(\tau,\sigma)\neq 0}d(p(\widehat{\sigma}),p(\widehat{\tau})).\] Similarly for a chain homotopy $P:C_*\to C_{*+1}$, \[\bd(P):= \sup_{P(\tau,\sigma)\neq 0}d(p(\widehat{\sigma}),p(\widehat{\tau})).\] If $f:C\to D$ is a chain equivalence with prescribed inverse $g$ and homotopies $P:C_*\to C_{*+1}$, $Q:D_*\to D_{*+1}$, then we define the \textit{bound} of the chain equivalence to be maximum of the bounds of $f,g,P$ and $Q$.
\qed\end{defn}

\begin{rmk}\label{makesmall}
When we measure in $X$ with the identity map as the control map, then the bound of a chain complex (or a chain equivalence) in $\A(X)$ is at most the maximum diameter of any simplex in $X$, i.e.\ $\mesh(X)$. Thus by subdividing we can get a chain complex with control as small as we like that when reassembled is chain equivalent in $\A(X)$ to the one we started with.
\qed\end{rmk}

The following bounded categories are due to Pedersen and Weibel:
\begin{defn}\label{bddcat}
Given a metric space $(X,d)$ and an additive category $\A$, let $\CC_X(\A)$ be the category whose objects are collections $\{M(x)\,|\,x\in X \}$ of objects in $\A$ indexed by $X$ in a locally finite way, written as a direct sum \[M = \sum_{x\in X} {M(x)}\] where $\forall x\in X$, $\forall r>0$, the set $\{ y\in X \,|\, d(x,y)<r\;\mathrm{and}\; M(y)\neq 0 \}$ is finite. A morphism of $\CC_X(\A)$, \[f = \{f_{y,x}\}: L = \sum_{x\in X}L(x) \to M = \sum_{y\in X}M(y) \] is a collection $\{f_{y,x}:L(x) \to M(y)\,|\,x,y\in X\}$ of morphisms in $\A$ such that each morphism has a bound $k=k(f)$, such that if $d(x,y)>k$ then $f_{x,y} = f_{y,x} = 0$.
\qed\end{defn}

\section{Bounded triangulations for the open cone}
A generalisation of Theorem \ref{maintopthm} holds for controlled and bounded algebra. In order to consider algebraic objects over the open cone $O(X^+)$ we need to be able to construct a triangulation of $X\times\R$ that is bounded over $O(X^+)$ from a tame and bounded triangulation of $X$. The idea is to compensate for the scaling in the metric of $X$ as we go up in the cone by subdividing $X\times \R$ more and more at a sufficiently fast rate. In this section we outline precisely how this is done.

Let $X$ be an $(n-1)$-dimensional locally finite simplicial complex. We construct a family of bounded triangulations of the product $X\times \R$ as follows. Let $\{v_i\}_{i\in\Z}$ be the following collection of points in $\R$: \[ v_i = \brcc{i,}{i\leqslant 0,}{\sum_{j=0}^{i}{\left(\frac{n}{n+1} \right)^j},}{i\geqslant 0.}\] These points are chosen so that for all $i\geqslant 0$,
\begin{equation*}
|v_i-v_{i+1}| = \frac{n}{n+1}|v_{i+1}-v_{i+2}|.
\end{equation*}

\begin{defn}
Let $t^0(X\times\R)$ denote the simplicial decomposition of $X\times \R$ that is $Sd^j\,X$ on $X\times\{ v_i\}$ for $j=\max \{0,i\}$ and prisms in between.  
\qed\end{defn}

\begin{ex}\label{Ex:decompositiont0}
Let $X=[0,1]$ so that dim$(X) = 1$ and $n=2$, then \[ \{v_i\}_{i\in\Z} = \{ \ldots,-2, -1,0,1, \frac{5}{2}, \frac{19}{4}, \frac{65}{8}, \ldots \}.\] The decomposition $t^0(X\times\R)$ is then as in the \Figref{Fig:t0XxR}.

\begin{figure}[ht]
\begin{center}
{
\psfrag{t0}{$t^0(X\times\R)$}
\includegraphics[width=2cm]{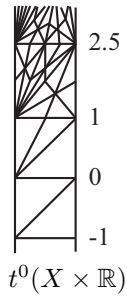}
}
\caption{The simplicial decompositions $t^0(X\times\R)$}
\label{Fig:t0XxR}
\end{center}
\end{figure}
\qed\end{ex}

\begin{rmk}\label{inclusionfunctorproblems}
 Note that we could obtain a triangulation for $O(X^+)$ from the triangulation $t^0(X\times\R)$ by quotienting it by the equivalence relation $\sim$ such that $O(X^+) = X\times\R/\sim$, i.e.\ by $(x,t)\sim (x^\prime,t)$ for all $t\leqslant 0$. This does not yield a locally finite triangulation of $O(X^+)$ if $X$ is not finite because the point $[(x,0)]$ would be contained in infinitely many simplices. Working with $X\times\R$ and measuring in $O(X^+)$ avoids this problem.
\qed\end{rmk}
\begin{rmk}
 The triangulation $t^0(X\times\R)$ is bounded over $O(X^+)$ but it may not be tame over $O(X^+)$ as the radius of simplices may tend to $0$ as we go higher up in the $\R$ direction. The triangulation is however locally tame as the simplex radii cannot approach zero except out towards infinity since $\comesh(X)>0 \Rightarrow \comesh(Sd\, X)>0$.
\qed\end{rmk}

\begin{defn}\label{exptrans} 
Using the $PL$ isomorphism $t^{-1}$ that sends $X\times [v_i,v_{i+1}]$ to $X\times [v_{i-1}, v_i]$ and its inverse we define the simplicial decompositions of $X\times \R$, $t^i(X\times\R)$, to be the images of the decomposition $t^0(X\times\R)$ under the map $t^i$. We call the map $t^i$ \textit{exponential translation by $i$.}
\qed\end{defn}

\begin{ex}
Continuing on from Example \ref{Ex:decompositiont0}, \Figref{Fig:XtimesR} below demonstrates a single exponential translation.
\begin{figure}[ht]
\begin{center}
{
\psfrag{-1}{$-1$}
\psfrag{0}{$0$}
\psfrag{1}{$1$}
\psfrag{2.5}{$2.5$}
\psfrag{t0}{$t^0(X\times\R)$}
\psfrag{t-1}{$t^{-1}(X\times\R)$}
\psfrag{t1}{$t^1(X\times\R)$}
\includegraphics[width=8cm]{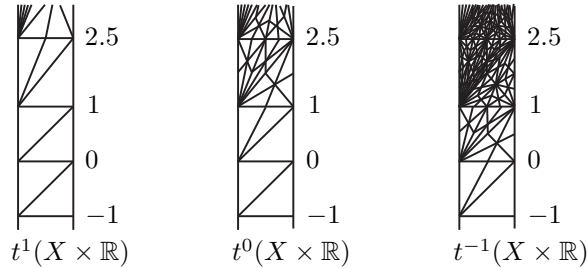}
}
\caption{The simplicial decompositions $t^i(X\times\R)$}
\label{Fig:XtimesR}
\end{center}
\end{figure}
\qed\end{ex}

\begin{rmk}
$t^i(X\times\R)$ is the barycentric subdivision of $t^{i+1}(X\times\R)$ on $X\times[v_i,\infty)$ and exactly the same as $t^{i+1}(X\times\R)$ on $X\times (-\infty, v_{i-1}]$.
\qed\end{rmk}

The exponential translation maps define isomorphisms of categories \[(t^i)_*: \left\{ \begin{array}{c} \A^*(t^j(X\times\R)) \to \A^*(t^{i+j}(X\times\R)) \\ \A_*(t^j(X\times\R)) \to \A_*(t^{i+j}(X\times\R)) \\ \mathbb{G}_{t^j(X\times\R)}(\A) \to \mathbb{G}_{t^{i+j}(X\times\R)}(\A) \end{array}\right. \]
for all $i,j\in \Z$.

The key feature of exponential translation is that it allows us to rescale the bound of an algebraic object in $\A(t^j(X\times\R))$ to be as small as we like, then when it is small enough it can be squeezed.

\begin{rmk}\label{slide}
Given a chain complex $C$ or a chain equivalence $f:C\to D$ in $\A(t^j(X\times\R))$ with bound less than $\ep$, then $t^{-1}(C)$ or $t^{-1}f:t^{-1}C\to t^{-1}D$ has bound less than $\frac{n}{n+1}\ep$ for $X\times[v_j,\infty)$ and less than $\ep$ for $X\times (-\infty,v_j)$ in $\A(t^{j-1}(X\times\R))$
\qed\end{rmk}
 
\section{Functor from controlled to bounded algebra}
We saw before that a chain equivalence $f:X\to Y$ with bound $\ep$ for all $\ep>0$ measured in $Y$ gives us a bounded chain equivalence $f\times\id: X\times\R \to Y\times \R$ measured in $O(Y^+)$. The same is true for algebra; a chain complex $C\in \brc{ch(\A^*(X))}{ch(\A_*(X))}$ provides us with a chain complex in $\brc{ch(\A^*(t^0(X\times\R)))}{ch(\A_*(t^0(X\times\R)))}$ which is bounded when measured in the open cone $O(X^+)$ since $t^0(X\times\R)$ is a bounded triangulation measured in the cone. This is done using a functor \[\textit{``}-\otimes\Z\textit{''}: \brc{ch(\A^*(X))\to ch(\A^*(t^0(X\times\R)))}{ch(\A_*(X)) \to ch(\A_*(t^0(X\times\R)))}\] which algebraically subdivides $C$ sufficiently quickly as we go up $X\times\R$ in the $\R$ direction to compensate for the growth in the metric. We carefully define this functor in this chapter dealing with two cases in parallel:
\begin{enumerate}
 \item $X$ is a finite-dimensional, locally finite simplicial complex.
 \item $X$ is a finite simplicial complex.
\end{enumerate}
We mention finite complexes as the general result we obtain can be packaged nicely in terms of the category $\CC_{O(X^+)}(\A)$ for a finite complex.

\begin{defn}\label{finiteinclusionfunctor}
Let $X$ be a finite simplicial complex. For any bounded and tame triangulation of $X\times\R$ there are inclusion functors
\[ \mathcal{I}:\left\{ \begin{array}{c} \A^*(X\times\R) \\ \A_*(X\times\R)\end{array} \right. \to \CC_{O(X^+)}(\A)\]
defined by associating simplices to the images of their barycentres in $O(X^+)$: 
\begin{eqnarray*}
\mathcal{I}(M)(x) &=& \sum_{\widehat{\sigma}\in j_X^{-1}(x)}M(\sigma)\\
\mathcal{I}(f)_{y,x}:\mathcal{I}(M)(x) \to \mathcal{I}(N)(y) &=& \{f_{\tau,\sigma}:M(\sigma)\to N(\tau)\}_{\widehat{\tau}\in j_X^{-1}(y),\widehat{\sigma}\in j_X^{-1}(x)}.
\end{eqnarray*}
These also induce inclusions 
\[ \mathcal{I}:\left\{ \begin{array}{c} ch(\A^*(X\times\R)) \\ ch(\A_*(X\times\R))\end{array} \right. \to ch(\CC_{O(X^+)}(\A)).\]
Since the $t^i(X\times\R)$ are bounded and locally tame we can consider $\A(t^i(X\times\R))$ as a subcategory of $\CC_{O(X^+)}(\A)$ and similarly for chain complexes.
\qed\end{defn}

\begin{rmk}
Note we do not get an inclusion functor if $X$ is not finite by remark \ref{inclusionfunctorproblems}. The problem is that we get possibly infinite sums associated to points in $X\times (-\infty,0]/\sim.$
\qed\end{rmk}

\begin{defn}\label{algcrosswithR}
Define a functor \[\textit{``}-\otimes\Z\textit{''}: \A(X) \to \A(t^0(X\times\{v_i\}_{i\in\Z}))\subset\A(t^0(X\times\R))\] by sending an object $M$ of $\A(X)$ to the object of $\A(t^0(X\times\{v_i\}_{i\in\Z}))$ that is $Sd^j\,M$ on $X\times\{v_i\}$ for $j=\max\{ 0,i\}$ and by sending a morphism $f:M\to N$ of $\A(X)$ to the morphism of $\A(t^0(X\times\{v_i\}_{i\in\Z}))$ that is $Sd^j\,f:Sd^j\, M \to Sd^j\, N$ on $X\times\{v_i\}$ again for $j=\max\{0,i\}.$ As before this also defines a functor \[\textit{``}-\otimes\Z\textit{''}: ch(\A(X)) \to ch(\A(t^0(X\times\{v_i\}_{i\in\Z}))).\] 
\qed\end{defn}

\begin{rmk}
If in addition $X$ is a finite complex we can compose $\textit{``}-\otimes\Z\textit{''}$ with the inclusion functor $\mathcal{I}$ of Definition \ref{finiteinclusionfunctor} to get a functor which we also call $\textit{``}-\otimes\Z\textit{''}$, \[\textit{``}-\otimes\Z\textit{''}: \A(X) \to \CC_{O(X^+)}(\A).\]
\qed\end{rmk}

\begin{rmk}\label{trivpartofalgmainthm}
Straight from the definitions and Lemma \ref{Lem:SubcontIsCont} we see that 
\[ C\simeq 0 \in \A(X) \Rightarrow \textit{``}C\otimes\Z\textit{''}\simeq 0 \in \A(t^0(X\times\{v_i\}_{i\in\Z})) \subset \A(t^0(X\times\R)).\]
\qed\end{rmk}

\begin{rmk}
If $X$ is finite then \[C\simeq 0 \in \A(X) \Rightarrow \textit{``}C\otimes\Z\textit{''}\simeq 0 \in \CC_{O(X^+)}(\A).\]
\qed\end{rmk}

In the next chapter we will prove that if $\textit{``}C\otimes\Z\textit{''}\simeq 0 \in \mathbb{G}_{t^0(X\times\R)}(\A)$ with finite bound measured in $O(X^+)$, then $C\simeq 0 \in \A(X)$. Thus these two conditions are equivalent. This will be the algebraic analogue of Theorem \ref{maintopthm}. Note how much easier it is to deduce Remark \ref{trivpartofalgmainthm} than it was to prove the corresponding topological versions in Theorem \ref{maintopthm}.

\begin{rmk}
If $X$ is finite then the algebraic analogue of Theorem \ref{maintopthm} is \[C\simeq 0 \in \A(X) \Leftrightarrow \textit{``}C\otimes\Z\textit{''}\simeq 0 \in \CC_{O(X^+)}(\A).\]
\end{rmk}

\chapter{Controlled algebraic Vietoris-like theorem for simplicial complexes}\label{chapten} 
In this chapter we reap all the rewards of our hard work, the culmination of which is to prove the main theorem of this thesis, which is stated so as to be reminiscent of Theorem \ref{maintopthm}: 
\begin{thm}\label{alganal}
Let $X$ be a finite-dimensional locally finite simplicial complex and let $C\in \A(X)$. Then the following are equivalent:
\begin{enumerate}
 \item $C(\sigma)\simeq 0 \in \A$ for all $\sigma \in X$,
 \item $C\simeq 0 \in \A(X)$,
 \item $\textit{``}C\otimes\Z\textit{''}\simeq 0\in\mathbb{G}_{t^0(X\times\R)}(\A)$ with finite bound measured in $O(X^+)$.
\end{enumerate}
\end{thm}

\begin{rmk}
For $X$ a finite complex, we can replace condition $3$ with \[\textit{``}C\otimes\Z\textit{''}
\simeq 0\in \CC_{O(X^+)}(\A).\] This is because $\textit{``}C\otimes\Z\textit{''}\simeq 0\in\mathbb{G}_{t^0(X\times\R)}(\A)$ with finite bound measured in $O(X^+)$ precisely when $\textit{``}C\otimes\Z\textit{''}\simeq 0\in\CC_{O(X^+)}(\A).$
\qed\end{rmk}
This theorem allows us to replace a collection of local conditions on $X$ by a single global condition on $X\times\R$\footnote{For a finite complex $X$ we get a single global condition on $O(X^+)$.} which in general should be less work to verify. 

Comparing Theorems \ref{alganal} and \ref{maintopthm} we see that it was easier to work with algebra than topology; with the algebraic theorem the equivalence of $(1)$ and $(2)$ is immediate from Proposition \ref{chcont} and $(2)\Rightarrow(3)$ is given by the construction of the functor \[\textit{``}-\otimes\Z\textit{''}: \A(X) \to \A(t^0(X\times\R)).\]The difficult implication is $(3)\Rightarrow(2)$. To prove this we use an algebraic squeezing theorem which we hope is of interest in its own right.

\section{Algebraic squeezing}
Given a chain complex $C\in ch(\A(X))$ that is contractible in $\mathbb{G}_X(\A)$ one might ask the question
\begin{q}
When does $C\in ch(\A(X))$ and $C\simeq 0 \in \mathbb{G}_X(\A)$ imply also that $C\simeq 0 \in \A(X)$?
\end{q}
The answer is: when the bound of the contraction $C\simeq 0 \in \mathbb{G}_X(\A)$ is sufficiently small. In fact we can do better than that:
\begin{thm}[Squeezing Theorem]\label{triangulise}
Let $X$ be a finite-dimensional locally finite simplicial complex. There exists an $\ep=\ep(X)>0$ and an integer $i=i(X)$ such that if there exists a chain equivalence $Sd^i\,C \to Sd^i\, D$ in $\mathbb{G}_{Sd^i\, X}(\A)$ with control $<\ep$ measured in $X$ for $C,D\in \A(X)$, then there exists a chain equivalence
\begin{displaymath}
 \xymatrix{ f: C \ar[r]^-{\sim} & D
}
\end{displaymath}
in $\A(X)$ without subdividing.
\end{thm}

\begin{proof}
Let $\ep^\prime(X)$ and $i(X)$ be chosen as in Proposition \ref{setupthesqueeze} and Corollary \ref{setupthedualsqueeze}. Then choose $\ep(X) = \frac{1}{5}\ep^\prime(X)$ so that by Proposition \ref{setupthedualsqueeze} there exist chain equivalences 
\begin{displaymath}
\xymatrix{ (\Delta^{lf}_*(X), d_{\Delta^{lf}_*(X)},0) \ar@<0.5ex>[rr]^-{s} && (\Delta^{lf}_*(Sd^i\,X), d_{\Delta^{lf}_*(Sd^i\,X)}, P) \ar@<0.5ex>[ll]^-{r} \\
 (\Delta^{-*}(X), \delta^{\Delta^{-*}(X)},0) \ar@<0.5ex>[rr]^-{r^*} && (\Delta^{-*}(Sd^i\,X), \delta^{\Delta^{-*}(Sd^i\,X)}, P^*) \ar@<0.5ex>[ll]^-{s^*}
}
\end{displaymath}
such that for all $\sigma\in X$, 
\begin{eqnarray}
 P(\Delta_*(N_{k\ep}(Sd^i \sigma))) &\subset& \Delta_{*+1}(N_{k\ep}(Sd^i \sigma)),\label{eqn1} \\
 r(\Delta_*(N_{k\ep}(Sd^i \sigma))) &\subset& \Delta_{*}(\sigma), \label{eqn2} \\
 r^*(\Delta^{-*}(\mathring{\sigma})) &\subset& \Delta^{-*}(\bigcup_{\tau\geqslant\sigma}(Sd^i\,\tau \backslash N_{5\ep}(\partial\tau\backslash \sigma))), \\
 s^*(\Delta^{-*}(Sd^i\, \mathring{\sigma})) &\subset& \Delta^{-*}(\bigcup_{\tau\geqslant \sigma}\mathring{\tau}), \\ 
 P^*(\Delta^{-*}(Sd^i\,\tau\backslash N_{k\ep}(\partial\tau\backslash\sigma))) &\subset& \Delta^{-*+1}(Sd^i\,\tau\backslash N_{k\ep/2}(\partial\tau\backslash\sigma)), \forall\tau\geqslant\sigma,
\end{eqnarray}
for all $0\leqslant k \leqslant 5$.

First suppose that $C,D\in \A^*(X)$. Using the chain equivalences above and Theorem \ref{algsubdiv} we have induced chain equivalences 
\begin{displaymath}
\xymatrix@1{ (C, d_C,0) \ar@<0.5ex>[r]^-{s_*} & (Sd^i\, C, d_{Sd^i\, C},(P_C)_*), \ar@<0.5ex>[l]^-{r_*} \\
(D, d_D,0) \ar@<0.5ex>[r]^-{s_*} & (Sd^i\, D, d_{Sd^i\, D},(P_D)_*), \ar@<0.5ex>[l]^-{r_*}
}
\end{displaymath}
where for $E=C,D$ the supports of $(P_E)_*$, $s_*$ and $r_*$ satisfy
\begin{eqnarray}
 s_*(E_*|_{\mathring{\sigma}}) &\subset& (Sd^i\,E)_{*}|_{\sigma}, \label{eqn7} \\
 (P_E)_*((Sd^i\,E)_*|_{N_{k\ep}(Sd^i \sigma)}) &\subset& (Sd^i\,E)_{*+1}|_{N_{k\ep}(Sd^i \sigma)}\;\mathrm{and} \label{eqn6} \\
 r_*((Sd^i\,E)_*|_{N_{k\ep}(Sd^i \sigma)}) &\subset& E_{*}|_{\sigma} \label{eqn8}
\end{eqnarray}
for all $0\leqslant k \leqslant 5$. Suppose there exists a chain equivalence 
\begin{displaymath}
 \xymatrix{ (Sd^i\, C,d_{Sd^i\,C},Q_1) \ar@<0.5ex>[r]^{f^i} & (Sd^i\, D,d_{Sd^i\,D},Q_2) \ar@<0.5ex>[l]^{g^i}
}
\end{displaymath}
with control $\ep$. By Lemma \ref{circcheq} composition gives a chain equivalence 
\begin{displaymath}
 \xymatrix{ (C,d_C,r_*(Q_1 + g^i\circ P_D\circ f^i)s_*) \ar@<0.5ex>[rr]^-{r_*\circ f^i \circ s_*} && (D,d_D,r_*(Q_2 + f^i\circ P_C\circ g^i)s_*). \ar@<0.5ex>[ll]^-{r_*\circ g^i \circ s_*}
}
\end{displaymath}
Examining this chain equivalence carefully we observe that it is in fact a chain equivalence in $\A^*(X)$:

Consider $r_*f^is_*$ applied to $C(\sigma)$ which we think of as supported on $\mathring{\sigma}$. By $(\ref{eqn7})$, $s_*(C(\sigma))$ is supported on $Sd^i\,\sigma$. Since $f^i$ has bound $\ep$, we see that $f^is_*(C(\sigma))$ is supported on $N_\ep(Sd^i\,\sigma)$. By $(\ref{eqn8})$, $r_*f^is_*(C(\sigma))$ is supported on $\sigma$, thus $r_*f^is_*$ is a morphism in $\A^*(X)$. For brevity in the following analyses we call this reasoning \textit{arguing by supports} and write
\begin{displaymath}
 \xymatrix{ r_*f^is_*: \mathring{\sigma} \ar[r]^-{(\ref{eqn7})} & Sd^i\, \sigma \ar[r] & N_\ep(Sd^i\, \sigma) \ar[r]^-{(\ref{eqn8})} & \sigma.
}
\end{displaymath}
See \Figref{Fig:Supports} for an example of this argument for a $2$-simplex.
\begin{figure}[h!]
\begin{center}
{
\psfrag{sstar}{$s_*$}
\psfrag{tstar}{$r_*$}
\psfrag{fi}{$f^i$}
\includegraphics[width=11cm]{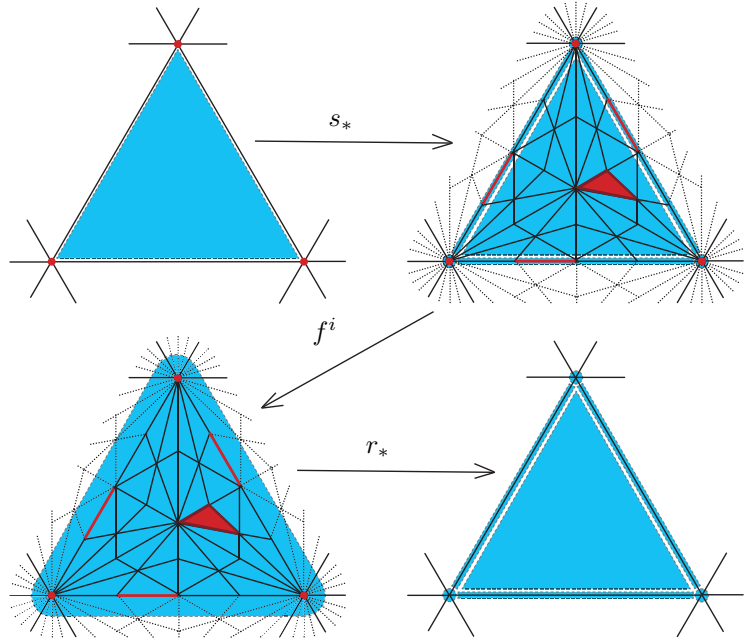}
}
\caption{Arguing by supports to show $r_*f^is_*\in \A^*(X)$.}
\label{Fig:Supports}
\end{center}
\end{figure}

By exactly the same argument we see that $r_*g^is_*$, $r_*Q_1s_*$ and $r_*Q_2s_*$ are all morphisms in $\A^*(X)$, noting that $g^i$, $Q_1$ and $Q_2$ all have bound $\ep$. This just leaves $r_*(g^i(P_D)_*f^i)s_*$ and $r_*(f^i(P_C)_*g^i)s_*$ to check. Arguing by supports both of these send 
\begin{displaymath}
 \xymatrix{ \mathring{\sigma} \ar[r]^-{(\ref{eqn7})} & Sd^i\, \sigma \ar[r] & N_\ep(Sd^i\, \sigma) \ar[r]^-{(\ref{eqn6})} & N_\ep(Sd^i\, \sigma) \ar[r] & N_{2\ep}(Sd^i\, \sigma) \ar[r]^-{(\ref{eqn8})} & \sigma,
}
\end{displaymath}
so they are also morphisms of $\A^*(X)$, thus $r_*f^is_*: C\to D$ is a chain equivalence in $\A^*(X)$ as required.

Next suppose that $C,D\in \A_*(X)$. Using the chain equivalences for the simplicial cochain complexes above and Theorem \ref{Dualsubdivision} we have induced chain equivalences 
\begin{displaymath}
\xymatrix@1{ (C, d_C,0) \ar@<0.5ex>[r]^-{(r^*)_*} & (Sd^i\, C, d_{Sd^i\, C},((P_C)^*)_*) \ar@<0.5ex>[l]^-{(s^*)_*} \\
(D, d_D,0) \ar@<0.5ex>[r]^-{(r^*)_*} & (Sd^i\, D, d_{Sd^i\, D},((P_D)^*)_*) \ar@<0.5ex>[l]^-{(s^*)_*}
}
\end{displaymath}
where for $E=C,D$ the supports of $((P_E)^*)_*$, $(s^*)_*$ and $(r^*)_*$ satisfy
\begin{eqnarray}
 (r^*)_*(E_*|_{\mathring{\sigma}}) &\subset& (Sd^i\,E)_*|_{\bigcup_{\tau\geqslant\sigma}(Sd^i\,\tau \backslash N_{5\ep}(\partial\tau\backslash \sigma))}, \label{eqn11} \\
 (P_E^*)_*((Sd^i\,E)_*|_{Sd^i\,\tau\backslash N_{k\ep}(\partial\tau\backslash\sigma)}) &\subset& (Sd^i\,E)_{*+1}|_{Sd^i\,\tau\backslash N_{k\ep/2}(\partial\tau\backslash\sigma)}, \forall\tau\geqslant\sigma, \label{eqn9} \\
 (s^*)_*((Sd^i\,E)_*|_{Sd^i\, \mathring{\sigma}}) &\subset& E_*|_{\bigcup_{\tau\geqslant \sigma}\mathring{\tau}}, \label{eqn10} 
\end{eqnarray}
for all $0\leqslant k \leqslant 5$. Suppose there exists a chain equivalence 
\begin{displaymath}
 \xymatrix{ (Sd^i\, C,d_{Sd^i\,C},Q_1) \ar@<0.5ex>[r]^{f^i} & (Sd^i\, D,d_{Sd^i\,D},Q_2) \ar@<0.5ex>[l]^{g^i}
}
\end{displaymath}
with control $\ep$. By Lemma \ref{circcheq} composition gives a chain equivalence 
\begin{displaymath}
 \xymatrix{ (C,d_C,(s^*)_*(Q_1 + g^i((P_D)^*)_*f^i)(r^*)_*) \ar@<0.5ex>[rr]^-{(s^*)_*f^i(r^*)_*} && (D,d_D,(s^*)_*(Q_2 + f^i((P_C)^*)_*g^i)(r^*)_*). \ar@<0.5ex>[ll]^-{(s^*)_*g^i(r^*)_*}
}
\end{displaymath}
As before, we examine this carefully and observe that it is a chain equivalence in $\A_*(X)$. All of $(s^*)_*f^i(r^*)_*$, $(s^*)_*g^i(r^*)_*$, $(s^*)_*Q_1(r^*)_*$ and $(s^*)_*Q_2(r^*)_*$ are morphisms in $\A_*(X)$ as they send
\begin{displaymath}
 \xymatrix{ \mathring{\sigma} \ar[r]^-{(\ref{eqn11})} & \bigcup_{\tau\geqslant\sigma}(Sd^i\,\tau \backslash N_{5\ep}(\partial\tau\backslash \sigma)) \ar[r] & \bigcup_{\tau\geqslant\sigma}(Sd^i\,\tau \backslash N_{4\ep}(\partial\tau\backslash \sigma)) \ar[r]^-{(\ref{eqn10})} & \bigcup_{\tau\geqslant \sigma}\mathring{\tau}. 
}
\end{displaymath}
Similarly, both  $(s^*)_*(g^i((P_D)^*)_*f^i)(r^*)_*$ and $(s^*)_*(f^i((P_C)^*)_*g^i)(r^*)_*$ send 
\begin{displaymath}
 \xymatrix@R=3mm@C=7.8mm{ \mathring{\sigma} \ar[r]^-{(\ref{eqn11})} & \bigcup_{\tau\geqslant\sigma}(Sd^i\,\tau \backslash N_{5\ep}(\partial\tau\backslash \sigma)) \ar[r] & \bigcup_{\tau\geqslant\sigma}(Sd^i\,\tau \backslash N_{4\ep}(\partial\tau\backslash \sigma)) \ar[r]^-{(\ref{eqn9})} & \bigcup_{\tau\geqslant\sigma}(Sd^i\,\tau \backslash N_{2\ep}(\partial\tau\backslash \sigma)) \\
&\ar[r] & \bigcup_{\tau\geqslant\sigma}(Sd^i\,\tau \backslash N_{\ep}(\partial\tau\backslash \sigma)) \ar[r]^-{(\ref{eqn10})} & \bigcup_{\tau\geqslant \sigma}\mathring{\tau}, 
}
\end{displaymath}
so are also morphisms in $\A_*(X)$. See \Figref{Fig:Supports2} for an example of this argument for a $1$-simplex.
\begin{figure}[ht]
\begin{center}
{
\psfrag{t}{}
\psfrag{A}[t][t]{$\mathring{\sigma}$}
\psfrag{B}[tl][tl]{$\bigcup_{\tau\geqslant\sigma}(Sd^i\,\tau \backslash N_{5\ep}(\partial\tau\backslash \sigma))$}
\psfrag{C}[t][t]{$\bigcup_{\tau\geqslant\sigma}(Sd^i\,\tau \backslash N_{1\ep}(\partial\tau\backslash \sigma))$}
\psfrag{D}[t][t]{$\bigcup_{\tau\geqslant \sigma}\mathring{\tau}$}
\psfrag{a}[t][t]{$(r^*)_*$}
\psfrag{b}[tl][tl]{$g^i((P_D)^*)_*f^i$}
\psfrag{c}[t][t]{$(s^*)_*$}
\includegraphics[width=9cm]{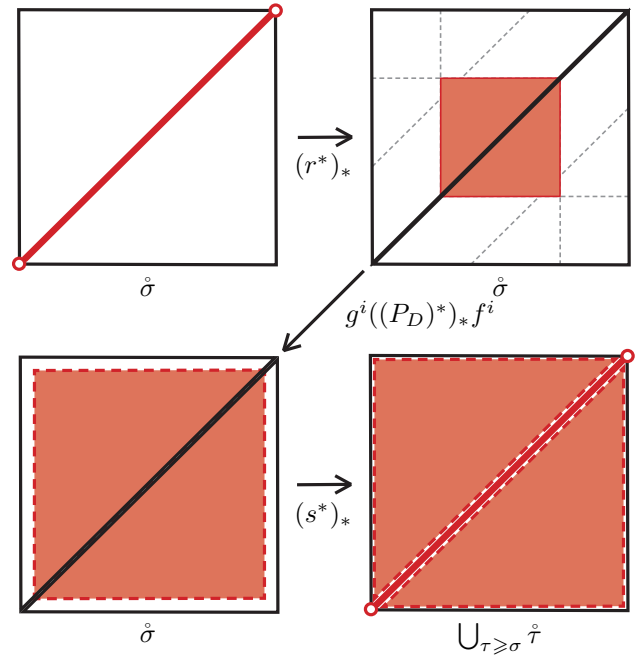}
}
\caption{Arguing by supports to show $(s^*)_*g^i((P_D)^*)_*f^i(r^*)_*\in \A_*(X)$.}
\label{Fig:Supports2}
\end{center}
\end{figure}
Thus $(s^*)_*f^i(r^*)_*:C\to D$ is a chain equivalence in $\A_*(X)$ as required.
\end{proof}

\section{Splitting}
Since $X\times\R$ is a product, exponential translates are of course all $PL$ isomorphic. We would like to exploit this fact on the chain level; a chain complex in $\A(t^0(X\times\R))$ whose exponential translates are all chain equivalent will turn out to be a product $\textit{``}C\otimes\Z\textit{''}$ which we can split to recover $C\in\A(X)$. Since exponential translation by $t^{-1}$ subdivides $t^j(X\times\R)$ into $t^{j-1}(X\times\R)$ we make the following definition:
\begin{defn}\label{exptranseq}
\begin{enumerate}[(i)]
 \item We say that a chain complex $C\in \A(t^j(X\times\R))$ is \textit{exponential translation equivalent} if for all $i\geqslant 0$, \[Sd_i\,C\simeq t^{-i}C \in \A(t^{j-i}(X\times\R)),\]where $Sd_i:\A(t^j(X\times\R))\to\A(t^{j-i}(X\times\R))$ is the subdivision functor obtained by viewing $t^{j-i}(X\times\R)$ as a subdivision of $t^j(X\times\R)$.
 \item We say that $C\in \CC_{O(X^+)}(\A)$ is \textit{exponential translation equivalent} if for all $i\geqslant 0$, \[t^{-i}C\simeq C \in \CC_{O(X^+)}(\A).\]
\end{enumerate} 
\noindent For a finite simplicial complex $X$, we see that these definitions coincide under the inclusion $\mathcal{I}: \A(t^j(X\times\R)) \to \CC_{O(X^+)}(\A)$ since $\mathcal{I}(Sd_i\,C)$ is boundedly chain equivalent to $C$ in $\CC_{O(X^+)}(\A)$ due to all the $t^k(X\times\R)$ being bounded triangulations.
\qed\end{defn}

\begin{rmk}\label{otherformofexptranseq}
Let $C\in \A(t^j(X\times\R))$ be exponential translation equivalent. Then for all $i\geqslant 0$, 
\[ \rr_it^{-i}C \simeq \rr_iSd_i C \simeq C,\]where $\rr_i: \A(t^{j-i}(X\times\R))\to \A(t^j(X\times\R))$ is the covariant assembly functor defined by assembling $t^{j-i}(X\times\R)$ into $t^j(X\times\R)$.
\qed\end{rmk}

\begin{ex}\label{Gexptrans}
For $C$ a chain complex in $ch(\A(X))$, we have that $\textit{``}C\otimes\Z\textit{''}$ is exponential translation equivalent. This is almost a tautology from the fact that by definition \[t^{-i}(t^j(X\times\R)) = Sd_i(t^j(X\times\R))\] and $\textit{``}C\otimes\Z\textit{''}$ is defined to be $Sd^j\, C$ on $t^0(X\times\{v_i\})$ where $j=\max\{0,i\}$.
\qed\end{ex}

\begin{ex}
The simplicial chain and cochain complexes of $t^0(X\times\R)$ are exponential translation equivalent since
\begin{displaymath}
\xymatrix@R=0mm@C=5mm{
Sd_i\Delta^{lf}_*(t^0(X\times\R)) \ar@{}[r]|-{=} & \Delta^{lf}_*(Sd_i\,t^0(X\times\R)) \ar@{}[r]|-{=} & \Delta^{lf}_*(t^{-i}(X\times\R)) \ar@{}[r]|-{=} &t^{-i}\Delta^{lf}_*(t^0(X\times\R)), \\
Sd_i\Delta^{-*}(t^0(X\times\R)) \ar@{}[r]|-{=} & \Delta^{-*}(Sd_i\,t^0(X\times\R)) \ar@{}[r]|-{=} & \Delta^{-*}(t^{-i}(X\times\R)) \ar@{}[r]|-{=} & t^{-i}\Delta^{-*}(t^0(X\times\R)),
}
\end{displaymath}
where the chain equivalences are given by Proposition \ref{chcont} and the fact that for all subdivision functors 
\begin{eqnarray*}
 Sd_i\,\Delta^{lf}_*(X) &=& \Delta^{lf}_*(Sd_i(X)) \\
 Sd_i\,\Delta^{-*}(X) &=& \Delta^{-*}(Sd_i(X))
\end{eqnarray*}
as algebraic subdivision generalises geometric subdivision.
\qed\end{ex}

\begin{thm}[Splitting Theorem]\label{crucsplit}
Let $C,D$ be exponential translation equivalent chain complexes in $\left\{ \begin{array}{c} \A^*(t^0(X\times\R)) \\ \A_*(t^0(X\times\R)). \end{array} \right.$ If there exists a chain equivalence $f:C\to D$ in $\mathbb{G}_{t^0(X\times\R)}(\A)$ with finite bound $0\leqslant B<\infty$ when measured in $O(X^+)$, then for all $i>B$ there exists a chain equivalence \[ f_i: C|_{t^0(X\times\{v_i\})} \to D|_{t^0(X\times\{v_i\})}\] in $\left\{ \begin{array}{c} \A^*(t^0(X\times\{v_i\})) \\ \A_*(t^0(X\times\{v_i\})). \end{array} \right.$ Projecting to $X\times\{1\}$ this can be considered as a chain equivalence in $\left\{ \begin{array}{c} \A^*(Sd^{i}\,X) \\ \A_*(Sd^{i}\,X) \end{array} \right.$ with bound tending to zero as $i\to \infty.$
\end{thm}

\begin{proof}
If $B=0$, just take \[ f_i=f|: C|_{t^0(X\times\{v_i\})} \to D|_{t^0(X\times\{v_i\})}.\]
So suppose $B>0$. For each radius $v_j$ such that $j>B$ we can find an interval $[v_{j_-},v_{j_+}]$ containing $(v_j-B, v_j+B)$ for $v_{j_\pm}\in \{v_i\}_{i\geqslant 0}$. Since $t^0(X\times [v_{j_-},v_{j_+}])$ has a bounded and tame triangulation it satisfies the conditions of the Squeezing Theorem, so there exist
\begin{eqnarray*}
\ep &=& \ep(t^0(X\times [v_{j_-},v_{j_+}]))\\
i &=& i(t^0(X\times [v_{j_-},v_{j_+}]))
\end{eqnarray*}
such that a chain equivalence $Sd^i\, C^\prime \to Sd^i\, D^\prime$ in $\mathbb{G}_{Sd^i\,(t^0(X\times [v_{j_-},v_{j_+}]))}(\A)$ can be squeezed to a chain equivalence in $\brc{\A^*(t^0(X\times [v_{j_-},v_{j_+}]))}{\A_*(t^0(X\times [v_{j_-},v_{j_+}])).}$

By Remark \ref{slide} we can translate $f$ to make the bound as small as we like: there exists a $k\geqslant i$ such that $t^{-k}f: t^{-k}C\to t^{-k}D$ is a chain equivalence in $\mathbb{G}_{t^{-k}(X\times\R)}(\A)$ with bound $<\frac{\ep}{3}$. By exponential translation equivalence we have chain equivalences
\begin{eqnarray*}
 t^{-k}C &\simeq& Sd_kC, \\
 t^{-k}D &\simeq& Sd_kD
\end{eqnarray*}
 in $\A(t^{-k}(X\times\R))$, where $Sd_k$ is the subdivision functor obtained from viewing $t^{-k}(X\times\R)$ as a subdivision of $t^0(X\times\R)$. Recall that $t^{-1}$ is barycentric subdivision on $t^i(X\times[v_i,\infty))$. Therefore, since $v_{j_-}>0$, on the block $t^{-k}(X\times[v_{j_-},v_{j_+}])$ we have that \[t^{-k}(X\times[v_{j_-},v_{j_+}]) = Sd^k\,t^0(X\times[v_{j_-},v_{j_+}]).\] So in particular the functor $Sd_k$ is the same as the $k^{th}$ iterated barycentric subdivision functor $Sd^k$ on this block. Restricted to this block\footnote{Strictly speaking we can choose this to hold for a neighbourhood of the block, then we can ignore what happens away from the block.} composing chain equivalences we get a chain equivalence
\begin{displaymath}
\xymatrix{\widetilde{f}: Sd^kC \ar[r] & t^{-k}C \ar[r]^{t^{-k}f|} & t^{-k}D \ar[r] & Sd^kD 
} 
\end{displaymath}
in $\mathbb{G}_{t^{-k}(X\times[v_{j_-},v_{j_+}])}(\A)$ where each equivalence has bound $\mesh(t^{-k}(X\times[v_{j_-},v_{j_+}]))< \frac{\ep}{3}$. So the combined bound is $<\ep$. We can now apply Theorem \ref{triangulise} to conclude that \[\brc{r_*\widetilde{f}s_*|_{t^0(X\times[v_{j_-},v_{j_+}])}}{s^*\widetilde{f}r^*|_{t^0(X\times[v_{j_-},v_{j_+}])}}:C|_{t^0(X\times[v_{j_-},v_{j_+}])} \to D|_{t^0(X\times[v_{j_-},v_{j_+}])}\] is a chain equivalence in $\brc{\A^*(t^0(X\times[v_{j_-},v_{j_+}]))}{\A_*(t^0(X\times[v_{j_-},v_{j_+}]))}$. By Proposition \ref{chcont} this means that \[\brc{r_*\widetilde{f}s_*|_{t^0(X\times\{v_j\})}}{s^*\widetilde{f}r^*|_{t^0(X\times\{v_j\})}}:C|_{t^0(X\times\{v_j\})} \to D|_{t^0(X\times\{v_j\})}\] is a chain equivalence in $\left\{ \begin{array}{c} \A^*(t^0(X\times\{v_j\})) = \A^*(Sd^j\, X) \\ \A_*(t^0(X\times\{v_j\})) = \A_*(Sd^j\, X) \end{array} \right.$ with bound $\mesh(Sd^{j}\,X)\to 0$ as $j\to\infty$ since $X$ is finite-dimensional. 
\end{proof}
With the previous two theorems the remainder of the proof of Theorem \ref{alganal} is now immediate.

\begin{proof}[Proof of Theorem \ref{alganal}]
$(1)\Leftrightarrow (2)$: Proposition \ref{chcont}.

\noindent $(2)\Rightarrow(3)$: Remark \ref{trivpartofalgmainthm}.

\noindent $(3)\Rightarrow (1)$: We already observed in Example \ref{Gexptrans} that $\textit{``}C\otimes\Z\textit{''}$ is exponential translation equivalent. The $0$ chain complex is trivially exponential translation equivalent, so applying Theorem \ref{crucsplit} to the bounded chain equivalence 
\begin{equation}\label{sta}
\textit{``}C\otimes\Z\textit{''} \to 0 
\end{equation}
given by the chain contraction $\textit{``}C\otimes\Z\textit{''} \simeq 0 \in \mathbb{G}_{t^0(X\times\R)}(\A)$ gives that \[\textit{``}C\otimes\Z\textit{''}[{t^0(X\times\{v_j\})}] \simeq 0 \in \left\{ \begin{array}{c} \A^*(t^0(X\times\{v_j\})) \\ \A_*(t^0(X\times\{v_j\})) \end{array} \right. \] for $j$ sufficiently large. That is to say that \[ Sd^j\, C \simeq 0 \in \left\{ \begin{array}{c} \A^*(Sd^j\, X) \\ \A_*(Sd^j\, X). \end{array} \right. \] By Lemma \ref{Lem:SubcontIsCont} this holds if and only if $C \simeq 0 \in \left\{ \begin{array}{c} \A^*(X) \\ \A_*(X) \end{array} \right.$ as required. 
\end{proof}

\begin{rmk}
It is the exponential translation equivalence that allows us to codimension one split over $X\times\R$, without this things are much more difficult.
\qed\end{rmk}

\section{Consequences for Poincar\'{e} duality}
Both of Theorems \ref{alganal} and \ref{triangulise} have applications to determining when a Poincar\'{e} duality space is a homology manifold.
\begin{defn}\label{PDstuff}
\begin{enumerate}[(i)]
 \item An \textit{$n$-dimensional Poincar\'{e} duality space} $X$ is a finite simplicial complex $X$ with a homology class $[X]\in H^{lf}_n(X)$ such that the cap product chain map \[ [X]\cap -: \Delta^{n-*}(X) \to \Delta^{lf}_*(X)\] is a chain equivalence, i.e.\ if $H^{n-*}(X) \cong H^{lf}_*(X).$
 \item An $n$-dimensional Poincar\'{e} duality space $X$ is said to be a \textit{homology manifold} if the cap product \[ [X]_x \cap -: \Delta^{n-*}(\{x\}) \to \Delta_*(X,X-\{x\})\] is a chain equivalence for each $x$, i.e.\ if $[X]_x \cap -: H^{n-*}(\{x\}) \cong H_*(X,X-\{x\}).$
 \item We will say that $X$ has \textit{$\ep$-controlled Poincar\'{e} duality} if there exists an $i$ such that the chain contraction \[\C([X]\cap - : \Delta^{n-*}(Sd^i\, X) \to \Delta^{lf}_*(Sd^i\, X))\simeq 0 \in \mathbb{G}_{Sd^i\,X}(\A)\] has control $\ep$.
\end{enumerate}
\qed\end{defn}

\begin{lem}\label{whenhommfld}
$X$ is a homology manifold if and only if it has $\ep$-controlled Poincar\'{e} duality for all $\ep>0$.
\end{lem}
\begin{proof}
Theorem 6.13 of \cite{RanSingularities}. See also \cite{QuinnHomMflds} p.271.
\end{proof}

Since a finite simplicial complex necessarily has a bounded and tame triangulation, we can get a Poincar\'{e} duality corollary to the Squeezing Theorem (Theorem \ref{triangulise}):

\begin{thm}[Poincar\'{e} Duality Squeezing]\label{PDsqueezing}
Let $X$ be a finite-dimensional locally finite simplicial complex. There exists an $\ep=\ep(X)>0$ and an integer $i=i(X)$ such that if $Sd^i X$ has $\ep$-controlled Poincar\'{e} duality then, subdividing as necessary, $X$ has $\ep$-controlled Poincar\'{e} duality for all $\ep>0$ and hence is a homology manifold.
\end{thm}

\begin{proof}
Take $\ep(X)$ and $i(X)$ as in the proof of Theorem \ref{triangulise}. Suppose $Sd^{i+1}\, X$ has $\ep$-controlled Poincar\'{e} duality. Then capping with the fundamental class gives a chain equivalence \[\phi_{i+1}: \Delta^{n-*}(Sd^{i+1}\, X) \to \Delta^{lf}_*(Sd^{i+1}\, X)\] in $\mathbb{G}_{Sd^{i+1}\, X}(\F(\Z))$ with control $\ep$. The chain complex $\Delta^{n-*}(Sd^{i+1}\, X)$ is naturally a chain complex in $\A_*(Sd^{i+1}\, X)$ by Example \ref{sup}. By applying the contravariant assembly functor $\ttt$, we can think instead of $\Delta^{n-*}(Sd^{i+1}\, X)$ as a chain complex in $\A^*(Sd^i\, X)$. Similarly, $\Delta^{lf}_*(Sd^{i+1}\, X)$ is naturally a chain complex in $\A^*(Sd^{i+1}\, X)$ and by applying the covariant assembly functor $\rr_r$ we think of it instead as a chain complex in $\A^*(Sd^i\, X)$. By Proposition \ref{subdivgeneralised}, $\phi_{i+1}$ can be thought of as an $(Sd^i\, X)$-based chain equivalence from $Sd^i\,(\Delta^{n-*}(Sd^1\, X))$ to $Sd^i\, (\Delta^{lf}_*(Sd^1\, X))$ for $\Delta^{n-*}(Sd^1\, X)$ and $\Delta^{lf}_*(Sd^1\, X)$ both thought of as chain complexes in $\A^*(X)$. Thus by the Squeezing Theorem there exists an $\A^*(X)$ chain equivalence \[\phi: \Delta^{n-*}(Sd^1\, X) \to \Delta^{lf}_*(Sd^1\, X).\] Now by Remark \ref{makesmall}, subdividing sufficiently yields a Poincar\'{e} duality chain equivalence with bound as small as we like.
\end{proof}

The algebraic Vietoris-like Theorem (Theorem \ref{alganal}) has the following corollary for Poincar\'{e} duality:

\begin{thm}
If $X\times\R$ has bounded Poincar\'{e} duality measured in $O(X^+)$ then $X$ has $\ep$-controlled Poincar\'{e} duality measured in $X$ for all $\ep>0$. This proves the converse of the footnote in \cite{epsurgthy} in the case of simplicial complexes over themselves.
\end{thm}

\begin{proof}
Give $X\times\R$ the simplicial decomposition $Sd^1\, t^0(X\times\R)$, and suppose we have a bounded $Sd^1\,t^0(X\times\R)$-graded Poincar\'{e} duality chain equivalence \[\phi: \Delta^{n+1-*}(Sd^1\, t^0(X\times\R)) \to \Delta^{lf}_*(Sd^1\, t^0(X\times\R)).\] As in Theorem \ref{PDsqueezing}, consider $\Delta^{n+1-*}(Sd^1\, t^0(X\times\R))$ and $\Delta^{lf}_*(Sd^1\, t^0(X\times\R))$ as chain complexes in $\A^*(t^0(X\times\{v_i\}))$ by contravariant assembling and covariantly assembling respectively. This gives a chain equivalence \[\phi^\prime: \ttt \Delta^{n+1-*}(Sd^1\,t^0(X\times\R)) \to \rr Sd\, \Delta^{lf}_*(t^0(X\times\R)),\] but $\rr Sd \simeq \id$ so we write this as \[\phi^\prime: \ttt \Delta^{n+1-*}(Sd^1\,t^0(X\times\R)) \to \Delta^{lf}_*(t^0(X\times\R)).\] This is a bounded $\mathbb{G}_{t^0(X\times\R)}(\A)$ chain equivalence of exponential translation equivalent chain complexes in $\A^*(t^0(X\times\{v_i\}))$, so we may apply the Splitting theorem (Theorem \ref{crucsplit}). This gives us chain equivalences 
\[\phi^\prime_i:\ttt \Delta^{n+1-*}(Sd^1\,t^0(X\times\R))(X\times\{v_i\}) \to \Delta^{lf}_*(t^0(X\times\R))(X\times\{v_i\})\] in $\A^*(t^0(X\times\{v_i\}))=\A^*(Sd^{i}\,X)$ for all $v_i>B$. More explicitly, since the contravariant assembly functor $\ttt$ assembles together open dual cells, these chain equivalences are
\[\phi_i:\Delta^{n+1-*}(\mathring{D}(Sd^{i+1}\,X\times\{v_i\}, Sd^1\,t^0(X\times\R))) \to \Delta^{lf}_*(Sd^i\, X\times\{v_i\}).\] Now, $\mathring{D}(Sd^{i+1}\,X\times\{v_i\}, Sd^1\,t^0(X\times\R))\cong (Sd^{i+1}\, X\times\{v_i\})\times (a,b)$, so in particular we have a chain equivalence \[\Delta^{n-*}(Sd\, X\times\{v_i\}) \cong \Delta^{n+1-*}(\mathring{D}(X\times\{v_i\}, t^0(X\times\R))). \] Composing this with $\phi_i$ we get the desired Poincar\'{e} duality chain equivalences \[\Delta^{n-*}(Sd^{i+1}\, X\times\{v_i\}) \to \Delta^{lf}_*(Sd^i\, X\times\{v_i\}).\]

I.e. we get $\ep$-controlled Poincar\'{e} duality chain equivalences for $X$ for all $\ep>0$. 
\end{proof}

Applying Lemma \ref{whenhommfld} we obtain the following:
\begin{cor}
$X$ is a homology manifold if and only if $X\times\R$ has bounded Poincar\'{e} duality measured in $O(X^+)$.
\end{cor}


%% file: appendices.tex
\appendix
\chapter{Locally finite homology} \label{appendixa}
In this thesis we work with unbounded simplicial complexes. To consider Poincar\'{e} duality for such spaces we need to work with locally finite simplicial chains. The contents of this appendix are taken from \cite{laitinen}. See \cite{RanHughes} for a more detailed account of locally finite homology.

Let $X$ be a finite-dimensional locally finite simplicial complex.
\begin{defn}
An \textit{exhaustion} $(K_i)$ of $X$ is a sequence of compact subcomplexes
\[K_0 \subset K_1 \subset K_2 \subset \ldots \]
such that \[X=\bigcup_{i}K_i\] and $K_i \subset \mathring{K}_{i+1}$ for all $i$.
\qed\end{defn}
Exhaustions always exist. (See \cite{rourkesanderson} proof of Theorem 2.2. page 12)

\begin{defn}\label{locallyfinitechains}
Define the locally finite simplicial chains of $X$ by \[\Delta_*^{lf}(X):= \varprojlim \Delta_*(X, X\backslash \mathring{K}_i)\] and the locally finite homology $H_*^{lf}(X)$ of $X$ by \[H_*^{lf}(X):= H_*(\Delta_*^{lf}(X)).\]
\qed\end{defn}

\begin{ex}
If $S_n$ denotes the set of $n$-simplices of $X$, then \[C_n^{lf}(X;R) = \prod_{S_n}R.\]
\qed\end{ex}

\begin{defn}
We say that $X$ is an \textit{$n$-dimensional Poincar\'{e} duality space} if there exists a homology class $[X]\in H_n^{lf}(X)$, called a fundamental class, such that the cap product chain map \[ [X]\cap -: \Delta^{n-*}(X) \to \Delta_*^{lf}(X)\] is a chain equivalence, i.e. if $H^{n-*}(X) \cong H_*^{lf}(X).$
\qed\end{defn}

\begin{ex}
Let $X$ be a PL manifold, then the fundamental class $[X]$ is the sum of all $n$-simplices of $X$ consistently oriented.
\qed\end{ex}

%% file: thesis.bbl
\newcommand{\etalchar}[1]{$^{#1}$}
\providecommand{\bysame}{\leavevmode\hbox to3em{\hrulefill}\thinspace}
\providecommand{\MR}{\relax\ifhmode\unskip\space\fi MR }
\providecommand{\MRhref}[2]{%
  \href{http://www.ams.org/mathscinet-getitem?mr=#1}{#2}
}
\providecommand{\href}[2]{#2}
\begin{thebibliography}{RCS{\etalchar{+}}96}

\bibitem[Aki72]{Akin72}
E.~Akin, \emph{Transverse cellular mappings of polyhedra}, Trans. Amer. Math.
  Soc. \textbf{169} (1972), 401--438. \MR{0326745 (48 \#5088)}

\bibitem[AM90]{AndMunk}
D.~R. Anderson and H.~J. Munkholm, \emph{Geometric modules and algebraic
  {$K$}-homology theory}, $K$-Theory \textbf{3} (1990), no.~6, 561--602.
  \MR{1071896 (91g:57033)}

\bibitem[Arm69]{trout}
S.~Armentrout, \emph{Cellular decompositions of {$3$}-manifolds that yield
  {$3$}-manifolds}, Bull. Amer. Math. Soc. \textbf{75} (1969), 453--456.
  \MR{0239578 (39 \#935)}

\bibitem[Bar03]{bartelssqueezing}
A.~Bartels, \emph{Squeezing and higher algebraic {$K$}-theory}, $K$-Theory
  \textbf{28} (2003), no.~1, 19--37. \MR{1988817 (2004f:19006)}

\bibitem[Beg50]{Begle1}
E.~G. Begle, \emph{The {V}ietoris mapping theorem for bicompact spaces}, Ann.
  of Math. (2) \textbf{51} (1950), 534--543. \MR{0035015 (11,677b)}

\bibitem[Beg56]{Begle2}
\bysame, \emph{The {V}ietoris mapping theorem for bicompact spaces. {II}},
  Michigan Math. J. \textbf{3} (1955--1956), 179--180. \MR{0082097 (18,497d)}

\bibitem[Bla51]{blankinship}
W.~A. Blankinship, \emph{Generalization of a construction of {A}ntoine}, Ann.
  of Math. (2) \textbf{53} (1951), 276--297. \MR{0040659 (12,730c)}

\bibitem[Bro60]{BrownSchoenflies}
M.~Brown, \emph{A proof of the generalized {S}choenflies theorem}, Bull. Amer.
  Math. Soc. \textbf{66} (1960), 74--76. \MR{0117695 (22 \#8470b)}

\bibitem[Bro61]{BrownMonotone}
\bysame, \emph{The monotone union of open {$n$}-cells is an open {$n$}-cell},
  Proc. Amer. Math. Soc. \textbf{12} (1961), 812--814. \MR{0126835 (23
  \#A4129)}

\bibitem[CD77]{CoramDuvall}
D.~S. Coram and P.~F. Duvall, \emph{Approximate fibrations}, Rocky Mountain J.
  Math. \textbf{7} (1977), no.~2, 275--288. \MR{0442921 (56 \#1296)}

\bibitem[CF79]{chapfer}
T.~A. Chapman and S.~Ferry, \emph{Approximating homotopy equivalences by
  homeomorphisms}, Amer. J. Math. \textbf{101} (1979), no.~3, 583--607.
  \MR{533192 (81f:57007b)}

\bibitem[CH69]{ConnHoll}
E.~H. Connell and J.~Hollingsworth, \emph{Geometric groups and {W}hitehead
  torsion}, Trans. Amer. Math. Soc. \textbf{140} (1969), 161--181. \MR{0242151
  (39 \#3485)}

\bibitem[Cha73]{hilbcube}
T.~A. Chapman, \emph{Cell-like mappings of {H}ilbert cube manifolds:
  applications to simple homotopy theory}, Bull. Amer. Math. Soc. \textbf{79}
  (1973), 1286--1291. \MR{0326740 (48 \#5083)}

\bibitem[Cha83]{Chapcontrol}
\bysame, \emph{Controlled simple homotopy theory and applications}, Lecture
  Notes in Mathematics, vol. 1009, Springer-Verlag, Berlin, 1983. \MR{711363
  (85d:57016)}

\bibitem[Coh67]{cohen}
M.~M. Cohen, \emph{Simplicial structures and transverse cellularity}, Ann. of
  Math. (2) \textbf{85} (1967), 218--245. \MR{0210143 (35 \#1037)}

\bibitem[Edw80]{Ed78}
R.~D. Edwards, \emph{The topology of manifolds and cell-like maps}, Proceedings
  of the {I}nternational {C}ongress of {M}athematicians ({H}elsinki, 1978)
  (Helsinki), Acad. Sci. Fennica, 1980, pp.~111--127. \MR{562601 (81g:57010)}

\bibitem[FA48]{ArtinFox48}
R.~H. Fox and E.~Artin, \emph{Some wild cells and spheres in three-dimensional
  space}, Ann. of Math. (2) \textbf{49} (1948), 979--990. \MR{0027512
  (10,317g)}

\bibitem[Fer79]{Ferryep}
S.~Ferry, \emph{Homotoping {$\varepsilon $}-maps to homeomorphisms}, Amer. J.
  Math. \textbf{101} (1979), no.~3, 567--582. \MR{533191 (81f:57007a)}

\bibitem[Fin68]{finney}
R.~L. Finney, \emph{Uniform limits of compact cell-like maps}, Notices Amer.
  Math. Soc. \textbf{15} (1968), 942, Abstract \#68T-G26.

\bibitem[FP95]{epsurgthy}
S.~Ferry and E.~K. Pedersen, \emph{Epsilon surgery theory}, Novikov
  conjectures, index theorems and rigidity, {V}ol.\ 2 ({O}berwolfach, 1993),
  London Math. Soc. Lecture Note Ser., vol. 227, Cambridge Univ. Press,
  Cambridge, 1995, pp.~167--226. \MR{1388311 (97g:57044)}

\bibitem[Hat02]{hatcher}
A.~Hatcher, \emph{Algebraic topology}, Cambridge University Press, Cambridge,
  2002. \MR{1867354 (2002k:55001)}

\bibitem[HR95]{HR95}
N.~Higson and J.~Roe, \emph{On the coarse {B}aum-{C}onnes conjecture}, Novikov
  conjectures, index theorems and rigidity, {V}ol.\ 2 ({O}berwolfach, 1993),
  London Math. Soc. Lecture Note Ser., vol. 227, Cambridge Univ. Press,
  Cambridge, 1995, pp.~227--254. \MR{1388312 (97f:58127)}

\bibitem[HR96]{RanHughes}
Bruce Hughes and Andrew Ranicki, \emph{Ends of complexes}, Cambridge Tracts in
  Mathematics, vol. 123, Cambridge University Press, Cambridge, 1996.
  \MR{1410261 (98f:57039)}

\bibitem[JRW09]{plmf}
B.~Jahren, J.~Rognes, and F.~Waldhausen, \emph{Spaces of {PL} manifolds and
  categories of simple maps {\rm http://folk.uio.no/rognes/papers/plmf.pdf}}.

\bibitem[Lac68]{LachCellANR}
R.~C. Lacher, \emph{Cell-like mappings of {${\rm ANR}'s$}}, Bull. Amer. Math.
  Soc. \textbf{74} (1968), 933--935. \MR{0244963 (39 \#6276)}

\bibitem[Lac69]{celllikemappings1}
\bysame, \emph{Cell-like mappings. {I}}, Pacific J. Math. \textbf{30} (1969),
  717--731. \MR{0251714 (40 \#4941)}

\bibitem[Lac77]{lach77}
\bysame, \emph{Cell-like mappings and their generalizations}, Bull. Amer. Math.
  Soc. \textbf{83} (1977), no.~4, 495--552. \MR{0645403 (58 \#31095)}

\bibitem[Lai96]{laitinen}
E.~Laitinen, \emph{End homology and duality}, Forum Math. \textbf{8} (1996),
  no.~1, 121--133. \MR{1366538 (97b:55006)}

\bibitem[ML63]{maclane}
S.~Mac~Lane, \emph{Homology}, Die Grundlehren der mathematischen
  Wissenschaften, Bd. 114, Academic Press Inc., Publishers, New York, 1963.
  \MR{0156879 (28 \#122)}

\bibitem[Ped84a]{PedInvariants}
E.~K. Pedersen, \emph{{$K_{-i}$}-invariants of chain complexes}, Topology
  ({L}eningrad, 1982), Lecture Notes in Math., vol. 1060, Springer, Berlin,
  1984, pp.~174--186. \MR{770237 (86g:18008)}

\bibitem[Ped84b]{PedKminusi}
\bysame, \emph{On the {$K\sb{-i}$}-functors}, J. Algebra \textbf{90} (1984),
  no.~2, 461--475. \MR{760023 (85k:18019)}

\bibitem[Pra10]{prassidis}
S.~Prassidis, \emph{Introduction to controlled topology and its applications},
  Cohomology of groups and algebraic {$K$}-theory, Adv. Lect. Math. (ALM),
  vol.~12, Int. Press, Somerville, MA, 2010, pp.~343--385. \MR{2655182}

\bibitem[PW89]{kthyhom}
E.~K. Pedersen and C.~A. Weibel, \emph{{$K$}-theory homology of spaces},
  Algebraic topology (Arcata, CA, 1986), Lecture Notes in Math., vol. 1370,
  Springer, Berlin, 1989, pp.~346--361. \MR{1000388 (90m:55007)}

\bibitem[Qui79]{Quinn1}
F.~Quinn, \emph{Ends of maps. {I}}, Ann. of Math. (2) \textbf{110} (1979),
  no.~2, 275--331. \MR{549490 (82k:57009)}

\bibitem[Qui82]{Quinn2}
\bysame, \emph{Ends of maps. {II}}, Invent. Math. \textbf{68} (1982), no.~3,
  353--424. \MR{669423 (84j:57011)}

\bibitem[Qui83]{QuinnHomMflds}
\bysame, \emph{Resolutions of homology manifolds, and the topological
  characterization of manifolds}, Invent. Math. \textbf{72} (1983), no.~2,
  267--284. \MR{700771 (85b:57023)}

\bibitem[Qui85]{QuinnGeomAlg}
\bysame, \emph{Geometric algebra}, Algebraic and geometric topology ({N}ew
  {B}runswick, {N}.{J}., 1983), Lecture Notes in Math., vol. 1126, Springer,
  Berlin, 1985, pp.~182--198. \MR{802791 (86m:57023)}

\bibitem[Ran92]{bluebk}
A.~Ranicki, \emph{Algebraic {$L$}-theory and topological manifolds}, Cambridge
  Tracts in Mathematics, vol. 102, Cambridge University Press, Cambridge, 1992.
  \MR{1211640 (94i:57051)}

\bibitem[Ran99]{RanSingularities}
\bysame, \emph{Singularities, double points, controlled topology and chain
  duality}, Doc. Math. \textbf{4} (1999), 1--59 (electronic). \MR{1677659
  (2000g:55007)}

\bibitem[RCS{\etalchar{+}}96]{hauptvermutungbook}
A.~Ranicki, A.~J. Casson, D.~P. Sullivan, M.~A. Armstrong, C.~P. Rourke, and
  G.~E. Cooke, \emph{The {H}auptvermutung book}, $K$-Monographs in Mathematics,
  vol.~1, Kluwer Academic Publishers, Dordrecht, 1996, A collection of papers
  of the topology of manifolds. \MR{1434100 (98c:57024)}

\bibitem[RS72]{rourkesanderson}
C.~P. Rourke and B.~Sanderson, \emph{Introduction to piecewise-linear
  topology}, Springer-Verlag, New York, 1972, Ergebnisse der Mathematik und
  ihrer Grenzgebiete, Band 69. \MR{0350744 (50 \#3236)}

\bibitem[RW90]{ranickiweiss}
A.~Ranicki and M.~Weiss, \emph{Chain complexes and assembly}, Math. Z.
  \textbf{204} (1990), no.~2, 157--185. \MR{1055984 (91f:55009)}

\bibitem[Sie72]{SiebCE}
L.~C. Siebenmann, \emph{Approximating cellular maps by homeomorphisms},
  Topology \textbf{11} (1972), 271--294. \MR{0295365 (45 \#4431)}

\bibitem[SS79]{SiebSull}
L.~Siebenmann and D.~Sullivan, \emph{On complexes that are {L}ipschitz
  manifolds}, Geometric topology ({P}roc. {G}eorgia {T}opology {C}onf.,
  {A}thens, {G}a., 1977), Academic Press, New York, 1979, pp.~503--525.
  \MR{537747 (80h:57027)}

\bibitem[Vie27]{Vietmapthm}
L.~Vietoris, \emph{\"{U}ber den h\"oheren {Z}usammenhang kompakter {R}\"aume
  und eine {K}lasse von zusammenhangstreuen {A}bbildungen}, Math. Ann.
  \textbf{97} (1927), no.~1, 454--472. \MR{1512371}

\bibitem[Zee64]{duncehat}
E.~C. Zeeman, \emph{On the dunce hat}, Topology \textbf{2} (1964), 341--358.
  \MR{0156351 (27 \#6275)}

\end{thebibliography}
